\documentclass[letterpaper,11pt,oneside,reqno]{amsart}

\usepackage{bm}
\usepackage{mathrsfs}
\usepackage{amsfonts,amsmath, amssymb,amsthm,amscd,stmaryrd}

\usepackage[mpexclude,DIV13]{typearea}
\usepackage{verbatim}
\usepackage{graphicx}
\usepackage[latin1]{inputenc}
\usepackage{latexsym}
\usepackage{lscape}
\usepackage{epsfig}
\include{bibtex}
\usepackage[colorlinks=true,linkcolor=blue]{hyperref}
\usepackage{setspace}
\usepackage{parskip}
\begin{document}
\bibliographystyle{alpha}
\newcommand{\cn}[1]{\overline{#1}}
\newcommand{\e}[0]{\varepsilon}
\newcommand{\bbf}[0]{\mathbf}
\newcommand{\bX}{{\bf X}}
\newcommand{\Pfree}[5]{\ensuremath{\mathbb{P}^{#1,#2,#3,#4,#5}}}
\newcommand{\PfreeShort}{\ensuremath{\mathbb{P}^{BB}}}

\newcommand{\WH}[8]{\ensuremath{\mathbb{W}^{#1,#2,#3,#4,#5,#6,#7}_{#8}}}
\newcommand{\Wfree}[5]{\ensuremath{\mathbb{W}^{#1,#2,#3,#4,#5}}}
\newcommand{\WHShort}[3]{\ensuremath{\mathbb{W}^{#1,#2}_{#3}}}
\newcommand{\WHShortCouple}[2]{\ensuremath{\mathbb{W}^{#1}_{#2}}}

\newcommand{\walk}[3]{\ensuremath{X^{#1,#2}_{#3}}}
\newcommand{\walkupdated}[3]{\ensuremath{\tilde{X}^{#1,#2}_{#3}}}
\newcommand{\walkfull}[2]{\ensuremath{X^{#1,#2}}}
\newcommand{\walkfullupdated}[2]{\ensuremath{\tilde{X}^{#1,#2}}}

\newcommand{\PH}[8]{\ensuremath{\mathbb{Q}^{#1,#2,#3,#4,#5,#6,#7}_{#8}}}
\newcommand{\PHShort}[1]{\ensuremath{\mathbb{Q}_{#1}}}
\newcommand{\PHExp}[8]{\ensuremath{\mathbb{F}^{#1,#2,#3,#4,#5,#6,#7}_{#8}}}

\newcommand{\D}[8]{\ensuremath{D^{#1,#2,#3,#4,#5,#6,#7}_{#8}}}
\newcommand{\DShort}[1]{\ensuremath{D_{#1}}}
\newcommand{\partfunc}[8]{\ensuremath{Z^{#1,#2,#3,#4,#5,#6,#7}_{#8}}}
\newcommand{\partfuncShort}[1]{\ensuremath{Z_{#1}}}
\newcommand{\bolt}[8]{\ensuremath{W^{#1,#2,#3,#4,#5,#6,#7}_{#8}}}
\newcommand{\boltShort}[1]{\ensuremath{W_{#1}}}
\newcommand{\boltNew}{\ensuremath{W}}
\newcommand{\QTLH}{\ensuremath{\mathfrak{H}}}
\newcommand{\QTLHgen}{\ensuremath{\mathfrak{L}}}

\newcommand{\whitenoise}{\ensuremath{\mathscr{\dot{W}}}}
\newcommand{\mf}{\mathfrak}
\newcommand{\di}{{\rm{disc}}}

\newcommand{\EE}{\ensuremath{\mathbb{E}}}
\newcommand{\PP}{\ensuremath{\mathbb{P}}}
\newcommand{\var}{\textrm{var}}
\newcommand{\N}{\ensuremath{\mathbb{N}}}
\newcommand{\R}{\ensuremath{\mathbb{R}}}
\newcommand{\C}{\ensuremath{\mathbb{C}}}
\newcommand{\Z}{\ensuremath{\mathbb{Z}}}
\newcommand{\Q}{\ensuremath{\mathbb{Q}}}
\newcommand{\T}{\ensuremath{\mathbb{T}}}
\newcommand{\E}[0]{\mathbb{E}}
\newcommand{\OO}[0]{\Omega}
\newcommand{\F}[0]{\mathfrak{F}}
\def \Ai {{\rm Ai}}
\newcommand{\G}[0]{\mathfrak{G}}
\newcommand{\ta}[0]{\theta}
\newcommand{\w}[0]{\omega}
\newcommand{\grad}{\nabla}
\renewcommand{\P}{\mathbb{P}}
\newcommand{\ra}[0]{\rightarrow}
\newcommand{\vectoro}{\overline}
\newcommand{\crairy}{\mathcal{CA}}
\newcommand{\nc}{\mathsf{NoTouch}}
\newcommand{\ncf}{\mathsf{NoTouch}^f}
\newcommand{\wxy}{\mathcal{W}_{k;\bar{x},\bar{y}}}
\newcommand{\AP}{\mathfrak{a}}
\newcommand{\cm}{\mathfrak{c}}
\newcommand{\bz}{\mathbf{z}}
\newtheorem{theorem}{Theorem}[section]
\newtheorem{partialtheorem}{Partial Theorem}[section]
\newtheorem{conj}[theorem]{Conjecture}
\newtheorem{lemma}[theorem]{Lemma}
\newtheorem{proposition}[theorem]{Proposition}
\newtheorem{corollary}[theorem]{Corollary}
\newtheorem{claim}[theorem]{Claim}
\newtheorem{experiment}[theorem]{Experimental Result}

\def\todo#1{\marginpar{\raggedright\footnotesize #1}}
\def\change#1{{\color{green}\todo{change}#1}}
\def\note#1{\textup{\textsf{\color{blue}(#1)}}}

\theoremstyle{definition}
\newtheorem{rem}[theorem]{Remark}

\theoremstyle{definition}
\newtheorem{com}[theorem]{Comment}

\theoremstyle{definition}
\newtheorem{definition}[theorem]{Definition}

\theoremstyle{definition}
\newtheorem{definitions}[theorem]{Definitions}

\theoremstyle{definition}
\newtheorem{conjecture}[theorem]{Conjecture}

\newcommand{\airysh}{\mathcal{A}}
\newcommand{\hfixed}{\mathcal{H}}
\newcommand{\afixed}{\mathcal{A}}
\newcommand{\canopynoarg}{\mathsf{C}}
\newcommand{\canopy}[3]{\ensuremath{\mathsf{C}_{#1,#2}^{#3}}}
\newcommand{\argmax}{x_{{\rm max}}}
\newcommand{\zmax}{z_{{\rm max}}}

\newcommand{\Rkle}{\ensuremath{\mathbb{R}^k_{>}}}
\newcommand{\Ronele}{\ensuremath{\mathbb{R}^k_{>}}}
\newcommand{\ewxy}{\mathcal{E}_{k;\bar{x},\bar{y}}}

\newcommand{\bxyf}{\mathcal{B}_{\bar{x},\bar{y},f}}
\newcommand{\bxyflr}{\mathcal{B}_{\bar{x},\bar{y},f}^{\ell,r}}

\newcommand{\bxyfone}{\mathcal{B}_{x_1,y_1,f}}

\newcommand{\ptac}{p}
\newcommand{\ptact}{v}

\newcommand{\fext}{\mathfrak{F}_{{\rm ext}}}
\newcommand{\gext}{\mathfrak{G}_{{\rm ext}}}
\newcommand{\xext}{{\rm xExt}(\mathfrak{c}_+)}

\newcommand{\dd}{\, {\rm d}}
\newcommand{\signc}{\Sigma}
\newcommand{\wxylr}{\mathcal{W}_{k;\bar{x},\bar{y}}^{\ell,r}}
\newcommand{\wxylrprime}{\mathcal{W}_{k;\bar{x}',\bar{y}'}^{\ell,r}}
\newcommand{\Rklezero}{\ensuremath{\mathbb{R}^k_{>0}}}
\newcommand{\XYfM}{\textrm{XY}^{f}_M}

\newcommand{\upright}{D}
\newcommand{\staircase}{SC}
\newcommand{\energy}{E}
\newcommand{\xmax}{{\rm max}_1}
\newcommand{\ymax}{{\rm max}_2}
\newcommand{\lppls}{\mathcal{L}}
\newcommand{\lpplsre}{\mathcal{L}^{{\rm re}}}
\newcommand{\lpplsarg}[1]{\mathcal{L}_{n}^{\fa \to #1}}
\newcommand{\larg}[3]{\mathcal{L}_{n}^{#1,#2;#3}}
\newcommand{\BP}{M}
\newcommand{\weight}{\mathsf{Wgt}}
\newcommand{\pairweight}{\mathsf{PairWgt}}
\newcommand{\sumweight}{\mathsf{SumWgt}}
\newcommand{\mpgood}{\mathcal{G}}
\newcommand{\mpg}{\mathsf{Fav}}
\newcommand{\mcgone}{\mathsf{Fav}_1}
\newcommand{\radnik}[2]{\mathsf{RN}_{#1,#2}}
\newcommand{\size}[2]{\mathsf{S}_{#1,#2}}
\newcommand{\pdr}{\mathsf{PolyDevReg}}
\newcommand{\pwr}{\mathsf{PolyWgtReg}}
\newcommand{\lwr}{\mathsf{LocWgtReg}}
\newcommand{\hwp}{\mathsf{HighWgtPoly}}
\newcommand{\fbr}{\mathsf{ForBouqReg}}
\newcommand{\bbr}{\mathsf{BackBouqReg}}
\newcommand{\fsc}{\mathsf{FavSurCon}}
\newcommand{\maxpoly}{\mathrm{MaxDisjtPoly}}
\newcommand{\maxswf}{\mathsf{MaxScSumWgtFl}}
\newcommand{\emaxswf}{\e \! - \! \maxswf}
\newcommand{\minswf}{\mathsf{MinScSumWgtFl}}
\newcommand{\eminswf}{\e \! - \! \minswf}
\newcommand{\surreg}{\mathcal{R}}
\newcommand{\scf}{\mathsf{FavSurgCond}}
\newcommand{\disjtpoly}{\mathsf{DisjtPoly}}
\newcommand{\intint}[1]{\llbracket 1,#1 \rrbracket}
\newcommand{\maxsym}{*}
\newcommand{\polynum}{\#\mathsf{Poly}}
\newcommand{\dlp}{\mathsf{DisjtLinePoly}}
\newcommand{\lowb}{\underline{B}}
\newcommand{\highb}{\overline{B}}
\newcommand{\tottt}{s_{1,2}^{2/3}}
\newcommand{\tot}{s_{1,2}}
\newcommand{\btone}{{\bf{t}}_1}
\newcommand{\bttwo}{{\bf{t}}_2}
\newcommand{\formerE}{C}
\newcommand{\rcon}{r_0}
\newcommand{\para}{Q}
\newcommand{\Cstrong}{E}

\newcommand{\mc}{\mathcal}
\newcommand{\vect}{\mathbf}
\newcommand{\bt}{\mathbf{t}}
\newcommand{\scB}{\mathscr{B}}
\newcommand{\scBres}{\mathscr{B}^{\mathrm{re}}}
\newcommand{\rightshadow}{\mathrm{RS}Z}
\newcommand{\dbm}{D}
\newcommand{\edgedbm}{D^{\rm edge}}
\newcommand{\gue}{\mathrm{GUE}}
\newcommand{\edgegue}{\mathrm{GUE}^{\mathrm{edge}}}
\newcommand{\eqdist}{\stackrel{(d)}{=}}
\newcommand{\geqdist}{\stackrel{(d)}{\succeq}}
\newcommand{\leqdist}{\stackrel{(d)}{\preceq}}
\newcommand{\scal}{{\rm sc}}
\newcommand{\fa}{x_0}
\newcommand{\hit}{H}
\newcommand{\scaledle}{\mathsf{Nr}\mc{L}}
\newcommand{\cenleup}{\mathscr{L}^{\uparrow}}
\newcommand{\cenledown}{\mathscr{L}^{\downarrow}}
\newcommand{\eln}{T}
\newcommand{\xmin}{{\rm Corner}^{\mfl,\mc{F}}}
\newcommand{\ymin}{{\rm Corner}^{\mfr,\mc{F}}}
\newcommand{\barxmin}{\overline{\rm Corner}^{\mfl,\mc{F}}}
\newcommand{\barymin}{\overline{\rm Corner}^{\mfr,\mc{F}}}
\newcommand{\qmin}{Q^{\mc{F}^1}}
\newcommand{\barqmin}{\bar{Q}^{\mc{F}^1}}
\setcounter{tocdepth}{2}
\newcommand{\test}{T}
\newcommand{\mfl}{\mf{l}}
\newcommand{\mfr}{\mf{r}}
\newcommand{\gfl}{\ell}
\newcommand{\gfr}{r}
\newcommand{\jre}{J}
\newcommand{\highfl}{{\rm HFL}}
\newcommand{\flyleap}{\mathsf{FlyLeap}}
\newcommand{\touch}{\mathsf{Touch}}
\newcommand{\notouch}{\mathsf{NoTouch}}
\newcommand{\close}{\mathsf{Close}}
\newcommand{\abovepar}{\mathsf{High}}
\newcommand{\vecint}{\bar{\iota}}
\newcommand{\cornthree}{{\rm Corner}^\mc{G}_{k,\mfl}}
\newcommand{\cornfour}{{\rm Corner}^\mc{H}_{k,\fa}}
\newcommand{\mpgg}{\mathsf{Fav}_{\mc{G}}}

\newcommand{\lefta}{M_{1,k+1}^{[-2\eln,\gfl]}}
\newcommand{\mida}{M_{1,k+1}^{[\gfl,\gfr]}}
\newcommand{\righta}{M_{1,k+1}^{[\gfr,2\eln]}}
\newcommand{\Var}{{\textrm{Var}}}
\newcommand{\ipdval}{d}
\newcommand{\ctemp}{d_0}

\newcommand{\wien}{W}
\newcommand{\pole}{P}
\newcommand{\pp}{p}

\newcommand{\const}{D_k}
\newcommand{\numcone}{14}
\newcommand{\numctwo}{13}
\newcommand{\numcthree}{6}
\newcommand{\rsC}{C}
\newcommand{\rsc}{c}
\newcommand{\cone}{c_1}
\newcommand{\Cone}{C_1}
\newcommand{\Ctwo}{C_2}
\newcommand{\smallc}{c_0}
\newcommand{\smallcprime}{c_1}
\newcommand{\smallcanother}{c_2}
\newcommand{\smallcnew}{c_3}
\newcommand{\Cda}{D}
\newcommand{\Kzero}{K_0}
\newcommand{\Rmac}{R}
\newcommand{\rmac}{r}
\newcommand{\conseqmac}{D}
\newcommand{\constn}{C'}
\newcommand{\coninit}{\Psi}
\newcommand{\condee}{\hat{D}}
\newcommand{\conbrac}{\hat{C}}
\newcommand{\Cnew}{\tilde{C}}
\newcommand{\Cbig}{C^*}
\newcommand{\Ctbd}{C_+}
\newcommand{\Ctbs}{C_-}
\newcommand{\cE}{\mathcal{E}}
\newcommand{\Cov}{Cov}

\newcommand{\imax}{i_{{\rm max}}}

\newcommand{\wlp}{{\rm WLP}}

\newcommand{\canopynumber}{\mathsf{Canopy}{\#}}

\newcommand{\cannum}{{\#}\mathsf{SC}}

\newcommand{\boundgood}{\mathsf{G}}
\newcommand{\lshift}{\mc{L}^{\rm shift}}
\newcommand{\deltapi}{\theta}
\newcommand{\rootneigh}{\mathrm{RNI}}
\newcommand{\rootneighuse}{\mathrm{RNI}}
\newcommand{\manycan}{\mathsf{ManyCanopy}}
\newcommand{\specialpt}{\mathrm{spec}}

\newcommand{\dist}{\vert\vert}
\newcommand{\fik}{\mc{F}_i^{[K,K+1]^c}}
\newcommand{\mcfa}{\mc{H}[\fa]}
\newcommand{\tent}{{\rm Tent}}
\newcommand{\goodk}{\mc{G}_{K,K+1}}
\newcommand{\pairsep}{{\rm PS}}
\newcommand{\mbf}{\mathsf{MBF}}
\newcommand{\nbd}{\mathsf{NoBigDrop}}
\newcommand{\bd}{\mathsf{BigDrop}}
\newcommand{\jleft}{j_{{\rm left}}}
\newcommand{\jright}{j_{{\rm right}}}
\newcommand{\smalljfluc}{\mathsf{SmallJFluc}}
\newcommand{\mfone}{M_{\mc{F}^1}}
\newcommand{\mfthree}{M_{\mc{G}}}
\newcommand{\rhomac}{P}
\newcommand{\phimac}{\varphi}
\newcommand{\chimac}{\chi}
\newcommand{\xnmac}{z_{\mathcal{L}}}
\newcommand{\Cwb}{E_0}
\newcommand{\initcond}{\mathcal{I}}
\newcommand{\neargeod}{\mathsf{NearGeod}}
\newcommand{\polyunique}{\mathrm{PolyUnique}}
\newcommand{\latecoal}{\mathsf{LateCoal}}
\newcommand{\nolatecoal}{\mathsf{NoLateCoal}}
\newcommand{\normalcoal}{\mathsf{NormalCoal}}
\newcommand{\regfluc}{\mathsf{RegFluc}}
\newcommand{\mdeltaweight}{\mathsf{Max}\Delta\mathsf{Wgt}}
\newcommand{\ovbar}[1]{\mkern 1.5mu\overline{\mkern-1.5mu#1\mkern-1.5mu}\mkern 1.5mu}
\newcommand{\fluc}{\mathsf{ScaledFluc}}
\newcommand{\fluct}{\mathsf{Fluc}}

\newcommand{\maxmin}{\pwr}
\newcommand{\nmac}{N}

\newcommand{\notlow}{{\rm NotLow}}

\newcommand{\rmreg}{{\rm Reg}}

\newcommand{\down}{\mathsf{Fall}}
\newcommand{\up}{\mathsf{Rise}}

\newcommand{\safeguard}{\mathsf{SafeGuard}}
\newcommand{\lowoverlap}{\mathsf{LowOverlap}}
\newcommand{\longex}{\mathsf{LongExcursions}}
\newcommand{\totexdur}{\mathsf{HighTotExcDur}_n^{0,t}}
\newcommand{\shortexc}{\mathsf{Short}_n^{0,t}}
\newcommand{\shortnonthinexc}{\mathsf{ShortNonThin}_n^{0,t}}
\newcommand{\longexc}{\mathsf{Long}_n^{0,t}}
\newcommand{\sigover}{\mathsf{SigOver}}
\newcommand{\steady}{\mathsf{Steady}}
\newcommand{\uni}{\mathsf{UnifBdd}}
\newcommand{\paradelta}{\Delta^{\cup}\,}

\newcommand{\aplus}{a_+}
\newcommand{\xminus}{x^-}

\newcommand{\maxdist}{\mathsf{MaxDist}}
\newcommand{\proxymimicry}{\mathsf{ProxyMimicry}_n^{0,t}(L,H)}
\newcommand{\nowideexc}{\mathsf{NoWideExc}_n^{0,t}}
\newcommand{\manyunlucky}{\mathsf{ManyUnlucky}_n^{0,t}}

\newcommand{\consistsep}{\mathsf{ConsistSep}_n^{0,t}}
\newcommand{\dmacd}{d}
\newcommand{\hata}{\sigma}
\newcommand{\hatapr}{\sigma}
\newcommand{\remainder}{r_0}
\newcommand{\nolow}{\mathsf{NoLow}_n}
\newcommand{\nohigh}{\mathsf{NoHigh}_n}
\newcommand{\low}{\mathsf{Low}_n}
\newcommand{\high}{\mathsf{High}_n}
\newcommand{\heightmac}{Y}
\newcommand{\duration}{{\rm dur}}
\newcommand{\scriptdur}{\mathscr{D}}

\newcommand{\redconst}{d_0}
\newcommand{\tza}{\theta}
\newcommand{\reg}{\mathsf{r}}
\newcommand{\tnonza}{\tau_0^\alpha}

\newcommand{\macroseventeen}{19}
\newcommand{\macrobig}{1370}
\newcommand{\cmac}{c_0}
\newcommand{\Cmac}{C_0}
\newcommand{\dmac}{d}

\title[Gaussian polymer near-ground states]{The geometry of near ground states \\ in Gaussian polymer models}

\author[S. Ganguly]{Shirshendu Ganguly}
\address{S. Ganguly\\
  Department of Statistics\\
 U.C. Berkeley \\
  401 Evans Hall \\
  Berkeley, CA, 94720-3840 \\
  U.S.A.}
  \email{sganguly@berkeley.edu}
  
\author[A. Hammond]{Alan Hammond}
\address{A. Hammond\\
  Departments of Mathematics and Statistics\\
 U.C. Berkeley \\
 899 Evans Hall \\
  Berkeley, CA, 94720-3840 \\
  U.S.A.}
  \email{alanmh@berkeley.edu}

\begin{abstract} 
The energy and geometry of maximizing paths in integrable last passage percolation models are governed by the characteristic KPZ scaling exponents of one-third and two-thirds.
When represented in scaled coordinates that respect these exponents, this random field of paths may be viewed as a complex energy landscape. We investigate the structure of valleys and connecting pathways in this landscape. The routed weight profile $\R \to \R$ associates to $x \in \R$
the maximum scaled energy obtainable by a path whose scaled journey from $(0,0)$ to $(0,1)$ passes through the point $(x,1/2)$. Developing tools of Brownian Gibbs analysis from \cite{BrownianReg} and~\cite{DeBridge}, we prove an assertion of strong similarity of this profile for Brownian last passage percolation  to Brownian motion of rate two on the unit-order scale. A sharp estimate on the rarity that two macroscopically different routes in the energy landscape offer energies close to the global maximum results. {We prove robust assertions concerning modulus of continuity for the energy and geometry of scaled maximizing paths, that develop the results and approach of~\cite{HammondSarkar}, 
delivering estimates valid on all scales above the microscopic.}
The geometry of excursions of near ground states about the maximizing path is investigated: indeed,  
we estimate the energetic  shortfall of scaled paths forced to closely mimic the geometry of the maximizing route while remaining disjoint from it.
We also provide bounds on the approximate gradient of the maximizing path, viewed as a function, ruling out sharp steep movement down to the microscopic scale. Our results find application in a companion study~\cite{Dynamics}
of the stability, and fragility, of last passage percolation under a dynamical perturbation.
\end{abstract}

\maketitle



\begingroup
\hypersetup{linktocpage=false}
\setcounter{tocdepth}{2}
\renewcommand{\baselinestretch}{0.9}\normalsize
\tableofcontents
\renewcommand{\baselinestretch}{1.0}\normalsize
\endgroup

\section{Introduction}

\subsection{KPZ universality, last passage percolation models, and scaled coordinates}
The $1 + 1$ dimensional Kardar-Parisi-Zhang [KPZ] universality class includes many microscopic models in which a random interface is suspended over a one-dimensional domain, whose growth in a direction normal to the surface competes with a smoothening surface tension in the presence of a local force that randomly roughens the surface.
Many planar last passage percolation [LPP] models exhibit these characteristics. In a planar LPP model, directed paths, moving in directions in the first quadrant, 
are assigned energy via a random environment, which is independent in disjoint regions. This energy is assigned  by integrating the environment's value along the path.
For a given pair of planar points, the directed path between them attaining the maximum energy  is called a geodesic.

For LPP models lying in the KPZ class, a geodesic moving in a non-axial direction, crossing  a large distance $n$ has an energy that is typically  linear in $n$ with a standard deviation of order~$n^{1/3}$. The associated random interface   mentioned at the outset is the function obtained when the lower geodesic endpoint is held fixed, and the geodesic energy is  a function of the other endpoint varying horizontally. In this particular case, where the first endpoint is fixed, the energy profile is termed `narrow wedge'.
 Non-trivial correlations in the interface occur between points with separation of order $n^{2/3}$. The same exponent governs the related notion of transversal fluctuation of the geodesic from the straight line joining its endpoints.
 Despite the predicted universality, these assertions have been rigorously demonstrated for only a few LPP models with certain exactly solvable features: the seminal work of Baik, Deift and Johansson~\cite{BDJ1999}
 established the one-third exponent, and the GUE Tracy-Widom distributional limit, for the case of Poissonian last passage percolation, while Johansson~~\cite{Johansson2000} derived the two-thirds power law for maximal transversal fluctuation for this model.

In view of these facts, it is natural to represent the field of geodesics in a scaled system of coordinates. Under this scaling, a northeasterly displacement of order~$n$  becomes a vertical displacement of one unit, while a horizontal displacement of order $n^{2/3}$ becomes a unit horizontal displacement. The system of energies also transfers to scaled coordinates, with the scaled geodesic energy being specified by centring about the mean value and normalizing by the typical scale of $n^{1/3}$.

In this way, the LPP geodesic that begins at $(0,0)$ and ends at $(n,n)$ has a scaled counterpart, which we will refer to as a polymer and label $\rho \big[ (0,0) \to (0,1) \big]$, that travels between $(0,0)$
 and~$(0,1)$. This polymer has a scaled energy, or weight, that we denote by $\weight \big[ (0,0) \to (0,1) \big]$, which in the high $n$ limit is distributed according to the GUE Tracy-Widom distribution.
  In the scaled LPP description more generally, a polymer  $\rho \big[ (x,s) \to (y,t) \big]$ is associated  to each pair of planar points $(x,s)$ and $(y,t)$ with $s < t$. The polymer's weight is denoted by  $\weight \big[ (x,s) \to (y,t) \big]$.
 
 \subsection{The energy landscape of scaled LPP and the structure of its valleys}

Many statistical mechanical systems may be described by a probability measure whose density $e^{-H(x)}$ with respect to a background measure $\mu$, supported on a space $X$, is specified by a Hamiltonian $H: X \to \R$ which may be viewed as an energy landscape over $X$. Such a system may be viewed as a particle that dwells randomly in $X$; the system is held at equilibrium by a Markovian dynamics in which the present state evolves locally according to a Metropolis rule governed by the relative values of the function $e^{-H(x)}$. 
The present state is thus a snapshot of a particle wandering in the energy landscape, which is typically attracted into the landscape's local valleys; its long-term behaviour is governed by the structure of valleys---their number; depths; and the heights and geometry of mountain passes that connect them. For Gaussian models of disorder, including Gaussian polymers, \cite{Chatterjee14} proved via an interpolation method that certain strong concentration properties exhibited by such systems are equivalent to an abundance of well-separated valleys---which abundance entails chaotic behaviour of observables when the system is slightly perturbed. The landscape geometry of several models has since then been studied.  Landscapes of general smooth Gaussian functions
on the sphere in high dimensions as well as those related to spin-glass models have been studied in \cite{auffinger1, auffinger2}. More recently, refined results for the number of valleys for  spin-glasses have been obtained in \cite{DEZ,CHL2018,Eldan}.

In this article, we investigate the structure of the energy landscape of {\em Brownian}  last passage percolation, a semi-discrete polymer wandering through Gaussian noise, in its scaled coordinate description.  This will find application in a companion study~\cite{Dynamics} of the transition from stability to chaos of this model subject to a dynamical perturbation.  Brownian LPP model will be recalled shortly; its noise environment is comprised of Brownian randomness. The model
carries a parameter $n \in \N$
which in rising approximates a limiting scaled description; in our present heuristic purpose, we omit mention of it; indeed we already did so, in indicating the meaning of the field of weights  $\weight \big[ (x,s) \to (y,t) \big]$.
 
Let $x \in \R$  and $a \in (0,1)$. Set $Z(x,a)$ equal to the supremum of weights of scaled paths on the route from $(0,0)$ to $(0,1)$ that pass through $(x,a)$. That is,
\begin{equation}\label{e.ztwoweight}
 Z(x,a) = \weight \big[ (0,0) \to (x,a) \big] +  \weight \big[  (x,a) \to (0,1) \big] \, .
\end{equation}
We refer to the random process $\R \to \R : x \to Z(x,a)$ as the {\em routed weight profile} at height $a$, because this process records weights of paths that are routed through a given location at this height. This profile is a cross-section of the LPP energy landscape that is pertinent for understanding the geometry of the polymer $\rho \big[ (0,0) \to (0,1) \big]$ and how effectively scaled paths that share the polymer's endpoints but that take alternative routes compete in weight with the polymer.
For example, the horizontal location at which  $\rho \big[ (0,0) \to (0,1) \big]$ traverses height $a$ is the maximizer~$M$ of  the random function $\R \to \R : x \to Z(x,a)$  (an almost surely unique location, as we will indicate in Lemma~\ref{l.routedprofile}(2)). For $x \in \R$, the quantity $Z(M,a) - Z(x,a) \geq 0$ is the shortfall in weight relative to the polymer's of a scaled LPP path from $(0,0)$ to $(0,1)$ that is constrained to pass via the point~$(x,a)$.

\subsection{Principal conclusions and themes in overview}
In this article, we will prove several conclusions concerning the energy landscape of scaled Brownian LPP.
 As we informally summarise them now, we continue to omit mention of the scaling parameter $n \in \N$: roughly, our assertions should be understood uniformly in high choices of this parameter.
 
 \subsubsection{Brownianity of the routed weight profile}\label{s.brownianity}  On scales larger the unit scale, the profile $x \to Z(x,a)$ is curved, following the parabola $x \to -2^{-1/2} a^{-1}(1-a)^{-1}x^2$. On the unit scale, however, it resembles Brownian motion. Building on a Brownian comparison result for narrow wedge weight profiles from \cite{DeBridge},  our first result, Theorem~\ref{t.brownianroutedprofile}, offers a strong attestation of this resemblance. 
For $a \in (0,1)$, the profile $z \to Z(x,a)$ enjoys a strong similarity with Brownian motion $B$ {\em of rate two}. Indeed, if $A$ denotes a collection of continuous functions on $[-1,1]$ that vanish at $-1$ for which the probability that $[-1,1] \to \R: x \to B(x) - B(-1)$
 is denoted by $\eta$, then the probability that the profile $[-1,1] \to \R : x \to Z(x,a) - Z(-1,a)$ belongs to $A$ is at most an expression of the form $\eta \cdot \exp \big\{ \Theta(1) (\log \eta^{-1})^{5/6} \big\}$. The latter, correction, term grows much more slowly than any inverse power of~$\eta$ in the limit of $\eta \searrow 0$. The constant implied by use of the notation $\Theta(1)$
 may be chosen uniformly as $a$ varies over any given compact set in $(0,1)$.

 \subsubsection{The rarity of twin peaks}
 If the maximizer $M$ of $x \to Z(x,1/2)$ lies on the right, so that say the positive probability event $M \in [1,3]$ is satisfied, the random variable 
 $$
 Z(M,1/2) \, - \, \sup \big\{ Z(x,1/2): x \in [-3,-1] \big\} \geq 0
 $$
 equals the shortfall in weight relative to the polymer's along those scaled paths from $(0,0)$ to $(0,1)$ that instead pass on the left, via locations in $[-3,-1]$, at the mid-life time one-half.
 If this random variable is less than a given small quantity $\sigma > 0$, the profile $x \to Z(x,1/2)$ resembles a pair of peaks, with the left hill's height rivalling the right hill's to within a distance of $\sigma$. When this {\em twin peaks'} event occurs, a local valley in the LPP energy landscape lies at a significant remove from the global valley while succeeding to  rival the latter's depth. An upper bound on this event's probability thus sheds light on the landscape's geometry. Given the strong resemblance of the profile to Brownian motion, the probability of the twin peaks' event is inherited from the counterpart Brownian probability. Our second principal conclusion, Theorem~\ref{t.nearmax}, asserts that twin peaks with discrepancy $\sigma$ arise in the routed weight profile with probability at most  $\sigma \cdot \exp \big\{ \Theta(1) (\log \sigma^{-1})^{5/6} \big\}$. 
  Figure~\ref{f.twinpeaksneartouch} offers a guide to twin peaks in the energy landscape via the equivalent notion of {\em near touch} for a natural decomposition of the routed weight profile.
 
\begin{figure}[ht]
\begin{center}
\includegraphics[height=7cm]{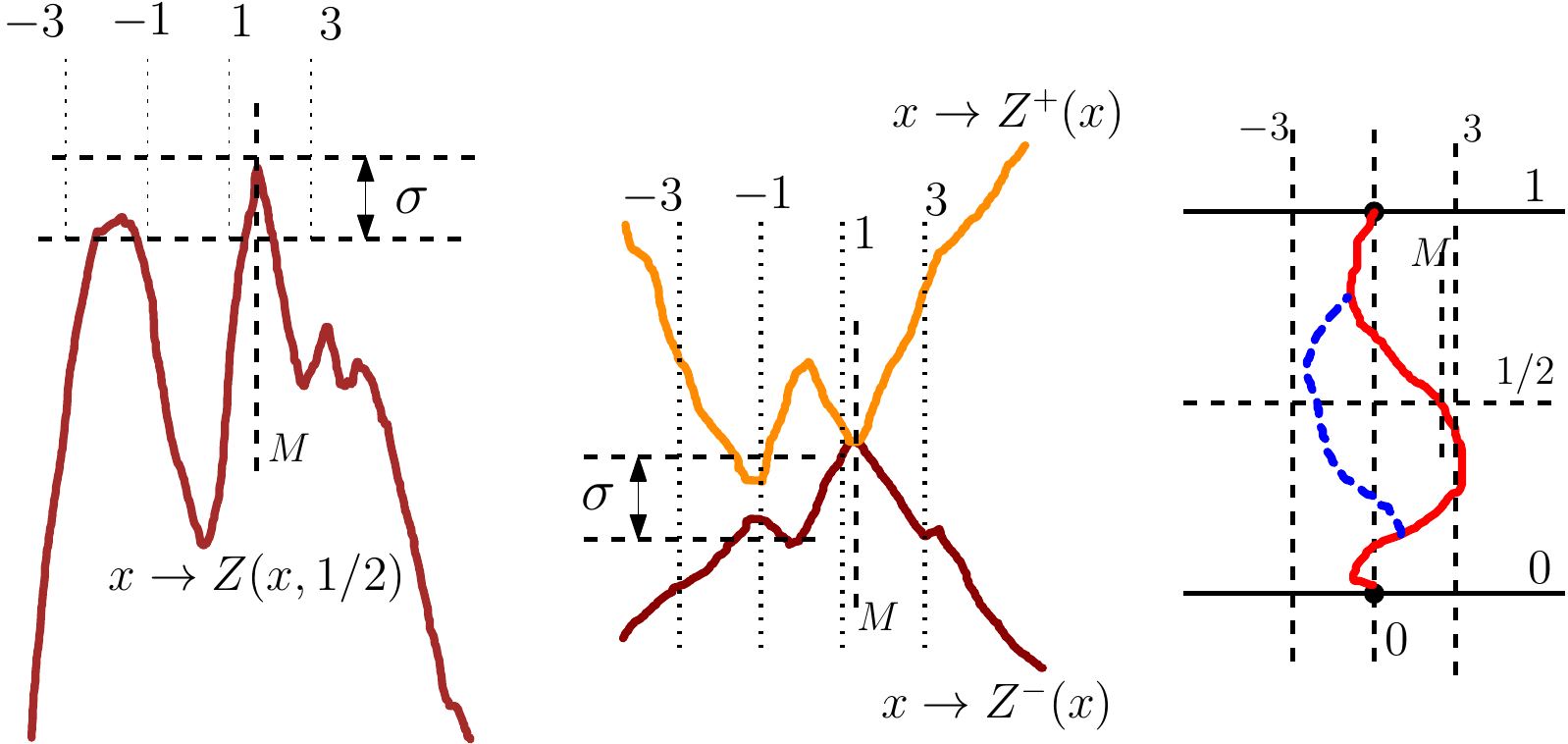}
\caption{{\em Left:} The routed weight profile $x \to Z(x,1/2)$ realizes the twin peaks' event. {\em Middle:} Let $Z^-(x) = \weight_n \big[ (0,0) \to (x,1/2) \big]$; and let $x \to Z^+(x)$ denote the vertical translate of $x \to -\weight_n \big[ (x,1/2) \to (0,1) \big]$ such that the graphs of $Z^-$ and $Z^+$ touch, but do not cross. The horizontal coordinate of the point of touch is~$M$, the maximizer of $Z(\cdot,1/2)$; the occurrence of twin peaks is now represented by a near touch on the part  of the two graphs in the strip $[-3,-1] \times \R$. {\em Right:} The bold polymer $\rho \big[ (0,0) \to (0,1) \big]$ has horizontal coordinate $M \in [1,3]$ at height one-half. The rival path, following the bold-dashed-bold route, attains a weight within $\sigma$ of the polymer's while swinging left, into $[-3,-1]$, at the mid-life time one-half.}
\label{f.twinpeaksneartouch}
\end{center}
\end{figure} 
 
 \subsubsection{Robust assertions of modulus of continuity for geometry and weight of polymers}
Polymers such as  $\rho = \rho \big[ (0,0) \to (0,1) \big]$ may be viewed as functions of the vertical coordinate;  in this way, we interpret $\rho: [0,1] \to \R$ as a random real-valued function.
A modulus of continuity for this function is known by~\cite{HammondSarkar,DOV} to take the form of a large constant multiple of $z \to z^{2/3} \big( \log z^{-1} \big)^{1/3}$. The weight of the polymer restricted to $[0,t]$ may also be viewed as a function of $t \in [0,1]$; these same references prove a modulus of continuity for this weight profile of order $z \to z^{1/3}  \big( \log z^{-1} \big)^{2/3}$.  We provide robust forms of such assertions in Theorems~\ref{t.toolfluc} and~\ref{t.weight}. These results are  valid in Brownian LPP uniformly in high values of its parameter $n \in \N$ and in variation of polymer endpoints over compact regions. Just as significantly, they control variation in polymer weight and geometry not merely in response to small $n$-independent changes in the vertical coordinate as do~\cite{HammondSarkar} and \cite{DOV}, but on any vertical scale down to the microscopic separation~$n^{-1}$. A related local fluctuation result was also a crucial ingredient in \cite{SlowBond}.

 \subsubsection{Slender excursions around the polymer are typically uncompetitive}
 The weight $Z(M,1/2)$ is realized by the maximum weight path on the route from $(0,0)$ to $(0,1)$, namely the polymer  $\rho = \rho \big[ (0,0) \to (0,1) \big]$. For $z \in \R$ small, the weight $Z(M+z,1/2)$
 is realized by a scaled path~$\phi$ that begins from $(0,0)$ by following the course of $\rho$; that departs this course at some height $h_1 \in (0,1/2)$ to visit $M+z$ at height one-half before rejoining $\rho$ at some height $h_2 \in (1/2,1)$; and that then follows the course of $\rho$ until $(0,1)$. As such, $[h_1,h_2]$  is an interval during which $\phi$ makes an excursion away from~$\rho$. The KPZ exponent of two-thirds for polymer geometry indicates that $(h_2 - h_1)^{2/3}$ has typical order $z$. That is, we expect $\phi$ to make an excursion, on an interval that contains one-half, for a duration~$h_2 - h_1$ of order $z^{3/2}$; to maintain a horizontal distance from~$\rho$ of order $z$ during much of the excursion's duration; and, in view of the Gaussian-order increment of the routed weight profile,  to accrue a shortfall in weight relative to $\rho$ of order $z^{1/2}$.
 
We present a conclusion, Theorem~\ref{t.nocloseness}, that validates this heuristic view.  It considers the maximum weight accrued by a path~$\phi$ that makes an excursion of duration $h \in (0,1)$  from the polymer $\rho$ in such a way that the horizontal discrepancy between $\rho$ and $\phi$ is at most $h^{2/3}\tza$ at most moments during the excursion's lifetime. The parameter $\tza > 0$ will be chosen to be small, so that a {\em slender} excursion is being demanded, one that deviates horizontally from the polymer by a factor of $\tza$ less than is expected. We will prove that the weight of any such path~$\phi$ is highly likely to fall short of the polymer weight by an order of  at least $h^{1/3} \tza^{-1}$. When $\tza$ is of unit order, this shortfall is predicted by the KPZ exponent of one-third for polymer weight; when $\tza \ll 1$, the factor of  $\tza^{-1}$ represents a weight penalty for the forcibly confined geometry endured by the excursion. Results of a similar flavor have appeared before in \cite{BGH18,Watermelon, BHS18,Bootstrapping}.
   
 \subsubsection{The polymer advances in a regular fashion microscopically}
 Our final main result, Theorem~\ref{t.notallcliffs}, concerns the microscopic structure of the trajectory of the polymer in Brownian LPP, and is more vividly expressed in unscaled coordinates. Consider then the Brownian LPP geodesic~$\Gamma_n$ that runs from $(0,0)$ to $(n,n)$. The geodesic's progression is globally diagonal; we prove that, even on the shortest of scales, this progression is manifest. A {\em cliff} in $\Gamma$ is a subpath of $\Gamma$ in which $\Gamma$ advances horizontally by one unit while advancing vertically by $A$ units. We prove that, when the positive parameter~$A$ is fixed at a high value, and except on an event of probability that decays at an exponential rate in~$n$, the proportion of $\Gamma$ that is comprised of cliffs is bounded away from one.
 
 \subsection{Probabilistic and geometric inquiry into KPZ universality}
 The study of
KPZ universality has advanced through physical 
insights, numerical analysis, and several techniques
of integrable or algebraic origin. Rather than hazard a 
summary of literature to support this one sentence history, we refer to the reader to~\cite{IvanSurvey} for a 
KPZ survey from 2012; in fact, integrable
and analytic approaches to KPZ have attracted huge interest around and since that time. A recent wave of 
KPZ research has brought probabilistic and geometric tools to the fore, making use of integrable aspects of the models of study as occasional inputs in arguments, 
albeit essential ones. Three examples are the solution~\cite{SlowBond} of the slow bond problem, in 
which the integrable model of exponential LPP is perturbed by altering the random environment along a one-dimensional subspace, and the resulting geometry and energy of geodesics is studied; 
the construction~\cite{AiryLE} of the  Airy line ensemble, a KPZ universal object encoding polymer weights in the narrow wedge case as a continuous system of mutually avoiding random curves;
and the construction~\cite{DOV} of the Airy sheet (or the directed landscape), a rich scaling limit for the weights of KPZ polymers in which these weights are obtained as LPP values in an environment specified by the Airy line ensemble after the subtraction of a parabola. 

The present article pursues the study of problems in KPZ in a probabilistic and geometric vein. 
In particular, our results Theorems~\ref{t.brownianroutedprofile} and~\ref{t.nearmax} on Brownianity and twin peaks' rarity for the routed weight profile lie in the domain of {\em Brownian Gibbs} analysis of LPP. 
The parabolic Airy line ensemble is in essence a mutually avoiding system of Brownian motions, subject to suitable boundary conditions. As we will indicate more clearly early in Section~\ref{s.twinpeaks}, this ensemble of random curves thus satisfies the Brownian Gibbs property, a simple and attractive resampling property involving Brownian motion and avoidance. The Brownian Gibbs technique led to the construction of the Airy line ensemble in~\cite{AiryLE}. The technique has been pursued in \cite{BrownianReg} and \cite{DeBridge} to yield strong inferences regarding the similarity to Brownian motion of the Airy$_2$ process, which is, after a parabolic shift, the scaling limit of the narrow wedge polymer weight profile in integrable LPP models. 

Our results on Brownianity and twin peaks' rarity develop this strand of research, 
begun in \cite{AiryLE}, and pursued in \cite{BrownianReg} and \cite{DeBridge}, 
so that results such as Theorems~\ref{t.brownianroutedprofile} and~\ref{t.nearmax} pertinent to the LPP energy landscape now become available.

By means of a more algebraic approach that analyses  representations involving Fredholm determinants,
strong Brownian comparison estimates have also been obtained in~\cite{MQR17}.  This work constructs a universal Markov process called the KPZ fixed point that describes the evolution of geodesic energy profiles starting from rather arbitrary initial data; an earlier result in \cite{QR1} identified domains of attraction for the one point fluctuations of the KPZ equation starting from general initial data. 
The assertion that the {A}iry$_2$ process closely resembles Brownian motion on the unit-order scale, and counterpart results for scaled geodesic energy profiles in LPP models, have been instrumental in several recent inquiries into geometric and fractal  properties of the KPZ fixed point. In \cite{BGcorrelation,BGZcorrelation, FO}, exponents governing temporal correlations induced by various initial data have been determined, thereby settling conjectures by Ferrari and Spohn from~\cite{FS}. (In \cite{CGH}, an analogous result for the KPZ equation is derived.)
Profile Brownianity also drives the fractal geometry and Hausdorff dimension  results for exceptional sets found in the space-time Airy sheet that are the subject of~\cite{BGHfractal1,BGHfractal2}.

\subsection{Stability and chaos in dynamical Brownian LPP}
Control on polymer geometry and weight; on the excursion geometry of LPP paths that in weight are competitive with the maximum; on the microscopic structure of geodesics---these geometric inferences are, we believe, robust tools that will serve to advance the analysis of LPP; its scaling limit; and its reaction to perturbation. Indeed, this article's results find application in 
a companion study~\cite{Dynamics} of the stability, and fragility, of Brownian LPP under dynamical perturbation. Very shortly, we will define Brownian LPP; for now, we note merely that its noise environment is specified by a countable system of independent Brownian motions. A dynamics may be introduced that leaves Brownian LPP invariant, in which each of these constituent Brownian motions is updated according to Ornstein-Uhlenbeck dynamics. In \cite{Dynamics} is identified the time-scale that heralds the transition from stability to chaos for dynamical Brownian LPP---the polymer from $(0,0)$ to $(0,1)$ is largely unperturbed in the stable zone and is profoundly altered in the chaotic phase. This time-scale takes the form $n^{-1/3 + o(1)}$ when a geodesic of extension~$n \in \N$ is considered; this corresponds to updating $n^{2/3 + o(1)}$ bits along the geodesic in a discrete LPP model. Every one of the results that we have indicated in the preceding overview has a role to play in proving this transition in~\cite{Dynamics}. 
The robust probabilistic and geometric results and technique 
that we present undergird the companion dynamical  LPP analysis and, we believe, will find further application in the study of scaled KPZ structure.

In the next two sections, we define Brownian LPP and introduce some of its basic objects; and we specify the transformation that specifies the scaled coordinates in which we couch our principal results and proofs. In the remaining introductory sections, we then present the statements of our main results, in the same order in which we have just summarized them.  

\subsection{Brownian last passage percolation [LPP]}\label{s.brlpp}
On a probability space equipped with a law labelled~$\PP$,
let $B:\R \times \Z \to \R$ denote a collection of 
independent  two-sided standard Brownian motions $B(\cdot,k):\R\to \R$, $k \in \Z$. The indexing of the 
domain in the form $\R \times \Z$ is unusual, with the other choice $\Z \times \R$ being more conventional. 
The choice of $\R \times \Z$ is made because it permits us to visualize this index set for the ensemble $B$'s 
curves as a subset of $\R^2$ with the usual Cartesian 
coordinate order being respected by the notation. 

Let $i,j \in \Z$ with $i \leq j$.
$\llbracket i,j \rrbracket$ will denote the integer interval $\{i,\cdots,j\}$.
For $x,y \in \R$ with $x \leq y$,
consider the collection of  non-decreasing lists 
 $\big\{ z_k: k \in \llbracket i+1,j \rrbracket \big\}$ of 
values $z_k \in [x,y]$. 
Adopting the convention that $z_i = x$ and $z_{j+1} = y$,
we associate to any such list the energy $\sum_{k=i}^j \big( B ( z_{k+1},k ) - B( z_k,k ) \big)$.
We then define the maximum energy to be
$$
M \big[ (x,i) \to (y,j) \big] \, = \, \sup \, \bigg\{ \, \sum_{k=i}^j \Big( B ( z_{k+1},k ) - B( z_k,k ) \Big) \, \bigg\} \, , 
$$
where the supremum is taken over all such lists. The random process $M \big[ (0,1) \to (\cdot,n) \big] :
[0,\infty) \to \R$ was introduced by~\cite{GlynnWhitt} and further studied in~\cite{O'ConnellYor}.

\subsubsection{Staircases}\label{s.staircase}
Set $\N = \{ 0,1,\cdots \}$.
For $i,j \in \N$ with $i \leq j$, and 
$x,y \in \R$ with $x \leq y$,
an energy has been ascribed to  any non-decreasing list 
$\big\{ z_k: k \in \llbracket i+1,j \rrbracket \big\}$ of values $z_k \in [x,y]$.
In order to emphasize the geometric aspects of this definition, 
we associate to each list a subset of $[x,y] \times [i,j] \subset \R^2$, that we call a staircase, which will be the 
range of a piecewise affine path.
 
To define the staircase above we again adopt the convention that $z_i = x$ and  $z_{j+1} = y$. 
The staircase will be specified as the union of certain horizontal planar line segments, and certain vertical 
ones.
The horizontal segments take the form $[ z_k,z_{k+1} ] \times \{ k \}$ for $k \in \llbracket i , j \rrbracket$.
A vertical planar line segment of unit length connects the 
right and left endpoints of each consecutive pair of horizontal segments. It is this collection of vertical line 
segments that form
the vertical segments of the staircase.

The resulting staircase may be depicted as the range of 
an alternately rightward and upward moving path from starting point $(x,i)$ to ending point $(y,j)$. 
The set of staircases with these starting and ending points will be denoted by $\staircase \big[ (x,i) \to (y,j) \big]$.
Since such staircases are in bijection with the collection of non-decreasing lists, any staircase $\phi \in \staircase \big[ (x,i) \to (y,j) \big]$
is assigned an energy $E(\phi) = \sum_{k=i}^j \big( B ( z_{k+1},k ) - B( z_k , k  ) \big)$ via the associated $z$-list.

\subsubsection{Energy maximizing staircases are called geodesics.}\label{s.geodesics}
A staircase  $\phi \in \staircase \big[ (x,i) \to (y,j) \big]$ 
whose energy  attains the maximum value $M \big[ (x,i) \to (y,j) \big]$ is called a geodesic from $(x,i)$ to~$(y,j)$.
That this geodesic exists for all choices of $x,y \in \R$ 
with $x \leq y$, is a simple consequence of the 
continuity of the constituent Brownian paths $B(k,\cdot).$
Further, for any given such choice of the pair $(x,y)$,  by~\cite[Lemma~$A.1$]{Patch},
 there is  an almost surely unique geodesic from  $(x,i)$ to~$(y,j)$. We denote it by $\Gamma \big[ (x,i) \to (y,j) \big]$.

\subsection{Scaled coordinates for Brownian LPP}\label{s:scaledcoordinates}

Members of the KPZ universality class enjoy scalings 
represented by the 
characteristic exponents of one-third 
and two-thirds. The one-third exponent governs the energetic fluctuation of the 
geodesic between $(0,0)$ and $(n,n)$, that is, if we write
\begin{equation}\label{e.weightfirst}
M \big[ (0,0) \to (n,n) \big] = 2n +  n^{1/3} \weight_n  \big[ (0,0) \to (0,1) \big] \, ,
\end{equation}
then the term $\weight_n  \big[ (0,0) \to (0,1) \big]$ is a random, tight in $n$ unit-order quantity. This is the 
scaled geodesic energy, which we will call {\em weight}. The exponent two-thirds appears in the fact that when 
geodesic energy $[0,\infty) \to \R: x \to 
M \big[ (0,0) \to (x,n) \big]$ is varied from $x =n$, it is 
changes of order~$n^{2/3}$ in~$x$ that result in non-trivial correlation. 

Given the above, it is natural to work in scaled coordinates under which the journey between $(0,0)$ 
and $(n,n)$ corresponds to the unit vertical journey 
between $(0,0)$ and $(0,1)$, while horizontal perturbation of the endpoint $(n,n)$ by magnitude 
$n^{2/3}$ corresponds to unit-order scaled horizontal perturbation. This will lead to the notion of  scaled 
energy, or weight,  associated to the image of any path in scaled coordinates. This is done next, namely, we 
specify the scaling map $R_n:\R^2 \to \R^2$ whose range specifies scaled coordinates; 
introduce notation for scaled paths; and specify the form of scaled energy.
\subsubsection{The scaling map.}\label{s.scalingmap}
For $n \in \N$, consider the $n$-indexed {\em scaling} 
map $R_n:\R^2 \to \R^2$ given by
\begin{equation}\label{e.scalingmap}
 R_n \big(v_1,v_2 \big) = \Big( 2^{-1} n^{-2/3}( v_1 - v_2) \, , \,   v_2/n \Big) \, .
\end{equation}

The scaling map naturally acts on subsets $C$ of $\R^2$ with
$R_n(C) = \big\{ R_n(x): x \in C \big\}$.

\subsubsection{Scaling transforms staircases to zigzags.}\label{s.staircasezigzag}
The image of any staircase under $R_n$
will be called an $n$-zigzag. The starting and ending 
points of an $n$-zigzag $Z$ are defined to be the image under $R_n$
of the corresponding points for the staircase $S$ such that $Z = R_n(S)$.

Note that the set of horizontal lines is invariant under $R_n$, while vertical lines are mapped to lines of 
gradient  $- 2 n^{-1/3}$.
Thus, an $n$-zigzag is the range of a piecewise affine 
path from the starting point to the ending point which alternately moves rightwards  along horizontal line 
segments  and northwesterly along sloping line segments with gradient  $- 2 n^{-1/3}$.

Note for example that, for given real choices of $x$ and $y$,
 a journey which in the original coordinates occurs between  $(2n^{2/3}x,0)$ and $(n  + 2n^{2/3}y,n)$ takes place 
in scaled coordinates between $(x,0)$ and $(y,1)$.
We may view the first coordinate as space and the second as time, though the latter interpretation should 
not be confused with dynamic time $t$; with this view in mind, the journey at hand is between $x$ and $y$ over the
 unit lifetime $[0,1]$.

 \subsubsection{Compatible triples}\label{s.compatible}
 Let $(n,s_1,s_2) \in \N \times \R_\leq^2$,  
 where we write $R_\leq^2 = \big\{ (s_1,s_2) \in \R^2: s_1 \leq s_2 \big\}$.
 Taking $x,y \in \R$, 
does there exist an $n$-zigzag from $(x,s_1)$ and $(y,s_2)$? 
Two conditions must be satisfied for an affirmative answer.

First: as far as the data $(n,s_1,s_2)$ is concerned, 
such an $n$-zigzag may exist only if 
 \begin{equation}\label{e.ctprop}
     \textrm{$s_1$ and $s_2$ are integer multiplies of $n^{-1}$} \, .
\end{equation}
We say that data $(n,s_1,s_2)  \in \N \times \R_\leq^2$ 
is a {\em compatible triple} if the above holds.
We will consistently impose this condition, whenever we seek to study $n$-zigzags whose lifetime is $[s_1,s_2]$.
The use of compatible triples should be thought of as a fairly minor detail. As the index $n$ increases, the 
$n^{-1}$-mesh becomes finer, so that the space of $n$-
zigzags better approximates a field of functions, defined 
on arbitrary finite intervals of the vertical coordinate, and taking values in the horizontal coordinate.

Associated to a compatible triple is the notation $\tot$, which will denote the difference $s_2 - s_1$. The law of 
the underlying Brownian ensemble $B: \R \times \Z  \to \R$ is invariant under integer shifts in the latter, curve 
indexing, coordinate. This translates to an invariance in 
law of scaled objects under vertical shifts by multiples of 
$n^{-1}$, thus making the parameter $\tot$
of far greater relevance than the individual values $s_1$ 
or $s_2$.

Returning to the above posed question, the second 
needed condition is that the horizontal coordinate of the 
unscaled counterpart of the latter endpoint must be at least the former which translates to the condition
\begin{equation}\label{e.xycond}
y -x \geq - 2^{-1}n^{1/3} \tot\, . 
\end{equation}

\subsubsection{Zigzag subpaths}\label{s.zigzagsubpaths}
Let $\phi$ denote an $n$-zigzag between elements $(x,s_1)$ and $(y,s_2)$ in $\R \times n^{-1}\Z$.
Let $(u,s_3)$ and $(v,s_4)$ be elements 
in $\phi \cap \big( [s_1,s_2] \cap n^{-1}\Z  \big)$. Suppose that $s_3 \leq s_4$ (and that $u \leq v$ if equality here holds), so 
that $(u,s_3)$ is encountered before $(v,s_4)$ in the journey along $\phi$. 
The removal of $(u,s_3)$ and $(v,s_4)$  from $\phi$ results in three connected components. The closure of 
one of these contains these two points and this closure will be denoted by $\phi_{(u,s_3) \to (v,s_4)}$. This is 
the zigzag subpath, or sub-zigzag, of $\phi$ between  $(u,s_3)$ and $(v,s_4)$.

\subsubsection{Staircase energy scales to zigzag weight.}
Let $n \in \N$ and $i,j \in \N$ satisfy $i < j$.
Any $n$-zigzag $Z$ from $(x,i/n)$ to $(y,j/n)$  is 
ascribed a scaled energy, which we will refer to as its 
weight, 
$\weight(Z) = \weight_n(Z)$, given by 
\begin{equation}\label{e.weightzigzag}
 \weight(Z) =  2^{-1/2} n^{-1/3} \Big( E(S) - 2(j - i)  - 2n^{2/3}(y-x) \Big) 
\end{equation}
where $Z$ is the image under $R_n$ of the staircase $S$.

\subsubsection{Maximum weight.}\label{s.maxweight} 
Let $n \in \N$. The quantity
 $\weight_n \big[ (0,0) \to (0,1) \big]$ specified in~(\ref{e.weightfirst}) is simply  
 the maximum weight ascribed to any $n$-zigzag from $(0,0)$ to $(0,1)$.
 
Let $(n,s_1,s_2) \in \N \times \R_\leq^2$ be a compatible triple. 
Suppose that $x,y \in \R$ satisfy 
 $y \geq x - 2^{-1} n^{1/3} \tot$. We will now define  $\weight_n \big[ (x,s_1) \to (y,s_2) \big]$ in a way such 
that this quantity equals the maximum weight of any $n$-zigzag from $(x,s_1)$ to $(y,s_2)$.
 We must set 
\begin{align}
 \,\,&   \weight_n \big[ (x,s_1) \to (y,s_2) \big]  \label{e.weightgeneral} \\
 &    =    2^{-1/2} n^{-1/3} \Big(  M \big[ (n s_1 + 2n^{2/3}x,n s_1) \to (n s_2 + 2n^{2/3}y,n s_2) \big] - 2n \tot -  2n^{2/3}(y-x) \Big) \, . \nonumber
\end{align}

The quantity  $\weight_n \big[ (x,s_1) \to (y,s_2) \big]$ may be expected to be, for given real choices of $x$ and $y$ that differ by order $\tot^{2/3}$, a unit-order random quantity; this collection of random variables is tight in the scaling parameter $n \in \N$ and in such choices of $s_1, s_2 \in n^{-1}\Z$ and $x,y \in \R$.

 \subsubsection{Highest weight zigzags are called polymers.}\label{s.polymer}
 An $n$-zigzag that attains the maximum weight given 
its endpoints will be called an $n$-polymer, or, usually, simply a polymer.
 Thus, under the scaling map, geodesics map to polymers.
As we recalled in Subsection~\ref{s.geodesics}, the 
geodesic with any given pair of endpoints is almost 
surely unique. For $x,y \in \R$ and $(n,s_1,s_2) \in \N \times \R_\leq^2$ a compatible triple, 
the almost surely unique $n$-polymer from $(x,s_1)$ to $(y,s_2)$ will be denoted by~$\rho_n \big[ (x,s_1) \to (y,s_2) \big]$; see Figure~\ref{f.scaling}.

\begin{figure}[ht]
\begin{center}
\includegraphics[height=7cm]{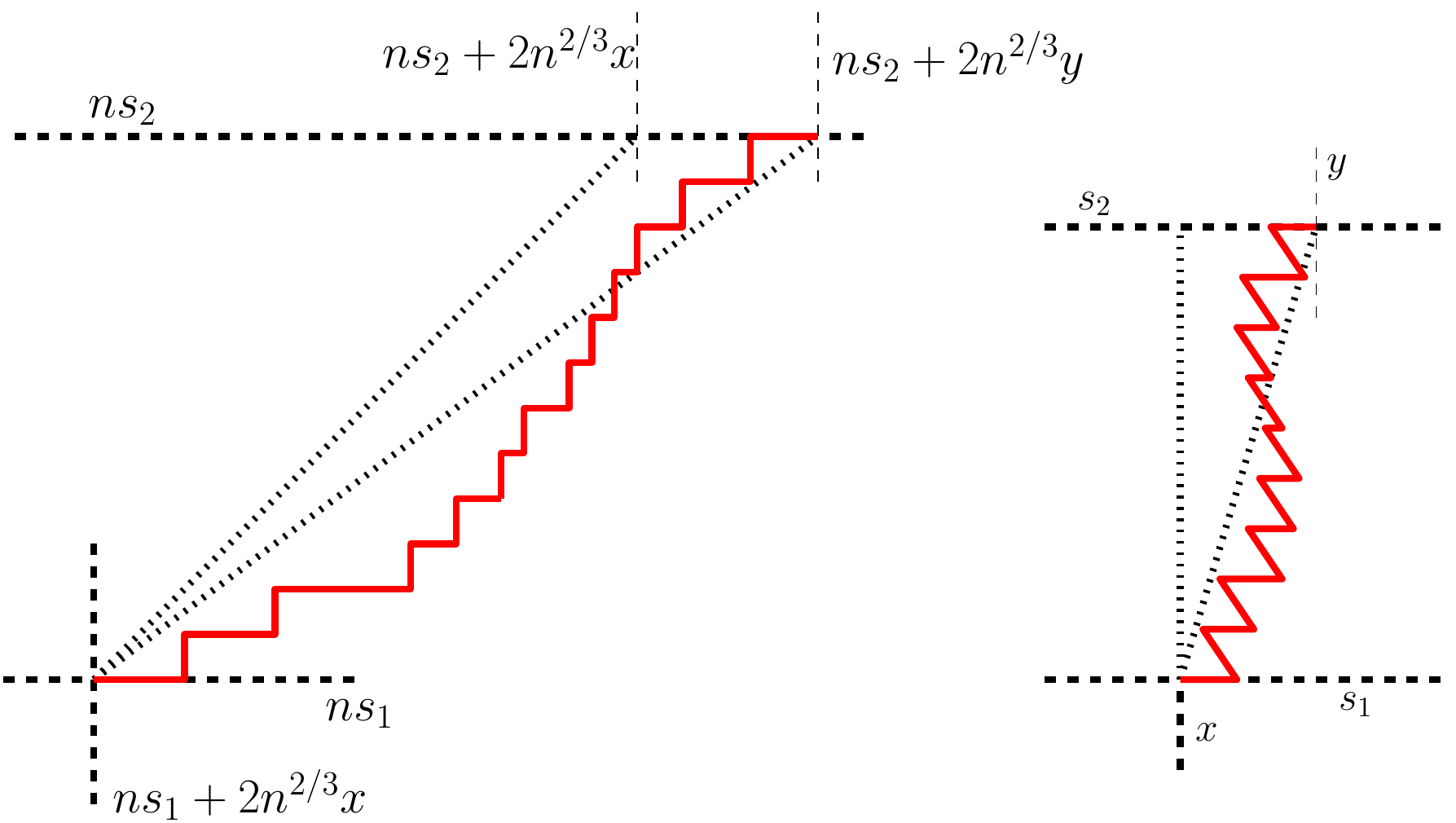}
\caption{Let $(n,s_1,s_2)$ be a compatible triple and let $x, y \in \R$. The endpoints of the geodesic in the left sketch are such that, when the scaling map~$R_n$ is applied to produce the right sketch,  the result is the $n$-polymer $\rho_n \big[ (x,s_1) \to (y,s_2) \big]$ from $(x,s_1)$ to $(y,s_2)$.
  }
\label{f.scaling}
\end{center}
\end{figure}  
Though not standard, since the term `polymer' is often used to refer to typical realizations of the path measure in LPP models at positive temperature, the above 
  usage of the term `polymer' for `scaled geodesic' is quite apt for our study, owing to the central role played by these objects.  
  
\subsubsection{Zigzags as near functions of the vertical coordinate}\label{s.zigzagfunction}
Suppose again that $\phi$ is an $n$-zigzag between 
points $(x,s_1)$ and $(y,s_2)$ in $\R \times n^{-1}\Z$.  For $s \in [s_1,s_2] \cap n^{-1}\Z$, we will write $\phi(s)$
for the supremum of values $x \in \R$ for which $(x,s) \in \phi$. This abuse of notation permits $\phi(s)$ to denote 
the horizontal coordinate of the point of departure from vertical coordinate $s$
in the journey along $\phi$ from $(x,s_1)$ to $(y,s_2)$. This convention is adopted partly because it captures 
the notion that the typical zigzags~$\phi$ we will consider---polymers or concatenations thereof---are 
closely approximable by a real-valued function of the vertical coordinate $s \in [s_1,s_2]$, at least when $n$ 
is high---indeed, the maximum length of the horizontal line segments in an $n$-polymer is readily seen to 
decay to zero in~$n$ with high probability. (Our few cliffs'  Theorem~\ref{t.notallcliffs} quantifies this assertion.)

\subsection{Brownianity and twin peaks' rarity for the routed weight profile}\label{s.browniantwinpeaks}

To specify the routed weight profile for scaled Brownian LPP,
let $n \in \N$ and $a \in n^{-1}\Z \cap (0,1)$. 
For $x \in \R$, let $\Psi_n(x,a)$ denote the set of $n$-zigzags $\phi$ that begin at $(0,0)$; end at $(0,1)$; and for which $x = \sup \big\{ z \in \R: (z,a) \in \phi \big\}$.
In other words, $\Psi_n(x,a)$ comprises those $n$-zigzags on the route from $(0,0)$ to $(0,1)$ whose point of departure from level $a$ occurs at $(x,a)$. We set $Z_n(x,a)$ equal to the supremum of the weights of elements of~$\Psi_n(x,a)$. In this way, the {\em routed weight profile} $Z_n(\cdot,a): \R \to \R$
records the maximum weight of zigzags that are constrained to exit level $a$ at a given horizontal location. We did not allude to this exit constraint in the heuristic discussion of Subsection~\ref{s.brownianity}: in the microscopic model, where $n \in \N$ is finite, this definition renders the maximizer location $M \in \R$ at which $Z_n(M,a) = \sup_{x \in \R} Z_n(x,a)$
unique, while maintaining that $Z_n(M,a) = \weight_n \big[ (0,0) \to (0,1) \big]$; in Lemma~\ref{l.routedprofile}, we will moreover see that, in the counterpart expression to~(\ref{e.ztwoweight}), the two right-hand terms are independent, even when $n$ is finite, when the present definition is adopted.

Next we make precise the notion of comparison that we will make to Brownian motion.
\begin{definition}\label{d.verybrownian}
Let $K \in \R$ and $d > 0$. Let $I$ denote the interval $[K-d,K+d]$.
We denote by $\mc{C}_{0,*}( I , \R )$
the space of continuous functions $f:I \to \R$ such that $f(K-d) = 0$.
For $\nu > 0$, let $\mc{B}_{0,*}^{\nu;I}$ denote the law on this function space given by Brownian motion $B:I \to \R$, $B(K-d) = 0$, of diffusion rate $\nu$.

{Let $g$ and $G$ be positive real parameters; and let $m \in \N$. A continuous random function
$X:I \to \R$ defined under a law~$\PP$  
is said to be $\big(g,G,\ipdval,m,\nu\big)$-Brownian if the following holds.} Let $A$ denote a Borel measurable subset of $\mc{C}_{0,*}( I , \R )$. Set $\eta = \mc{B}_{0,*}^{\nu;I}(A)$. Then the condition that $e^{-gm^{1/12}} \leq \eta \leq g \wedge e^{-G \ipdval^6}$ implies that
\begin{equation}\label{e.verybrownian}
\PP \Big( I \to \R: x \to X(x) - X(K-d) \, \, \textrm{belongs to $A$} \Big) \, \leq \, \eta \cdot G \exp \Big\{ G \ipdval \big( \log \eta^{-1} \big)^{5/6} \Big\} \, .
\end{equation}
\end{definition}

In heuristic overview, we compared the routed weight profile to Brownian motion of rate two on $[-1,1]$. Our rigorous result crucially relying on the main result in \cite{DeBridge} makes the comparison on any compact interval:
 after the addition of the linear term
 $2^{1/2}\big( a(1-a) \big)^{-1} Rx$, the  profile $x \to Z_n(x,a)$  is very similar to Brownian motion of rate two,  in the locale of any given $R \in \R$. 
\begin{theorem}\label{t.brownianroutedprofile}
There exist positive constants $c$, $G$ and $g$ such that the following holds. 
Let $n \in \N$ and $a \in n^{-1}\Z \cap (0,1)$. Suppose that $\vert R \vert \leq 2^{-1}c n^{1/9} \big( a \wedge (1 - a - n^{-1}) \big)^{7/9}$. Let $\ell \in \R$ satisfy $-2^{-1}n^{1/3}a \leq R - \ell$ and $R + \ell \leq  2^{-1}n^{1/3}(1 - a- n^{-1})$. The process $Z_n(\cdot,a):[R-\ell,R+\ell] \to \R$ may be expressed in the form 
$$
Z_n(x,a) = X(x) - 
\Big( 2^{1/2}\big( a(1-a) \big)^{-1} R + \e \Big) x \, , 
$$
where $X:[R-\ell,R+\ell] \to \R$ is $\big( g,G' \ell^6,   \ell, \min \{ a,1-a \} n ,2 \big)$-Brownian. Here, the constant $G'$ is up to an absolute positive factor equal to $G^{17/6} g^{-5/6} \big( a \wedge (1-a - n^{-1}) \big)^{-34/3}$; and   
 $\e = \e(a,R,n)$, given by $\e = 2^{1/2} \big( (1 - a - n^{-1})(1-a) \big)^{-1} R  n^{-1}$, is an error term without dependence on~$x$. 
\end{theorem}

Given the above quantitative comparison to Brownian motion, the next result presents our conclusion regarding the rarity  of {\em twin peaks}. The probability that there exists $x \in \R$ such that $Z_n(x,a)$
rivals the maximum value of $Z_n(\cdot,a)$, with $Z_n(x,a)$ being less than this maximum by a small multiple~$\hata$ of the square-root distance $\big( x - \rho_n(a) \big)^{1/2}$ is bounded above by the product of $\hata$ and a lower-order correction $\exp \big\{ \Theta(1) \big( \log \hata^{-1} \big)^{5/6} \big\}$; a further factor of $e^{-\Theta(1) R^2 \ell}$ penalizes the maximizer for being of a large order~$R > 0$.

\begin{theorem}\label{t.nearmax}
For $K$ any compact interval of $(0,1)$,
there exist positive constants $H = H(K)$ and $h = h(K)$ and an integer $n_0 = n_0(K)$ such that
the following holds.
Let $n \in \N$, $R \in \R$, $\ell \geq 1$, $\ell' > 0$, $a \in n^{-1}\Z \cap K$, $\hata > 0$ and $\e > 0$.
 Suppose that $n \geq n_0$,
 $\vert R \vert \leq h n^{1/9}$,
$\ell \in (3\e, h n^{1/\macrobig})$ and $\ell' \in (3\e,\ell]$.
Denoting $\hata \wedge 1$ by $\hata_*$, we have that  
\begin{eqnarray*}
& & \PP \Big( M \in [R - \ell/3,R+\ell/3]  \, , \, \sup_{x \in \R: \vert x - M \vert  \in [\e,\ell'/3]} \big( Z_n(x,a) +  \hata (x - M)^{1/2} \big)  \geq Z_n(M,a) \Big) \\
 & \leq & \log \big( \ell' \e^{-1} \big) \max \Big\{ \hata_* \cdot 
 \exp \big\{  - hR^2 \ell + H \ell^{\macroseventeen} \big( 1 + R^2 + \log \hata_*^{-1}  \big)^{5/6} \big\} ,   \exp \big\{ - h n^{1/12}   \big\} \Big\} \, , 
\end{eqnarray*}
where $M$ denotes $\rho_n(a)$, the almost surely unique maximizer of $x \to Z_n(x,a)$.
\end{theorem}
The right-hand factor of  $\log \big( \ell' \e^{-1} \big)$ reflects a union bound indexed by dyadic scales intersecting the interval $[\e,\ell'/3]$. There is no non-smallness condition on the scale~$\e$. The probability upper bound  $\exp \big\{ - h n^{1/12}   \big\}$ becomes operative for extremely small values of $\sigma$.

\subsection{Robust modulus of continuity for the geometry and weight of polymers}

The next result offers a quantified prelimiting expression for the $z \to z^{2/3} \big( \log z^{-1} \big)^{1/3}$ modulus of continuity for polymer geometry in a fashion that is uniform as the polymer's endpoints vary over a compact region and that holds on all scales above the microscopic separation~$n^{-1}$.\\\\

 \begin{theorem}\label{t.toolfluc} 
\leavevmode
 \begin{enumerate}
\item There exist positive $H$,  $h$ and $r_0$, and $n_0 \in \N$, such that, when $n \in \N$ satisfies $n \geq n_0$, $k \in \N$ satisfies $2^k \leq h n$ and $r \in \R$ satisfies  $r_0 \leq r \leq n^{1/10}$, it is with probability at least $1 - H \exp \big\{ - h r^3 k \big\}$ that the following event occurs.
Let $x,y \in \R$ be of absolute value at most~$r$.  Let $h_1,h_2 \in n^{-1}\Z \cap [0,1]$ satisfy $h_{1,2} \in (2^{-k-1},2^{-k}]$ and
let $u,v \in \R$ be such that $(u,h_1)$ and $(v,h_2)$ belong to $\rho_n \big[ (x,0) \to (y,1) \big]$. 
Then
$$
\big\vert v -u \big\vert \leq H h_{1,2}^{2/3} \big( \log (1 + h_{1,2}^{-1}) \big)^{1/3}r \, .
$$
\item  There exist positive $G$, $H$,  $h$ and $r_0$, and $n_0 \in \N$, such that, when $n \in \N$ satisfies $n \geq n_0$, and $r \in \R$ satisfies  $r_0 \leq r \leq n^{1/10}$, it is with probability at least $1 - H n^{- h r^3}$  that the following event occurs.
As above, let $x,y \in \R$ be of absolute value at most~$r$, and let $u,v \in \R$ be such that $(u,h_1)$ and $(v,h_2)$ belong to $\rho_n \big[ (x,0) \to (y,1) \big]$. 
 Consider any $h_1,h_2 \in n^{-1}\Z \cap [0,1]$ that satisfy  $h_{1,2} < H n^{-1}$.
Then
$$
\big\vert v -u \big\vert \leq G n^{-2/3} ( \log n )^{1/3}r
 \, .
$$
\end{enumerate}
 \end{theorem}
As a special case, we gain control on the maximum fluctuation of such polymers.

 \begin{corollary}\label{c.lateral}
There exist positive $H$,  $h$ and $r_0$, and $n_0 \in \N$, such that, when $n \in \N$ satisfies $n \geq n_0$, and $r \in \R$ satisfies  $r_0 \leq r \leq n^{1/10}$, it is with probability at least $1 - H \exp \big\{ - h r^3\big\}$ that the following holds.
Let $x,y \in \R$ be of absolute value at most $r$. If $(u,h') \in \R \times \big( n^{-1}\Z \cap [0,1] \big)$ lies in   $\rho_n \big[ (x,0) \to (y,1) \big]$, then  
$\vert u \vert \leq H r$.
 \end{corollary}
 A control, similar to that offered by Theorem~\ref{t.toolfluc}, on the  $z \to z^{1/3} \big( \log z^{-1} \big)^{2/3}$ modulus of continuity for polymer weight is available.
 \begin{theorem}\label{t.weight}
 \leavevmode
 \begin{enumerate} 
 \item
There exist positive $H$,  $h$ and $r_0$, and $n_0 \in \N$, such that, when $n \in \N$ satisfies $n \geq n_0$; $k \in \N$ satisfies $2^k \leq hn$;  and $r \in \R$ satisfies $r \geq r_0$, it is with probability at least $1 - H \exp \big\{ - h r^3 k \big\}$ that the following occurs.
Let $h_1,h_2 \in n^{-1}\Z \cap [0,1]$ satisfy $h_{1,2} \in ( 2^{-k-1}, 2^{-k}]$ and 
  $r \leq (n h_{1,2})^{1/64}$;
let $x,y \in \R$ be of absolute value at most $r$; and let
$u,v \in \R$ be such that $(u,h_1)$ and $(v,h_2)$ belong to $\rho_n \big[ (x,0) \to (y,1) \big]$. Then
$$
\big\vert \weight_n \big[ (u,h_1) \to (v,h_2) \big] \big\vert \leq H^2 r^2 \cdot
 h_{1,2}^{1/3} \big( \log  h_{1,2}^{-1} \big)^{2/3} 
\, .
$$
\item 
There exist positive $G$, $H$ and $r_0$, and $n_0 \in \N$, such that, when $n \in \N$ satisfies $n \geq n_0$, and $r \in \R$ satisfies $r \geq r_0$, it is with probability at least
 $1 - H n^{- h r^3}$  that, if
  $h_{1,2} < H n^{-1}$;  $x,y \in \R$  have absolute value at most $r$; and 
$u,v \in \R$ are such that $(u,h_1)$ and $(v,h_2)$ belong to $\rho_n \big[ (x,0) \to (y,1) \big]$; 
 then $\big\vert \weight_n \big[ (u,h_1) \to (v,h_2) \big] \big\vert \leq G r^2 \cdot n^{-1/3}  ( \log n )^{2/3}$. 
 \end{enumerate}
 \end{theorem}

 We present two further results, which  emerge in the course of the proof of Theorem~\ref{t.toolfluc}. The former offers uniform control on the maximum fluctuation of polymers, in which we permit to vary  the polymer endpoints and the moment during the lifetime of the polymer at which fluctuation is measured.  The latter proves the rarity of short polymers of extreme weight that begin and end in a unit-order region.

In order to  state the first result, we specify a measure of the fluctuation of the polymer $\rho_n \big[ (x,s_1)\to (y,s_2) \big]$ at the intermediate moment $h \in [s_1,s_2] \cap n^{-1} \Z$, measuring the horizontal distance between the polymer at this height $h$ relative to the height-$h$ location $\tfrac{s_2 - h}{\tot} x + \tfrac{h- s_1}{\tot} y$ of the line that interpolates $(x,s_1)$ and $(y,s_2)$. We set 
\begin{equation}\label{e.flucdef}
\fluct_n \big[ (x,s_1) \to (y,s_2) ; h \big] 
  =   \sup \Big\{ \big\vert u - \tfrac{s_2 - h}{\tot} x - \tfrac{h- s_1}{\tot} y \big\vert:  u \in \R \, , \, (u,h) \in \rho_n \big[ (x,s_1) \to (y,s_2) \big] \Big\} \, .
\end{equation}
The typical order of this quantity is $\lambda^{2/3}$, where $\lambda$ equals $(h - s_1) \wedge (s_2 - h)$, with $\wedge$ denoting minimum.

\begin{theorem}\label{t.deviation}
Let $K > 0$, $r \geq r_0$, $a \in (0,1/4]$, $n \in \N$ and $s_1,s_2 \in n^{-1}\Z \cap [0,1]$ satisfy $s_1 \leq s_2$; $n \tot a$ and $n \tot(1-a)$ are at least $\Theta(1)$; $K a^{1/3} \leq \Theta(1)$; and $\vert K \vert \leq (n \tot)^{2/3}$.
Then
\begin{equation}\label{e.deviationvariant} 
  \PP \Big( \sup \fluct_n \big[ (x,h_1) \to (y,h_2) ; h \big] \geq  r (a \tot)^{2/3}\big(\log a^{-1} \big)^{1/3} \Big) \leq \Theta(1) K^2 a^{-10/3}  a^{\Theta(1) r^3 } \, ,
\end{equation}
where the supremum is taken over $x,y \in [-K,K]\cdot s_{1,2}^{2/3}$, $h_1 \in n^{-1}\Z \cap [s_1,s_1 + \tot/3]$, $h_2 \in n^{-1}\Z \cap [s_2 - \tot/3,s_2]$ and 
$h \in n^{-1}\Z$ such that $\tfrac{h-h_1}{h_{1,2}} \in [a,2a] \cup [1-2a,1-a]$.
\end{theorem}

To express our result on weights, 
let $\low(\zeta,\ell,L,M)$ denote the event that 
$$
 \tot^{-1/3} \weight^{\cup}_n \big[  (x,s_1) \to  (y, s_2) \big]  
$$ 
is less than $-\zeta$ for some pair $(x,s_1), (y,s_2) \in \R  \times  n^{-1}\Z \cap [0,1]$ with $|x| \vee |y| \le M$, $\vert x - y \vert \leq 2^{-2\ell/3} L$ and $s_{1,2} \in (2^{-\ell-1},2^{-\ell}]$. Let $\high(\zeta,\ell,L,M)$ denote the event that the displayed quantity exceeds~$\zeta$  for some such pair. 
\begin{proposition}\label{p.onepoint}  
When $n \geq \Theta(1) 2^\ell$, $L \leq \Theta(1) (n 2^{-\ell})^{1/46}$, $\Theta(1) \leq \zeta \leq \Theta(1) (n 2^{-\ell})^{1/30}$ and $M>0$,
$$
\PP \Big(  \low\big(\zeta,\ell,L,M\big) \cup  \high\big(\zeta,\ell,L,M\big) \Big) \, \leq \, \Theta(1)  2^{5\ell/3}  M L
\exp \big\{ - \Theta(1) \zeta ^{3/2} \big\}
 \, .
$$ 
\end{proposition}

\subsection{Slim pickings for slender excursions}\label{s.slimpickings}
Here we present results asserting that zigzags constrained to stay close to a deterministic path or the polymer are typically uncompetitive in weight.   

Consider a given zigzag $\phi$ 
from $(0,0)$ to $(0,1)$. Let $(x,s_1),(y,s_2) \in \R\times n^{-1}\Z \cap [0,1]$, $s_1 < s_2$, be two points, neither of which necessarily lies in $\phi$.
A zigzag 
$\psi$ from $(x,s_1)$ to $(y,s_2)$ that is disjoint from~$\rho_n$ will be called an excursion, even though this name might more properly be reserved for the case where $\psi$'s endpoints lie in $\phi$.
For an excursion $\psi$, 
consider the  set of $s \in [s_1,s_2] \cap n^{-1}\Z$ for which, to use the language of Subsection~\ref{s.zigzagfunction}, 
$\vert \psi(s) - \phi(s) \vert$ is at most $\tot^{2/3}  \tza$, where $\tza > 0$ is given.
If this set has  cardinality at least $(1-\chi) \big\vert [s_1,s_2] \cap n^{-1}\Z \big\vert$
and contains the values $s_1$ and $s_2$, then the excursion $\psi$  is called $(\phi,\tza,1 - \chi)$-close.
 The parameter $\tza > 0$ measures 
constraint in movement beyond the factor $\tot^{2/3}$ that is dictated by KPZ scaling, and our notion of closeness indicates that this constraint is satisfied at a high percentage of levels in $[s_1,s_2]$.
See Figure~\ref{f.slenderexcursion}.

\begin{figure}[t]
\centering{\epsfig{file=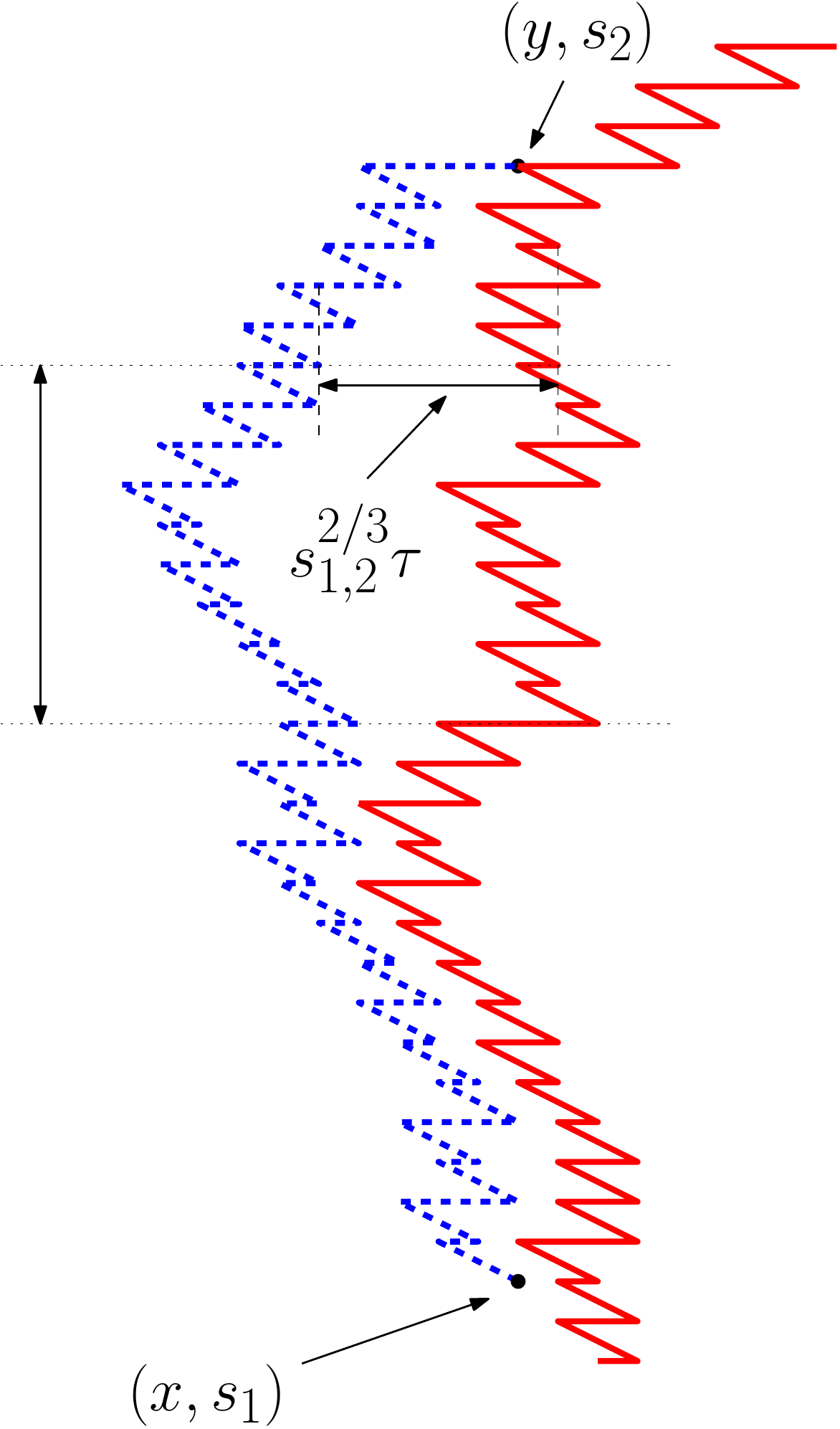, scale=0.4}}
\caption{The dashed zigzag $\psi$ from $(x,s_1)$ to $(y,s_2)$ is an excursion around the bold zigzag $\phi$.
 The vertical double-arrowed line indicates the set of coordinates---an interval in this instance---at which departures from vertical levels differ by more than the quantity appearing in the definition of a slender excursion. Indeed, if such levels exclude $s_1$ and $s_2$ and constitute a fraction less than $\chi$ of all levels in $[s_1,s_2]$, the excursion is $(\phi,\tza,1-\chi)$-close.}\label{f.slenderexcursion}
\end{figure}

The supremum of the weights of $(\phi,\tza,1-\chi)$-close excursions will be denoted by 
$$
\weight_n \big[  (x,s_1) \to (y,s_2)  \, ; \, (\phi, \tza,1 - \chi)\textrm{-close}  \big] \, .
$$ 

 For $\ell \in \N$ and $d_0 >0$, let $\mathsf{LowSlenderExcursion}\big( \ell, \tza,1-\chi;  \phi \big)$
denote the event that   
$$ 
\sup \, \tot^{-1/3} \weight_n \big[  (x,s_1) \to (y,s_2) \, ; \, (\phi,\tza,1-\chi)\textrm{-close}  \big] \leq - 
  \redconst {\tza}^{-1}  \, ,$$
 where the supremum is taken by varying the points $(x,s_1) , (y,s_2) \in \R \, \times \, [0,1] \cap n^{-1}\Z$ over choices such that $2^{-1 - \ell} \leq \tot \leq 2^{-\ell}$.

Our first assertion of slender slim pickings concerns excursions about the polymer $\rho_n$.
\begin{theorem}\label{t.slenderpolymer} There exist  constants $d_0,C>0$ such that we may find $\chi_0 \in (0,1)$, $d_2 > 0$ and $n_0 \in \N$ 
for which  $\chi \in (0,\chi_0)$  and $n \geq n_0$ imply that when $ \tza^{-1/4} > C\log n$, and  
 $\ell \in \N$ satisfies $2^{\ell}\leq n \tza^{40}$,

$$ 
\PP \Big( \neg \,  \mathsf{LowSlenderExcursion}(\ell,\tza , 1 - \chi;  \rho_n) \Big) \leq  \exp \Big\{ - d_2 \tza^{-1/2}\Big\} \, .
$$ 
\end{theorem}

Theorem~\ref{t.slenderpolymer} will follow from our second assertion:
any zigzag that is constrained to stay close to a deterministic zigzag $\phi$ is typically uncompetitive in weight. 
To this end, for  $\phi$ any $n$-zigzag from $(0,0)$ to $(0,1)$, let 
$\weight^*_n \big[  (x,s_1) \to (y,s_2)  \, ; \, (\phi, \tza,1 - \chi)\textrm{-close}  \big]$ 
 denote the supremum of the weights of $(\phi,\tza,1-\chi)$-close zigzags~$\psi$. Note that  $\psi$ varies over a class of zigzags; it is not assumed to be  an excursion about $\phi$, and may intersect $\phi$.
The superscript asterisk in the new notation indicates this distinction. 

Let $\mathsf{LowSlenderWeight}^*(\ell, \tza,1-\chi;  \phi)$
denote the event specified above, with $\weight_n$ replaced by $\weight_n^*$.
 
The upcoming assertion, that $\mathsf{LowSlenderWeight}^*$  is typical, will suppose a certain regularity on $\phi$.

We say that an $n$-zigzag $\phi$ from $(0,0)$ to $(0,1)$ is {\em $\mathfrak{R}$-regular} if, whenever 
 $h_1,h_2 \in n^{-1}\Z \cap [0,1]$ and $u,v \in \R$ satisfy $(u,h_1),(v,h_2) \in \phi$, we have that
\begin{equation}\label{regular-criterion}
\big\vert v -u \big\vert \leq h_{1,2}^{2/3} \mathfrak{R} \, .
\end{equation}

\begin{theorem}\label{t.nocloseness}  There exist constants $d_0,C>0$ such that
we may find $\chi_0 \in (0,1)$, $d_2 > 0$ and $n_0 \in \N$ 
for which  $\chi \in (0,\chi_0)$  and $n \geq n_0$ imply that,  when $ \tza^{-1/4} > C\log n$ and  
 $\ell \in \N$ satisfies $2^{\ell}\leq n \tza^{40}$,
$$ 
\PP \Big( \neg \,  \mathsf{LowSlenderWeight}^*(\ell,\tza , 1 - \chi;  \phi) \Big) \leq  \exp \big\{ - d_2 \tza^{-3/2} \big\}
$$ 
for any {\em given} zigzag $\phi$ from $(0,0)$ to $(0,1)$ which is $\theta^{-1/4}$-regular.
\end{theorem}
Theorem \ref{t.slenderpolymer} will follow from Theorem \ref{t.nocloseness} and an  application of the FKG inequality to the effect that conditioning on $\rho_n$ is negative on the environment exterior to $\rho_n$.  This is why  we consider excursions, namely zigzags that are disjoint from $\rho_n$. Theorem \ref{t.nocloseness} is then applicable because 
  $\mathsf{LowSlenderExcursion}(\ell, \tza , 1 - \chi;  \rho_n)$ is a decreasing event on this exterior environment.  Theorem~\ref{t.toolfluc} will show that $\rho_n$ typically is suitably regular.

\subsection{There are few cliffs along the geodesic}

The geodesic~$\Gamma_n$ from $(0,0)$ to $(n,n)$ progresses in a roughly diagonal fashion, even at the microscopic scale. 
We now state this {\em few cliffs'} assertion, using unscaled coordinates because these are suited to discussing the microscopic scale.

For $n \in \N$, let $\Gamma_n \subset [0,n]^2$ denote the almost surely unique staircase of maximum energy  with starting and ending points $(0,0)$ and $(n,n)$ in static Brownian LPP. For $A \in \N$ a positive integer, we divide the vertical coordinate interval $[0,n]$ into consecutive subintervals of lengths $A$, as well as a remaining subinterval of shorter length if need be. We set $m$ to be the greatest integer {\em strictly} less than $n/A$, so that there are  $m$ subintervals of length~$A$. There is also one remaining interval, whose length is at least one and at most $A$.

We record a sequence $\big\{ X_i: i \in \llbracket 0, m+1 \rrbracket \big\}$ of horizontal coordinates of departure of $\Gamma_n$ from the horizontal borders of  the consecutive strips indexed by the vertical subintervals that we have just defined. Formally, when $i \in \llbracket 0 , m \rrbracket$, $X_i$ is equal to the supremum of those $x \in [0,n]$ for which $(x,iA) \in \Gamma_n$; while for $i = m+1$, $X_i = \sup \big\{ x \in [0,n]: (x,n) \in \Gamma_n \big\}$, so that $X_i = n$.

Now set, for each such index $i$, $Z_i = \lfloor X_i \rfloor$. The sequence  $\big\{ Z_i: i \in \llbracket 0, m+1 \rrbracket \big\}$ is a non-decreasing list of integers lying in $\llbracket 0,n\rrbracket$ 
that offers a unit-scale coarse-grained description of the horizontal progress of the staircase $\Gamma_n$
as consecutive vertical milestones at separation $A$ are passed.
This description is equally captured by the difference function
\begin{equation}\label{e.zdifferencefunction}
 \Psi: \llbracket 0,m \rrbracket \to \llbracket 0, n \rrbracket \, , \, \Psi(i) = Z_{i+1} - Z_i \, .
\end{equation}
Note that $\sum_{i=0}^m \Psi(i) = Z_{m+1} - Z_0$ is at most $n$. 
We now specify a set $\mc{I}$ of indices marking slow horizontal advance---cliffs in the graph of $\Gamma_n$; we set $\mathcal{I}(\Gamma_n)$ equal to the set of $i \in \llbracket 0,m \rrbracket$ for which $\Psi(i)$ is at most two.
\begin{theorem}\label{t.notallcliffs}
There exist $A > 0$, $\alpha_0 \in (1/2,1)$, $h > 0$ and $n_0 \in \N$, such that, when $n \geq n_0$ and $\alpha \geq \alpha_0$ satisfies $\alpha m \in \N$ where $m = \lfloor n/A \rfloor$, 
$$
 \PP \Big( \big\vert \mc{I}(\Gamma_n) \big\vert < \alpha m \Big) \geq 1 - e^{-h n} \, .
$$
\end{theorem}

{\bf Acknowledgments.} The authors thank Milind Hegde for several useful discussions. The first author is partially supported by NSF grant DMS-$1855688$, NSF CAREER Award
DMS-$1945172$, and a Sloan Research Fellowship.  
The second author is supported by NSF grant DMS-$1512908$ and by a Miller Professorship from the Miller Institute for Basic Research in Science at U.C. Berkeley.

\section{Brownian regularity and twin peaks}\label{s.twinpeaks} 

Here we prove the Brownianity and twin peaks' rarity assertions Theorems~\ref{t.brownianroutedprofile} and~\ref{t.nearmax} concerning the routed weight profile $x \to Z_n(x,a)$. We begin by noting 
some similiarities between the specification~(\ref{e.ztwoweight}) of this profile in heuristic discussion and the actual definition at the start of Section~\ref{s.browniantwinpeaks}.
\begin{lemma}\label{l.routedprofile}
Set $\aplus = a+ n^{-1}$ and $\xminus = x - 2^{-1}n^{-2/3}$.
\begin{enumerate}
\item The routed weight profile is given by a sum of independent weight profiles,
\begin{equation}\label{e.routedweightprofile}
 Z_n(x,a) \, = \,  \weight_n \big[(0,0) \to (x,a) \big] \, + \, \weight_n \big[(\xminus,\aplus) \to (0,1) \big]  \, . 
\end{equation}
\item Almost surely, the maximizer of $Z_n(\cdot,a)$, namely the value of $x \in \R$ for which $Z_n(x,a)$ equals the supremum of $Z_n(z,a)$ over $z \in \R$, is unique and equals $\rho_n(a)$.
\end{enumerate}
\end{lemma}
{\bf Proof: (1).} 
Let $\psi$ denote an $n$-zigzag that begins at $(0,0)$, ends at $(0,1)$, and for which $x = \sup \big\{ z \in \R: (z,a) \in \psi \big\}$. Let $\psi^-$ denote the initial zigzag of $\psi$ that ends at $(x,a)$. Note that $\psi$ reaches $\R \times \{ a + n^{-1} \}$ at $(\xminus,\aplus)$. Let $\psi^+$ denote the final sub-zigzag of $\psi$ that begins at $(\xminus,\aplus)$. Thus, $\weight_n (\psi) = \weight_n (\psi^-)+ \weight_n (\psi^+)$. By definition, $Z_n(x,a)$ equals the supremum of  $\weight_n (\psi)$ over such $\psi$. We see that $Z_n(x,a)$ is at most the right-hand side of~(\ref{e.routedweightprofile}). But equality may be obtained by varying $(\psi^-,\psi^+)$ subject to the endpoint constraints that specify this pair.
Moreover, the two right-hand terms in~(\ref{e.routedweightprofile}) are independent because they are respectively measurable with respect to randomness indexed by the disjoint regions $\R \times [0,a]$ and $\R \times [a+n^{-1},1]$. \\
{\bf (2).} The polymer $\rho_n$ is almost surely unique by~\cite[Lemma~$4.6(1)$]{Patch}.
Since 
$\rho_n(a)$ is by definition the location of departure of the polymer $\rho_n$ from $\R \times \{ a \}$, we see that it is the maximizer of $x \to Z_n(x,a)$. \qed

The notation $\aplus$ and $\xminus$ is adopted henceforth. It reflects the two denoted quantities being merely microscopically perturbed copies of $a$ and $x$.

The proof of Theorem~\ref{t.nearmax} will harness 
Theorem~\ref{t.brownianroutedprofile}.
The derivation of the latter result is comprised of four steps; two further steps will yield  the former.
\begin{enumerate}
\item The two right-hand weights in the formula~(\ref{e.routedweightprofile}) for the routed weight profile may be viewed as functions of the variable $x \in \R$. In a simple $a$-dependent change of coordinates, we will present {\em normalized} counterparts to these profiles. These normalized profiles are Brownian of rate one on the unit scale in a sense that is uniform in $a$ and $n$; they are globally governed by the curvature of a shared parabola, $x \to - 2^{-1/2}x^2$.
\item  We will recall from~\cite{NonIntPoly} that any normalized profile may be embedded via the Robinson-Schensted-Knuth correspondence as the uppermost curve of a {\em regular} ensemble. The latter object is a random ordered system of continuous curves which enjoy an attractive probabilistic resampling, the Brownian Gibbs property, alongside certain boundary conditions.  
\item We will recall from \cite{DeBridge} that the curves in a regular ensemble enjoy strong similarity to Brownian motion of rate one, in the sense of Definition~\ref{d.verybrownian}. 
\item The profile $x \to Z_n(x,a)$ is thus seen via~(\ref{e.routedweightprofile})  to be the sum of two independent processes that bear a demanding comparison to standard Brownian motion; it would thus seem---and we will prove---that this profile withstands such a comparison to Brownian motion of rate two. In this way, we will obtain Theorem~\ref{t.brownianroutedprofile}.
\item Twin peaks are rare for Brownian motion (of rate two).
\item Thus, and as Theorem~\ref{t.nearmax} asserts, they are also rare for the profile $x \to Z_n(x,a)$. 
\end{enumerate}
The six ensuing subsections give rigorous renderings of these respective steps.

\subsection{Forward and backward weight profiles}\label{s.encode}

By the formula~(\ref{e.routedweightprofile}), the routed weight profile is exhibited as a sum of two independent random processes. The first may be labelled `forward', because the origin $(0,0)$ is fixed, and the spatial variable $x$ is attached to the more advanced height $a \in (0,1)$. The latter process may be called `backward', because the fixed point $(0,1)$ is more advanced that the height $a^+ = a + n^{-1}$ of the endpoint $(x^-,a^+)$ that varies with $x$. 

It is useful to discuss further the forward and backward processes, and we consider a given compatible triple $(n,s_1,s_2) \in \N \times \R^2_\leq$  in order to do so.
Consider the forward weight profile, given by
$$
\weight_n \big[ (0,s_1) \to (\cdot,s_2)] : \big[  - 2^{-1} n^{1/3} \tot  , \infty\big) \to \R \, ;
$$
and the backward profile,
$$
\weight_n \big[ (\cdot,s_1) \to (0,s_2)] :  \big( - \infty ,   2^{-1} n^{1/3} \tot \big] \to \R \, .
$$
 
It is valuable to vividly picture these two profiles. 
Each locally resembles Brownian motion but globally follows the contour of the parabola $- 2^{-1/2} (y-x)^2 \tot^{-1}$ 
as a function of $y$ or $x$ in the forward or backward case. Each profile adopts values of order $\tot^{1/3}$
when $x$ and $y$ differ by an order of~$\tot^{2/3}$. More negative values, dictated by parabolic curvature, are witnessed outside this region. This description holds sway in a region that expands from the origin as the parameter $n$ rises.

Clearly, then, our profiles have fundamental differences according to the value of $\tot$: sharply peaked ensemble curves when $\tot$ is small, and much flatter curves when $\tot$ is large.
 A simple further parabolic transformation will serve to put the profiles on a much more equal footing. 
 Since the profiles are already scaled objects, we will use the term `normalized' to allude to the newly transformed counterparts.

That is, we define the {\em normalized} forward profile 
$$
 \scaledle_{n;(x,s_1)}^{\uparrow;s_2}: \big[- 2^{-1} (n \tot)^{1/3}  , \infty \big) \to \R \, ,
$$
setting,
 for $z \geq - 2^{-1} n^{1/3} \tot^{-2/3}$, 
\begin{equation}\label{e.normfwdprofile1}
\scaledle_{n;(x,s_1)}^{\uparrow;s_2}(  z )
 = \tot^{-1/3} \, \weight_n \big[(x,s_1) \to (x + \tot^{2/3}z,s_2) \big] \, .
\end{equation}
The normalized backward profile  
$\scaledle_{n;s_1}^{\downarrow;(y,s_2)}: \big(-\infty, 2^{-1} (n \tot)^{1/3}   \big] \to \R$
is specified by setting 
\begin{equation}\label{e.normbckprofile1}
\scaledle_{n;s_1}^{\downarrow;(y,s_2)}( z )
 = \tot^{-1/3} \, \weight_n \big[(y + \tot^{2/3}z,s_1) \to (y,s_2) \big]
\end{equation}
for  $z \leq  2^{-1} n^{1/3} \tot^{-2/3}$.

The new curves locally resemble Brownian motion as before, but they have been centred and squeezed so that now the parabola that dictates their overall shape is $- 2^{-1/2} z^2$.
This picture is accurate in a region that expands as the parameter $n \tot$ rises.

\subsection{Brownian Gibbs line ensembles}

The Robinson-Schensted-Knuth correspondence permits any given forward and backward weight profile to be embedded as the uppermost curve in an ordered system---or line ensemble---of random continuous curves that enjoy an attractive and valuable probabilistic resampling called the Brownian Gibbs property. 
The notion of a  Brownian Gibbs line ensemble was introduced in~\cite{AiryLE} to capture a system of ordered curves that arise by conditioning Brownian motions or bridges on mutual avoidance. The precise definition is not logically needed in this article, but we offer an informal summary next, and then indicate how our normalized profiles satisfy this definition.

\subsubsection{An overview}  
Let $n \in \N$ and 
let $I$ be a closed interval in the real line.
A $\intint{n}$-indexed line ensemble defined on $I$ is a random collection of continuous curves  $\mc{L}:\intint{n} \times I \to \R$ specified under a probability measure $\PP$. The $i\textsuperscript{th}$ curve is thus $\mc{L}(i,\cdot): I \to \R$. (The adjective `line' has been applied to these systems perhaps because of their origin in such models as Poissonian LPP, where the counterpart object has piecewise constant curves. We will omit it henceforth.)
An ensemble is called {\em ordered} if $\mc{L}(i,x) > \mc{L}(i+1,x)$ whenever $i \in \intint{n-1}$ and $x$ lies in the interior of $I$.
The curves may thus assume a common value at any finite endpoint of $I$.
We will consider ordered ensembles that satisfy a condition called the Brownian Gibbs property.
Colloquially, we may say that an ordered ensemble is called Brownian Gibbs if it arises from a system of Brownian bridges or Brownian motions defined on $I$ by conditioning on the mutual avoidance of the curves at all times in $I$.
 
\subsubsection{Defining $(c,C)$-regular ensembles} 
We are interested in ensembles that are not merely Brownian Gibbs but that hew to the shape of a parabola and have one-point distributions for the uppermost curve that enjoy tightness properties. We will employ the next definition, which specifies a  $(\bar\phi,\rsc,\rsC)$-regular ensemble from~\cite[Definition~$2.4$]{BrownianReg},
in the special case where the vector $\bar\phi$  equals  $(1/3,1/9,1/3)$.

\begin{definition}\label{d.regularsequence} 
Consider a Brownian Gibbs ensemble of  the form 
$$
\mc{L}: \intint{\nmac} \times \big[ - \xnmac , \infty \big) \to \R  \,  ,
$$
and which is defined on a probability space under the law~$\PP$.
The number $\nmac = \nmac(\mathcal{L})$ of ensemble curves and the absolute value $\xnmac$ of the finite endpoint may take any values in $\N$ and $[0,\infty)$.

Let 
$\para:\R \to \R$ denote the parabola $\para(x) = 2^{-1/2} x^2$.
 
Let $\rsC$ and $\rsc$ be two positive constants. The ensemble $\mc{L}$
is said to be $(\rsc,\rsC)$-regular
 if the following conditions are satisfied.
\begin{enumerate}
\item {\bf Endpoint escape.} $\xnmac \geq  \rsc N^{1/3}$.
\item {\bf One-point lower tail.} If $z \in [ -\xnmac, \infty)$ satisfies $\vert z \vert \leq \rsc \nmac^{1/9}$, then
$$
\PP \Big( \mc{L} \big( 1,z\big) + \para(z) \leq - s \Big) \leq \rsC \exp \big\{ - \rsc s^{3/2} \big\}
$$
for all $s \in \big[1, \nmac^{1/3} \big]$.
\item {\bf One-point upper tail.}  If $z \in [ -\xnmac, \infty)$ satisfies $\vert z \vert \leq \rsc \nmac^{1/9}$, then
$$
\PP \Big( \mc{L} \big( 1,z\big) +  \para(z) \geq  s \Big) \leq \rsC \exp \big\{ - \rsc s^{3/2} \big\}
$$
for all $s \in [1, \infty)$.
\end{enumerate}
We will call these conditions ${\rm Reg}(1)$, ${\rm Reg}(2)$ and ${\rm Reg}(3)$.

A Brownian Gibbs ensemble of the form 
$$
\mc{L}: \intint{\nmac} \times \big( -\infty , \xnmac  \big] \to \R
$$
is also said to be $(\rsc,\rsC)$-regular if the reflected ensemble $\mc{L}( \cdot, - \cdot)$ is. This is equivalent to the above conditions when instances of $[ - \xnmac, \infty)$
are replaced by $(-\infty, \xnmac]$.
\end{definition}

\subsubsection{The normalized forward and backward profiles may be embedded in regular ensembles}\label{s.normalized}

We say that a random function of the form $\mc{L}: [-\xnmac,\infty) \to \R$ or  $\mc{L}: (-\infty,\xnmac] \to \R$ is $(c,C,m)$-regular
if there exists an $m$-curve $(c,C)$-regular ensemble of which it is the lowest indexed curve.

Our reason for invoking the theory of regular Brownian Gibbs ensembles is that the normalized Brownian LPP profiles are regular.
\begin{proposition}\label{p.shiftbrownian}
There exist values for the positive parameters $C$ and $c$ such that, for $n \in \N$ and $a \in n^{-1}\Z \cap (0,1)$, the following hold.
\begin{enumerate}
\item 
The process $\scaledle_{n;(0,0)}^{\uparrow;a}: \big[- 2^{-1} (n a)^{1/3}  , \infty \big) \to \R$
is $\big(c,C,na+1\big)$-regular.
\item 
The process $\scaledle_{n;\aplus}^{\downarrow;(0,1)}: \big( - \infty , 2^{-1} \big(n (1- \aplus)\big)^{1/3}   \big] \to \R$
is $\big(c,C,n(1-a)\big)$-regular.
\end{enumerate}
\end{proposition}
{\bf Proof.} Values of $C$ and $c$ that validate these two statements are offered by~
\cite[Proposition~$4.2$]{NonIntPoly}. \qed

The reader may consult the fifth paragraph of Section~$5.8$ of \cite{NonIntPoly} for a point of departure to the proof of~\cite[Proposition~$4.2$]{NonIntPoly}.

\subsection{The Brownianity of the narrow wedge weight profile}
We now state  the principal result of~\cite{DeBridge}, asserting the Brownianity of scaled Brownian LPP polymer weight profiles in the narrow wedge case (recall the notion of approximate Brownianity from Definition \ref{d.verybrownian}), when one endpoint is fixed, and the other varies horizontally. This conclusion is expressed in terms of regular ensembles in~\cite{DeBridge}. Our concern is merely with the uppermost curve and we record the result only in this case. 

\begin{theorem}\label{t.debridge} 

Suppose that $\mc{L}_m$ is an $m$-curve $(c,C)$-regular ensemble for some $m \in \N$ and $C,c \in (0,\infty)$. Let $\ipdval \geq 1$ and $K \in \R$ satisfy $[K-\ipdval,K+ \ipdval] \subset \rsc/2 \cdot [-m^{1/9},m^{1/9}]$. There exist values of the positive parameters $g$ and $G$, chosen without dependence on $m$, $K$ or $d$, such that  the random function $[K - \ipdval,K+\ipdval] \to \R: x \to \mc{L}_m(1,x) + 2^{1/2}Kx$
is  $\big(g,G,\ipdval,m,1\big)$-Brownian.

\end{theorem}

{\bf Proof.} The result follows from~\cite[Theorem~$3.11$]{DeBridge} by considering the curve with the lowest index $k=1$
in the ensemble $\mc{L}_m$. We now indicate conditions on the claimed constants $G$ and $g$ 
that render valid the application of this theorem---a mundane check phrased in terms of parameters $D_1$ and $C_1$ from the quoted result. 
We choose $G$ to be at least the value of this parameter as specified in~\cite[Theorem~$3.11$]{DeBridge} while satisfying $G \geq 24^6 D_1^{-3}$ and $G \geq 4932 D_1^{5/2}$.
We choose $g > 0$  to satisfy $g \leq e^{-1} \wedge (17)^{-1} C_1^{-1} D_1^{-1}$.
We further demand that $g \leq \big( \rsc/2 \wedge 2^{1/2} \big) D_1^{-1}$; and ensure that this parameter  is small enough that 
 $\exp( - gm_0^{1/2} ) \geq e^{-1}$
for $m_0 = (c/3)^{-18} \vee  6^{36}$, in order that the condition $e^{-g m^{1/12}} \leq \eta$ be impossible to satisfy unless $m \geq m_0$. 
We impose the lower bound of $e^{-1}$ since the upper bound that $\eta$ is assumed to satisfy in Definition \ref{d.verybrownian} implies that $\eta \leq e^{-1}$ in view of $G,d \geq 1$.
\qed

\subsection{Brownianity for the routed weight profile}

We now prove Theorem~\ref{t.brownianroutedprofile}. 
We begin by rewriting the basic formula Lemma~\ref{l.routedprofile}(1) in normalized form: 
\begin{equation}\label{e.zrewrite}
Z_n(x,a) =   a^{1/3} \scaledle_{n;(0,0)}^{\uparrow;a} \big(  a^{-2/3} x \big) +   (1-\aplus)^{1/3} \scaledle_{n;\aplus}^{\downarrow;(0,1)} \big(  (1-\aplus)^{-2/3} \xminus \big) \, . 
\end{equation}

Theorem~\ref{t.brownianroutedprofile} is concerned with the Brownian character of the profile $x \to Z_n(x,a)$ in a neighbourhood of a given $R \in \R$. This profile is curved parabolically in a global sense, so that a suitable drift must be identified for the Brownian motion with which we seek to make comparison. We find the drift by determining the counterpart drifts for the two right-hand terms in~(\ref{e.zrewrite}). We are thus led to make a further simple change of coordinates. 

Set $\ell:\R^2 \to \R$ to equal $\ell(x,y) = - 2^{-1/2}x^2 - 2^{1/2}x(y-x)$; thus, $y \to \ell(x,y)$
is the line tangent at $x \in \R$ to the parabola $z \to - 2^{-1/2}z^2$.

Let $s \in n^{-1}\Z \cap (0,1)$. For $x \in \R$, define the {\em shifted} forward profile 
$$
\scaledle_{n;x;(0,0)}^{{\rm shift}\uparrow;s}:\big[- 2^{-1} (n s)^{1/3} -x , \infty \big) \to \R
$$
by setting 
$$
 \scaledle_{n;x;(0,0)}^{{\rm shift}\uparrow;s} (   z ) = 
 \scaledle_{n;(0,0)}^{\uparrow;s} (  x +  z ) - \ell \big( x, x +z  \big) \, .
$$
For $y \in \R$, define the {\em shifted} backward profile 
$\scaledle_{n;y;a}^{{\rm shift}\downarrow;(0,1)}:\big( - \infty,  2^{-1} (n \tot)^{1/3} - y \big] \to \R$
by setting 
$$
 \scaledle_{n;y;a}^{{\rm shift}\downarrow;(0,1)} (   z ) = 
 \scaledle_{n;a}^{\downarrow;(0,1)} (   y +  z ) - \ell \big(  y,  y +z \big) \, .
$$
Further set 
\begin{equation}\label{e.uparrowshift}
Z_n^\uparrow(x,a) =  a^{1/3}  \scaledle_{n;a^{-2/3}R;(0,0)}^{{\rm shift}\uparrow;a} \big( a^{-2/3}(x-R) \big)
\end{equation}
and 
$$
Z_n^\downarrow(x,a) =  (1-\aplus)^{1/3}
 \scaledle_{n;(1-\aplus)^{-2/3}R;\aplus}^{{\rm shift}\downarrow;(0,1)} \Big(  \big(1-\aplus\big)^{-2/3}(\xminus-R) \Big) \, .
$$
Since $\ell(x,y) = 2^{-1/2}x^2 - 2^{1/2}xy$, we find then from~(\ref{e.zrewrite}) that 
\begin{equation}\label{e.zupdownarrow}
Z_n(x,a) = 
Z_n^\uparrow(x,a) +
Z_n^\downarrow(x,a)  + \Theta(x) + \mc{E}(x) \, ,
\end{equation}
where $\Theta(x) = 2^{-1/2} \big( a(1-a) \big)^{-1} R(R-2x)$ and 
$$
\mc{E}(x) = 2^{-1/2} \big( (1 - \aplus)(1-a) \big)^{-1} R(R-2x) n^{-1} + 2^{-1/2}(1 - \aplus)^{-1}R n^{-2/3} \, .
$$

In the next result, we see how shifted coordinates, which have made possible the formula~(\ref{e.zupdownarrow}), put us in excellent shape to derive Theorem~\ref{t.brownianroutedprofile}. Indeed, as we will argue shortly, largely on the basis of the upcoming lemma and Theorem~\ref{t.debridge}, the first two right-hand terms of this formula are independent processes that are very similar to Brownian motion of rate one; while the third term~$\Theta$ records the drift inherited from parabolic curvature that is manifest in the locale of the location~$R \in \R$.
\begin{lemma}\label{l.shiftbrownian}
There exist values for the positive parameters $C$, $c$, $G$ and $g$ such that, for $n \in \N$ and $a \in n^{-1}\Z \cap (0,1)$, the following hold.
\begin{enumerate}
\item Suppose that $\vert x \vert \leq c/2 \cdot (na + 1)^{1/9}$. Then the process
$\scaledle_{n;x;(0,0)}^{{\rm shift}\uparrow;a}:\big[- 2^{-1} (n a)^{1/3} -x , \infty \big) \to \R$
is $\big(c/2,C,na + 1\big)$-regular.
\item Suppose that $\vert y \vert \leq c/2 \cdot \big(n(1-a)\big)^{1/9}$. Then the process
$$
\scaledle_{n;y;a}^{{\rm shift}\downarrow;(0,1)}:\big( - \infty,  2^{-1} \big(n (1-a)\big)^{1/3} - y \big] \to \R
$$ 
is $\big(c/2,C,n(1-a)\big)$-regular.
\item
The processes  $Z_n^\uparrow(\cdot,a),Z_n^\downarrow(\cdot,a):[R,R + \ell] \to \R$ are independent.
\item Suppose that $\vert R \vert \leq 2^{-1}c n^{1/9} a^{7/9}$.
Let $\ell \in \R$ satisfy $\ell - R \leq  2^{-1}n^{1/3}a$. Then the process 
$Z_n^\uparrow(\cdot,a):[R-\ell,R+\ell] \to \R$ is $\big(g,G a^{-4},\ell,an + 1,1\big)$-Brownian.
\item  Suppose that $\vert R \vert \leq 2^{-1}c n^{1/9} (1 - \aplus)^{7/9}$.
Let $\ell \in \R$ satisfy $R + \ell \leq  2^{-1}n^{1/3}(1 - \aplus)$. Then the process $Z_n^\downarrow(\cdot,a):[R-\ell,R + \ell] \to \R$ is $\big(g,G (1 - \aplus)^{-4},\ell,(1-a)n,1\big)$-Brownian.
\end{enumerate}
\end{lemma}
{\bf Proof.}
{\bf (1,2).} These are due to~\cite[Lemma~$3.4$]{ModCon} or 
\cite[Lemma~$2.26$] {BrownianReg}.

{\bf (3).} The regions $\R \times [0,a]$ and $\R \times [a+n^{-1},1]$
that respectively specify $Z_n^\uparrow(\cdot,a)$ and  $Z_n^\downarrow(\cdot,a)$ are disjoint, so that these processes are independent.

{\bf (4).} By the third part of the lemma, and Theorem~\ref{t.debridge}, there exist positive values $G$ and $g$ such that $\scaledle_{n;a^{-2/3}R;(0,0)}^{{\rm shift}\uparrow;a}: [-a^{-2/3}\ell,a^{-2/3}\ell]\to \R$ is $\big(g,G,a^{-2/3}\ell,an + 1,1\big)$-Brownian, where we have used $[-a^{-2/3}\ell,a^{-2/3}\ell] \subset \big[- 2^{-1} (n a)^{1/3} - a^{-2/3}R , \infty \big)$ and $a^{-2/3}\vert R \vert \leq 2^{-1}c (na + 1)^{1/9}$. The former condition is implied by 
$a^{-2/3} \ell \leq 2^{-1}(na)^{1/3} + a^{-2/3}R$ and thus by our hypothesis that $\ell - R \leq 2^{-1} n^{1/3}a$.  The latter is implied by $\vert R \vert \leq 2^{-1}c n^{1/9} a^{7/9}$.  

In light of this, and~(\ref{e.uparrowshift}), we may apply the next presented Lemma~\ref{l.browniantobrownian}
with 
 $\kappa = a$ to verify Lemma~\ref{l.shiftbrownian}(6) holds.

{\bf (5).} Invoking similarly the fourth part of the lemma, 
$$
\scaledle_{n;(1-\aplus)^{-2/3}R;\aplus}^{{\rm shift}\downarrow;(0,1)}: \big[-(1 - \aplus)^{-2/3}\ell,(1 - \aplus)^{-2/3}\ell\big]\to \R
$$ 
is seen to be $\big(g,G,(1 - \aplus)^{-2/3}\ell, (1-\aplus)n,1\big)$-Brownian. This time, we need 
$$
\big[ - (1-\aplus)^{-2/3}\ell,(1 - \aplus)^{-2/3}\ell \big] \subset \big( -\infty,  2^{-1} \big(n (1 - \aplus) \big)^{1/3} - (1 - \aplus)^{-2/3}R \big]
$$ 
and $(1 - \aplus)^{-2/3}\vert R \vert \leq 2^{-1}c (n(1-\aplus) + 1)^{1/9}$.
 The former condition is implied by $R + \ell \leq  2^{-1}n^{1/3}(1 - \aplus)$; the latter by $\vert R \vert \leq 2^{-1}c n^{1/9} (1 - \aplus)^{7/9}$. \qed

The $a$-dependent spatial-temporal scaling in~(\ref{e.uparrowshift}) respects Brownian motion; it is unsurprising then that this scaling in essence leaves invariant a property of similarity to Brownian motion. The next result is a rigorous interpretation of this notion.
\begin{lemma}\label{l.browniantobrownian}
Let $G,g,\ell > 0$, $\kappa \in (0,1)$, $m \in \N$ and $R \in \R$.  
Suppose that  the random function $\mc{L}: \big[-\kappa^{-2/3}\ell,\kappa^{-2/3}\ell\big] \to \R$ is 
$\big(g,G,\kappa^{-2/3}\ell,m,1\big)$-Brownian.
 Set $\mc{L}^\kappa(x) = \kappa^{1/3} \mc{L} \big( \kappa^{-2/3}(x-R) \big)$. Then  $\mc{L}^\kappa: [R-\ell,R + \ell] \to \R$ is $\big(g,G \kappa^{-4},\ell,m,1\big)$-Brownian.
\end{lemma}
{\bf Proof.} 
Recalling the notation of Definition~\ref{d.verybrownian}, let $B \subseteq \mc{C}_{0,*} \big( [R-\ell,R+\ell] , \R \big)$. Set 
$$
B^* = \Big\{ f \in \mc{C}_{0,*} \big( [-\kappa^{-2/3}\ell,\kappa^{-2/3}\ell] , \R \big) : x \to \kappa^{1/3} f \big( \kappa^{-2/3}(x-R) \big) \in B \Big\} \, .
$$
Then Brownian scaling implies that, when $B$ is Borel measurable, $\mc{B}_{0,*}^{[0,\kappa^{-2/3}\ell]}(B^*) = \mc{B}_{0,*}^{[R,R+ \ell]}(B)$. 

Taking $I = \big[-\kappa^{-2/3}\ell,\kappa^{-2/3}\ell\big] $, the process $\mc{L}: \big[-\kappa^{-2/3}\ell,\kappa^{-2/3}\ell\big] \to \R$
meets the condition on $X: I \to \R$ in Definition~\ref{d.verybrownian}
for the parameter quintet $\big(g,G,\kappa^{-2/3}\ell,m,1\big)$. It is our task to verify that  $\mc{L}^\kappa:[R-\ell,R + \ell] \to \R$  does so for the quintet $\big(g,G \kappa^{-4},\ell,m,1\big)$. The third parameter, which is one-half the length of the domain interval, has decreased, by a factor of $\kappa^{2/3}$.
By the preceding paragraph, the value of the Brownian probability $\eta$ is shared in the definition as it applies to the processes $\mc{L}$ and~$\mc{L}^\kappa$. If we denote the second element of the latter quintet by $G'$, we may note that we must demand of it that $G' \geq G \big( \kappa^{-2/3} \big)^6 = G \kappa^{-4}$---so that the hypothesis $\eta \leq e^{-G' \ell^6}$ is implied by $\eta \leq \exp \big\{ -G \big( \kappa^{-2/3}\ell \big)^6 \big\}$; and that $G' \geq G \kappa^{-2/3}$---so that $G' \ell \geq G \kappa^{-2/3} \ell$ may be applied to obtain the right-hand side in the display in Definition~\ref{d.verybrownian}. The choice $G' = G \kappa^{-4}$ meets these two requirements. Since the three further parameters, $g$, $G$ and $m$, transmit unaltered, we obtain Lemma~\ref{l.browniantobrownian}. \qed

Each profile
$x \to Z_n^\uparrow(x,a)$ and
$x \to Z_n^\downarrow(x,a)$ will shortly be shown to be very similar to standard Brownian motion by an argument that harnesses Lemmas~\ref{l.shiftbrownian} and~\ref{l.browniantobrownian} to the fundamental estimate Theorem~\ref{t.debridge}.  The profile $x \to Z_n(x,a)$, after linear adjustment, will then be seen via~(\ref{e.zupdownarrow}) to resemble Brownian motion of rate two (as  Theorem~\ref{t.brownianroutedprofile} asserts), provided that we argue that Brownianity in the sense of Definition~\ref{d.verybrownian} is stable under addition of processes. After we establish this in Lemma~\ref{l.brownianadditivity}, we will be ready to give a short proof of Theorem~\ref{t.brownianroutedprofile}.  First, however, we present a result that permits us to dispense with Definition~\ref{d.verybrownian}'s inconsequential but practically irksome Brownian probability hypothesis $\eta \leq g \wedge e^{-G \ipdval^d}$.

\begin{lemma}\label{l.noupperbound}
Under the circumstances of Definition~\ref{d.verybrownian}, suppose on the parameter $\eta$, instead of the condition  $e^{-gm^{1/12}} \leq \eta \leq g \wedge e^{-G d^6}$ , that merely the lower bound $\eta \geq e^{-gm^{1/12}}$  holds. Suppose also that $g \in (0,1)$ and that $G,d  \geq 1$. Then
$$
\PP \Big( I \to \R: x \to X(x) - X(K-d) \, \, \textrm{belongs to $A$} \Big) \, \leq \, \eta \cdot     \exp \Big\{ 5 G^{11/6} \ipdval^6  g^{-5/6}      \Big\} \exp \Big\{ G \ipdval \big( \log \eta^{-1} \big)^{5/6} \Big\} \, .
$$  
\end{lemma}

{\bf Proof.}
Note that by Theorem \ref{t.debridge} the condition  $e^{-gm^{1/12}} \leq \eta$ implies that
\begin{equation}\label{removalconstraint}
\PP \Big( I \to \R: x \to X(x) - X(K-d) \, \, \textrm{belongs to $A$} \Big) \, \leq \, \eta \cdot  H G\exp \Big\{ G \ipdval \big( \log \eta^{-1} \big)^{5/6} \Big\} \, ,
\end{equation}
where $H = \big( \max \big\{ g^{-1} , e^{G \ipdval^6} \big\} +1 \big)  \exp \Big\{ G \ipdval \big( \log \max \big\{ g^{-1} , e^{G \ipdval^6} \big\}   \big)^{5/6} \Big\}$. The above follows from Theorem \ref{t.debridge} since $H\ge 1$ and moreover for $\eta\ge g \wedge e^{-G d^6},$ we have $ \eta H\ge 1$ while $G\exp \Big\{ G \ipdval \big( \log \eta^{-1} \big)^{5/6}\Big\}$ is always at least $1$ since $G\ge 1$ by hypothesis.

Since $g \in (0,1)$, note that 
$$
 H \leq 2  g^{-1}  e^{G \ipdval^6}  \exp \Big\{ G \ipdval \big( \log ( g^{-1} e^{G \ipdval^6} )   \big)^{5/6} \Big\} \leq 2  g^{-1}  e^{G \ipdval^6}  \exp \Big\{ G^{11/6} \ipdval^6  g^{-5/6}      \Big\} \, ,
$$
where we have applied
$a e^x \leq e^{ax}$ for $a,x \geq 1$ with $a = g^{-1}$ and $x = G \ipdval^6$. Thus, 
$$
GH \leq 2  g^{-1}G  e^{G \ipdval^6}  \exp \Big\{ G^{11/6} \ipdval^6  g^{-5/6}      \Big\} \leq g^{-1} e^{3G \ipdval^6}  \exp \Big\{ G^{11/6} \ipdval^6  g^{-5/6}      \Big\} \leq g^{-1}  \exp \Big\{ 4 G^{11/6} \ipdval^6  g^{-5/6}      \Big\} \, ,
$$
where we used $G \geq 1 \geq g$ and $d \geq 1$. Noting that $g^{-1} \leq \exp \big\{ g^{-5/6} \big\}$ alongside $G,d \geq 1$ completes the proof. \qed

 Recall Definition \ref{d.verybrownian}. 
\begin{lemma}[Additive stability for Brownian regularity]\label{l.brownianadditivity}
Let $\nu_i \in (0,\infty)$; $n_i \in \N$ for $i \in \{1,2\}$; $g,G,\ell > 0$; and $R\in \R$.
Let the random functions $X,Y :[R-\ell,R+\ell] \to \R$ be independent under the law $\PP$. Suppose that $X$ is $\big( g,G,\ell,m_1,\nu_1\big)$-Brownian, and that $Y$ is $\big( g,G,\ell,m_2,\nu_2\big)$-Brownian. Then $X+Y$ is $\big( g ,G', \ell  ,m_1 \wedge m_2 , \nu_1 + \nu_2 \big)$-Brownian, where $G'$ is a multiple of $G^{17/6} g^{-5/6}  \ell^6$
by an absolute positive factor.
\end{lemma}

{\bf Proof.} Write $\mc{C} =  \mc{C}_{0,*}\big( I ,\R \big)$ where $I =  [R-\ell,R+\ell]$. Our argument will rely on analysing the different values of the Radon-Nikodym derivative of $\hat{X}:=X(\cdot) - X(R - \ell)$ with respect to Brownian motion. So we first apply the Lebesgue decomposition theorem to obtain a Borel set $\mc{X}_1\subset \mc{C}$ such that $\hat X$ is absolutely continuous  with respect to $\mc{B}^{\nu_1 + \nu_2;I}_{0,*}$ on $\mc{X}_1$ and $\mc{B}^{\nu_1 + \nu_2;I}_{0,*}(\mc{X}_1^c)=0.$
Similarly for $\hat Y$, we define $\mc{X}_2$.

We first claim the following bounds:
\begin{align}\label{aprioribounds1}
\P(\hat X\in \mc{X}_1^c)&\le   G\exp \Big\{ 5G \ell g^{-5/6} m_1^{5/72}  \Big\} e^{-g m_1^{1/12}} \\
\nonumber
\P(\hat Y\in \mc{X}_2^c)&\le  G\exp \Big\{ 5G \ell g^{-5/6} m_2^{5/72}  \Big\} e^{-g m_2^{1/12}} .
\end{align} 
We only discuss the first bound since the argument for the second is similar. Note that since  $\mc{B}^{\nu_1 + \nu_2;I}_{0,*}(\mc{X}_1^c)=0,$ one cannot directly appeal to the fact that $X$ is $\big( g,G,\ell,m_1,\nu_1\big)$-Brownian. However, a simple enlargement argument gives a Borel set $S$ such that $\mc{X}_1^c\subset S$ and $\mc{B}^{\nu_1 + \nu_2;I}_{0,*}(S)=e^{-g m_1^{1/12}}.$
Thus we get 
$$
\P(\hat X\in \mc{X}_1^c)\le \P(\hat X\in S)\le G\exp \Big\{ G \ell g^{5/6} m_1^{5/72}  \Big\} e^{-g m_1^{1/12}} \, .
$$

Next, let $A \subset \mc{C}$ be Borel measurable. Set $\hatapr = \mc{B}^{\nu_1 + \nu_2;I}_{0,*}(A)$, and suppose that $\hatapr$ is at least the quantity $\exp \big\{ - g (m_1 \wedge m_2)^{1/12} \big\}$.  
For $f \in \mc{C}$, set $A_f = \big\{ g \in \mc{C}\cap \mc{X}_2 : f +g \in A \big\}$.

Now let  $F: \mc{C} \to [0,\infty)$ denote the Radon-Nikodym derivative
of the law of $\hat X$ with respect to the law $\mc{B}^{\nu_1;I}_{0,*}$ on $\mc{X}_1$; and let $G: \mc{C} \to [0,\infty)$ denote the counterpart with the replacements $\hat X \to \hat Y$,
 $\nu_1 \to  \nu_2$ and $\mc{X}_1\to \mc{X}_2$ made.

For $k \in \N$, write $D_k = \big\{ f \in \mc{C}: 2^k \leq G(f) < 2^{k+1} \big\}$ for $k\ge 1$ and let $D_0 = \big\{ f \in \mc{C}:  G(f) < 2 \big\}$. Set  $\eta_k = \mc{B}^{\nu_2;I}_{0,*} \big( \cup_{j=k}^\infty D_j \big)$.

Note that  $\PP \big( \hat Y \in \cup_{j=k}^\infty D_j \big) \geq 2^k \eta_k$.
However, by Lemma~\ref{l.noupperbound}, the condition  
\begin{equation}\label{e.etakcond}
e^{-g m_2^{1/12}} \leq \eta_k  \, , 
\end{equation}
implies that 
\begin{equation}\label{levelsets}
 \PP \big( \hat Y \in \cup_{j=k}^\infty D_j \big) \, \leq  \,  \exp \Big\{ 5 G^{11/6} \ell^6  g^{-5/6}      \Big\} \exp \Big\{ G \ell \big( \log \eta_k^{-1} \big)^{5/6} \Big\} \eta_k
 \, .
\end{equation}
Thus, when (\ref{e.etakcond}) holds, $2^k \leq \hat{G} \exp \big\{ G \ell \big( \log \eta_k^{-1} \big)^{5/6} \big\}$ with $\hat{G} =  \exp \big\{ 5 G^{11/6} \ell^6  g^{-5/6}      \big\}$, whence 
$$
\eta_k \leq \exp \big\{ - (G \ell)^{-6/5} \big( k \log 2  - \log \hat{G} \big)^{6/5}  \big\} \, .
$$
Hence, $\eta_k \leq \exp \big\{ - (G \ell)^{-6/5} \big( k \log 2  - \log \hat{G} \big)^{6/5}  \big\} \vee e^{-g m_2^{1/12}}$, whether or not~(\ref{e.etakcond}) holds.

Set $J \in \N$ to be minimal such that $\exp \big\{ - (G \ell)^{-6/5} \big( J \log 2  - \log \hat{G} \big)^{6/5}  \big\}  \leq \hatapr$. Note that $J$ is at most $2 G\ell \big( \log \hatapr^{-1} \big)^{5/6}$ provided that 
$G\ell \big( \log \hatapr^{-1} \big)^{5/6} \geq  \log \hat{G}$.

If $\eta_J \geq e^{-g m_2^{1/12}}$, then $\eta_J \leq \hatapr$, so that 
 \begin{equation}\label{e.yjj}
 \PP \big( \hat Y \in \cup_{j=J}^\infty D_j \big) \leq    \exp \big\{ 5 G^{11/6} \ell^6  g^{-5/6}      \big\} \exp \big\{ G \ell \big( \log \hatapr^{-1} \big)^{5/6} \big\} \hatapr \, ,
 \end{equation}
  where we used that $\hatapr \in (0,\hatapr_0)$ for $\hatapr_0 > 0$ small enough that $(0,\hatapr_0) \to \R: x \to \exp \big\{ G \ell \big( \log x^{-1} \big)^{5/6} \big\} x$ is increasing.
If $\eta_J < e^{-g m_2^{1/12}}$, then, by an enlargement argument,  
$$
\PP \big( \hat Y \in \cup_{j=J}^\infty D_j \big) \leq    \exp \big\{ 5 G^{11/6} \ell^6  g^{-5/6}      \big\} \exp \Big\{ G \ell g^{5/6} m_2^{5/72}  \Big\} e^{-g m_2^{1/12}} \, .
$$ 
Since $\hatapr \geq e^{-g m_2^{1/12}}$, we find that~(\ref{e.yjj}) holds, whether or not $\eta_J \geq e^{-g m_2^{1/12}}$. 

Writing $B: I \to \R$, $B(R-\ell) =0$, to denote a Brownian motion of diffusion rate one under the law~$\PP$, independent of $X$,
note  that
\begin{align}\label{decomposition12}
 \PP \big( \hat X+ \hat Y \in A  \big) & =   \PP \big( \hat Y \in A_{\hat X} \big) \\
 \nonumber
 &\le \P(\hat X \in \mc{X}^c_1)+\P(\hat Y \in \mc{X}^c_2)+ \sum_{k = 0}^{J-1} \PP \big( \hat Y \in A_{\hat X} \cap D_k, \hat Y \in \mc{X}_2,\hat X \in \mc{X}_1 \big) \\
 \nonumber
 & \,\,\,+ \,  \PP \big( \hat Y \in \cup_{j = J}^\infty D_j, \hat Y \in \mc{X}_2  \big) \, .
 \end{align}
The first two terms will be bounded using \eqref{aprioribounds1}. For the next two terms, observe that 
 \begin{align}
 \nonumber
& \quad  \sum_{k = 0}^{J-1} \PP \big( \hat Y \in A_{\hat X} \cap D_k,\hat X \in \mc{X}_1 \big)+  \PP \big( \hat Y \in \cup_{j = J}^\infty D_j \big)\\
 \nonumber
 & \leq  \sum_{k = 0}^{J-1} 2^{k+1}
 \PP \big( B \in A_{\hat X} \cap D_k, {\hat X} \in \mc{X}_1\big) \, \, + \,  
  \exp \big\{ 4 G^{11/6} \ell^6  g^{-5/6}      \big\} \exp \big\{ G \ell \big( \log \hatapr^{-1} \big)^{5/6} \big\} \hatapr \\
  \nonumber
  & \leq   2^J     
  \PP (  B \in A_{\hat X},{\hat X} \in \mc{X}_1 ) \, \, + \,  \exp \big\{ 5 G^{11/6} \ell^6  g^{-5/6}      \big\} \exp \big\{ G \ell \big( \log \hatapr^{-1} \big)^{5/6} \big\} \hatapr \\
  \label{finalbound12}
  & \leq    2^{2 G\ell ( \log \hatapr^{-1} )^{5/6}}   \PP (  B \in A_{\hat X},{\hat X} \in \mc{X}_1 ) \, + \,   \exp \big\{ 5 G^{11/6} \ell^6  g^{-5/6}      \big\} \exp \big\{ G \ell \big( \log \hatapr^{-1} \big)^{5/6} \big\} \hatapr \, .
\end{align}
In a specification similar to that of $D_k$, set $E_k = \big\{ f \in \mc{C}: 2^k \leq F(f) < 2^{k+1} \big\}$ for $k\ge 1,$ and $E_0 = \big\{ f \in \mc{C}:  F(f) \leq 2 \big\}$. By a verbatim argument that invokes $\hatapr \geq e^{-g m_1^{1/12}}$,  we see that $\PP \big( \hat X \in \cup_{j=J}^\infty E_j \big)$ is at most the right-hand side of~(\ref{e.yjj}).   

Write $B':I \to \R$, $B'(R - \ell) = 0$, for a further Brownian motion of diffusion rate one, defined under the law $\PP$ and chosen independently of $B$.
Note then that 
\begin{eqnarray*}
& & \PP \big( B \in A_{\hat X}, {\hat X}\in \mc{X}_1 \big) \\
 & = & \sum_{k = 0}^{J-1} 
\PP \big( B \in A_{\hat X} , {\hat X} \in E_k \big)
\, \, + \,  \PP \big( {\hat X} \in \cup_{j = J}^\infty E_j \big)  \\ 
 & \leq &  \sum_{k = 0}^{J-1} 2^{k+1}
\PP \big( B \in A_{B'} , B' \in E_k \big)   \, \, + \,  \exp \big\{ 5 G^{11/6} \ell^6  g^{-5/6}      \big\} \exp \big\{ G \ell \big( \log \hatapr^{-1} \big)^{5/6} \big\} \sigma
 \\
   & \leq &  2^J \PP \big( B \in A_{B'} \big)   \, + \,  \exp \big\{ 5 G^{11/6} \ell^6  g^{-5/6}      \big\} \exp \big\{ G \ell \big( \log \hatapr^{-1} \big)^{5/6} \big\} \hatapr
   \\
   & = & 2^{2 G\ell ( \log \hatapr^{-1} )^{5/6}}  \PP \big( B + B' \in A  \big) \, + \,  \exp \big\{ 5 G^{11/6} \ell^6  g^{-5/6}      \big\} \exp \big\{ G \ell \big( \log a^{-1} \big)^{5/6} \big\} \hatapr \\
    & = &   \Big( \exp \big\{ 2 \log 2 \cdot G\ell \big( \log \hatapr^{-1} \big)^{5/6} \big\}   \, + \,  \exp \big\{ 5 G^{11/6} \ell^6  g^{-5/6}      \big\} \exp \big\{ G \ell \big( \log \hatapr^{-1} \big)^{5/6} \big\} \Big) \, \hatapr \, \\
     & \leq & \exp \big\{  2 \cdot 5 G^{11/6} \ell^6  g^{-5/6} \cdot 2\log 2 \cdot G \ell \big( \log \hatapr^{-1} \big)^{5/6} \big\} \, \hatapr \, \\
     & = & 
 \exp \big\{  20 \log 2 \cdot  G^{17/6} g^{-5/6}  \ell^7    \big( \log \hatapr^{-1} \big)^{5/6} \big\}  \cdot \hatapr \, ,
\end{eqnarray*} 
where to get the inequality in the second line, for the term involving the sum, we simply use the definition of $E_k,$ and independence of $B$ and $X$ as well as of $B$ and $B',$ to pass from $\PP \big( B \in A_{\hat X} , {\hat X} \in E_k \big)$ to  $2^{k+1}\PP \big( B \in A_B' , B' \in E_k \big)$. The second term is bounded using the already stated bound on $\PP \big( {\hat X} \in \cup_{j = J}^\infty E_j \big)$. The final inequality uses $G,\ell \geq 1$, $g \leq 1$, $\hatapr \leq e^{-1}$ and $2e^{x} \leq e^{2x}$ for $x \geq 1$.

The preceding display and \eqref{finalbound12} yield 
\begin{align*}
&\sum_{k = 0}^{J-1} \PP \big( \hat Y \in A_{\hat X} \cap D_k,\hat X \in \mc{X}_1 \big)+  \PP \big( \hat Y \in \cup_{j = J}^\infty D_j \big)\\
&\le 2^{2 G\ell ( \log \hatapr^{-1} )^{5/6}}    \exp \big\{  20 \log 2 \cdot  G^{17/6} g^{-5/6}  \ell^7    \big( \log \hatapr^{-1} \big)^{5/6} \big\}  \cdot \hatapr \\ \,
&\hspace{2in} + \,   \exp \big\{ 5 G^{11/6} \ell^6  g^{-5/6}      \big\} \exp \big\{ G \ell \big( \log \hatapr^{-1} \big)^{5/6} \big\} \hatapr \, .
\end{align*}
This, along with \eqref{decomposition12} and \eqref{aprioribounds1}, implies that
$$
 \PP \big( \hat X+ \hat Y \in A  \big) \le \exp \big\{ \Theta(1) G^{17/6} g^{-5/6}  \ell^7 \big( \log \hatapr^{-1} \big)^{5/6} \big\} \hatapr
$$
and hence completes the proof of Lemma~\ref{l.brownianadditivity}. \qed

{\bf Proof of Theorem~\ref{t.brownianroutedprofile}.} 
If we set $X(x) = Z^\uparrow_n(x,a) +  Z^\downarrow_n(x,a)$, then (\ref{e.zupdownarrow})  implies that 
$$
Z_n(x,a) = X(x) - 
\Big( 2^{1/2}\big( a(1-a) \big)^{-1} R + \e \Big) x + c(a,R,n)
$$
where 
$\e = 2^{1/2} \big( (1 - a - n^{-1})(1-a) \big)^{-1} R  n^{-1}$ and $c(a,R,n)$ is a constant.
 The process $X$ is found to be  $\big( g,G', \max \{ a^{-2/3} , (1-a)^{-2/3} \} \ell, \min \{ a,1-a \} n ,2 \big)$-Brownian for the value of $G'$ given in Theorem~\ref{t.brownianroutedprofile} by means of Lemma~\ref{l.shiftbrownian}(3,4,5) and Lemma~\ref{l.brownianadditivity}.
 It remains only to remove the constant $c(a,R,n)$ from the representation of $Z_n(x,a)$ in order to complete the proof of  Theorem~\ref{t.brownianroutedprofile}. The stated Brownian property of $X$ is unaltered by the addition of a constant to this process; thus, we may absorb the constant by adding it to $X$. \qed

\subsection{The rarity of twin peaks in the Brownian case}
Here, we state and prove our Brownian twin peaks' rarity result, Proposition~\ref{p.midneartouch}. Although the result is not new, we could not locate a reference;  the following quantitative form is suitable for our applications.

The argument rests in part on 
Proposition~\ref{p.brownianepsilonexc}, a near-return probability estimate for Brownian meander, which is stated and proved in the later part of the section.

Let $K$ and $r$ be positive. Let $B: [-r,r] \to \R$, $B(0) = 0$, denote standard Brownian motion under a law labelled $\PP$. Set $W:[-r,r] \to \R$, $W(x) = B(x) + Kx$. Let $M \in [-r,r]$ denote the almost surely unique point at which the process $W$ attains its maximum. 
Let $\mathsf{Mid}$ denote the {\em middle-third event} that $M \in [-r/3,r/3]$.
For parameters $\e \in (0,r/6)$ and $\hata \in (0,1)$, let $\mathsf{NT} = \mathsf{NT}(B)$  denote the {\em near-touch event}
that there exists a value of $z \in [-r,r]$ such that $\vert z - M \vert \in [\e,2\e]$
for which $W(z) \geq W(M) - \hata \e^{1/2}$. The near-touch event will be considered only when the middle-third event occurs, so that the condition that $z \in [-r,r]$ will be implied by the demand that $\vert z - M \vert \leq 2\e$.

\begin{proposition}\label{p.midneartouch}
For a constant $D > 0$ that is independent of $K \geq 0$, $r > 0$, $\e \in (0,r/6)$ and $\hata \in (0,1)$,
we have that 
$$
\PP \big( \mathsf{NT} \cap \mathsf{Mid} \big) \leq  
D K^{-1}r^{-1/2}  \exp \big\{ - K^2 r/18  \big\} \min \big\{ \hata , 1 \big\} \, ,
$$ 
as well as 
$\PP \big( \mathsf{NT} \cap \mathsf{Mid} \big) \leq  D \hata$.
\end{proposition}
{\bf Proof.} The case of general $r > 0$ may be reduced to that where $r=1$ by considering $r^{-1/2}W(rx): [-1,1] \to \R$. Brownian scaling entails that the $r=1$ result implies the general result when the replacement of $K$ by $K r^{1/2}$ is made. Thus, it suffices to prove the lemma with $r$ set equal to one, a choice which we now make.

 We will first argue that
\begin{equation}\label{e.midprob}
\PP(\mathsf{Mid}) \leq 3 \cdot 2^{3/2} \pi^{-1/2} K^{-1}  \exp \big\{ - \tfrac{K^2}{18}  \big\} \, .
\end{equation}
Note that $\mathsf{Mid}$ entails that $W(1) < W(M)$ with $M \in [-1/3,1/3]$, so that $B(1) < B(M) - 2K/3$, which forces one of $B(1)$ and $B(M)$ to exceed $K/3$ in absolute value. Thus,
$$
 \mathsf{Mid} \subseteq \Big\{ \sup \big\{ \vert B(x) \vert : x \in [-1,1] \big\} \geq K/3 \Big\} \, .
$$
By symmetry and the reflection principle, the probability of this right-hand event is at most 
\begin{equation}\label{e.reflection}
2 \cdot 2 \PP \big( B(1) \geq K/3 \big) \leq 4 (2\pi)^{-1/2} (K/3)^{-1} \exp \big\{ - 2^{-1} (K/3)^2 \big\} = 3 \cdot 2^{3/2} \pi^{-1/2} K^{-1}  \exp \big\{ - \tfrac{K^2}{18}  \big\} \, ,
\end{equation}
 the displayed inequality by a standard upper bound on the tail of a Gaussian random variable. Thus, we obtain~(\ref{e.midprob}).

The next result will be important as we turn to analysing the conditional probability of $\mathsf{NT}$ given $\mathsf{Mid}$. For $r > 0$ and $y \leq 0$, we write $\mc{B}_{0,y}^{[0,r]}$ for the law of Brownian 
bridge $B$ of diffusion rate one on $[0,r]$ with $B(0) = 0$ and $B(r) = y$; thus, $\mc{B}_{0,y}^{[0,r]} \big( \cdot \big\vert B < 0\big)$ is the law resulting from conditioning $B$ on $B(x)<0$ for $x \in (0,r]$.
\begin{lemma}\label{l.excursioncond}
Under the conditional law $\PP \big( \cdot \big\vert \mathsf{Mid}  \big)$, consider the processes 
  $X^-:[0,M+1] \to \R$ and
 $X^+:[0,1-M] \to \R$,  given by $X^-(x) = B(M-x) - B(M)$ and $X^+(x) = B(M+x) - B(M)$. 
\begin{enumerate} 
\item Conditionally on the value of $M \in [-1/3,1/3]$ and on $X^+(1-M)$ being any given value $y \leq 0$, the conditional distribution of~$X^+$ is given by the law $\mc{B}_{0,y}^{[0,1-M]} \big( \cdot \big\vert B < 0\big)$.
\item Conditionally on the value of $M \in [-1/3,1/3]$ and on $X^-(M+1)$ being any given value $y \leq 0$, the conditional distribution of~$X^-$ is given by the law $\mc{B}_{0,y}^{[0,M+1]} \big( \cdot \big\vert B < 0\big)$.
\end{enumerate}
\end{lemma}
{\bf Proof: (1).} Under the stated conditioning, the process $X^+:[0,1-M] \to \R$ is given by  $[0,1-M] \to \R: x \to X(x)$, where $X(x) = B(x) + Kx$, with $B$ being standard Brownian motion, conditioned on $X(1-M) = y$ and on $X(x) < 0$ for $x \in (0,1-M)$. Conditioning $X$ as given by the preceding formula on $X(1-M) = y$ results in the Brownian bridge law $\mc{B}_{0,y}^{[0,1-M]}$, as we may readily verify by
decomposing $B:[0,1-M]\to \R$ as the sum of a Brownian bridge and a linear term with an independent Gaussian coefficient.
 The further conditioning $X(x) < 0$ results in the conditional distribution stated in the lemma. \\
{\bf (2).} This almost verbatim argument is omitted. \qed

We wish to apply the shortly upcoming Proposition~\ref{p.brownianepsilonexc} 
alongside Lemma~\ref{l.excursioncond} to find that there exists $D > 0$ such that, for any $y < 0$, $$\mc{B}_{0,y}^{[0,1-M]} \big(\mathsf{NT}  \big\vert B < 0\big)+\mc{B}_{0,y}^{[0,M+1]} \big(\mathsf{NT}  \big\vert B < 0\big) \le D \hata.$$

Above, we are abusing notation a little in denoting by $\mathsf{NT}$ the event that there exists a value of $z \in [-\e,2\e]$ such that 
for which $B(z) \geq - \hata \e^{1/2}$.

We set the proposition's parameters: 
$r = \e$ and 
$s = 1\pm M$. The proposition's hypothesis  $s \geq 3r$ is valid because  because $M \leq 1/3$ when $\mathsf{Mid}$ occurs, so that $\e$ is merely supposed to be at most a given positive constant.
  Equipped with the just stated outcome, we find that
\begin{eqnarray*}
 \PP \big( \mathsf{Mid} \cap \mathsf{NT} \big) & = &
 \PP \big( \mathsf{Mid}  \big) 
 \PP \big( \mathsf{NT} \big\vert \mathsf{Mid} \big) \\
 & \leq & 3 \cdot 2^{3/2} \pi^{-1/2} K^{-1}  \exp \big\{ - \tfrac{K^2}{18}  \big\}
 \E\Big( \mc{B}_{0,y}^{[0,1-M]} \big(\mathsf{NT}  \big\vert B < 0\big)+\mc{B}_{0,y}^{[0,M+1]} \big(\mathsf{NT}  \big\vert B < 0\big)\Big)\\
 & \leq & 3 \cdot 2^{3/2} \pi^{-1/2} K^{-1}  \exp \big\{ - \tfrac{K^2}{18}  \big\} D \hata \, .
\end{eqnarray*}
It is in the second inequality that the conclusion of Proposition~\ref{p.brownianepsilonexc} is used. 
The first inequality is due to~(\ref{e.midprob}), and the mean $\E$ on this inequality's right-hand side is taken over $y$ and $M$. 
Note that the displayed assertion in Proposition~\ref{p.midneartouch} may be viewed as a pair of bounds due to the right-hand factor of $\min \{ \sigma, 1 \}$.
The proof of the bound including the factor of $\sigma$ has just been completed, while the bound without this factor is implied by~(\ref{e.midprob}). 
The latter assertion of Proposition~\ref{p.midneartouch} follows from $\PP \big( \mathsf{NT} \big\vert \mathsf{Mid} \big) \leq D \hata$. \qed

\begin{proposition}\label{p.brownianepsilonexc} 
Let $s$, $r$ be positive, with $s \geq 3r$; let $y \leq 0$; and let $\e > 0$. 
Let $X:[0,s] \to \R$  be a random process specified under the law~$\PP$ whose law is  $\mc{B}_{0,y}^{[0,s]} \big( \cdot \big\vert B < 0\big)$.
Let $E = E(X,r,\e)$ denote the event that $\sup_{x \in [r,2r]} X(x)$ is at least $- r^{1/2} \e$. There exists a constant $D > 0$ such that, for any such $s$, $r$, $y$ and $\e$, $\PP(E) \leq D \e$.
\end{proposition}

Some preliminaries will be of aid in proving this proposition.

Let $f:[u_0,v_0] \to \R$ denote a continuous function defined on a compact real interval. For any closed subinterval $[u,v] \subseteq [u_0,v_0]$, let $f^{[u,v]}:[u,v] \to \R$ denote the {\em bridge}---that is, the continuous function with vanishing endpoint values---that is an affine translate of $f$'s restriction to $[u,v]$. Namely,
$$
 f^{[u,v]}(x) \, = \, f (x) - 
 \tfrac{v - x}{v-u}  f(u) 
 - \tfrac{x-u}{v-u} f(v)  \, \, \, \, \textrm{for} \, \, \, \,  x \in  [u,v ] \, .
$$

Let $x,y \in [u_0,v_0]$ satisfy $x < y$.
We view the interval $[x,y]$ as a union $L \cup R$ of a left and a right subinterval,  
setting $L = [x,m]$ and $R = [m,y]$, where $m = (x+y)/2$. 

Note that the function $f$  is characterized by the list of data:
\begin{itemize}
\item the restriction of $f$ to $[u_0,x] \cup [y,v_0]$;
\item the left bridge $f^L$ and the right bridge $f^R$; and
\item the relative midpoint value $f\big((x+y)/2\big) - \big( f(x) + f(y) \big)/2$.
\end{itemize}
Indeed, given the listed data, $f$ may be reconstructed by recording its values on $[u_0,x] \cup [y,v_0]$; by recording the value $f\big((x+y)/2\big)$ via the first and third items; and by recovering its remaining values by adding to the affine interpolations of the endpoint values on $L$ and $R$ the respective bridges $f^L$ and $f^R$.
\begin{lemma}\label{l.threeitem}
Let $B:[u_0,v_0] \to \R$ have the law $\mc{B}_{0,*}^{[u_0,v_0]}$. The elements in the three-item list for~$B$ are independent. The bridges in the second item have respective laws---$\mc{B}_{0,0}^L$
and $\mc{B}_{0,0}^R$---of Brownian bridge on $L$ and $R$ with vanishing endpoint values. The law of the real random variable in the third item is Gaussian with mean zero and variance $(y-x)/4$. 
\end{lemma}
{\bf Proof.} Given the first element in the three item list for $B$, the third item takes the indicated form by explicit computation. By L\'evy's construction of Brownian motion~\cite{MortersPeres}, the two second item bridges are then independent standard Brownian bridges. 
\qed

{\bf Proof of Proposition~\ref{p.brownianepsilonexc}.} For $t > 0$, set $X_t:[0,s/t] \to \R$,  $X_t(x) = t^{-1/2} X(tx)$. Note that $X_t(s/t) = t^{-1/2}y$.
By Brownian scaling, we note that $X_r$ has the law of $X$ indexed by parameters $(r/t,s/t,t^{-1/2}y,\e)$ in place of $(r,s,y,\e)$.
Moreover, the spatial-temporal scaling $(x,y) \to \big(tx,t^{1/2}y\big)$ sends the event $E(X,r,\e)$ to the event $E(X_t,r/t,\e)$.
It is thus enough to prove the proposition for a given value of $r > 0$.
We will do so with $r = 2$.

Set $L = [1,3]$, $R = [3,5]$, and write $N = X(3) - \big( X(1) + X(5) \big)/2$. 
Consider the three-item list that represents $X$ in the case that $[x,y] = [1,5]$.
The two bridges in the second item are $X^L$ and $X^R$, and the relative value in the third item equals $N$.

We will prove the proposition by analysing a random experiment in which the process $X$ is sampled and then altered to produce a coupled process $X^r$. 
This resampled process $X^r$ will share the law of~$X$. The discussion of the experiment will include four claims each of whose proofs is given straight after the claim in question is stated.  

In the experiment, $X$ is first sampled and represented in the format of the three-item list.
The third element is discarded and resampled to equal $N^r$, a random variable that is selected independently according to the standard Gaussian law. 
Let $X^r: [0,s] \to \R$ denote
the process arising from the resampled three-item list. 

Let $S$ denote the event that $X^r(x) < 0$ for all $x \in (0,s)$. If $S$ occurs, we set $Z = X^r$. In the other case, we sample an independent copy of the process $X$ and repeat the procedure. This process continues until $Z:[0,\infty) \to \R$ is specified.

{\em Claim~$1$.} The process $Z:[0,s] \to \R$ has the law of $X$.\\
{\em Proof.} It is enough to argue that the conditional distribution of $X^r$
given the occurrence of $S$ equals the law of $X$. We will establish the {\em stronger assertion} that, 
given the first and second items in the three-item list that specifies $X$,
the conditional distribution of the third item that specifies $X$, and of the third item that specifies $X^r$ given $S$, coincide.
 
Thus suppose given the first and second items that specify $X$. Let $\heightmac$ denote the  random height, measurable with respect to these items, such that setting the third item value equal to $\heightmac$ ensures that $X:[0,s] \to \R$ assumes the value zero in $[1,5]$, but is never positive in this interval. Since the value $(y-x)/4$ in Lemma~\ref{l.threeitem} equals one in the present case, this lemma implies that, given the first and second items, the conditional distribution of the third item in the specification of $X$ is the law of a standard Gaussian random variable conditioned to be less than $\heightmac$. But the characterization of the event $S$ 
given the first and second items is simply that the standard Gaussian random variable~$N^r$ be less than $\heightmac$. This confirms the stronger assertion and completes the proof of Claim~$1$. \qed

In view of Claim~$1$, it suffices for the proof of Proposition~\ref{p.brownianepsilonexc} to argue that, for some $D > 0$, the bound  
$\PP \big( E(Z,2,\e) \big) \leq D \e$ holds for all $\e > 0$. 
 
 Let $K \in \N^+$ denote the step at which the procedure terminates. 
Thus, $S = \{ K = 1\}$. 
Let $\mc{F}$ denote the $\sigma$-algebra generated by the first and second items in the three-item list for $X$.
Let $H \subset \R$ denote the $\mc{F}$-measurable random set of $h \in \R$
such that the specification of the third item value to~$h$ alongside the given first and second items causes the event $S \cap E\big( X^r,2,\e \big)$ to occur. 

{\em Claim~$2$.} The random set $H$ is an $\mc{F}$-measurable interval whose length is almost surely at most~$2^{3/2}\e$. \\
{\em Proof.} Given~$\mc{F}$, the event $S \cap E\big( X^r,2,\e \big)$  is characterized by the condition that $\sup \big\{ X^r(z): z \in [2,4] \big\} \in (-2^{1/2}\e,0)$. When $N^r = \heightmac$, this supremum equals zero; let $Z \in [2,4]$
satisfy $X^r(Z) = 0$ when $N^r = \heightmac$. Then $X^r(Z) = - \tfrac{2 - \vert Z - 3\vert}{2}\lambda \leq - 2^{-1} \lambda$ when $N^r = \heightmac - \lambda$, so that the supremum is at most $-2^{1/2}\e$ when $\lambda \geq 2^{3/2}\e$. Thus, $H \subset [\heightmac - 2^{3/2}\e,\heightmac]$. The monotonicity of $X^r$ in $N^r$ implies that $H$ is an interval, so that Claim~$2$ is validated. \qed

Denoting by $\mu$ the standard Gaussian law, note that   
\begin{equation}\label{e.excevente}
\PP \big( E(Z,2,\e) \, \big\vert \, K = 1  \big) = \frac{\PP( N_r \in H )}{\PP(S)}  = \frac{\mu(H)}{\PP(S)} \, .
\end{equation}

{\em Claim~$3$.}  $\PP(S) \geq 4^{-1} \mu(1,\infty) e^{-2}$. \\
{\em Proof.} Consider four independent events: $\big\{ X(1) \leq -1 \big\} \cap \big\{ X(5) < X(1) \big\}$;  the supremum of the bridge $X^{[1,3]}$
is at most one; likewise for the bridge $X^{[3,5]}$; and $N^r < 0$.
 Lower bounds on the probabilities of these four events are: $2^{-1}\mu(1,\infty)$; $e^{-1}$; $e^{-1}$; and $1/2$. Indeed, $X(1)$ is stochastically dominated by a standard Gaussian random variable, as is $X(5) - X(1)$ conditionally on the value of $X(1)$; $\mc{B}_{0,0}^{[1,3]} \big( \sup_{x \in [1,3]} B(x) \geq r \big) = e^{-r^2}$
for $r \geq 0$ by Brownian scaling and equation $(3.40)$ in~\cite[Chapter~$4$]{KaratzasShreve}; while the third bound follows similarly to the second, and the fourth is trivial.
The four events ensure that $S$ occurs, whence Claim~$3$. \qed

{\em Claim~$4$.} For $x \in \R$ and $a > 0$, $\mu \big[ x, x + a \big] \leq \big( 2 \pi \big)^{-1/2} a$. \\
{\em Proof.} The standard Gaussian density is at most $(2\pi)^{-1/2}$. \qed

Applying Claims~$2$,~$3$ and~$4$ to~(\ref{e.excevente}), we learn that
$$
\PP \big( E(Z,2,\e) \, \big\vert \, K = 1  \big) \leq D \e
$$
where $D = 2 \pi^{-1/2} \cdot 4 \mu(1,\infty)^{-1} e^2$.
Note then that
$$
 \PP \big( E(X,2,\e) \big) = \PP \big( E(Z,2,\e) \big) =  \PP \big( E(Z,2,\e) \, \big\vert \,  K = 1  \big) \, ,
$$
so that  $\PP \big( E(X,2,\e) \big) \leq D\e$ is seen to hold for the same choice of $D > 0$. Thus we obtain Proposition~\ref{p.brownianepsilonexc}. \qed

\subsection{The rarity of twin peaks for the routed weight profile}

Theorem~\ref{t.nearmax} would now seem to be readily at hand on the basis of Theorem~\ref{t.brownianroutedprofile} and Proposition~\ref{p.midneartouch}. 
There is a gap to be bridged, however. To explain this, let $X = Y^\perp$   mean that `the random function $X$ is very similar to the random function $Y$ in the sense of Definition~\ref{d.verybrownian}'.
 By Theorem~\ref{t.brownianroutedprofile}, $Z_n(\cdot,a):[R-\ell,R+\ell] \to \R$ has (up to an additive shift) the form $B^\perp + \ell$, where $B$ is Brownian motion of rate two and $\ell$ is the linear function $\ell(x) = Kx$, with $K = 2^{1/2} \big( a(1-a)\big)^{-1}R$. Proposition~\ref{p.midneartouch} delivers pertinent information about the process $B + \ell$.  The extra little element we thus need is to understand that a random function that has the form $B^\perp + \ell$ also has the form $(B + \ell)^\perp$. This element is furnished in the proof of the next result, which translates Proposition~\ref{p.midneartouch} into a form which in unison with 
 Theorem~\ref{t.brownianroutedprofile} will then readily deliver 
Theorem~\ref{t.nearmax}.\\

\begin{corollary}\label{c.mid}
Let $R \in \R$, $\ell \geq 1$, $a \in n^{-1}\Z \cap (0,1)$, $\hata > 0$ and $\e > 0$.
Let $M$ denote the maximizer of  $Z_n(\cdot,a):[R-\ell,R+\ell] \to \R$. Suppose that $\vert R \vert \leq 2^{-1}c n^{1/9} \big( a \wedge (1 - a - n^{-1}) \big)^{7/9}$, where $c$ appears in Theorem~\ref{t.brownianroutedprofile}. Let $\ell > 0$ satisfy $-2^{-1}n^{1/3}a \leq R - \ell$ and $R + \ell \leq  2^{-1}n^{1/3}(1 - a- n^{-1})$. 
Abbreviating $Z = Z_n(\cdot,a)$,
let $\mathsf{Mid}(Z)$ denote the event that $M$ lies in $[R - \ell/3,R+\ell/3]$. 
Let $\mathsf{NT}(Z)$ denote the event that there exists $x \in [R-\ell,R+\ell]$ such that 
$\vert x - M \vert \in [\e,2\e]$ and 
$Z(x) \geq Z(M) - \hata\e^{1/2}$.
When $a$ lies in a compact interval in $(0,1)$, then there exist constants $H,h > 0$ and $n_0 \in \N$ determined by this compact interval such that, if we further suppose that  $\ell \leq h n^{1/\macrobig}$ and $n \geq n_0$, then
$$
\PP \Big( \mathsf{Mid}(Z) \cap \mathsf{NT}(Z) \Big) \, \leq \, \max \Big\{ \hata_* \cdot 
 \exp \big\{  - hR^2 \ell + H \ell^{\macroseventeen} \big( 1 + R^2 + \log \hata_*^{-1}  \big)^{5/6} \big\} ,   \exp \big\{ - h n^{1/12}   \big\} \Big\} \, , 
$$
where we denote $\hata_* = \min \{ \sigma,1 \}$.
\end{corollary}
{\bf Proof.} Set $X^*:[R-\ell,R+\ell] \to \R$, $X^*(R) = 0$, to be Brownian motion of diffusion rate two. Thus $X^*$ is the pure counterpart to the process $X:[R-\ell,R+\ell] \to \R$ from  Theorem~\ref{t.brownianroutedprofile} which is $\big( g,G' \ell^6,  \ell, \min \{ a,1-a \} n ,2 \big)$-Brownian for $G' = \Theta(1) G^{17/6} g^{-5/6} \big( a \wedge (1-a - n^{-1}) \big)^{-34/3}$.

Now define $\mathsf{MidNT}^*(X^*)$ to be the event that the process $[R-\ell,R+\ell] \to \R: x \to X^*(x) + Kx$ realizes the event $\mathsf{Mid} \cap \mathsf{NT}$. That is, 
$\mathsf{MidNT}^*$ is the set of those  continuous functions $f : [R-\ell,R+\ell] \to \R$ that vanish at $R$ and for which
$$
 \textrm{the map $[R-\ell,R+\ell] \to \R: x \to f(x) + K(x-R)$ belongs to $\mathsf{Mid} \cap \mathsf{NT}$} \, .
$$  
Let $Z^*:[R-\ell,R+\ell] \to \R$ be given by $Z^*(x) = X^*(x) + K(x-R)$, where $K = - 2^{1/2}\big( a(1-a) \big)^{-1} R + \e$, with $\e$ specified by Theorem~\ref{t.brownianroutedprofile}.
Given the form of $\e$, we see that, provided that $n \geq n_0$, we have $d_1 \vert R \vert \leq \vert K \vert \leq D_1 \vert R \vert$, where the positive constants $d_1$ and $D_1$ and the natural number $n_0$ are determined by the compact interval in $(0,1)$ in which $a \in n^{-1}\Z$ is supposed to lie.

Note that, by definition,
$$
\PP \Big( \mathsf{Mid}(Z^*) \cap \mathsf{NT}(Z^*) \Big) =  \PP \Big( \mathsf{MidNT}^*(X^*)  \Big) \, .
$$
Let $p \in (0,1)$ denote this probability. Recalling from above that the process $X:[R-\ell,R+\ell] \to \R$ is $\big( g,G' \ell^6,  \ell, \min \{ a,1-a \} n ,2 \big)$-Brownian for $G' =  \Theta(1) G^{17/6} g^{-5/6} ( a \wedge (1-a - n^{-1}) )^{-34/3}$.  
By Lemma~\ref{l.noupperbound}, we thus have that, when $\exp \big\{ - g \big( \min \{ a,1-a \} n \big)^{1/12} \big\} \leq p$, 
$$
\PP \Big( \mathsf{MidNT}^*(X)  \Big)  \, \leq \, q \cdot \exp \big\{ 5 (G')^{11/6} \ell^{6 + 77/6} g^{-5/6} \big\} \exp \Big\{ G' \ell \big( \log q^{-1} \big)^{5/6} \Big\} \, .
$$
for any value $q \in [p,1]$. 
 This left-hand side equals 
$\PP \big( \mathsf{Mid}(Z) \cap \mathsf{NT}(Z)  \big)$ by definition, while the value of $p$ satisfies
$$
 p \leq 
 D
 K^{-1}\ell^{-1/2}  \exp \big\{ - K^2 \ell/18  \big\} \min \{ \hata , 1 \} 
$$
as well as $p \leq D \hata$
by Proposition~\ref{p.midneartouch}. Choose $q = 
D
K^{-1}\ell^{-1/2}  \exp \big\{ - K^2 \ell/18  \big\}  \min \{ \hata , 1 \}$.
Recalling the notation $\hata_* = \hata \wedge 1$,
we find that
\begin{eqnarray*}
\PP \big( \mathsf{Mid}(Z) \cap \mathsf{NT}(Z) \big)   & \leq &  
D
K^{-1}\ell^{-1/2}  \exp \big\{ - K^2 \ell/18  \big\} \hata_*  \cdot g^{-1} \exp \big\{ 5 (G')^{11/6} \ell^{113/6} g^{-5/6} \big\} \\
 & & \qquad \qquad \times \, \exp \Big\{ G' \ell \big( \log \big(  D^{-1} K \ell^{1/2}  \exp \big\{  K^2 \ell/18  \big\} \hata_*^{-1}  \big) \big)^{5/6} \Big\} \, .
\end{eqnarray*}
Using $d_1 \vert R \vert \leq \vert K \vert \leq D_1 \vert R \vert$, and absorbing the $D^{-1} K \ell^{1/2}$ factor in the third exponential term into $\exp \big\{  K^2 \ell/18  \big\}$,
we obtain, for positive constants $H$ and $h$ determined by the compact interval in $(0,1)$ in which $a$ is supposed to lie, 
$$
\PP \big( \mathsf{Mid}(Z) \cap \mathsf{NT}(Z) \big) \leq H \exp \big\{ - h R^2 \ell \big\}  \exp \big\{ H  \ell^{\macroseventeen}  g^{-5/6} \big\} \exp \big\{ H \ell^7 \big( R^2 \ell + \log \hata_*^{-1}  \big)^{5/6} \big\} \,  \hata_* \, .
$$
(The dependence on $G$ and $g$ has been absorbed by $H$.)
Further note that, the factor $K^{-1}\ell^{-1/2}$, which is problematic for small $K \ell^{1/2}\geq 0$,
has been omitted by  making $H$ large enough; indeed, $\exp \big\{ - h R^2 \ell \big\}$ approaches one as  $K \ell^{1/2} \searrow 0$, while the other  right-hand exponential terms are at least one, and the left-hand side, being a probability, is at most one.

Noting that $\ell \geq 1$, and suitably increasing $H$, this upper bound is at most 
$$
 \exp \Big\{  - hR^2 \ell + H \ell^{\macroseventeen} \big( 1 + R^2 + \log \hata_*^{-1}  \big)^{5/6} \Big\} \, \hata_* \, .
$$

This completes the proof of Corollary~\ref{c.mid} in the case that  $\exp \big\{ - g \big( \min \{ a,1-a \} n \big)^{1/12} \big\} \leq p$. 
Suppose now that the opposing inequality holds. Choosing $q = \exp \big\{ - g \big( \min \{ a,1-a \} n \big)^{1/12} \big\}$, we find that
\begin{eqnarray*}
\PP \big( \mathsf{Mid}(Z) \cap \mathsf{NT}(Z) \big)   & \leq &  \exp \big\{ - g \big( \min \{ a,1-a \} n \big)^{1/12} \big\} \cdot g^{-1} \exp \big\{ 4 (G')^{11/6} \ell^6 g^{-5/6} \big\} \\
 & & \qquad \qquad \times \, \exp \Big\{ G' \ell g^{5/6} \big( \min \{ a,1-a \} n \big)^{5/72}   \Big\} \, .
\end{eqnarray*}
For positive constants $H$ and $h$ determined by the compact interval in $(0,1)$ in which $a$ is supposed to lie, we obtain
$$
\PP \big( \mathsf{Mid}(Z) \cap \mathsf{NT}(Z) \big)    \leq   \exp \big\{ - h n^{1/12} + H  \ell^{\macroseventeen}   + H \ell^7  n^{5/72}   \Big\} \, ;
$$
or, more simply,
$$
\PP \big( \mathsf{Mid}(Z) \cap \mathsf{NT}(Z) \big)    \leq   \exp \big\{ - h n^{1/12} + H \ell^{\macroseventeen}  n^{5/72}   \Big\} \, .
$$
Since we suppose that $\ell$ is at most a small constant multiple of $n^{1/\macrobig}$, we also have, after a decrease in the value of $h >0$,
$$
\PP \big( \mathsf{Mid}(Z) \cap \mathsf{NT}(Z) \big)    \leq   \exp \big\{ - h n^{1/12}   \big\} \, .
$$
This completes the proof of Corollary~\ref{c.mid}. \qed

{\bf Proof of Theorem~\ref{t.nearmax}.}
Corollary~\ref{c.mid} implies that,
when $a \in n^{-1}\Z$ lies in a compact interval in $(0,1)$, there exist constants $H,h > 0$ determined by this compact interval such that, for $\e \in (0,\ell/3)$,
\begin{eqnarray*}
& & \PP \Big( M \in [R - \ell/3,R+\ell/3]  \, , \, \sup_{x \in \R: \vert x - M \vert \in [\e,2\e]} Z_n(x,a) \geq Z_n(M,a) - \hata \e^{1/2} \Big) \\
 & \leq & \max \Big\{ \hata \cdot 
 \exp \big\{  - hR^2 \ell + H \ell^{\macroseventeen} \big( 1 + R^2 + \log \hata^{-1}  \big)^{5/6} \big\} ,   \exp \big\{ - h n^{1/12}   \big\} \Big\} \, . 
\end{eqnarray*}
Summing this bound over dyadic scales $[2^j \e, 2^{j+1}\e]$ from that indexed by $j =0$ until a final truncated scale of the form $[2^j \e, \ell']$, we learn that
\begin{eqnarray*}
& & \PP \Big( M \in [R - \ell/3,R+\ell/3]  \, , \, \sup_{x \in \R: \vert x - M \vert \in [\e,\ell/3]} \big(Z_n(x,a)+ \hata 2^{-1/2} (x - M)^{1/2}\big) \geq Z_n(M,a)  \Big) \\
 & \leq & \lceil \log_2 \big( \ell' \e^{-1} \big) \rceil \max \Big\{ \hata \cdot 
 \exp \big\{  - hR^2 \ell + H \ell^{\macroseventeen} \big( 1 + R^2 + \log \hata^{-1}  \big)^{5/6} \big\} ,   \exp \big\{ - h n^{1/12}   \big\} \Big\} \, . 
\end{eqnarray*}
Relabelling $\hata$ and adjusting the values of $H$ and $h$, we obtain
 Theorem~\ref{t.nearmax}. \qed

\section{Fluctuation in polymer weight and geometry}

Here we prove our robust modulus of continuity assertions Theorem~\ref{t.toolfluc}, which concerns polymer geometry, and Theorem~\ref{t.weight}, which concerns polymer weight.
As a consequence, we will prove Corollary~\ref{c.lateral}, which addresses the maximum fluctuation of polymers. We further prove Theorem~\ref{t.deviation}, a result which offers control on the fluctuation of polymers that is uniform in compact endpoint variation and in variation on a given dyadic scale for the polymer lifetime proportion at which fluctuation is measured.

There are eight subsections. The first introduces some basic tools needed on several later occasions in this article. 
Four sets of preliminaries that are needed for  the proofs of  Theorems~\ref{t.toolfluc} and~\ref{t.weight} are respectively treated in the ensuing four subsections: the assertion of a strong form of invariance for parabolic weight in Brownian LPP; a result, in the style of Corollary~\ref{c.lateral}, concerning the maximum fluctuation of polymers; control on polymer weight that is uniform in compact endpoint variation; and control on large local fluctuations in unscaled geodesic energy.
 The sixth subsection proves Theorems~\ref{t.toolfluc} and~\ref{t.weight}  and Corollary~\ref{c.lateral}. The seventh proves  Theorem~\ref{t.deviation} and the eighth Proposition~\ref{p.onepoint}.

\subsection{Some basics}
\subsubsection{The scaling principle}\label{s.scalingprinciple}
Write $\R^2_< = \big\{ (x,y) \in \R^2: x < y\big\}$. Let $(n,s_1,s_2) \in \N \times \R^2_<$ be a compatible triple. The quantity $n \tot$ is a positive integer, in view of the defining property~(\ref{e.ctprop}).
The scaling map $R_k: \R^2 \to \R^2$ has been defined whenever $k \in \N^+$, and thus we may speak of $R_n$  and $R_{n \tot}$.
The map $R_n$ is the composition of $R_{n \tot}$ and the transform $S_{\tot^{-1}}$ given by $\R^2 \to \R^2: (a,b) \to \big(a\tot^{-2/3},b\tot^{-1}\big)$. That is, the system of $n\tot$-zigzags is transformed into the system of $n$-zigzags
by an application of  $S_{\tot^{-1}}$.  Note 
that $\weight_n \big[ (x,s_1) \to (y,s_2) \big] =  \tot^{1/3} \weight_{n \tot}  \big[ (x \tot^{-2/3},\kappa) \to (y \tot^{-2/3},\kappa + 1) \big]$, where $\kappa = s_1 \tot^{-1}$; indeed this weight transformation law is valid for all zigzags, rather than just polymers, in view of~(\ref{e.weightzigzag}).
 
We may summarise these inferences by saying that the system of $n\tot$-zigzags, including their weight data, is transformed into the $n$-zigzag system, and its accompanying weight data, by the transformation $\big( a,b,c \big) \to \big( a \tot^{-1/3}  ,b \tot^{-2/3} , c \tot^{-1} \big)$, where the components refer to the changes suffered in weight,  
and
 horizontal and  vertical coordinates.
 This fact leads us to what we call the {\em scaling principle}. 
 
 \noindent{\em The scaling principle.} 
 Let $(n,s_1,s_2) \in \N \times \R^2_<$ be a compatible triple.
 Any statement concerning the system of $n$-zigzags, including weight information, is equivalent to the corresponding statement concerning the system of $n\tot$-zigzags, provided that the following changes are made:
 \begin{itemize}
 \item the index $n$ is replaced by $n\tot$;
 \item any time is multiplied by $\tot^{-1}$;
 \item any weight is multiplied by $\tot^{1/3}$;
 \item and any horizontal distance is multiplied by $\tot^{-2/3}$.
 \end{itemize}

\subsubsection{Tail bounds on one-point polymer weight}\label{s.parabolicweight}

Let  $\weight^\cup_n \big[ (x,h_1) \to (y,h_2) \big]$ denote the parabolically adjusted weight  $\weight_n \big[ (x,h_1) \to (y,h_2) \big] + 2^{-1/2}(y-x)^2 h_{1,2}^{-1}$.
We will have need on several occasions for control on the upper and lower tail of this random variable.
\begin{lemma}\label{l.onepointbounds}
There exist positive constants $C$ and $c$, and $n_0 \in \N$, such that the following holds. Let $n \in \N$ and $x,y \in \R$
satisfy $n \geq n_0$ and $\vert x -y \vert \leq c n^{1/9}$.
\begin{enumerate}
\item For $t \geq 0$,
$$
\PP \Big( \weight^\cup_n \big[ (x,0) \to (y,1) \big]  \geq t \Big) \leq C \exp \big\{ - c t^{3/2} \big\} \, .
$$
\item For  $t \geq 0$,
$$
\PP \Big( \weight^\cup_n \big[ (x,0) \to (y,1) \big]  \leq - t \Big) \leq C \exp \big\{ - c t^{3/2} \big\} \, .
$$
\end{enumerate}
\end{lemma}
{\bf Proof.} This result follows from \cite[Proposition~$3.6$]{DeBridge} and translation invariance of Brownian LPP. \qed

\subsubsection{Polymer uniqueness and ordering}\label{s.uniquenessordering}

A polymer with given endpoints is almost surely~unique.

\begin{lemma}\cite[Lemma~$4.6(1)$]{Patch}\label{l.polyunique}
  Let $x,y \in \R$. 
 There exists an $n$-zigzag from $(x,0)$ to $(y,1)$
 if and only if $y \geq x - n^{1/3}/2$.
 When the last condition is satisfied, there is almost surely a unique $n$-polymer from $(x,0)$ to $(y,1)$.
 \end{lemma}

A rather simple sandwiching fact about polymers will also be needed. Let $(x_1,x_2),(y_1,y_2) \in \R^2$ and  consider
a zigzag  $Z_1$ from $(x_1,s_1)$ to $(y_1,s_2)$ and another
 $Z_2$ from $(x_2,s_1)$ to $(y_2,s_2)$. 
We declare that $Z_1 \preceq Z_2$ if `$Z_2$ lies on or to the right of $Z_1$': formally, if $Z_2$ is contained in the union of the closed horizontal planar line segments whose left endpoints lie in $Z_1$.
\begin{lemma}\cite[Lemma $5.7$]{NonIntPoly}\label{l.sandwich}
Let $(n,s_1,s_2)$ be a compatible triple, and let  $(x_1,x_2)$ and $(y_1,y_2)$ belong to $\R_\leq^2$. Suppose that there is a unique $n$-polymer from $(x_i,s_1)$ to $(y_i,s_2)$, both when $i=1$ and $i=2$. (This circumstance occurs almost surely, and the resulting polymers have been labelled $\rho_n \big[ (x_1,s_1) \to (y_1,s_2) \big]$ and  $\rho_n \big[ (x_2,s_1) \to (y_2,s_2) \big]$.)
Now let $\rho$ denote any $n$-polymer that begins in $[x_1,x_2] \times \{ s_1\}$
and ends in  $[y_1,y_2] \times \{ s_2 \}$. Then 
$$
\rho_n  \big[ (x_1,s_1) \to (y_1,s_2) \big]  \, \preceq \, \rho \, \preceq \, \rho_n \big[ (x_2,s_1) \to (y_2,s_2) \big] \, .
$$
\end{lemma}

\subsubsection{Boldface notation for parameters in statement applications}\label{s.boldface}
Some later used outside results come equipped with parameters that must be set in any given application. When such applications are made, we employ a boldface notation to indicate the parameter labels of the results being applied. This device permits occasional reuse of symbols and disarms notational conflict.

\subsection{Invariance of the polymer weight field}

Given two collections of real-valued random variables $M^1$ and $M^2$ indexed by pairs $(x,h_1),(y,h_2) \in \R \times \Z$, we write $M^1 \equiv M_2$ to indicate that the two collections have the same law. This notation will be applied in the unscaled picture, when the fields of random variables are energies of geodesics, such as $M \big[ (x,h_1) \to (y,h_2) \big]$. Thus, the conditions $x \leq y$ and $h_1 \leq h_2$ must be imposed, to ensure that these energies are well defined. We abuse notation by setting values such as  $M \big[ (x,h_1) \to (y,h_2) \big]$ equal to zero when they are not well defined, so that the indexing by $(x,h_1),(y,h_2) \in \R \times \Z$ is admissible. The $\equiv$ notation will also be adopted in the scaled picture, where two collections of real-valued random variables $W^1$ and $W^2$ indexed by pairs $(x,h_1),(y,h_2) \in \R \times n^{-1}\Z$
are said to satisfy $W^1 \equiv W^2$ when they have the same law, and where a similar extension of domain is applied to permit this choice of domain for the two collections.

For $K \in \R$, let $\tau_K: \R^2 \to \R^2$ denote the shear map 
\begin{equation}\label{e.shearmap}
\tau(x,y) = (x + K y,y)
\end{equation}
The shear map will permit us to straighten sharply sloping corridors.  Our invariance result asserts that parabolic weight is statistically almost unchanged under application of the shear map. Later, when we prove Proposition~\ref{p.threepart}(2,3), it will permit us to propagate control on polymer weights from roughly square rectangles to much wider ones.

\begin{proposition}\label{p.invariance}
Suppose given two disjoint compact real intervals $I_1$ and $I_2$.
For $n \in \N$ and $K > 0$, let $\mc{K}_n(K)$ denote the set of quadruples $(x,y,h_1,h_2) \in \R^4$,
where the real values $x$ and $y$ vary over $[-2K,2K]$
and
where $h_1$ and $h_2$ respectively vary over elements of $n^{-1}\Z$ lying in $I_1$ and~$I_2$.
Suppose that $n \geq \Theta(1)$ and $\Theta(1) \leq K \leq \Theta(1) n^{1/18}$. Then 
$$
 \weight_n^\cup \big[ \tau_K(x,h_1) \to \tau_K(y,h_2) \big] \equiv 
 \weight_n^\cup \big[ (x,h_1) \to (y,h_2) \big] + \mc{E}(x,y,h_1,h_2)
$$
where the error terms $\mc{E} (x,y,h_1,h_2)$ are random but small, satisfying the uniform tail bound
\begin{eqnarray}
 & & \P \bigg( \sup \Big\{ \big\vert \mc{E}(x,y,h_1,h_2) \big\vert :  (x,y,h_1,h_2) \in \mc{K}_n(K) \Big\} \geq \Theta(1) n^{-1/9} \bigg) \nonumber \\
 &
  \leq & 
  \Theta(1) n^{2 + 2/3} K^{-2}  \exp \big\{ - \Theta(1) n^{1/12} K^{-3/2} \big\} +  
\Theta(1) K^2  \exp \big\{ - \Theta(1) n^{1/12} \big\}  \label{e.invariance}
 \, .
\end{eqnarray}
The $\Theta(1)$ terms may depend on the pair $(I_1,I_2)$ and on no other parameter.
\end{proposition}
The parameters $h_1$ and $h_2$ vary subject to $h_{1,2} = \Theta(1)$ in this result. During the upcoming derivation, we will monitor dependence on $h_{1,2}$
more closely, with a view to the potential for applications where $h_{1,2} \ll 1$, and will impose $h_{1,2} = \Theta(1)$ as we close out the proof of Proposition~\ref{p.invariance}.

The first element on our route to proving the proposition is the assertion of a strong form of  energetic invariance that is enjoyed by Brownian LPP. 
\begin{lemma}\label{l.invariance}
For $\kappa > 0$,
$M \big[ (x,h_1) \to (y,h_2) \big] \equiv \kappa^{1/2} M \big[ (x,h_1) \to (x + \kappa^{-1}(y-x),h_2) \big]$.
\end{lemma}
{\bf Proof.} The law of the Brownian motions $B: \R \times \Z \to \R$ that constituent the noise environment of Brownian LPP is invariant under the scaling
$B(z,n) \to \kappa^{1/2} B(x + \kappa^{-1}(z-x),n)$. \qed

By the definition of the shear map~$\tau_K$ and the specification~(\ref{e.weightgeneral}) of weight in terms of its unscaled counterpart, energy, we see that
$\weight_n \big[  \tau_K(x,h_1) \to \tau_K(y,h_2)  \big]$  is equal to

\begin{equation}\label{e:scaledcoordinates1} 
 2^{-1/2} n^{-1/3} \Big( M \big[ \big( nh_1 + 2n^{2/3}(x + K h_1),h_1 \big) \to   \big( nh_1 + 2n^{2/3}(y + K h_2),h_2 \big) \big]
  - (2n +K) h_{1,2} - 2 n^{2/3}(y-x) \Big) \, .
\end{equation} 
Set $x' = x \big(1 + 2 n^{-1/3} K \big)^{-1}$ and $y' = y \big(1 + 2 n^{-1/3} K \big)^{-1}$. Applying  Lemma~\ref{l.invariance}, we see that 
\begin{eqnarray*}
& & M \big[ \big( nh_1 + 2n^{2/3}(x + K h_1),h_1 \big) \to   \big( nh_1 + 2n^{2/3}(y + K h_2),h_2 \big) \big] \\
&
 \equiv &   \big(1 + 2 n^{-1/3} K \big)^{1/2} 
M \big[ \big( nh_1 + 2n^{2/3}x',h_1 \big) \to   \big( nh_1 + 2n^{2/3}y',h_2 \big) \big] \, ,
\end{eqnarray*}
where note that, in the usage of the notation $\equiv$, the two fields are viewed as functions of the variables  $(x,h_1)$ and $(y,h_2)$, with the dependence on $x$ and $y$ being communicated via
$x'$ and $y'$ in the right-hand case. Translating to scaled coordinates, we see that $\weight_n \big[  \tau_K(x,h_1) \to \tau_K(y,h_2)  \big]$ takes the form
$$
 \big(1 + 2 n^{-1/3} K \big)^{1/2}  \weight_n \big[  (x',h_1) \to (y',h_2)  \big] + R_1 + R_2 + R_3 \, ,
$$
where 
\begin{align*}
R_1 &= 2^{1/2} n^{2/3} h_{1,2} \big(  \big( 1 + 2 n^{-1/3}K \big)^{1/2} - 1 \big);\\ 
R_2 &= 2^{1/2} n^{1/3} (y-x)  \big(  \big( 1 + 2 n^{-1/3}K \big)^{-1/2} - 1 \big); \, \textrm{and} \\
R_3 &= - 2^{1/2}n^{1/3}K h_{1,2} \, .
\end{align*}
 
The term $R_1$ is seen to take the form $2^{1/2} n^{1/3} h_{1,2} K - 2^{-1/2} h_{1,2} K^2 + \Theta(n^{-1/6})h_{1,2}$, where the $\Theta(n^{-1/6})$ term is due to $\vert K \vert \leq \Theta(1) n^{1/18}$.
We have that
$$
 R_2 = -2^{1/2} (y-x) K + 2^{-1/2}3 n^{-1/3} (y-x) K^2 + \Theta \big( n^{-2/3} K^3 \big) \vert y - x \vert \, .
$$
Thus, $R_2 = -2^{1/2} (y-x) K + \Theta(n^{-1/6})$, since $\vert x - y \vert \leq 2K$ and $\vert K \vert \leq \Theta(1) n^{1/18}$.

Considering now the parabolically adjusted weight $\weight_n^\cup \big[  \tau_K(x,h_1) \to \tau_K(y,h_2)  \big]$ that is the subject of Proposition~\ref{p.invariance},
we see that it has the form 
$$
\weight_n \big[  \tau_K(x,h_1) \to \tau_K(y,h_2)  \big] + 2^{-1/2} \big( y-x + Kh_{1,2} \big)^2 h_{1,2}^{-1} \, .
$$
We expand the right-hand square $(a+b)^2$, $a = y-x$ and $b = K h_{1,2}$, and make use of the noted forms for $R_1$, $R_2$ and $R_3$, to find that
 $\weight_n^\cup \big[  \tau_K(x,h_1) \to \tau_K(y,h_2)  \big]$  equals (in the sense of the relation $\equiv$)
 $$
 \big(1 + 2 n^{-1/3} K \big)^{1/2}  \weight_n \big[  (x',h_1) \to (y',h_2)  \big] + 2^{-1/2} (y-x)^2 h_{1,2}^{-1} \, + \, \Theta(1)n^{-1/6}h_{1,2} \, .
 $$
Since $x'$ and $y'$ are small perturbations of $x$ and $y$, we can already recognise the desired weight $\weight_n^\cup \big[  (x,h_1) \to (y,h_2)  \big]$ in this display.
We next summarise our progress by stating and proving a lemma that clarifies the form of the discrepancy between the obtained and desired terms. A second lemma offers control on the tail of the discrepancy, so that the two lemmas will directly yield Proposition~\ref{p.invariance}.  We employ the shorthand  $\weight^\cup\Delta (x,y;x',y')$ to denote the parabolically adjusted weight difference
$$
\weight^\cup_{n h_{1,2}} \big[ \big(x'h_{1,2}^{-2/3},0\big)\to \big( y'h_{1,2}^{-2/3},1 \big) \big] -  \weight^\cup_{n h_{1,2}} \big[ \big(x h_{1,2}^{-2/3},0\big)\to \big( y h_{1,2}^{-2/3},1 \big) \big] \, .
$$ 
\begin{lemma}\label{l.obtained.desired}
$$
 \weight_n^\cup \big[  \tau_K(x,h_1) \to \tau_K(y,h_2)  \big]  \equiv  \weight_n^\cup \big[  (x,h_1) \to (y,h_2)  \big]  +  \Theta(1)n^{-1/6}h_{1,2}  + {\rm Error}_1(x,y) + {\rm Error}_2(x,y) \, ,
$$
where  
$$
{\rm Error}_1(x,y) \equiv  \big( 1 + \Theta(n^{-5/18}) \big) \Big( h_{1,2}^{1/3} \weight^\cup\Delta (x,y;x',y')  \,  + \,  \, \Theta(1) h_{1,2}^{-4/3} K^3 n^{-1/3} \Big) \, ;
$$
and 
$$
{\rm Error}_2(x,y) \equiv 
\Theta(n^{-5/18})  \weight_n \big[  (x,h_1) \to (y,h_2)  \big]  \, .
$$
\end{lemma}
{\bf Proof.}
Setting $E = \weight_n \big[  (x',h_1) \to (y',h_2)  \big] -  \weight_n \big[  (x,h_1) \to (y,h_2)  \big]$, and recalling that $\vert K \vert$ is at most $\Theta(1)n^{1/18}$, we find that
\begin{eqnarray*}
 \weight_n^\cup \big[  \tau_K(x,h_1) \to \tau_K(y,h_2)  \big] 
& \equiv & \big( 1 + \Theta(n^{-5/18}) \big)  \Big( \weight_n^\cup \big[  (x,h_1) \to (y,h_2)  \big] + E \Big) \\
& = & \weight_n^\cup \big[  (x,h_1) \to (y,h_2)  \big]  + {\rm Error} \, ,
\end{eqnarray*}
with  ${\rm Error} = {\rm Error}_1 + {\rm Error}_2
  +  \Theta(1)n^{-1/6}h_{1,2}$, where  ${\rm Error}_2$ has the form asserted in the lemma; and where ${\rm Error}_1 = 
 \big( 1 + \Theta(n^{-5/18}) \big) E$, for  a system $E$ that  by the scaling principle is seen to satisfy
$$
 E \, \equiv \, h_{1,2}^{1/3} \Big( \weight_{n h_{1,2}} \big[ \big(x'h_{1,2}^{-2/3},0\big)\to \big( y'h_{1,2}^{-2/3},1 \big) \big] -  \weight_{n h_{1,2}} \big[ \big(x h_{1,2}^{-2/3},0\big)\to \big( y h_{1,2}^{-2/3},1 \big) \big]  \Big) \, ;
$$
and thus 
$$
E \, \equiv  \, h_{1,2}^{1/3} \weight^\cup\Delta (x,y;x',y')  \,   - \, 2^{-1/2} (x' - y')^2 h_{1,2}^{-4/3} \, + \,  2^{-1/2} (x - y)^2 h_{1,2}^{-4/3} \, .
$$
Note that $2^{-1/2} (x' - y')^2 h_{1,2}^{-4/3}  -   2^{-1/2} (x - y)^2 h_{1,2}^{-4/3}$ equals 
$2^{-1/2} h_{1,2}^{-4/3} \Theta \big( \vert y - x \vert ( \vert x \vert + \vert y \vert ) \big) n^{-1/3}K$,
since $x' = x \big( 1 + \Theta(n^{-1/3}K) \big)$ and  $y' = y \big( 1 + \Theta(n^{-1/3}K) \big)$. Note that  $\Theta \big( \vert y - x \vert ( \vert x \vert + \vert y \vert ) \big) = \Theta(K^2)$
since $\vert x \vert$ and $\vert y \vert$ are at most $2K$; we have obtained the desired form for ${\rm Error}_1$ and thus Lemma~\ref{l.obtained.desired}. \qed 
\begin{lemma}\label{l.twoerrors}
Recall that $x,y \in \R$ are in absolute value at most $2K$.
\begin{enumerate}
\item Suppose that $h_{1,2} \leq \Theta(1)$,  $n h_{1,2} \geq \Theta(1)$ and
  $\Theta(1) h_{1,2}^{-1/18} \leq   K  \leq \Theta(1) n^{1/18} h_{1,2}^{13/18}$.
Then 
 $\P \Big(  \sup_{\vert x \vert , \vert y \vert \leq K} \big\vert {\rm Error}_1(x,y) \big\vert \geq \Theta(n^{-1/9}) \Big) \leq
 \Theta(1) n^{2/3} K^{-2} \exp \big\{ - \Theta(1) n^{1/12} K^{-3/2} \big\}$.
\item Suppose that  $n \geq \Theta(1)$, $K \leq \Theta(1) n^{1/18}$ and $h_{1,2} = \Theta(1)$.  Then 
$$
\P \Big(  \sup_{\vert x \vert , \vert y \vert \leq K}  \big\vert {\rm Error}_2(x,y) \big\vert \geq \Theta(n^{-1/6}) \Big) \leq
\Theta(1) K^2  \exp \big\{ - \Theta(1) n^{1/12} \big\}  \, .
$$
\end{enumerate}
\end{lemma}

{\bf Proof of Proposition~\ref{p.invariance}.} Following directly from the two preceding lemmas is a form of the sought result in which $h_1$ and $h_2$ are given; the factor of $\Theta(1)n^2$ is absent from line~(\ref{e.invariance}) in this version.
When a sum indexed by $h_1 \in n^{-1}\Z \cap I_1$ and $h_2 \in n^{-1}\Z \cap I_2$ is performed, the factor of $\Theta(1) n^2$ enters the right-hand side. \qed 

{\bf Proof of Lemma~\ref{l.twoerrors}:(1).}
Our argument principally rests on establishing that
\begin{eqnarray} & & \P \bigg( \sup \Big\{     \big\vert \weight^\cup\Delta (x,y;x',y') \big\vert  :  \vert x \vert, \vert y \vert \leq K    \Big\} 
  \geq  
 \Theta(1) h_{1,2}^{-1/3} n^{-1/9}  \bigg) \label{e.weightcup} \\
 & \leq &  \Theta(1) n^{2/3} K^{-2}  \exp \big\{ - \Theta(1) n^{1/12} K^{-3/2} \big\}
  \, . \nonumber
\end{eqnarray}
 Indeed,  ${\rm Error}_1(x,y)$ is a sum of two terms, and the displayed bound controls the first. The second, $\Theta(1) h_{1,2}^{-4/3} K^3 n^{-1/3}$,
 is at most $\Theta(n^{-1/9})$ because the needed condition $K \leq \Theta(1) h_{1,2}^{4/9} n^{2/27}$ is implied by the hypotheses of  Lemma~\ref{l.twoerrors}(1); thus, it indeed suffices to prove~(\ref{e.weightcup}).

The difference  $\weight^\cup\Delta (x,y;x',y')$  of parabolic weights will be addressed by~\cite[Theorem~$1.1$]{ModCon}.
Indeed, if we denote by $I$ and $J$ two compact real intervals, and set, for $r > 0$,
$$
p_{I,J}(r) = \P \bigg( \sup \Big\{ \weight^\cup_{nh_1} \big[ (u_1,0) \to (v_1,1) \big] -  \weight^\cup_{nh_1} \big[ (u_2,0) \to (v_2,1) \big] : u_1,u_2 \in I , v_1,v_2 \in J \Big\} \geq r  \bigg) \, ,
$$
then~\cite[Theorem~$1.1$]{ModCon} will shortly provide an upper bound on $p_{I,J}(r)$. To make a choice of the interval-pair $(I,J)$ for which such a bound will aid our derivation of~(\ref{e.weightcup}), we begin by noting that 
$$
 \max \big\{ \vert x' - x \vert , \vert y' - y \vert \big\} \leq 4K^2 n^{-1/3} \, .
$$ 
Indeed, $\vert x' - x \vert = \vert x \vert \big( 1 - (1 + 2 n^{-1/3}K)^{-1} \big) \leq 8 K^2 n^{-1/3}$ since $\vert x \vert \leq 2K$ and $2n^{-1/3}K \leq 1/2$; and similarly for $\vert y' -y \vert$.
In considering, as we do, $\weight^\cup\Delta (x,y;x',y')$, we see that the pair of starting endpoints $x' h_{1,2}^{-2/3}$ and $x h_{1,2}^{-2/3}$ are at distance at most $8K^2 n^{-1/3} h_{1,2}^{-2/3}$ and have absolute value at most $2K h_{1,2}^{-2/3}$. As such, we may find a set $\mc{I}$ of compact intervals, each contained in $\big[-2K h_{1,2}^{-2/3}, 2K h_{1,2}^{-2/3} \big]$ and of length~$16K^2 n^{-1/3} h_{1,2}^{-2/3}$, in such a way that $\vert \mc{I}\vert \leq \Theta(1) n^{1/3} K^{-1}$ while, for every $x \in [-K,K]$, there exists $I \in \mc{I}$ for which $x,x' \in I$. 
 
This of course implies that, for any $x,y \in \R$ in absolute value at most $2K$, we may find $(I,J) \in \mc{I}^2$ such that $x,x' \in I$ and $y,y' \in J$. By a union bound indexed by $\mc{I}^2$, we see that the left-hand side of~(\ref{e.weightcup}) is at most 
\begin{equation}\label{e.ij}
 \Theta(1) n^{2/3} K^{-2} \cdot
p_{I,J}\Big( \Theta(1) h_{1,2}^{-1/3} n^{-1/9} \Big) \, , 
\end{equation}
where $(I,J)$  is some element of $\mc{I}^2$. 
Thus we see that our usage of~\cite[Theorem~$1.1$]{ModCon}  should be made so as to find an upper bound on $p_{I,J}\big( \Theta(1) h_{1,2}^{-1/3} n^{-1/9} \big)$.
We set this result's parameters so that ${\bf n} = n h_{1,2}$; ${\bf x}$ and ${\bf y}$ are the left endpoints of the intervals $I$ and $J$ in~(\ref{e.ij}); ${\bm \e}$ equals the interval length $8 K^2 n^{-1/3} h_{1,2}^{-2/3}$; and with ${\bf R}$ chosen so that ${\bm \e}^{1/2} {\bf R} =  \Theta(1) h_{1,2}^{-1/3} n^{-1/9}$---which is to say, ${\bf R} = \Theta(1) n^{1/18} K^{-1}$.
 What we learn from this application of~\cite[Theorem~$1.1$]{ModCon}  is that the $p_{I,J}$ term in~(\ref{e.ij}) is at most $\exp \big\{ - \Theta(1) n^{1/12} K^{-3/2} \big\}$.
Thus are~(\ref{e.weightcup}) and Lemma~\ref{l.twoerrors}(1) obtained. It remains only to verify that the hypotheses of  Lemma~\ref{l.twoerrors}(1) are adequate to permit the application of~\cite[Theorem~$1.1$]{ModCon}. The application requires five conditions:
\begin{eqnarray*}
& &  [1] \, \, K^2 n^{-1/3} h_{1,2}^{-2/3} \leq \Theta(1) \, ;
 [2] \, \,  n h_{1,2} \geq \Theta(1) ,\ ;
 [3] \, \,  K h_{1,2}^{-2/3} \leq \Theta(1) n^{1/18} h_{1,2}^{1/18} \, ; \\
 & & [4] \, \, n^{1/18} K^{-1} \geq \Theta(1) \, ;
 \, \, \textrm{and} \,\, [5] \, \, n^{1/18} K^{-1} \leq \Theta(1)n^{1/18} h_{1,2}^{1/18} \, .
\end{eqnarray*}
 We consider the further conditions
 $$
 [6] \, \, K \leq n^{1/6} h_{1,2}^{1/3}  \Theta(1) \, ;  [7] \, \,  K  \leq \Theta(1) n^{1/18} h_{1,2}^{13/18} \, ;
 [8] \, \, h_{1,2} \leq 1 \, ; \, \, \textrm{and} \, \,
 [9] \, \,  K \geq \Theta(1) h_{1,2}^{-1/18}
 $$ 
Denoting equivalence and implication by $\leftrightarrow$ and $\to$, note that 
   $[1] \leftrightarrow [6]$; $[3] \leftrightarrow [7]$; $[7] \to [6]$; $[7,8] \to 4$; and 
 $[5] \leftrightarrow [9]$. Thus $[2,7,8,9] \to [1,2,3,4,5]$. Since $[2,7,8,9]$ are the hypotheses of Lemma~\ref{l.twoerrors}(1)
 and  $[1,2,3,4,5]$ permit its proof, this completes the needed hypothesis verification for this result. \qed

 {\bf (2).} The tail of the parabolically shifted weight $\weight^\cup_n \big[  (x,h_1) \to (y,h_2)  \big]$
is addressed by~\cite[Proposition~$1.5$]{ModCon}. Indeed, when this result is applied via the scaling principle, and with parameter settings ${\bf n} = n h_{1,2}$,
 ${\bf x} = x h_{1,2}^{-2/3}$, ${\bf y} = y h_{1,2}^{-2/3}$ and  ${\bf t} = R$, the outcome is the bound 
$$
 \P  \bigg( \sup \Big\{ \big\vert \weight^\cup_n \big[  (x+u,h_1) \to (y+v,h_2)  \big] \big\vert : u \in h_{1,2}^{2/3} \cdot [0,1] , v \in  h_{1,2}^{2/3}\cdot [0,1] \Big\} \geq h_{1,2}^{1/3} R  \bigg) \leq  
 \exp \big\{ - \Theta(1) R^{3/2} \big\} \, ,
$$
where we should bear in mind for the upcoming selection of $R > 0$ that this application demands that $R \leq \Theta(1) n^{1/18} h_{1,2}^{1/18}$.
A union bound over a mesh of  order $\big( Kh_{1,2}^{-2/3} \big)^2$ points $(x,y) \in [-K,K]^2$ then yields that
$$
 \P  \bigg( \sup \Big\{ \big\vert \weight^\cup_n \big[  (x,h_1) \to (y,h_2)  \big] \big\vert : x,y \in [-K,K]  \Big\} \geq h_{1,2}^{1/3} R  \bigg) \leq  
 \Theta(1) K^2 h_{1,2}^{-4/3} \exp \big\{ - \Theta(1) R^{3/2} \big\} \, .
$$
If we set   $h_{1,2}^{1/3} R$ equal to $\Theta(1)n^{1/18}$---a choice made so that the needed upper bound on $R$ is satisfied, provided that $h_{1,2}$ has unit order $\Theta(1)$, as we impose that it does---then this right-hand side takes the form $\Theta(1) K^2 h_{1,2}^{-4/3} \exp \big\{ - \Theta(1) n^{1/12} h_{1,2}^{-1/2} \big\}$.
 Since $h_{1,2} = \Theta(1)$, this has the form of the right-hand quantity in the bound maintained by Lemma~\ref{l.twoerrors}(2); thus,
if
the superscript~$\cup$ for $\weight_n$ were omitted from the preceding display---so that weight without parabolic shift were instead addressed---we would obtain the result that we seek.
 This alteration is permitted because the parabolic term $(y-x)^2 h_{1,2}^{-1}$, being at most $4K^2 h_{1,2}^{-1}$, is at most $\Theta(1) n^{2/9}$, in view  $K \leq \Theta(1) n^{1/9}$ (which is implied by our hypothesis on $K$) and $h_{1,2} = \Theta(1)$.  Thus, we obtain  Lemma~\ref{l.twoerrors}(2), subject to verifying that the hypotheses of this result are enough to permit usage of inputs during its proof. These inputs are 
the application of~\cite[Proposition~$1.5$]{ModCon}, which requires the conditions 
$$
 [1] \, \,  n h_{1,2} \geq \Theta(1) ; \, [2] \, \,  \vert x - y \vert h_{1,2}^{-2/3} \leq \Theta(1) n^{1/18} h_{1,2}^{2/3} ; \,
 [3] \, \, R \geq \Theta(1) ; \, \, \textrm{and} \, \,
 [4] \, \, R \leq \Theta(1)n^{1/18} h_{1,2}^{1/18} \, ;
 $$
 and control on a parabolic term, which requires 
 $[5] \, \, K \leq n^{1/9} h_{1,2}^{1/2}$.
 Recall that $\vert x - y \vert \leq 2K$, since $\vert x \vert$ and $\vert y \vert$ are at most $K$. In the preceding proof, we have set $R = \Theta(1)n^{1/18}h_{1,2}^{-1/3}$.
 The hypotheses of Lemma~\ref{l.twoerrors}(2), namely that $K \leq \Theta(1) n^{1/18}$, $n \geq \Theta(1)$ and $h_{1,2} = \Theta(1)$, are readily seen to imply $[1,2,3,4,5]$; so the needed hypothesis verification has been carried out for the proof of  Lemma~\ref{l.twoerrors}(2).  \qed
 
 \subsection{Local energetic fluctuation}
 
 In this short section, we offer control in Proposition~\ref{p.microfluc} on the tail of large local fluctuation in the unscaled geodesic energy, relying on oscillation estimates for Brownian motion which constitutes the underlying noise.
 
\begin{definition}
Let $n \in \N$ and $K,\sigma >0$. Let $\Delta_n(K,\sigma) \subseteq \big( \R \times \Z \big)^2$
denote the set of pairs $(x_1,s_1),(x_2,s_2) \in \R \times \Z$ with $x_1 \leq x_2 \leq x_1 + \sigma$;  $s_1 \leq s_2 \leq s_1 + \sigma$; and with each of $x_1$, $x_2$, $s_1$ and~$s_2$ at most $nK$ in absolute value.
\end{definition}

\begin{proposition}\label{p.microfluc}
For $K,\sigma > 0$ and  $r \geq 0$,
  \begin{eqnarray*}
 & & \P \bigg( \sup \Big\{ \big\vert M \big[ (x_1,s_1)\to(x_2,s_2) \big] \big\vert : \big(  (x_1,s_1),(x_2,s_2) \big) \in \Delta_n(K,\sigma) \Big\}  \geq r \bigg) \\ 
 & \leq & n \cdot 2^5 \pi^{-1/2} K \sigma^{-1/2} (\sigma + 1) r^{-1} \exp \big\{ - 2^{-4} r^2  (\sigma + 1)^{-2} \sigma^{-1} \big\} \, .
  \end{eqnarray*}
\end{proposition}

 Let $L > 0$, and let $f:[0,L] \to \R$. For $I \subseteq [0,L]$ and $r \in (0,L]$, set 
 $$
 \omega_I(f,r) \, 
 = \,  \sup \, \Big\{ \, \vert f(y) - f(x) \vert : 0 \leq x \leq y \leq x+r \leq L \, , \, x \in I \, \Big\} \, .
 $$
 In this way, $(0,L] \to \R: r \to \omega_{[0,L]}(f,r)$ is the modulus of continuity of the function $f$.
 \begin{lemma}\label{l.brownianmodcon}
 For $L > 0$,  let $B:[0,L] \to \R$ have the law of standard Brownian motion. For  $r \in (0,L]$ and $x > 0$,
 $$
  \P \Big( \omega_{[0,L]}\big( B, r \big) \geq x  \Big) \leq 2^4 \pi^{-1/2} L r^{-1/2} x^{-1} \exp \big\{ - 2^{-4} x^2 r^{-1} \big\} \, .
 $$
 \end{lemma}
{\bf Proof.}  Set $I = r\Z \cap [0,L]$. It is readily verified that $\omega_I\big( B, 2r \big) \geq \omega_{[0,L]}\big( B, r \big)/2$. This enables the first of the next displayed bounds,
  \begin{eqnarray*}
   \P \Big(  \omega_{[0,L]}\big( B, r \big) \geq x \Big) & \leq & \P \Big(  \omega_I\big( B, 2r \big) \geq x/2 \Big) 
     \leq  \vert I \vert \cdot \P \Big( \sup \big\{ \vert B(z) \vert : -1 \leq z \leq 1 \big\} \geq 2^{-3/2}xr^{-1/2} \Big) \\
     & \leq & 
    2L r^{-1} \cdot  2^3 \pi^{-1/2} L r^{-1/2} x^{-1} \exp \big\{ - 2^{-4} x^2 r^{-1} \big\} \, ,
  \end{eqnarray*}
  where the second bound is due to Brownian translation and scaling symmetries and, in whose third,
  $L \geq r$ is used to bound $\vert I \vert$ above by $2L/r$, and the reasoning in~(\ref{e.reflection}) bounds above the probability term. \qed
  
  {\bf Proof of Proposition~\ref{p.microfluc}.}
The concerned supremum is at most 
$$
 \big( \sigma + 1 \big) \cdot \sup \Big\{  \omega_{[-nK,nK]} \big(   x \to B(x,s) , \sigma \big) :  s \in \Z \cap [-K,K] \Big\} \, ,
$$
where recall that the constituent curves $B:\R \times \Z \to \R$ in the Brownian LPP noise environment are independent two-sided standard Brownian motions.   The displayed supremum thus has probability of being at least $r/(\sigma + 1)$ given by the bound offered by Lemma~\ref{l.brownianmodcon} with ${\bf L} = 2nK$, ${\bf r} = \sigma$ and ${\bf x} = r/(\sigma + 1)$. 
The result is Proposition~\ref{p.microfluc}. \qed

\subsection{Global polymer fluctuation}

In this section, we present and prove the next result, 
which concerns the location 
at which the polymer 
$\rho_n = \rho_n \big[ (0,0) \to (x,1) \big]$ departs a level $a \in n^{-1}\Z \cap (0,1)$. 
This theorem is in essence a weaker form of Corollary~\ref{c.lateral}; it will be needed for our principal proofs.

In a shorthand usage, 
$\rho_n^x$ will denote $\rho_n \big[ (0,0) \to (x,1) \big]$ (a polymer that is almost surely unique for given $x \in \R$, by Lemma~\ref{l.polyunique}). By the convention of Subsection~\ref{s.zigzagfunction}, $\rho_n^x(a)$ thus denotes the concerned location, at which a forward-in-time tracing of this polymer departs from the horizontal line $\R \times \{ a \}$.
\begin{theorem}\label{t.maximizersecond}
For $K$ any compact interval of $(0,1)$,
there exist positive constants $H = H(K)$ and $h = h(K)$ and an integer $n_0 = n_0(K)$ such that, if  
  $n \in \N$, $R \in \R$, $a \in n^{-1}\Z \cap K$ and $x \in \R$ satisfy $n \geq n_0$, 
 $\vert R \vert \leq  h n^{1/9}$ and $\vert x \vert \leq n^{2/3}$, then 
$$
 \PP \Big( \big\vert \rho_n^x(a) - xa \big\vert \geq R  \Big) \, \leq \, H  \exp \big\{  - hR^3 \big\}  \, . 
$$
\end{theorem}
Next is a result that includes an important special case of Theorem~\ref{t.maximizersecond}. We write $\rho_n$ for $\rho_n^0 = \rho_n \big[ (0,0) \to (0,1) \big]$.
\begin{proposition}\label{p.maximizer}
For $K$ any compact interval of $(0,1)$,
there exist positive constants $H = H(K)$ and $h = h(K)$ and an integer $n_0 = n_0(K)$ such that, if $n \in \N$, $R \in \R$  and $a \in n^{-1}\Z \cap K$ satisfy $n \geq n_0$ and
 $\vert R \vert \leq h n^{1/9}$, then 
$$
 \PP \Big( \rho_n(a) \in [R - 1,R+1] \Big) \, \leq \, H  \exp \big\{  - hR^3  \big\}
$$
and
$$
 \PP \Big(  \big\vert \rho_n(a) \big\vert \geq h n^{1/9}  \Big) \, \leq \, H \exp \big\{ - h n^{1/3} \big\} \, .
$$  
\end{proposition}
{\bf Proof.}
Recall from Lemma~\ref{l.routedprofile}(2) that $\rho_n(a)$ is the almost surely unique maximizer of the random function $x \to Z_n(x,a)$, and recall the formula~(\ref{e.zrewrite}) for the latter process.

To prove the first bound, we first note from~(\ref{e.zrewrite}) that the event that 
$Z_n(x,a) \geq - r$ for given $x,r \in \R$ entails that either  $a^{1/3} \scaledle_{n;(0,0)}^{\uparrow;a} \big(   a^{-2/3} x \big) \geq - 2^{-1}r$ or  
$$
(1-\aplus)^{1/3} \scaledle_{n;\aplus}^{\downarrow;(0,1)} \big(  (1-\aplus)^{-2/3} \xminus \big)  \geq - 2^{-1} r \, .
$$
Thus,  the event $\big\{ \sup_{x \in [R-1,R+1]} Z_n(x,a) \geq  - R^2 \big\}$ is contained in
\begin{eqnarray*}
   & & \bigg\{ \sup_{x \in a^{-2/3} \cdot [R-1,R+1]} \scaledle_{n;(0,0)}^{\uparrow;a} (   x )   \geq  - 2^{-1} a^{-1/3} R^2 \bigg\} \\
 & \cup & 
    \bigg\{ \sup_{x \in (1-\aplus)^{-2/3}  \cdot[R-1,R+1]} 
 \scaledle_{n;\aplus}^{\downarrow;(0,1)} ( \xminus )    \geq  - 2^{-1} (1-\aplus)^{-1/3} R^2 \bigg\} \, . 
\end{eqnarray*}
We now apply  the regularity of normalized ensembles expressed in Proposition~\ref{p.shiftbrownian}(1,2), and a tail bound~\cite[Proposition~$2.28$]{BrownianReg} for the deviation of regular ensemble top curves relative to parabolic curvature.  
Indeed, using the boldface notation, we choose the latter result's parameter ${\bf n}$ equal to $na$ and to $n(1 - \aplus)$, and thus find that the probability of each event in the above union is at most 
$$
 H \max \Big\{  \exp \big\{ - h R^3 \big\}  ,   \exp \big\{ - h n^{1/2} \big\} \Big\} \, ,
$$
where the existence of positive constants $H$ and $h$ is ensured because we may select the lower bound $n_0(K)$ so that $na$ and $n(1-\aplus)$ are at least the lower bound on ${\bf n}$ demanded in \cite[Proposition~$2.28$]{BrownianReg}. 

Since $\vert R \vert \leq h n^{1/9}$, we find that 
$$
 \PP \bigg( \sup_{x \in [R-1,R+1]} Z_n(x,a) \geq  - R^2 \bigg) \leq  2H  \exp \big\{ - h R^3 \big\}  \, .
 $$
 However, $\sup_{x \in \R} Z_n(x,a)$ equals $\weight_n \big[ (0,0) \to (0,1) \big]$ and thus has probability to 
 exceed $- R^2$ which is seen to be least $1 - C \exp \big\{ - c R^3  \big\}$ provided that  $R \leq h n^{1/3}$ (with $h = c^{1/2}$) by Lemma~\ref{l.onepointbounds}(2).
 
 From these inferences, the former assertion of Proposition~\ref{p.maximizer} follows after a relabelling of $H$ and~$h$.

To prove the latter bound, we first note from~(\ref{e.zrewrite}) that the event that 
$Z_n(x,a) \geq - g n^{2/9}$ for given $x \in \R$ entails that either  $\scaledle_{n;(0,0)}^{\uparrow;a} \big(  a^{-2/3} x \big) \geq - 2^{-1}g n^{2/9}$ or  
$$
(1-\aplus)^{1/3} \scaledle_{n;\aplus}^{\downarrow;(0,1)} \big( (1-\aplus)^{-2/3} \xminus \big)  \geq - 2^{-1} g n^{2/9} \, .
$$

By Proposition~\ref{p.shiftbrownian}(1,2),
 and \cite[Proposition~$2.30$]{BrownianReg}, for any $g' > 0$, there exist positive $G$ and $g > 0$, such that, for $a \in n^{-1}\Z \cap K$, the probabilities
$$
 \PP \Big( \exists \, y \in \R : \vert y \vert \geq a^{-2/3} g' n^{1/9} \, , \, \scaledle_{n;(0,0)}^{\uparrow;a} (  y ) \geq - 2^{-1}g n^{2/9} \Big)
$$
and
$$
 \PP \Big( \exists \, y \in \R : \vert y \vert \geq (1 - \aplus)^{-2/3} g' n^{1/9} \, , \,  \scaledle_{n;\aplus}^{\downarrow;(0,1)} (  y ) \geq - 2^{-1}g n^{2/9} \Big)
$$
are at most $G \exp \big\{ - g n^{1/3} \big\}$. Thus, the probability that $Z_n(x,a)$ attains a value of at least $- g n^{2/9}$ for $x \in \R$
 satisfying $\vert x \vert \geq g' n^{1/9}$ is at most $2G \exp \big\{ - g n^{1/3} \big\}$. However, $\sup_{x \in \R} Z_n(x,a)$ equals $\weight_n \big[ (0,0) \to (0,1) \big]$ and thus has probability to 
 exceed $- g n^{2/9}$ which is seen by 
 Lemma~\ref{l.onepointbounds}(2) to be at least $1 - C \exp \big\{ - c g^{3/2} n^{1/3} \big\}$.
 From these inferences, the latter assertion of Proposition~\ref{p.maximizer} follows.
\qed

To obtain Theorem~\ref{t.maximizersecond} from Proposition~\ref{p.maximizer}, we need to reach a comparable conclusion about polymers of route $(0,0) \to (x,1)$
as we have in the case $x=0$. In this task, we benefit from the strong invariance property of Brownian LPP that we indicated in Lemma~\ref{l.invariance}, but which is now presented in a scaled guise in the next lemma.  
Recall that~$\rho_n^x(a)$ denotes the location of departure  $\rho_n \big[ (0,0) \to (x,1) \big](a)$ of this polymer from the horizontal line~$\R~\!\!\times~\!\!\{ a \}$.
\begin{lemma}\label{l.rhoxa}
 \leavevmode
\begin{enumerate}
\item 
Let $x > - 2 n^{2/3}$.  The random processes $n^{-1}\Z \cap [0,1] \to \R$ that map $a$ to $\rho_n^x(a) - xa$ and  $\big( 1 + 2^{-1} n^{-2/3} x \big)\rho_n(a)$ are equal in law. 
\item 
Let $b > - 2n^{2/3}$. 
Let $\tau_b: \R^2 \to \R^2$ denote the shear map $\tau(x,y) = (x + yb,y)$.  The field of polymers $\rho_n$ is indexed by starting and ending points in $\R \times n^{-1}\Z$
(with a formal empty-set value for inadmissible choices of index). 
The image under $\tau_b$ of this field  is equal in law to the field of polymers $\big( 1 + 2^{-1} n^{-2/3} b \big)\rho_n$. 
\end{enumerate} 
\end{lemma}
{\bf Proof: (1).} The geodesics $\Gamma_n \big[ (0,0) \to (n + 2n^{2/3}x,n)\big]$
and  $\Gamma_n \big[ (0,0) \to (n,n)\big]$
are mapped to $\rho_n^x$ and~$\rho_n$ by the scaling map~$R_n$ in~(\ref{e.scalingmap}).
We {\em claim} that the former geodesic is mapped to a distributional copy of the latter by applying the transformation $\R \times \Z \to \R \times \Z$, $(z,h) \to \big(z, (1 + 2^{-1} n^{-2/3} x)^{-1} h\big)$. Indeed, the image of the noise environment of static Brownian LPP under this transformation is a copy of Brownian LPP in which the constituent Brownian motions have rate $1 + 2^{-1} n^{-2/3} x$. The energy of any given staircase is thus simply the product of  $\big( 1 + 2^{-1} n^{-2/3} x \big)^{1/2}$ and the energy of the staircase in Brownian LPP where the motions have the standard rate of one. Since this factor is present in computing the energy of all staircases, its presence does not affect the status of the geodesic, so that the claim is confirmed. Given the claim, we  apply  the scaling map~$R_n$ to obtain Lemma~\ref{l.rhoxa}(1).

{\bf (2).} The above inferences apply equally when we instead consider the $(z_1,z_2)$-indexed  fields of geodesics $\Gamma_n \big[ (2n^{2/3}z_1,0) \to (n + 2n^{2/3}(z_2+x),n)\big]$ 
and  $\Gamma_n \big[ (2n^{2/3}z_1,0) \to (2n^{2/3}n + z_2,n)\big]$ where $z_1,z_2 \in \R$. \qed

Allying Proposition~\ref{p.maximizer} with Lemma~\ref{l.rhoxa} leads to the next result, which equips us to give a short proof of Theorem~\ref{t.maximizersecond}.
\begin{corollary}\label{c.maximizer}
For $K$ any compact interval of $(0,1)$,
there exist positive constants $H = H(K)$ and $h = h(K)$ and an integer $n_0 = n_0(K)$ such that, if  
  $n \in \N$, $R \in \R$, $a \in n^{-1}\Z \cap K$ and $x \in \R$ satisfy $n \geq n_0$, 
 $\vert R \vert \leq h n^{1/9}$ and  $\vert x \vert \leq n^{2/3}$, then 
$$
 \PP \Big( \rho_n^x(a) - xa \in [R - 1,R+1] \Big) \, \leq \, H \exp \big\{  - hR^3 \big\} \, , 
$$
and
$$
 \PP \Big(  \big\vert \rho^x_n(a) - xa  \big\vert \geq h n^{1/9}  \Big) \, \leq \, H \exp \big\{ - h n^{1/3} \big\} \, .
$$  
\end{corollary}
{\bf Proof.} The dilation factor $1 + 2^{-1} n^{-2/3} x$ in Lemma~\ref{l.rhoxa}
lies on the interval $[1/2,3/2]$ under the present hypothesis that $\vert x \vert \leq n^{2/3}$. Denote this factor by $d$.
When the probabilities in Corollary~\ref{c.maximizer} are addressed by Proposition~\ref{p.maximizer} via Lemma~\ref{l.rhoxa}, the proposition is applied with ${\bf R} = d^{-1}R, d^{-1}R-1$ and $d^{-1}R+1$ in alliance with a simple union bound.  Thus we obtain the corollary by suitably adjusting the values of the constants~$H$ and $h$. \qed

{\bf Proof of Theorem~\ref{t.maximizersecond}.} Note that the set $\R \setminus (-R,R)$
is a subset of the union of intervals: 
$$
  \big(-\infty, - h n^{1/9}\big] \cup \big[ h n^{1/9}, \infty \big)  \, \cup \bigcup_{i= \lceil R \rceil}^{ \lfloor h n^{1/9} \rfloor} \Big( \big[ -i - 1,  -i + 1  \big] \cup  \big[ i - 1,  i + 1  \big] \Big) \, . 
$$
By the latter assertion of Corollary~\ref{c.maximizer}, the probability that $\rho_n^x(a) - xa$ takes a value in the union of the first two displayed intervals is bounded above by  $H \exp \big\{ - h n^{1/3} \big\}$.
Let $i \in \N$ satisfy $i \leq h n^{1/9}$. By this corollary's former assertion applied with ${\bf R}$ equal to either $i$ or $-i$,
 the   probability that $\rho_n^x(a) - xa$ assumes a value in $[i-1,i+1]$, or in $[-i-1,-i+1]$, 
 is at most  $H \exp \big\{ - h i^3 \big\}$. Summing these bounds over $i \in \llbracket  \lceil R \rceil, \lfloor h n^{1/9} \rfloor \rrbracket$, and employing the first inference alongside $\vert R \vert \leq h n^{1/9}$, we obtain Theorem~\ref{t.maximizersecond} after a relabelling of $H$ and $h$. \qed

\subsection{Compact uniform control on polymer weight}
In this section, we present a further tool needed for the proofs of Theorems~\ref{t.toolfluc} and~\ref{t.weight}. Control is gained on the tail of the maximum absolute value of the parabolic weight of polymers whose endpoints are varied over compact regions. 

\begin{proposition}\label{p.tailweight}
 For $n \geq \Theta(1)$, $\Theta(1) \leq R \leq n^{1/30}$  and $0 <K \leq n^{1/46}$, 
$$
 \P \left(   \sup_{\begin{subarray} c (x,h_1) \in [-K,K] \times [-3,-1] \\   (y,h_2) \in [-K,K] \times [1,3] \end{subarray}} \Big\vert   \weight^\cup_n \big[ (x,h_1) \to (y,h_2) \big]  \Big\vert \geq R   \right) \, \leq \, \Theta(1) K^2 \exp \big\{ - \Theta(1) R^{3/2} \big\} \, .
$$
\end{proposition}
The argument that we will give for Proposition~\ref{p.tailweight} mimics that of \cite[Propositions~$10.1$ and~$10.5$]{SlowBond}.

\noindent{\bf Proof of Proposition~\ref{p.tailweight}.}
The next presented proposition is sufficient to prove this assertion. Indeed, Proposition~\ref{p.threepart}(2) and~(3) imply it. \qed
 
\begin{proposition}\label{p.threepart}
 For $n \geq \Theta(1)$ and $\Theta(1) \leq R \leq \Theta(1)  n^{4/9} (\log n)^{-2}$, the following hold.
\begin{enumerate}
\item 
The probability that 
\begin{equation}\label{e.partone}
  \inf \Big\{   \weight^\cup_n \big[ (0,0) \to (y,h)  \big]  : y \in [-1,1] , h \in [1,3]  \Big\} \leq -R 
\end{equation}
is at most $\Theta(1)  \exp \big\{ -\Theta(1) R^{3/2} \big\}$.
\item For $K \leq \Theta(1)n^{1/18}$,
the probability that  the condition 
$$
  \inf \Big\{   \weight^\cup_n \big[ (x,h_1) \to (y,h_2) \big]   : x,y \in [-K,K] , h_1 \in [-3,-1],h_2 \in [1,3]  \Big\} \leq -R 
$$ 
holds is at most $\Theta(1) K^2 \exp \big\{ -\Theta(1) R^{3/2} \big\} + \Theta(1) n^3 K^4 \exp \big\{ - \Theta(1) n^{1/12} K^{-3/2} \big\}$.
\item Likewise for the condition 
$$
  \sup \Big\{   \weight^\cup_n \big[ (x,h_1) \to (y,h_2) \big]  : x,y \in [-K,K] , h_1 \in [-3,-1] , h_2 \in [1,3]  \Big\} \geq R \, .
$$  
\end{enumerate}
\end{proposition} 
\noindent{\bf Proof.} We will starting by proving the proposition's first part; and then the later parts when $K=1$. Note that the parabolic discrepancy  $\weight^\cup_n \big[ (x,h_1) \to (y,h_2) \big]  - \weight_n \big[ (x,h_1) \to (y,h_2) \big]$ is uniformly bounded in absolute value over all concerned routes $(0,0) \to (y,h)$ or $(x,h_1) \to (y,h_2)$ in these cases. Thus, we may prove these assertions with $\weight_n$ in place of $\weight_n^\cup$.

{\bf (1).}
We start by defining an infinite tree~$T = (V,E)$, embedded in the plane and rooted at $(0,0)$, each of whose vertices has four offspring.
The children of each vertex will be called left-low, left-high, right-low and right-high. The root is the unique vertex in generation zero. Its left-low child is at $(-1/2,1/2)$; its left-high child at $(-1/2,1/2 + 1)$; its right-low child at $(1/2,1/2)$; and its right-high child at $(1/2,1/2 +1)$.
Let $w$ denote the four-vector of planar points whose coordinates are these respective locations. The offspring of any child~$c$ of the root lie at $c + w/2$.
Iteratively, suppose that the locations of any vertex in $T$ of generation at most $m \in \N$ have been determined. Vertices in generation $m+1$
are placed in locations in the set $\big\{ v + 2^{-m}w: v \in V_n \big\}$, with an edge running from parent to child.

It is straightforward that the closure of the vertex set $V$ contains $[-1,1] \times [1,2]$. Let $(y,h)$ be an element in the latter set for which $h \in n^{-1}\Z$. In order to find a lower bound on $\weight_n \big[ (0,0) \to (y,h) \big]$, we aim to consider the sum of the weights of polymers that interpolate the endpoints of the edges in~$T$ along the end of the tree that runs from $(0,0)$ to the element $(y,h)$ of the closure of $V$. However, such endpoints lie in the plane, rather than in $\R \times n^{-1}\Z$; we begin by finding a nearby path through the latter index set. This path~$P$ will follow the tree end until its distance from the destination $(y,h)$ is a large multiple (of order $\log n$) of the microscopic spacing~$n^{-1}$; then it will jump directly to $(y,h)$, so that the approximating path is of finite length.

To any planar point $(z,s) \in \R^2$, we associate $(z,s)^{\downarrow}$, the element of $\R \times n^{-1}\Z$ that is first encountered on a journey due south that commences at $(z,s)$. 

We now specify $j \in \N$ to denote the smallest integer such that $2^{-j} \leq K_0 n^{-1}\log n$, where $K_0$ is a large constant.

Let $\big\{ (x_i,s_i): i \in \N \big\}$ denote a sequence of adjacent elements of $V$ with $(x_0,s_0) = 0$ whose limit equals $(y,h)$. Set $\big( x_i(n), s_i(n) \big) = (x_i,s_i)^{\downarrow}$. 
Respecify $\big(x_j(n),s_j(n) \big)$ to equal $(y,h)$. The path $P$ has elements $\big(x_i(n),s_i(n) \big)$ for $i \in \llbracket 0, j \rrbracket$, with edges between consecutively indexed elements. 

The path $P$ indeed offers a lower bound on the polymer weight from $(0,0)$ to $(y,h)$. Namely,
\begin{equation}\label{e.weightpath}
 \weight_n  \big[ (0,0) \to (y,h) \big] \, \geq \,  \sum_{i=0}^{j-1} \weight_n \big[ \big(x_i(n),s_i(n) \big) \to \big(x_{i+1}(n),s_{i+1}(n) \big) \big] \, .
\end{equation}
An edge in $T$ that connects a vertex~$(u,h_1)$ in generation $m \in \N$
to one of its children~$(v,h_2)$ satisfies $h_{1,2} \in 2^{1-m} \cdot \{ 1,3 \}$
and $\vert v -u \vert = 2^{-1 -m}$. Let $R > 0$. The edge
  is called $R$-typical if 
\begin{equation}\label{e.dtypical}
 2^{(m+1)/3} \Big\vert  \weight_n \big[  (u,h_1)^{\downarrow} \to (v,h_2)^{\downarrow} \big]  \Big\vert \leq R (m+1)^{2/3} \, .
\end{equation} 
The left-hand quantity is a normalized weight---it is random but of unit order, satisfying tail bounds that are uniform over edges in the tree~$T$. Indeed, the scaling principle from Section~\ref{s.scalingprinciple} and Lemma~\ref{l.onepointbounds} imply that any given edge crossing between generations $m$ and $m+1$ such that $n2^{-(m+1)}\geq n_0$---a condition verified when $m \leq j$---fails to be $R$-typical with probability that is at most~$C \exp \big\{ - c R^{3/2}(m+1) \big\}$. The parabolic curvature term in Lemma~\ref{l.onepointbounds} has been discarded because it takes the form $h_{1,2}^{-4/3}(u-v)^2= h_{1,2}^{-4/3}\cdot h_{1,2}^2= h_{1,2}^{2/3}\leq 1$.

We now specify the event $\mathsf{Typical} = \mathsf{Typical}(R)$ that every edge in the tree between generations $m$ and $m+1$ with $m \in \llbracket 0,j-1 \rrbracket$ is $R$-typical. We see that
$$
 \P \big( \neg \, \mathsf{Typical} \big)\, \leq \, \sum_{m=0}^{j-1} 4^m C \exp \big\{ - c R^{3/2} (m+1) \big\} \, ,
$$ 
whose right-hand side is at most $2C \exp \big\{ - 2^{-1} c R^{3/2} \big\}$ provided that $R \geq \Theta(1)$.

In its final step $\big(x_{j-1}(n),s_{j-1}(n) \big) \to \big(x_j(n),s_j(n) \big)$,  $P$ departs from the path beaten along $T$. In this regard,
we {\em claim} that
\begin{eqnarray}
 & & \PP \bigg(  \Big\vert \weight_n \big[ \big(x_{j-1}(n),s_{j-1}(n) \big) \to \big(x_j(n),s_j(n) \big) \big] \Big\vert \geq 1 \bigg) \label{e.finalterm} \\
 & \leq &  
 \Theta(1)
 n^{2/3} (\log n)^{1/2}  K_0^{1/2}   \exp \big\{ - 
 2^{-5} 
 n^{2/3}  (K_0 \log n + 2)^{-3}  \big\}
 \, . \nonumber
\end{eqnarray}
The differences $x_j(n) - x_{j-1}(n)$ and  $s_j(n) - s_{j-1}(n)$ are at most $K_0 n^{-1} \log n + n^{-1}$ in absolute value; and the four coordinates $x_{j-1}(n)$, $s_{j-1}(n)$, $x_j(n)$ and  $s_j(n)$ are at most two in this sense. 
Merely an order of $\log n$ microscopic levels separate the vertical coordinates  $s_{j-1}(n)$ and $s_j(n)$, making the unscaled picture a suitable context for proving the claim.
Indeed, specifying weight in terms of energy via~(\ref{e.weightgeneral}), and applying Proposition~\ref{p.microfluc} with ${\bf K} = 3$,  ${\bm \sigma} = K_0 \log n + 1$ and ${\bf r} = 2^{-1/2}n^{1/3}$,  we obtain the claimed~(\ref{e.finalterm}).

Applying~(\ref{e.dtypical}) and~(\ref{e.finalterm}) to~(\ref{e.weightpath}), we find that, when $\Theta(1) \leq R \leq \Theta(1) K_0^{-2} n^{4/9} (\log n)^{-2}$,
$$
\P \left( \inf_{\begin{subarray}  {c} y  \in [-1,1] \\ h \in [1,2] \end{subarray}}   \weight_n \big[ (0,0) \to (y,h) \big]   \leq -\kappa R \right) \leq \Theta(1) \exp \big\{ -  \Theta(1) R^{3/2} \big\} 
$$
for suitable positive $d$, $H$ and $h$;  here, we set $\kappa = 1+\sum_{m=0}^\infty (m+1)^{2/3} 2^{-(m+1)/3}$. The appearance of $\kappa$ may be absorbed by the usage of $\Theta(1)$ notation.  The result is Proposition~\ref{p.threepart}(1) with the instance of $[1,3]$ in~(\ref{e.partone}) replaced by $[1,2]$. We need to obtain the counterpart statement where the interval in question is $[2,3]$. We move the tree~$T$ upwards by one unit and add to it the edge that connects $(0,0)$ to $(0,1)$. We treat this edge in the preceding analysis as if it connects vertices of generation zero and one, and follow the rest of the analysis unchanged. This completes the proof of Proposition~\ref{p.threepart}(1).

\noindent{\bf (2) for $K=1$.} 
We prove this assertion in the {\em stronger form} where the probability upper bound is  $\Theta(1) K^2 \exp \big\{ -\Theta(1) R^{3/2} \big\}$.

Let $x,y \in [-1,1]$ and note that  $\weight_n \big[ (x,h_1) \to (y,h_2) \big]$ is at least the sum of  $\weight_n \big[ (x,h_1) \to (0,0) \big]$ and  $\weight_n \big[ (0,0) \to (y,h_2) \big]$. Since $\big\{ \weight_n \big[ (x,h_1) \to (0,0) \big] : (x,h_1) \in [-K,K] \times [-3,-1]\big\}$ has the law of   $\big\{ \weight_n \big[ (0,0)  \to (y,h_2) \big] : (y,h_2) \in [-K,K] \times [1,3] \big\}$ 
when the identification of $(x,h_1)$ with $(-x,-h_1)$ is made, 
two applications of Proposition~\ref{p.threepart}(1) and a union bound yield Proposition~\ref{p.threepart}(2), up to a relabelling of the constants $H$ and $h$.

\noindent{\bf (3) for $K=1$.} We also prove this assertion in the above mentioned stronger form.

Note that the occurrence of the condition
$$
\sup \Big\{  \weight_n \big[ (x,-1) \to (y,1) \big] :  (x,h_1) \in [-K,K] \times [-3,-1] \, , \,  (y,h_2) \in [-K,K] \times [1,3] \Big\} > R\,,
$$
alongside the conditions
\begin{equation}\label{e.twocond}
\min \, \Bigg\{ \, \inf_{\begin{subarray}{c}x \in [-K,K]\\ h_1 \in [-3,1] \end{subarray}} \weight_n \big[ (0,-4) \to (x,-1) \big] \, , \,  \inf_{\begin{subarray}{c}y \in [-K,K]\\ h_2 \in [1,3] \end{subarray}} \weight_n \big[ (y,h_2) \to (0,4) \big]  \, \Bigg\} \, \geq \, -  R/4\,, 
\end{equation}
entails that 
\begin{equation}\label{e.weightconc}
\weight_{n;(0,-4)}^{(0,4)}  > R/2 \, .
\end{equation}
Bounds on the failure probabilities of the two conditions~(\ref{e.twocond}) arise by applying Proposition~\ref{p.threepart}(2) in light of simple scaling properties. The one-point upper tail control offered by Lemma~\ref{l.onepointbounds}(1) via the scaling principle  provides an upper bound on the probability of~(\ref{e.weightconc}). Thus we obtain Proposition~\ref{p.threepart}(3). \qed

\noindent{\bf (2,3) for general $K$.} Proposition~\ref{p.invariance} permits us to reduce this derivation to the just obtained stronger form of the special case where $K=1$,
at the expense of the appearance of the additive term $\Theta(1) n^3 K^4 \exp \big\{ - \Theta(1) n^{1/12} K^{-3/2} \big\}$. Indeed, this further term takes the form of a multiple $\Theta(1) K^2$
of an upper bound on the probability appearing in  Proposition~\ref{p.invariance}. The reason for this form is that the indexing set-pair 
$\big( [-K,-K] \times [-3,-1] , [-K,K] \times [1,3] \big)$ may be covered by a union of at most $\Theta(1) K^2$ sets that are {\em distortions} of the  {\em standard} set-pair $\big( [-1,-1] \times [-3,-1] , [-1,1] \times [1,3] \big)$. By distortion, we mean that both elements in the pair are the image of their standard counterpart by an application of a given shear map of the form $\tau_\kappa$, where $\vert \kappa \vert \leq \Theta(K)$, composed with a horizontal translation. Proposition~\ref{p.invariance} with ${\bf K} = \kappa$ offers control over weights indexed by pairs $(x,h_1)$ and $(y,h_2)$ lying in a given image in this sense, and a union bound is then taken over the $\Theta(1)K^2$ distortions, so that the general~$K$ result is obtained with the indicated additive term. Note that it is these applications of Proposition~\ref{p.invariance} which are responsible for the imposition of the hypothesis that $K = \Theta(1) n^{1/18}$ in Proposition~\ref{p.threepart}. This completes the proof of this proposition.  \qed

\subsection{Modulus of continuity for polymer weight and geometry}
Here we prove Theorems~\ref{t.toolfluc} and~\ref{t.weight}, and Corollary~\ref{c.lateral}. To survey the route ahead,  it is perhaps helpful to recall first that the two theorems are expressed via a parameter $k \in \N$,
and that they control polymer fluctuation and subpath weights uniformly along stretches of vertical extension at most $2^{-k}$. 
Theorem~\ref{t.maximizersecond} offers control on the mid-life fluctuation of a polymer that traverses between given endpoints. By use of polymer ordering, a uniform form of this control may be gained for the fluctuation at height one-half of all polymers on routes $(x,0) \to (y,1)$ for $x,y \in \R$ of absolute value at most $r$. A mesh of points in a rectangular pattern with heights zero, one-half and one---and with horizontal distance between adjacent points shrunk in essence by a factor of $2^{2/3}$ from the level zero value of order~$r$---may thus be constructed such that polymers progress without sudden left or right movements between consecutive levels in the mesh.  Theorem~\ref{t.maximizersecond}  may be applied to mesh endpoints at adjacent vertical levels, so that, in view of polymer sandwiching, control on polymer fluctuation is gained at heights one-quarter and three-quarters. The construction proceeds iteratively, down to a dyadic scale that we will label $j \in \N$. The outcome is that  Lemmas~\ref{l.viable} and~\ref{l.viableconseq} assert that, with a high probability indicated in Lemma~\ref{l.open}(3), polymer fluctuation between consecutive mesh heights at level $j$---at distance $2^{-j}$---is controlled as Theorem~\ref{t.toolfluc} asserts, with an upper bound of the form $\Theta(1) 2^{-2j/3} (\log j)^{1/3}$.
For Theorem~\ref{t.toolfluc} to be obtained, what remains is to give the Kolmogorov continuity criterion argument, in which a sum over dyadic scales $j$ at least $k$ is performed to find a similar fluctuation upper bound between generic heights at displacement of order $2^{-k}$. There is a slight twist: the mesh construction becomes unmanageable at close to the microscopic scale, and a separate but simple device, Lemma~\ref{l.localcontrol}, is used to handle the very short scale. This apparatus also delivers  Theorem~\ref{t.weight}, which concerns polymer subpath weight. The aspect ratio of rectangles 
whose consecutive corners are vertically or horizontally adjacent in the mesh at scale~$j$ respects the two-thirds KPZ exponent up to a factor that is polylogarithmic in $j$. 
We  may thus apply (via the scaling principle) the uniform control gained in Proposition~\ref{p.tailweight} on weights for polymers crossing  rectangles to derive Theorem~\ref{t.weight}.

We begin, then, by constructing the meshes. Let $n \in \N$.
For each $i \in \N$ for which $2^i \leq n-1$, let $Y_i$ denote a subset of $n^{-1} \llbracket 0,n \rrbracket$
of cardinality $2^i +1$ such that, for each $k \in \llbracket 0,2^i \rrbracket$, there exists $y \in Y_i$ satisfying $\big\vert  y - k 2^{-i} \big\vert \leq 2^{-1}n^{-1}$. For example, we take $Y_0 = \{ 0,1 \}$ and $Y_1 = \{ 0, n^{-1} \lfloor n/2 \rfloor, 1 \}$. Note further that this set sequence may be constructed so that each set $Y_i$ for $i \in \N^+$ contains its predecessor, with the set $Y_i \setminus Y_{i-1}$ of newcomers interlacing the set $Y_{i-1}$ of existing members.

Let $r$ and $\kappa$ be positive. Set $X_0 = \{ -r,r\}$. For $i \in \N^+$ and $j \in \N$, set $X_i = X_i(r)$ according to 
$$
X_i(r) =  2^{-2i/3}(  j +1 )^{1/3} 
\kappa 
r \Z \cap (-2r,2r) \, .
$$ 
The parameter $r$ appears in Theorem~\ref{t.toolfluc}. The just introduced $\kappa$ is presently unspecified; it will shortly be joined by two further positive parameters $\kappa_1$ and $\kappa_2$, with the triple being specified when the relation that is demanded of them is derived.

Suppose given $j \in \N$. Let $i \in \llbracket 0,j \rrbracket$. We consider the graph $G_i = G_i(j) = \big( V_i, E_i \big)$ whose vertex set~$V_i$ equals $X_i \times Y_i$. Note that two dyadic scale parameters $i$ and $j$ (with $i \leq j$) are implicated in this definition, because, although we omit reference to $j$ from the notation $X_i$, this parameter enters via this object. What we called the mesh of scale~$j$ in overview is $G_j = G_j(j)$---in our inductive gaining of control on polymer fluctuation, we will treat the parameter $j \in \N$ as fixed, with a view to controlling fluctuation at vertical displacement $\Theta(1)2^{-j}$; we will descend through the scales $2^{-i}$, considering the $G_i(j)$ as $i$ increases to its final value of $j$. 

 Two elements $v,w \in V_i$ are {\em vertically adjacent} if their vertical coordinates $v_2$ and $w_2$ differ, and no element of $Y_i$ lies in the open interval delimited by $u_2$ and $v_2$. The elements $v$ and $w$ are {\em horizontally close} if their horizontal coordinates differ by at most
$\kappa_1 2^{-2i/3}(j + 1)^{1/3} r$, where $\kappa_1$ is a parameter that is at least $\kappa$. The edge between $v$ and $w$ belongs to the edge-set $E_i$ precisely when the vertices $v$ and $w$ are vertically adjacent and horizontally close. Note that $V_i$ may be identified with a rectangle in the lattice $\Z^2$ and that, when this identification is made, an element of $E_i$ is an edge between vertices that differ by one vertical unit and by at most $\kappa_1 \kappa^{-1}$ horizontal units. See Figure \ref{f.kolmogorovcontinuity}.

\begin{figure}[t]
\centering{\epsfig{file=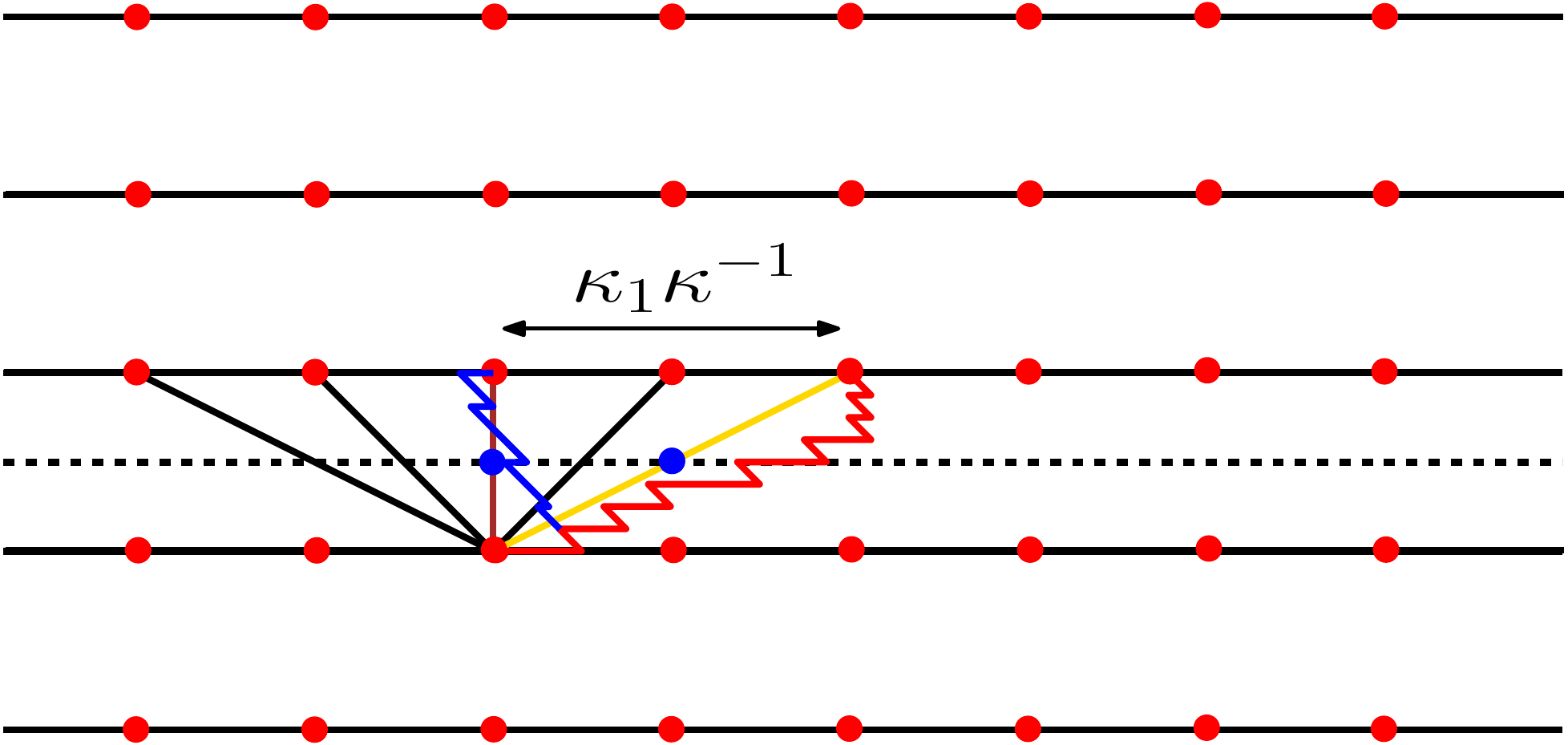, scale=0.4}}
\caption{The graph structure $G_i$ on  vertex set $V_i$, in which edges in $E_i$ connect elements that are vertically adjacent and horizontally close. The rightmost depicted edge is closed on account of fluctuation of the polymer away from the midpoint of the straight line representing the edge, while the vertical edge is open. 
}\label{f.kolmogorovcontinuity}
\end{figure}

We now assign a status of open, or closed, to each edge in each of the graphs $G_i$.
The assignation of this status will be random, and determined by a common realization of static Brownian LPP, governed by a law labelled~$\PP$ in a manner we now specify.  

Let $e \in E_i$ be an edge of the graph $G_i$ that connects the vertices $(x_1,h_1)$
and $(x_2,h_2)$, with $h_1 < h_2$. 
The distance $h_{1,2} = h_2 - h_1$ satisfies $\vert h_{1,2} - 2^{-i} \vert \leq n^{-1}$. There is a unique element $h^+$ of $Y_{i+1}$ in $(h_1,h_2)$, and its distance from both $h_1$ and $h_2$ differs from $2^{-i-1}$ by at most $n^{-1}$.

Let $\ell \big[ (x_1,h_1) \to (x_2,h_2) \big]$ denote the planar line segment with endpoints $(x_1,h_1)$
and $(x_2,h_2)$. Extending the notational abuse introduced in Subsection~\ref{s.zigzagfunction}, 
we write $\ell \big[ (x_1,h_1) \to (x_2,h_2) \big](h)$ for the horizontal coordinate at which this line segment visits the vertical coordinate $h \in [h_1,h_2]$.
 
Let $\kappa_2 > 0$ be a further parameter. The edge $e \in E_i$ will be declared to be $r$-{\em open} if 
\begin{equation}\label{e.open}
\Big\vert \, \rho_n \big[ (x_1,h_1) \to (x_2,h_2) \big] (h^+) - \ell \big[ (x_1,h_1) \to (x_2,h_2) \big](h^+) \, \Big\vert \, \leq \, \kappa_2 2^{-2i/3} 
\big( j + 1 \big)^{1/3}
 r \, .
\end{equation} 
An edge $e \in E_i$ that is not $r$-open is $r$-{\em closed}.

\begin{lemma}\label{l.open}
There exist positive $h$, $g_1$, $g_2$, $g_3$ and $r_0$, and $n_0 \in \N$, such that, when $n \in \N$ and  $j \in \N$ satisfy $n \geq n_0$ and $2^j \leq hn$, and $r \in \R$ satisfies  $r_0 \leq r \leq n^{1/10}$, the following hold.
\begin{enumerate}
\item Let $i \in \llbracket 0,j \rrbracket$. The $\PP$-probability that a given edge $e \in E_i$ is $( j + 1)^{1/3} r$-closed is at most $\exp \big\{ - g_1 r^3 (j + 1) \big\}$.
\item  Let $i \in \llbracket 0,j \rrbracket$.  
The $\PP$-probability that a  $( j + 1)^{1/3}  r$-closed edge in $E_i$ exists is at most
$\exp \big\{ - g_2 r^3 (j+1) \big\}$.
\item The $\PP$-probability that there exist $i \in \llbracket 0, j \rrbracket$
and an edge in $E_i$ that is $\big( j + 1 \big)^{1/3}r$-closed is at most $\exp \big\{ - g_3 r^3 (j+1) \big\}$. 
\end{enumerate}  
\end{lemma}
The various constants in this result may depend on $\kappa$, $\kappa_1$ and $\kappa_2$.

{\bf Proof of Lemma~\ref{l.open}: (1).} By the scaling principle from Section~\ref{s.scalingprinciple}, the probability $\PP \big( e \, \textrm{is closed} \big)$ 
takes the form 
$$
\PP \Big( \big\vert \rho_m \big[ (0,0) \to (x,1) \big] (a) - xa \big\vert \geq \big(j + 1 \big)^{1/3}r \big( 1 + \varepsilon \big) \Big) 
$$
where $m \in \N$ satisfies $\big\vert m - 2^{-i}n \big\vert < 1$; $a \in m^{-1}\Z \cap (0,1)$ satisfies $\vert a - 1/2 \vert \leq m^{-1}$; $x \in \R$ satisfies $\vert x \vert \leq \big(j + 1 \big)^{1/3}r$; and $\varepsilon$ is a small error term, satisfying $\vert \varepsilon \vert \leq m^{-1}$.

Since $\vert \varepsilon  \vert \leq 1/2$, the displayed probability is bounded above by Theorem~\ref{t.maximizersecond}
with ${\bf n} = m$ and ${\bf R} = 3/2 \cdot \big(j + 1 \big)^{1/3}r$. 
The theorem implies that the probability in question is at most the quantity $\exp \big\{ - h(3/2)^3 r^3 (j + 1) \big\}$. Setting $g_1 = h(3/2)^3$ yields the lemma's first part.

{\bf (2).} By the preceding part and a union bound, the probability in question is found to be at most 
$2^{5i/3}r \exp \big\{ - g_1 r^3 (j+1) \big\}$.
Since $i \leq j$ and $r \geq r_0$, the desired bound results by making suitable positive choices for $g_2$ and $r_0$.

{\bf (3).} The second part of the lemma is summed over $i \in \llbracket 0,j \rrbracket$ to obtain this result. \qed

Let $i \in \llbracket 0, j \rrbracket$.
A {\em horizontal piece} of scale~$i$ is a closed horizontal planar interval whose endpoints are consecutive elements in $X_i \times Y_i$. 
(This means that the concerned elements of $X_i$ differ by $2^{-2i/3} (j+1)^{1/3}r$.)
If the vertical coordinates of two horizontal pieces are {\em vertically adjacent}, we apply the latter term to the pair of pieces.

Let $\phi$ be an $n$-zigzag from $(x,0)$ to $(y,1)$ where $\vert x \vert$ and $\vert y \vert$ are at most $r$. Let $P$ denote the set of horizontal pieces of scale~$i$ that contain a point of departure of $\phi$ from a horizontal planar line segment (of the form $\R \times \{ h \}$ for some $h \in n^{-1}\Z \cap [0,1]$).
Consider any pair $\chi = \big\{ [x_1,x_2] \times \{h_1 \} , [y_1,y_2] \times \{ h_2 \} \big\}$ of vertically adjacent horizontal pieces in $P$. 
The pair $\chi$ is called {\em good} if two conditions are met. First, $\vert y_1 - x_1 \vert$, which equals $\vert y_2 - x_2 \vert$, must be at most $\kappa_1 \cdot 2^{-2i/3} (j+1)^{1/3}r$. Second, $\vert x_1 \vert$ and $\vert y_1 \vert$ must be at most $(j+1)^{1/3} r \big( 4 -  2^{-2i/3}  ( 1 - 2^{-2/3})^{-1} \kappa \big)$.

We call $\phi$ viable at scale $i$ if every pair of vertically adjacent horizontal pieces of scale~$i$ in $P$ is good. In a viable zigzag, fluctuation on a vertical mesh of scale $2^{-i}$ is consistently controlled; and the global horizontal location is also controlled via the upper bounds on $\vert x_1 \vert$ and $\vert y_1 \vert$, which contain a negative term with a factor of $2^{-2i/3}$ in order to facilitate the induction on $i$ that will deliver the next result. 
\begin{lemma}\label{l.viable}
Let $j \in \N$ satisfy $2^j \leq hn$, where the constant $h > 0$ is furnished by Lemma~\ref{l.open}. 
Suppose that every edge in $E_i$ is open for all $i \in \llbracket 0,j \rrbracket$.
Let $x$ and $y$ be two reals of absolute value at most $r$. 
When $n \geq 2^{39}$, the polymer $\rho_n \big[ (x,0) \to (y,1) \big]$ is viable at scale~$j$.
\end{lemma}
{\bf Proof.} We will prove by induction on $i \in \llbracket 0, j \rrbracket$ that, under the hypothesis of the lemma, any polymer of the given form is viable at scale~$i$. First take $i = 0$. The edges with endpoint pairs $\big\{ (-r,0),(-r,1) \big\}$ and $\big\{ (r,0),(r,1) \big\}$ have elements whose horizontal coordinates are shared by members of a pair and which are in absolute value equal to $r$. Thus these pairs are good, and the concerned polymer is viable at scale zero.

Now consider $i \in \intint{j}$, and assume that the inductive hypothesis holds for values of the index that are lower than~$i$. Write $\rho = \rho_n \big[ (x,0) \to (y,1) \big]$.
Let $[x_1,x_2] \times \{ h_1\}$ and $[y_1,y_2] \times \{ h_2 \}$ be two vertically adjacent horizontal pieces of scale $i$ with $h_1 < h_2$ that contain the point of departure of $\rho(h)$ at the respective heights $h \in \{ h_1,h_2\}$.
Thus $h_1$ and $h_2$ are elements of $Y_i$.

In order to demonstrate that this pair of vertically adjacent horizontal pieces is good, and thus complete the proof of the inductive step, we must show that
\begin{equation}\label{e.yonexone}
 \big\vert y_1 - x_1 \big\vert \leq \kappa_1 \cdot 2^{-2i/3} (j+1)^{1/3} r  \, ;
 \end{equation} 
\begin{equation}\label{e.ytwoxtwo} 
\big\vert y_2 - x_2 \big\vert \leq  \kappa_1 \cdot 2^{-2i/3} (j+1)^{1/3} r  \, ;
\end{equation}
and 
\begin{equation}\label{e.maxytwoxtwo} 
\max \big\{ \vert x_2 \vert , \vert y_2 \vert \big\} \leq  (j+1)^{1/3}r \big( 4 -   2^{-2i/3}\lambda \kappa \big) \, ,
\end{equation}
where we set $\lambda =  (1 - 2^{-2/3})^{-1}$.

We will argue that  $y_1 - x_1  \geq - \kappa_1 \cdot 2^{-2i/3} (j+1)^{1/3} r$. Indeed, the bound  $y_2 - x_2 \leq \kappa_1 \cdot 2^{-2i/3} (j+1)^{1/3} r$ follows from an almost identical argument to the one that we are about to give. These two bounds imply~(\ref{e.yonexone}) and~(\ref{e.ytwoxtwo}) because $y_2 - x_2 $ is in fact equal to  $y_1 - x_1$.

Since a horizontal piece of scale $i$ has length $2^{-2i/3} (j+1)^{1/3} \kappa r$,
we see that
\begin{equation}\label{e.xyoffset}
y_1 \geq \rho(h_2) - 2^{-2i/3} (j+1)^{1/3}\kappa r \, .
\end{equation}

One or other of $h_1$ and $h_2$ also belongs to $Y_{i-1}$. Suppose that $h_1 \in Y_{i-1}$; the other case entails no further complication. Let $h_3$ denote the lowest element of $Y_{i-1}$ that exceeds $h_1$.

Let $[u_1,u_2] \times \{ h_1 \}$ be the horizontal piece of scale $i-1$ that contains
$\rho(h_1)$.
Let $[v_1,v_2] \times \{ h_3 \}$ be a horizontal piece of scale $i-1$ that contains $\rho(h_3)$. 

By the inductive hypothesis, $\vert u_1 \vert \leq (j+1)^{1/3}r \big( 4 - 2^{-2(i-1)/3}\lambda \kappa \big)$. Note that 
\begin{equation}\label{e.xonebound}
\vert u_1 - x_1 \vert \leq 2^{-2(i-1)/3}(j+1)^{1/3}\kappa r
\end{equation}
because $\rho(h_1)$ lies in $[u_1,u_2]$ and $[x_1,x_2]$, so that the distance between $x_1$ and~$u_1$ may be at most the length $u_2 - u_1$, which is the longer of these two intervals. Thus, 
$$
\vert x_1 \vert \leq (j+1)^{1/3}r \big( 4 -  2^{-2(i-1)/3}\lambda \kappa  + 2^{-2(i-1)/3} \kappa \big) \, .
$$
This implies that $\vert x_1 \vert \leq (j+1)^{1/3} r \big( 4 -  2^{-2i/3}\lambda \kappa \big)$ because 
  $2^{2/3} (\lambda - 1)$ is equal to $\lambda$.  
A symmetric argument furnishes the same bound for $\vert x_2 \vert$. Thus do we obtain~(\ref{e.maxytwoxtwo}).

 The planar intervals  $[u_1,u_2] \times \{ h_1 \}$ and $[v_1,v_2] \times \{ h_3 \}$
are vertically consecutive horizontal intervals of scale $i-1$ that $\phi$ intersects. 
By the inductive hypothesis, the edge with endpoints $(u_1,h_1)$ and $(v_1,h_3)$ is thus seen to belong to $E_{j-1}$. By the lemma's hypothesis, this edge is open.  By~(\ref{e.open}), we learn that
\begin{equation}\label{e.rhohtwo}
 \rho (h_2) - \ell \big[ (u_1,h_1) \to (v_1,h_3) \big](h_2)  \, \geq \, - \kappa_2 2^{-2(i-1)/3} (j+1)^{1/3} r \, .
\end{equation}
The levels $h_1$, $h_2$ and $h_3$ each differ by at most $n^{-1}$
from the respective elements of a three-term arithmetic progression of real numbers with consecutive difference $2^{-j}$. Thus,
 $\big\vert h_2 - (h_1+h_3)/2 \big\vert \leq   2^{1-j} n^{-1} \leq 2 \vert h_3 - h_1 \vert n^{-1}$.
We know by the inductive hypothesis that $u_1$ and $v_1$ have absolute value at most $2 (j+1)^{1/3} r$. We thus see that
\begin{equation}\label{e.elluv}
 \ell \big[ (u_1,h_1) \to (v_1,h_3) \big](h_2) \, \geq \, \tfrac{u_1 + v_1}{2} \, - \, 2n^{-1} \cdot 4 (j+1)^{1/3} r \, .
\end{equation}
By~(\ref{e.xyoffset}),~(\ref{e.rhohtwo}) and~(\ref{e.elluv}), 
$$
 y_1 \geq u_1 - \tfrac{u_1 - v_1}{2} - 8 n^{-1} (j+1)^{1/3} r  - 2^{-2(i-1)/3} (j+1)^{1/3} \kappa_2 r - 2^{-2i/3} (j+1)^{1/3}\kappa r \, .
$$
Note that $v_1 - u_1 \geq - \kappa_1 \cdot 2^{-2(i-1)/3}(i+1)^{1/3} r$ because, as we have noted, the edge with endpoints $(u_1,h_1)$ and $(v_1,h_3)$ lies in $E_{i-1}$. Also using~(\ref{e.xonebound}), we find that
\begin{eqnarray*}
 y_1 & \geq & x_1 - 2^{-2i/3} (j+1)^{1/3}  2^{2/3}  \kappa r  -  2^{-1} \kappa_1 \cdot 2^{-2(i-1)/3} (j+1)^{1/3} r \\
 & & \qquad -\,  8  (j+1)^{1/3} n^{-1}  r \, - \,  2^{-2(i-1)/3} (j+1)^{1/3} \kappa_2 r \, - \, 2^{-2i/3} (j+1)^{1/3}\kappa r
\end{eqnarray*}
and, since $2^i \leq n$,
$$
 y_1 - x_1 \geq - 2^{-2i/3} (j+1)^{1/3} r \, \Big( (2^{2/3}  +1)  \kappa   + 2^{-1/3}\kappa_1    + 2^{2/3}\kappa_2 
  + 8 n^{-1/3}  \Big) \, .
$$
We now choose $\kappa$, $\kappa_1$ and $\kappa_2$ positive so that $(2^{2/3}+1) \kappa + 2^{-1/3} \kappa_1 + 2^{2/3} \kappa_2  < \kappa_1 - 2^{-10}$, alongside 
the already supposed $\kappa_1 \geq \kappa$.
 We find then that, for $n \geq 2^{39}$,
$$
 y_1 - x_1 \geq - \kappa_1  2^{-2i/3} (j+1)^{1/3} r  \, .
$$
This is the bound that we sought to show in order to verify~(\ref{e.yonexone}) and~(\ref{e.ytwoxtwo}). Since we already obtained~(\ref{e.maxytwoxtwo}), the proof of the inductive step in deriving Lemma~\ref{l.viable} is complete. \qed

\begin{lemma}\label{l.viableconseq}
Let $j \in \N$, and let $x,y \in \R$ be such that $\vert x \vert$ and $\vert y \vert$ are at most $r$. For $i \in \llbracket 0,j \rrbracket$, let $\phi$ be an $n$-zigzag between $(x,0)$ and $(y,1)$ that is viable at scale~$i$ for each $i \in \llbracket 0, j \rrbracket$. For any such~$i$, let $s_1,s_2 \in n^{-1}\Z \cap (0,1)$
denote {\em consecutive} elements of $Y_i$. Then
\begin{equation}\label{e.claimkolmogorov}
  \big\vert \phi(s_2) - \phi(s_1) \big\vert \leq
 \big( \kappa_1 + \kappa \big) 2^{-2i/3} (j+1)^{1/3} r \, .
\end{equation}
\end{lemma}
{\bf Proof.}
The horizontal pieces of scale $i$ to which $\big( \phi(s_1),s_1 \big)$
and $\big( \phi(s_2),s_2 \big)$ belong have length $2^{-2i/3} (j+1)^{1/3} \kappa r$
and are bordered on the left by a pair of points forming the endpoints of an edge in $E_i$; thus, these left-hand endpoints have horizontal separation of at 
$\kappa_1 \cdot 2^{-2i/3} (j+1)^{1/3} r$. From this, Lemma~\ref{l.viableconseq} is seen to hold. \qed

 \begin{corollary}\label{c.fluc}
There exist positive $H$, $h$ and $r_0$, and $n_0 \in \N$, such that, when $n \in \N$ and  $j \in \N$ satisfy $n \geq n_0$ and $2^j \leq hn$, and $r \in \R$ satisfies  $r_0 \leq r \leq n^{1/10}$, it is with probability at least $1 - H \exp \big\{ - h r^3 (j+1)\big\}$
that, for every $x,y \in \R$ of absolute value at most $r$,  and for any {\em consecutive} elements $h_1,h_2 \in n^{-1}\Z \cap [0,1]$ of $Y_j$,
$$
\Big\vert \rho_n \big[ (x,0) \to (y,1) \big](h_1) -  
\rho_n \big[ (x,0) \to (y,1) \big](h_2) \Big\vert \leq H h_{1,2}^{2/3} ( j+1 )^{1/3}r \, .
$$
 \end{corollary}
 {\bf Proof.}  This statement follows from Lemma~\ref{l.open}(3), and Lemmas~\ref{l.viable} and~\ref{l.viableconseq}. \qed

\begin{lemma}\label{l.localcontrol}
Let $\phi$ be an $n$-zigzag. 
Let $(x,s_1)$ and $(y,s_2)$ be elements in $\big(\R \times n^{-1}\Z \big) \cap \phi$ with $s_1 < s_2$. Let $(z,s)$ denote an element of $\phi$ that is encountered after $(x,s_1)$
but before $(y,s_2)$ as $\phi$ is traced in the sense of increasing height. Then 
$\vert z - x \vert$ and 
$\vert z - y \vert$
are at most $\vert y - x \vert + 2^{-1} n^{1/3} \tot$.
\end{lemma}
{\bf Proof.} By the deterministic properties of an $n$-zigzag outlined in Section \ref{s.staircasezigzag}, the leftmost position that $z$ may adopt is $x - 2^{-1}n^{1/3} \tot$. The rightmost such position is $y +  2^{-1}n^{1/3} \tot$. \qed

{\bf Proof of Theorem~\ref{t.toolfluc}:(1).} 
This argument is in the style of the derivation of the Kolmogorov continuity criterion. 

Recall that the statement we seek to verify comes equipped with a parameter $k \in \N$ that satisfies $2^k \leq h n$ for a small constant $h > 0$.
The statement claims the existence of an event of probability at least 
 $1 - H \exp \big\{ - h r^3 k \big\}$ 
on which the conclusion holds. We choose this event to be the intersection over indices $j \in \N$ satisfying $k \leq j \leq \log_2 (hn)$
of the event in Corollary~\ref{c.fluc}, modifying the values of~$H$ and $h$ so that this lower bound on probability holds. In deriving the inequality in the conclusion of Theorem~\ref{t.toolfluc}(1), we are thus permitted to invoke the conclusion of this corollary for any such index~$j$.

Recall further that we suppose $h_{1,2} \in (2^{-k-1},2^{-k}]$. Let $[s_1,s_2]$ be an interval of maximum length among those that are contained in $[h_1,h_2]$ and whose endpoints $s_1$ and $s_2$ are consecutive elements of $Y_i$ for some index $i$. We denote by $k^*$ the index $i$ thus selected; note that $k^*$ is one among $k$, $k+1$ or $k+2$. 
We select in the interval $[h_1,s_1]$ an interval of maximum length delimited by a pair $(s_3,s_4)$ of consecutive elements of some mesh $Y_i$; necessarily,  
$s_4 = s_1$, with the concerned value of $i$ being at least $k^*$. Likewise, an interval is selected within $[s_2,h_2]$, with the resulting endpoint pair $(s_5,s_6)$ satisfying $s_5 = s_2$.
 We write $K^*$ for the maximum index $i \in \N$ for which $2^{-i} \geq hn$, where the positive constant $h$ is contributed by Corollary~\ref{c.fluc}.  
 The selection of intervals is iterated, both to the left and to the right. It runs upwards through dyadic scales, and is stopped when all intervals of scale $i$ at most $K^*$ have been selected.  
 
 Among the closed intervals obtained in the procedure, the represented scales $i$ satisfy $k^* \leq i \leq K^*$, with $K^* = \lfloor \log_2 ( hn ) \rfloor$, and with each such scale appearing at most twice;
the union  $[s_1,s_2]$ of the intervals satisfies $0 \leq s_1 - h_1 \leq 2h^{-1}n^{-1}$ and  $0 \leq h_2 - s_2  \leq 2h^{-1}n^{-1}$.

Recall that we consider elements $(u,h_1)$ and $(v,h_2)$ of $\rho$, where here we denote  $\rho = \rho_n \big[ (x,0) \to (y,1) \big]$ for given $x,y \in \R$ with $\max \{ \vert x \vert , \vert y \vert \} \leq r$. To obtain the sought upper bound on $\vert u-v \vert$, we write 
\begin{equation}\label{e.uvtriangle}
 \vert u -v  \vert \leq  \vert u - \rho(s_1) \vert + \vert \rho(s_2) - \rho(s_1) \vert + \vert \rho(s_2) -v \vert \, ,
\end{equation}
whose middle right-hand term is seen to be at most $2 H \sum_{i = k^*}^{K^*} 2^{-2i/3} (i + 1)^{1/3}$, and thus at most $H 2^{-2k/3} (k + 1)^{1/3}$ after increase of~$H$, by the form of the procedure that constructed $[s_1,s_2]$ alongside Corollary~\ref{c.fluc}.
After an increase of $H$, this upper bound is seen to take the form $H 2^{-2k/3} (k + 1)^{1/3}$; or equally the form $H h_{1,2}^{2/3} \big( \log h_{1,2}^{-1} \big)^{1/3}$.
Thus, Theorem~\ref{t.toolfluc}(1) will be obtained, provided that we verify that the first and third right-hand terms in~(\ref{e.uvtriangle}) are smaller than the middle term.
Seeking to prove this, we let $s_0$ denote the greatest element of $Y_{K^*}$ that is less than $s_1$, and let $s_3$ be the least element of this set that exceeds $s_2$. 
Since $s_0 < h_1 \leq s_1$, $(u,h_1)$ lies on the subpath of $\rho$ between $\big(\rho(s_0),s_0\big)$ and  $\big(\rho(s_1),s_1\big)$.
Similarly,  $(v,h_2)$ lies on the subpath of $\rho$ between $\big(\rho(s_2),s_2\big)$ and  $\big(\rho(s_3),s_3\big)$.
 Lemma~\ref{l.localcontrol} thus implies that
 $\vert u - \rho(s_1) \vert$ is at most  $\vert \rho(s_0) - \rho(s_1) \vert + 2^{-1} n^{1/3} s_{0,1}$; and that
  $\vert \rho(s_2) -v \vert$ is at most  $\vert \rho(s_2) - \rho(s_3) \vert + 2^{-1} n^{1/3} s_{2,3}$. 
The pairs $(s_0,s_1)$ and $(s_2,s_3)$ comprise consecutive elements of $Y_{K^*}$, where recall that $K^* = \lfloor \log_2 (hn) \rfloor$; 
so that $s_{0,1}$ and $s_{2,3}$ are 
at most $(2 h^{-1} +1)n^{-1}$.
By the conclusion of Corollary~\ref{c.fluc},  $\vert \rho(s_0) - \rho(s_1) \vert$ and $\vert \rho(s_2) - \rho(s_3) \vert$ are thus seen to be at most of order $n^{-2/3} \big( \log n \big)^{1/3}r$ after suitable adjustment to $H$. The sought upper bound, of order $2^{-2k/3} (k + 1)^{1/3}$, holds because $2^k \leq n$. 
This completes the proof of Theorem~\ref{t.toolfluc}(1).

{\bf (2).} Recall that instead we suppose that $h_{1,2} < H n^{-1}$. Let $[s_0,s_3]$ denote an interval  containing $[h_1,h_2]$ whose endpoints 
are consecutive elements of the mesh $Y_{K^*}$; here, $K^*$ continues to denote $\lfloor \log_2 (hn) \rfloor$,
with $h$ now decreased from its value in Corollary~\ref{c.fluc} suitably to ensure the existence of such $s_0$ and $s_3$.
The notation $s_0$ and $s_3$ is used to indicate the similarity of these quantities with the usage in the preceding case.
 Now, however, $s_{0,3}$ is at most  $2n^{-1}$. 
Thus, Corollary~\ref{c.fluc} with~$j$ chosen so that $2^j = \Theta(n)$ yields that, on an event of probability at least 1 - $n^{-hr^3}$, the bound  $\vert \rho(s_0) - \rho(s_3) \vert \leq G n^{-2/3} \big( \log n \big)^{1/3}r$ for suitably high $G$ and for all concerned choices of $h_1$ and $h_2$. Since $(u,h_1)$ and $(v,h_2)$ lie on the subpath of $\rho$ between 
$\big(\rho(s_0),s_0\big)$ and $\big(\rho(s_3),s_3\big)$, Lemma~\ref{l.localcontrol} thus implies that $\vert u - \rho (s_0) \vert$ and $\vert v - \rho(s_0) \vert$ are at most $G n^{-2/3} \big( \log n \big)^{1/3}r + n^{-2/3}H$. Applying the triangle inequality, we learn that Theorem~\ref{t.toolfluc}(2) holds with a suitably increased value of $G$. \qed

 {\bf Proof of Corollary~\ref{c.lateral}.} Take ${\bf j}= 0$, ${\bf u} = x$, ${\bf h_1} = 0$, ${\bf v} = u$ and ${\bf h_2} = h$ in Theorem~\ref{t.toolfluc}. Since $\vert x \vert \leq r$, the corollary follows by increasing $H > 0$. \qed

{\bf Proof of Theorem~\ref{t.weight}: (1).} The mooted event of probability at least $1 - H \exp \big\{ - h r^3 k \big\}$  will be chosen to ensure that the conclusions of Theorem~\ref{t.toolfluc} and Corollary~\ref{c.lateral} hold.

We have $h_{1,2} \in (2^{-k-1},2^{-k}]$. By the conclusion of Theorem~\ref{t.toolfluc}, 
\begin{equation}\label{e.uvdifference}
\vert v - u \vert \leq H h_{1,2}^{2/3} \big( \log (1 + h_{1,2}^{-1}) \big)^{1/3}r \, .
\end{equation}
By the conclusion of Corollary~\ref{c.lateral}, $\vert u \vert \leq H r$. 

Let $R \subset \R^2$ denote the rectangle $\big[ 0, Y \big] \times \big[0,2^{-k} \big]$, where $Y$ denotes  $H 2^{-2k/3} k^{1/3}r$. 
The lower-third $R^-$ of $R$ is $[0,Y] \times \big[0,2^{-k}/3 \big]$; its upper-third $R^+$ is  $[0,Y] \times \big[2/3 \cdot 2^{-k} , 2^{-k} \big]$. To any translate $T = R+(x,y)$ indexed by $(x,y)$
is evidently associated a lower-third $T^-$ and an upper-third $T^+$.  

In light of the noted bounds, we may find a collection $\mc{C}$ of translates of $R$ by vectors in $\R \times n^{-1}\Z$ such that $\vert \mc{C} \vert$ is at most a constant multiple of $2^{5k/3} k^{-1/3}$ such that $(u,h_1) \in T^-$ and $(v,h_2) \in T^+$ for some element $T$ of $\mc{C}$.

We apply Proposition~\ref{p.tailweight} with ${\bf K}$ of order  $k^{1/3}r$, ${\bf R} = K$ and ${\bf n} = 6n 2^{-k}$ via the scaling principle, and use a union bound, to find that, for $K$ sufficiently high,
$$
 \PP \Big( \sup \big\vert \weight_n \big[ (u,h_1) \to (v,h_2) \big] \big\vert \geq \Theta(1) (K + H^2 k^{2/3} r^2) 2^{-k/3} \Big) \, \leq \, \vert \mc{C} \vert \cdot  k^{2/3}r^2 G \exp \big\{ - d K^{3/2} \big\} \, ,
$$
where the supremum is taken over all choices $(u,h_1)$ and $(v,h_2)$ with  $h_1,h_2 \in n^{-1}\Z \cap [0,1]$ and  $h_{1,2} \in (2^{-k-1},2^{-k}]$ that belong to $\rho_n \big[ (x,0) \to (y,1) \big]$
where $x$ and $y$ vary over real values of absolute value at most $r$. Because Proposition~\ref{p.tailweight} treats parabolically adjusted weight, we need to take account of parabolic curvature, and, in view of~(\ref{e.uvdifference}),  we do so by means of the above term  $H^2 k^{2/3} r^2 2^{-k/3}$. 
The upper bounds on ${\bf K}$ and ${\bf R}$ hypothesised by the proposition are satisfied in view of the assumption that $r \leq (n h_{1,2})^{1/50}$.

By choosing $K$ to be a large multiple of $\big( \log (r 2^k) \big)^{2/3}$, the right-hand factors of $\vert \mc{C} \vert  k^{2/3}$, which grows as a power of $2^k$, and of $r^2$, may be removed, at the expense of a decrease in the value of $d > 0$. The form of Theorem~\ref{t.weight}(1) entails that we desire this failure probability upper bound  $\exp \big\{ - d K^{3/2} \big\}$ be at most $H \exp \big\{ - h r^3 k \big\}$. This condition is ensured if we adjust $K$ so that it equals a large constant multiple of $r^2 k^{2/3}$. This adjustment is an increase to the value of $K$: indeed,  since $r$ is supposed to be high, and $h_{1,2} \leq 1$, the adjusted value of $K$ is at least the maximum of $r^2$ and  the given large multiple of $\big( \log (r 2^k) \big)^{2/3}$.
(The choice $K = \Theta(1)r^2 k^{2/3}$ imposes the constraint $r k^{1/3} \leq \Theta(1) (n 2^{-k})^{1/60}$ to admit the above usage of  Proposition~\ref{p.tailweight}. The upper bound on $r$ in Theorem~\ref{t.weight}(1) implies this constraint.)
 Since the adjusted value of $K$ may be absorbed into the preceding display's term  $H^2 k^{2/3} r^2$ that arose from parabolic curvature, we complete the derivation of Theorem~\ref{t.weight}(1) by decreasing $h > 0$ if need be.

{\bf (2).} Suppose instead that $h_{1,2} < H n^{-1}$. The points $(u,h_1)$ and $(v,h_2)$ lie on a polymer of the form 
$\rho \big[ (x,0) \to (y,1) \big]$ where $\vert x \vert$ and $\vert y \vert$ are at most $r$. In this case, we will rely on control on Brownian oscillation and adopt the unscaled perspective to finish the proof. 
By Theorem \ref{t.toolfluc}(2), with probability
 at least $1 - H n^{- h r^3}$,
$$
\big\vert v -u \big\vert \leq G_1 n^{-2/3} ( \log n )^{1/3}r
 \, .
$$
Thus, recalling \eqref{e.weightgeneral}, it suffices  to prove 
that, with probability at least $1 - H n^{- h r^3},$
$$M \big[ (u,i) \to (v,j) \big] \le \frac{G}{10} r^2 \cdot  ( \log n )^{2/3}$$ for all $1\le i\le j\le n$, with $j-i\le H$, $u\le v,$ $|u-v|\le G_1( \log n )^{1/3}r,$ and $|u|\le O\big(n( \log n )^{1/3}r\big)$.
To prove this bound, we will simply bound the probability that 
$$
M \big[ (u,i) \to (v,i+1) \big] \ge \frac{G}{H} r^2 \cdot  ( \log n )^{2/3} \, ,
$$ 
for some $u,v$ as above. Since we may suppose that $G\ge 2H,$ simple Brownian oscillation estimates, relying on the reflection principle, yield that this probability is at most 
$$
O \Big( n \exp\big\{-r^4   ( \log n )^{4/3}/r  (\log n )^{1/3}\big\} \Big) \le \exp\big\{- 2^{-1} r^3 \log n \big\} \, = \, n^{-2^{-1}r^3} \, ,
$$ 
for all large enough $r$.
Thus the proof of Theorem~\ref{t.weight}(2) is complete. \qed

\subsection{Polymer fluctuation tails, uniform in variation of endpoints and lifetime fraction}
In this section, we prove Theorem~\ref{t.deviation}. 

Let $a \in (0,1)$ and $r > 0$. Define the {\em lower zone} $Z_n^-(a,r)$ to be the product of $[-1,1] \cdot  r a^{2/3}\big(\log a^{-1} \big)^{1/3}$
and $n^{-1}\Z \cap [0,a/4]$;  and the {\em upper zone} $Z_n^+(a,r)$  to be the product of $[-1,1] \cdot  r a^{2/3}\big(\log a^{-1} \big)^{1/3}$
and $n^{-1}\Z \cap [1 - a/4,1]$.

Recall $\fluct_n \big[\cdot, \cdot \big]$ from \eqref{e.flucdef}.
\begin{lemma}\label{l.deviation} There exist positive constants $C_1,C_2,C_3$ and $C_4$ such that, 
for $r \geq C_1$, $a \in (0,1/16]$ and $n \in \N$ for which $\min \big\{ na,n(1-a)\big\} \geq C_2$,
  $$
  \PP \Big( \sup \fluct_n \big[ (x,h_1) \to (y,h_2) ; h \big] \geq  3r a^{2/3}\big(\log a^{-1} \big)^{1/3} \Big) \leq  C_4 a^{C_3 r^3 } \, ,
  $$
where the supremum is taken over  $(x,h_1)$ in the lower zone $Z^-_n(a,r)$;  $(y,h_2)$ in the upper zone $Z^+_n(a,r)$; 
  and $h \in n^{-1}\Z \cap \big( [a/2,4a] \cup [1-4a,1-a/2] \big)$.
 \end{lemma}
 {\bf Proof.} Let $z \in \R$.
Define the event $\mathsf{Narrow}_n(z,a,r)$ that
\begin{eqnarray}
 & & \rho_n \big[ (z,0)\to (z,1) \big] \cap  \bigg( \R \times \Big(   n^{-1}\Z \cap \big( \big[0,4a\big] \cup \big[1-4a,1\big] \big) \Big) \bigg) \label{e.strip} \\
 & \subseteq &  \big[z - \alpha(a,r),z+  \alpha(a,r) \big]  \times \big[0,1 \big] \, , \nonumber
\end{eqnarray}
where $\alpha(a,r) = 4^{-1}r a^{2/3}\big(\log a^{-1} \big)^{1/3}$.

To argue that this event is typical, we make two sets of applications of Theorem~\ref{t.toolfluc},
first with ${\bf h_1} = 0$ and then with ${\bf h_2} = 1$. 
In each, we apply Theorem~\ref{t.toolfluc}(1) with ${\bf k}$ ranging upwards from the maximal $i \in \N$ 
for which $2^{-i} \geq a/4$; and then use Theorem~\ref{t.toolfluc}(2) to treat the smallest scale.
We take   ${\bf r} = r $ in these applications.
 Since $a \geq \Theta(1) n^{-1}$, what we learn by doing so is that it is with probability at least $1-\Theta(1) a^{\Theta(1) r^3}$ that~(\ref{e.strip}) holds when $\alpha(a,r)$ is multiplied by a large constant.  
 By replacing $r$ by a small constant multiple of this quantity, and relabelling,
  we find that, for $z \in \R$,
 \begin{equation}\label{e.narrowprob}
\PP \Big( \neg \, \mathsf{Narrow}_n(z,a,r) \Big) \leq \Theta(1) a^{\Theta(1)r^3} \, .
\end{equation}
Set $z^- = - 5/4 \cdot r a^{2/3} \big( \log a^{-1} \big)^{1/3}$ and $z^+ = 5/4 \cdot r a^{2/3} \big( \log a^{-1} \big)^{1/3}$.
When $\mathsf{Narrow}_n\big(z^-,a,r\big)$ and $\mathsf{Narrow}_n\big(z^+,a,r\big)$ occur, every polymer of the form
$\rho_n\big[ (x,h_1)\to (y,h_2) \big]$, with $(x,h_1)$ in the lower zone $Z_n^-(a,r)$ and  $(y,h_2)$ in the upper zone $Z_n^+(a,r)$,
has the property that its endpoint locations $(x,h_1)$ and $(y,h_2)$ are bounded on the left by $\rho_n\big[ (z^-,0)\to (z^-,1) \big]$, and on the right by $\rho_n\big[ (z^+,0)\to (z^+,1) \big]$.
By polymer ordering  Lemma~\ref{l.sandwich}, every point in $\rho_n\big[ (x,h_1)\to (y,h_2) \big]$ is also thus bounded. However, $\rho_n\big[ (z^-,0)\to (z^-,1) \big]$ and $\rho_n\big[ (z^+,0)\to (z^+,1) \big]$
remain in vertical strips 
$$
\big[z - \alpha(a,r),z+  \alpha(a,r) \big]  \times \big[0,1 \big] \, , 
$$
with $z$ equal to $z^-$ or $z^+$, during  $\big[0,4a\big] \cup \big[1-4a,1\big]$ as indicated in~(\ref{e.strip}). We learn that $\rho_n\big[ (x,h_1)\to (y,h_2) \big]$ remains in the strip $\big[-\frac32  r a^{2/3} \big( \log a^{-1} \big)^{1/3},\frac32 r a^{2/3} \big( \log a^{-1} \big)^{1/3}\big]$
during  $\big[0,4a\big] \cup \big[1-4a,1\big]$. 
This implies that the quantity $\fluct_n \big[ (x,h_1) \to (y,h_2) ; h \big]$ appearing in Lemma~\ref{l.deviation} is for $h \in n^{-1}\Z \cap \big( [a/2,4a] \cup [1-4a,1-a/2] \big)$ at most  $3r a^{2/3} \big( \log a^{-1} \big)^{1/3}$, since $(x,h_1) \in Z_n^-(a,r)$ and  $(y,h_2) \in Z_n^+(a,r)$.

Thus the proof of Lemma~\ref{l.deviation} is complete, because,  the upper bound $\Theta(1) a^{\Theta(1) r^3}$ in this result is in light of~(\ref{e.narrowprob}) a bound on the probability that 
$\mathsf{Narrow}_n\big(z^-,a,r\big) \cup \mathsf{Narrow}_n\big(z^+,a,r\big)$ fails to occur. \qed

{\bf Proof of Theorem~\ref{t.deviation}.} By the scaling principle, it suffices to treat that case that $s_1 =0$ and $s_2 = 1$.
Specify the starting region $S$ equal to the product of $[-K,K]$ and $n^{-1}\Z \cap [0,1/3]$; and the ending region $E$ equal to the product of $[-K,K]$ and $n^{-1}\Z \cap [2/3,1]$.
We are concerned with journeys between $(x,h_1)$ and $(y,h_2)$, where~$(x,h_1)$ lies in the starting region~$S$ and $(y,h_2)$ in the ending region~$E$.
We plan to argue that typically such journeys simultaneously have the desired property of fluctuation by applying Lemma~\ref{l.deviation} via the scaling principle and using a union bound. We first construct a family of maps which will map any pair of points in $S\times E$  to a pair of points in  $Z_n^-(a,r)\times Z_n^+(a,r)$ as in Lemma~\ref{l.deviation}, which then will allow us to finish the proof by applying the lemma.

For $K \in \R$, recall from \eqref{e.shearmap} the shear map $\tau_K:\R^2 \to \R^2$, $\tau(x,y) = (x+Ky,y)$. 
Further, define a {\em KPZ dilation} to be a map of the form $\R^2 \to \R^2: (x,y) \to (\zeta^{2/3} x,\zeta y)$ which sends the line $\R \times \{ 1\}$ to a line of the form $\R \times \{ y\}$ for $y \in n^{-1} \N$ with $y > 0$; that is, we ask that $\zeta \in n^{-1}\N^+$. A {\em vertical shift} is a map of the form $\R^2 \to \R^2: (x,y) \to (x,y+h)$, where $h \in n^{-1}\Z$. 
A {\em horizontal shift}   is a map of the form $\R^2 \to \R^2: (x,y) \to (x+u,y)$, where $u \in \R$. 

Let $\Theta$ denote the class of maps from $\R^2$ to $\R^2$ that take the form $\phi_h \circ \phi_v \circ \phi_d \circ \phi_s$, where $\phi_s$ is a shear map $\tau_\kappa$ with $\vert \kappa \vert \leq n^{2/3}$; $\phi_d$ is a KPZ dilation; $\phi_v$ is a vertical shift; and $\phi_h$ is a horizontal shift. 

A basic covering pair for the product set $S \times E$ is a pair $(B^-,B^+)$, where there exists an element $\theta \in \Theta$ for which $B^-$ is the image under $\theta$
of the lower zone  $Z_n^-(a,r)$; $B^+$ is the image under $\theta$
of the upper zone  $Z_n^+(a,r)$ {where $a$ and $r$ are as in the statement of Theorem \ref{t.deviation},} with $B^- \cap S \not= \emptyset$; and $B^+ \cap E \not= \emptyset$.
The {\em covering number} is defined to be the minimum cardinality of a set of basic covering pairs for $S \times E$ such that $S \times E \subseteq \bigcup \big( B^- \times B^+ \big)$ where the union ranges over pairs $(B^-,B^+)$ in the set.

\begin{lemma}\label{l.coveringnumber}
Suppose that $K a^{1/3}$ is bounded above.
The covering number is at most a constant multiple of $a^{-10/3} K^2$.
\end{lemma}
{\bf Proof.}
Let $(x,h_1) \in S$ and $(y,h_2) \in E$. We want to locate a basic covering pair $(B^-,B^+)$ with  $(x,h_1) \in B^-$ and $(y,h_2) \in B^+$. Associated to $(B^-,B^+)$ is the composition of a shear, a KPZ dilation, a vertical shift and a horizontal shift. We attempt to recover these operations by undoing them. 
\begin{enumerate}
\item {\em The inverse horizontal shift.} Shift the two points horizontally by a common displacement that is a multiple of $3^{-3/2} a^{2/3}$ so that the resulting points  $(x',h_1)$ and $(y',h_2)$ satisfy the condition that
$\vert x' \vert \leq 3^{-3/2} a^{2/3}$.  
\item {\em 
Next is the inverse vertical shift.} Shift the two new points down by the maximum multiple of $n^{-1} \lfloor n a/12 \rfloor$ so that the resulting points $(x',h_1')$ and $(y',h_2')$ lie in the upper half-plane. Note that $0 \leq h_1' < a/12$ and $1/3 < h_2' \leq 1.$  
\item {\em 
Next the inverse KPZ dilation.} We map $(x',h_1')$ and $(y',h_2')$ to $(\hat{x},\hat{h}_1)$ and $(\hat{y},\hat{h}_2)$ via an inverse KPZ dilation $(x,y) \to (\zeta^{-2/3}x,\zeta^{-1}y)$, where $\zeta^{-1} \in [1,3]$ is chosen to ensure that $0 \leq \hat{h}_1 \leq a/4$  and $1 - a/4 \leq \hat{h}_2 \leq 1$. This choice of $\zeta$ may be made from a set of cardinality of order $a^{-1}$, where this entropy factor is adequate to ensure the desired bounds on $\hat{y}$. 
\item {\em Now, we undo the shear map.} Let $(\hat{x}_1,\hat{h}_1)$ and $(\hat{y}_1,\hat{h}_2)$
denote the image of $(\hat{x},\hat{h}_1)$ and $(\hat{y},\hat{h}_2)$ under an inverse shear map $\tau_{-\kappa}$, selected to ensure that $\vert \hat{y}_1 - \hat{x}_1 \vert \leq a^{2/3}$.
Since $0 \leq \hat{x} \leq a^{2/3} \leq 1 \leq K$ and  $\vert \hat{y} \vert \leq 3K$, the value of $\kappa$ may be chosen among a constant multiple of $K a^{-2/3}$ options to guarantee this outcome. 
\item {\em A final shift.} Having undone the several maps, we hoped to obtain $(\hat{x}_1,\hat{h}_1) \in Z_n^-(a,r)$ and $(\hat{y}_1,\hat{h}_2) \in Z_n^+(a,r)$. The vertical coordinates satisfy the desired conditions $\hat{h}_1 \in [0,a/4]$ and $\hat{h}_2 \in [1 - a/4,1]$; and the horizontal displacement $\vert \hat{y}_1 - \hat{x}_1 \vert$, being at most $a^{2/3}$, is consistent with our aim. But a final horizontal shift is needed, to ensure that the horizontal coordinates are both at most $r a^{2/3}\big(\log a^{-1} \big)^{1/3}$. Since $\vert \hat{x}_1 \vert$ is readily seen to have order at most $Ka$, and $r \geq 1$ as well as $a \leq e^{-1}$, we see that this final shift may be chosen from among an order of $K a^{1/3}$ choices. However, since we hypothesise that $Ka^{1/3}$ is most a large constant, the entropic term associated to this final step is bounded.
\end{enumerate}
The product of upper bounds neglecting bounded factors on the number of choices for the maps employed in the respective steps is equal to
$$
 Ka^{-2/3} \times a^{-1} \times a^{-1} \times K a^{-2/3} \times 1  = K^2 a^{-10/3} \, .
$$
If we apply the inverses of our inverse maps---in reverse order!---to the lower and upper zones   $Z_n^-(a,r)$ and  $Z_n^+(a,r)$, we will obtain $B^-$ and $B^+$, elements of a basic covering pair that respectively contain the given points  $(x,h_1) \in S$ and $(y,h_2) \in E$. Since the constructed map is one among a collection whose cardinality is at most the displayed quantity up to a bounded factor, we have completed the proof of Lemma~\ref{l.coveringnumber}. \qed

In seeking to obtain Theorem~\ref{t.deviation} in the case that $s_1 =0$ and $s_2 =1$, it is enough, in view of $h_{1,2} \leq 1$, to bound above the probability of the event---that we will denote by $A$---that there exist  $(x,h_1) \in S$ and $(y,h_2) \in E$ for which
\begin{eqnarray*}
 & & \textrm{there exists  a moment~$h$ at which a fraction lying in $[a,2a] \cup [1-2a,1-a]$} \\
 & & \textrm{of the lifetime $[h_1,h_2]$ has elapsed such that} \\
  & & h_{1,2}^{-2/3} \big\vert \rho_n \big[ (x,h_1) \to (y,h_2) \big](h) - \ell \big[ (x,h_1) \to (y,h_2) \big](h) \big\vert \, \, \textrm{is at least $R a^{2/3} \big(\log a^{-1} \big)^{1/3}$} \, ,
 \end{eqnarray*}
 when $R=5r$, where the positive parameter $R$ is displayed to permit convenient later reference to the display, and where the choice of a multiple of five---or, indeed, of any given multiple, but five will work for our purpose---is admissible via the absorptive proclivity of $\Theta(1)$ terms in the theorem. 
  We cover $E \times S$ by a union indexed by $i \in \mc{C}$ of basic covering boxes $(B^-_i,B^+_i)$ which by Lemma~\ref{l.coveringnumber} we know may be chosen so that $\vert \mc{C} \vert$ is at most a constant multiple of $K^2 a^{-10/3}$. 
  
  Fix a given index $i \in \mc{C}$ and consider the event  
  that there exist $(x,h_1) \in B^-_i$ and $(y,h_2) \in B^+_i$ such that the displayed circumstance takes place with $R=r$.  
We have constructed a composite function that maps $B^-_i$ into  $Z_n^-(a,r)$
and   $B^+_i$ into  $Z_n^+(a,r)$. The composite is the outcome of a five-step composition. After each step, the given points $(x,h_1) \in B^-_i$ and   $(y,h_1) \in B^+_i$ have been mapped to certain locations; we may consider the probability that there exist  $(x,h_1) \in B^-_i$ and $(y,h_2) \in B^+_i$ such that the last displayed event occurs when $(x,h_1)$ and $(y,h_2)$ are replaced by these locations.  For the step indexed by $j \in \llbracket 0, 5 \rrbracket$, we denote this probability by $p_j(R,i)$. 

Note that the probability $\PP(A)$ that we seek to bound in order to obtain 
 Theorem~\ref{t.deviation} when $s_1 =0$ and $s_2 =1$ is bounded above by a constant multiple of $K^2 a^{-10/3} \sup_{i \in \mc{C}} p_0(5r,i)$.
 Thus, if we can show that $p_0(5r,i)$ is at most $a^{dr^3}$ for each $i \in \mc{C}$, we will have proved this theorem.
 
 On the other hand,
$p_5(3^{4/3}r,i)$ for any given $i \in \mc{C}$ is at most the probability that the last displayed event occurs with $R = 3^{4/3}r$ and when $(x,h_1)$ and $(y,h_2)$ are respectively replaced by certain given elements of  $Z_n^-(a,r)$ and $Z_n^+(a,r)$. 
As such, $p_5(3^{4/3}r,i)$ is bounded above by Lemma~\ref{l.deviation}: indeed, this result implies that $p_5\big(3^{4/3}r,i \big) \leq 
\Theta(1)
a^{\Theta(1)
r^3}$ because $h_{1,2} \geq 1/3$ (since $h_1 \leq 1/3$ and $h_2 \geq 2/3$) and because, if $h \in [0,1]$ satisfies $\tfrac{h - h_1}{h_{1,2}} \in [a,2a] \cup [1-2a,1-a]$ for $h_1 \in [0,a/4]$ and $h_2 \in [1- a/4,1]$, then $h \in [a/2,4a] \cup [1-4a,1-a/2]$. 

To close out the proof of Theorem~\ref{t.deviation} , it suffices to show that $p_0(5r,i) \leq p_5\big(3^{4/3}r,i \big)$. Indeed, this bound proves this result when the left-hand instance of $r$ in~(\ref{e.deviationvariant}) is replaced by $5r$; but, as we have noted, the stated form may then be obtained since $\Theta(1)$ notation is used.
We seek then to obtain the just stated bound. Of the five maps involved in the composition, four are shifts or a KPZ dilation.  For each index advance $j \to j+1$ in which one of these maps is involved, the scaling principle shows that $p_j(s,i)$ equals $p_{j+1}(s,i)$. The remaining map is the fourth---the inverse shear map $\tau_{-\kappa}$---involved in the index change $3 \to 4$. Since the value of the parameter $\kappa$ specifying this map is at most a constant multiple of $K$, and $\vert K \vert$ is at most a small constant multiple of $n^{2/3}$, 
$\kappa$  itself is at most a small multiple of $n^{2/3}$; thus,
Lemma~\ref{l.rhoxa}(2)
and $5 > 3^{4/3}$ imply that  $p_3\big(5r,i \big)$ is at most $p_4\big(3^{4/3}r,i\big)$. This confirms 
that $p_0(5r,i) \leq p_5\big(3^{4/3}r,i \big)$ and completes the proof of  Theorem~\ref{t.deviation}. \qed 

\subsection{Polymer weight tails, uniform in variation of endpoints}
In this section, we prove Proposition~\ref{p.onepoint}.

{\bf Proof of Proposition~\ref{p.onepoint}.} Proposition~\ref{p.tailweight} is concerned with the tail of parabolic weight for polymers that begin in the lower third and end in the upper third of the rectangle $[-K,K] \, \times \, n^{-1}\Z \cap [-3,3]$. With $\nu = 6^{-1} \cdot 5/4 \cdot 2^{-\ell}$, this rectangle is mapped under 
the transformation $(x,y) \to (\nu^{2/3}x,\nu y)$ and a vertical shift  to the rectangle
$$
[-K,K]\cdot  6^{-2/3} (5/4)^{2/3} 2^{-2\ell/3} \,  \times \, m^{-1}\Z  \cap  \big[ 0,5/4 \cdot 2^{-\ell} \big] \, ,
$$
where $m = n \cdot 4/5 \cdot 2^\ell$. Consider the collection of translates of the displayed rectangle by vectors of the form $\big( K \cdot  6^{-2/3} (5/4)^{2/3} 2^{-2\ell/3}  \cdot j , 2^{-\ell} \cdot 1/12 \cdot k \big)$, with $(j,k) \in \Z^2$. Let $\Psi$ denote the subcollection indexed by those $(j,k)$ for which the translation vector has horizontal component lying in $[-M,M]$ and vertical component lying in $[0,1]$. Set $K > 0$ so that $L = K \cdot 6^{-2/3} (5/4)^{2/3}$.
Then $\Psi$ is a set of cardinality $\Theta(1) M K^{-1} 2^{5\ell/3}$ such that any pair of elements $(x,s_1)$ and $(y,s_2)$ implicated in the definition in the events 
$\low\big(\zeta,\ell,L,M\big)$ and $\high\big(\zeta,\ell,L,M\big)$ belongs to some element of $\Psi$. 
What we learn from this is that we may apply Proposition~\ref{p.tailweight} with ${\bf n} = n \cdot 5/4 \cdot 2^{-\ell}$, ${\bf K} = K$ and ${\bf R} = (4/5)^{1/3} 2^{-1/3} \zeta$ via translation invariance and the scaling principle, and use a union bound, to conclude that the probability in Proposition~\ref{p.onepoint} is at most $\Theta(1) M K^{-1} 2^{5\ell/3} \Theta(1) K^2 \exp \big\{ - \Theta(1) \zeta^{3/2} \big\}$.  
(That the selection of {\bf R} is satisfactory depends on $\tot \geq 2^{-\ell-1}$.)
Since the obtained bound takes the desired form, and the hypotheses on $n$, $L$ and $\zeta$ in Proposition~\ref{p.onepoint}
enable this application of Proposition~\ref{p.tailweight}, the proof of Proposition~\ref{p.onepoint} is complete. \qed

\section{Slim pickings for slender excursions}\label{s.slenderresults}

Here we prove our result Theorem~\ref{t.nocloseness} asserting the significant shortfall in weight accrued by zigzags that are constrained to follow excursions relative to a given zigzag that are narrower than the width dictated by the KPZ scaling exponent of two-thirds.

 More precisely, but still in summary, this theorem
 concerns the maximum weight that may be accrued by an $n$-zigzag~$\psi$ that is constrained to pursue a slender excursion relative to a given zigzag~$\phi$
 of duration of order $2^{-\ell}$. By slender, we mean that, at a high but fixed proportion $1 - \chi$ of heights along the excursion, the width between $\psi$ and $\phi$ is at most a small multiple 
 $\tza$
 of the characteristic separation~$2^{-2\ell/3}$. There is a degree of choice in the levels, of proportion $\chi$, at which~$\psi$ is not bound by this slenderness constraint.
Proposition~\ref{p.closelow} is a result {\em en route} to Theorem~\ref{t.nocloseness} in which a counterpart conclusion is reached when this set of levels is instead fixed. 
 
 There are four subsections. In the first, we give a brief LPP-based proof of a result that we will need: the mean of the GUE Tracy-Widom  distribution is negative. 
 (See~\cite[Lemma~$A.4$]{Watermelon} for another proof pointed out by Ivan Corwin.)  In the second section, we record two needed results, including a form of~\cite[Theorem~$1.1$]{ModCon} that
concerns variation of polymer weight under endpoint perturbation.  In the third section, we prove Proposition~\ref{p.closelow}. In the fourth, we sum out over the levels fixed in this proposition in order to obtain Theorem~\ref{t.nocloseness}.

 \subsection{The negative mean of the GUE Tracy-Widom distribution}
 
\begin{proposition}\label{p.negativemean}
There exist $d > 0$ and $n_0 \in \N$ such that, for $n \geq n_0$ and $\vert x \vert \leq 2^{-1} c n^{1/19}$, 
$$
\E \, \weight_n \big[ (0,0) \to (x,1) \big] \leq -d \, .
$$
\end{proposition}
It is Proposition~\ref{p.negativemean} that we will later employ, but this result has the following interesting consequence.
\begin{corollary}\label{c.negativemean}
The mean of the GUE Tracy-Widom distribution is negative.
\end{corollary}
{\bf Proof.} The limit in law as $n \to \infty$ of  $\weight_n \big[ (0,0) \to (0,1) \big]$ has the distribution of $2^{1/3}X$, where, by 
 \cite{TracyWidom,Baryshnikov}, $X$ has the GUE Tracy-Widom law.
 Thus the result follows from Proposition~\ref{p.negativemean}. \qed

The next result is the principal component of Proposition~\ref{p.negativemean}.

\begin{lemma}\label{l.twoways}
There exist $d_1 > 0$ and $n_0 \in \N$ such that, for $n \geq n_0$ and $\vert x \vert \leq 2^{-1} c n^{1/19}$, 
\begin{equation}\label{e.twoways}
\E \, \Big[  \weight_n \big[ (0,0) \to (x,1) \big] \vee  \weight_n \big[ (0,0) \to (x+1,1) \big]  \Big] \geq 
\E \, \weight_n \big[ (0,0) \to (x,1) \big] + d_1 
\end{equation}
provided that $x \leq 0$ while, if $x > 0$, the same inference holds when $x+1$ is replaced by $x-1$.
\end{lemma}
{\bf Proof.} We will give a Brownian Gibbs argument. The random profile $x \to  \weight_n \big[ (0,0) \to (x,1) \big]$
is $(c,C,n+1)$-regular in the sense of Subsection~\ref{s.normalized} by Proposition~\ref{p.shiftbrownian}(1) (with ${\bf a} = 1$) therein.
We will write $\mc{L}$ for the $(n+1)$-curve regular ensemble  whose uppermost curve $\mc{L}(1,x)$ equals 
$\weight_n \big[ (0,0) \to (x,1) \big]$.

Suppose first that $x  \in [-1/2,1/2]$.
For $D > 0$, let $N = N(D)$ denote the event that the weight $\weight_n \big[ (0,0) \to (x,1) \big]$ is at most $D$, and that $\weight_n \big[ (0,0) \to (x+2,1) \big]$ is at least $-D$.  
 We make two claims.

{\em Claim~$1$.} There exist $D_1 > 0$, $d_2 \in (0,1)$ and $n_1 \in \N$ such that, when $n \geq n_1$,
it is with probability at least $d_2$ that $N(D_1)$ occurs.  

{\em Claim~$2$.} For $n \in \N$ and $D > 0$, the conditional probability of 
$$
\weight_n \big[ (0,0) \to (x+1,1) \big] \geq \weight_n \big[ (0,0) \to (x,1) \big]  + 1
$$  
given $N(D)$ is at least $\nu_{0,1/2} \big( D+1,\infty \big)$, where $\nu_{0,1/2}$ denotes the Gaussian law of mean zero and variance one-half. The same statement holds when the displayed left-hand side is replaced by  $\weight_n \big[ (0,0) \to (x-1,1) \big]$. 

{\em Proof of Claim~$1$.} 
The ensemble $\mc{L}$ satisfies ${\rm Reg}(2)$ and ${\rm Reg}(3)$ in Definition~\ref{d.regularsequence}, and from this, the result follows.

{\em Proof of Claim~$2$.} The event on which we condition in this claim is that  $\mc{L}(1,x) \leq D$ and $\mc{L}(1,x+2) \geq -D$. If we further condition on the value $(u,v)$  of $\big( \mc{L}(1,x) , \mc{L}(1,x+2) \big)$, and on the form $f$ of $\mc{L}(2,\cdot): [x,x+2] \to \R$, then the conditional distribution of $\mc{L}(1,\cdot): [x,x+2] \to \R$ is given by Brownian bridge $B:[x,x+2] \to \R$, with $B(x) = u$ and $B(x+2) = v$---whose law we label $\mc{B}_{u,v}^{[x,x+2]}$---conditioned on $B(z) \geq f(z)$ for $z \in [x,x+2]$.
Note that
$$
\mc{B}_{u,v}^{[0,2]} \big( B(1) \geq u + 1 \big\vert B > f \big) \geq 
\mc{B}_{u,v}^{[0,2]} \big( B(1) \geq u + 1  \big) \geq \mc{B}_{0,0}^{[0,2]} \big( B(1) \geq D + 1  \big) = \nu_{0,1/2} \big( D+1, \infty\big) \, , 
$$
the first inequality by the monotonicity offered in   \cite[Lemma 2.18]{BrownianReg} (a result originally proved in \cite{AiryLE}); the second by the affine scaling property of Brownian bridge; and the third by the law of the midpoint value of standard Brownian bridge. This completes the proof of the first assertion of Claim~$2$. The second assertion has an almost identical proof.

The two claims show that, for $x \in [-1/2,1/2]$, 
\begin{equation}\label{e.xplusone}
\E \, \Big[  \weight_n \big[ (0,0) \to (x,1) \big] \vee  \weight_n \big[ (0,0) \to (x+1,1) \big]  \Big] \geq 
\E \, \weight_n \big[ (0,0) \to (x,1) \big] + d_2 \, , 
\end{equation}
as well as the bound after we replace $x+1$ by $x-1$. 
This proves Lemma~\ref{l.twoways} in the case that $x \in [-1/2,1/2]$. By \cite[Lemma 2.26]{BrownianReg}---a tool of near parabolic invariance that propagates spatial information from unit-order to much broader intervals---we learn from~(\ref{e.xplusone}) that, for  $\vert x \vert \leq 2^{-1} c n^{1/19}$,
$$
\E \, \Big[  \weight^\cup_n \big[ (0,0) \to (x,1) \big] \vee  \weight^\cup_n \big[ (0,0) \to (x+1,1) \big]  \Big] \geq 
\E \, \weight^\cup_n \big[ (0,0) \to (x,1) \big] + d_2 \, , 
$$
where recall that $\weight^\cup_n \big[ (0,0) \to (x,1) \big]$ denotes the parabolically adjusted weight  $\weight_n \big[ (0,0) \to (x,1) \big] + 2^{-1/2}x^2$.
Subtracting $2^{-1/2}x^2$ yields~(\ref{e.twoways}) for $x \leq -1/2$. We obtain~(\ref{e.twoways}) for $x \geq 1/2$ by the same argument, with the role of~(\ref{e.xplusone})
 played by its counterpart where $x-1$ replaces $x+1$. This completes the proof of Lemma~\ref{l.twoways}. \qed

{\bf Proof of Proposition~\ref{p.negativemean}.} We will prove this result by showing that 
$\E \, \weight_n \big[ (0,0) \to (x,1) \big]$ is at most~$-d_1$, where Lemma~\ref{l.twoways} furnishes $d_1 > 0$. To verify this, suppose first that $x \leq 0$. We claim that
\begin{equation}\label{e.limitsuplhs}
 \liminf \, \,  \sup_{z \in \R} K^{-1} \weight_n \big[ (0,0) \to (z,K) \big] \, \geq \, \frac{n}{n+1} \Big( \E \, \weight_n \big[ (0,0) \to (z,1) \big] + d_1 \Big) \, ,
\end{equation}
where the limit infimum is taken as $K \to \infty$ through $K \in \N$. This claim is substantiated by constructing an $n$-zigzag that begins at $(0,0)$. It travels first either to $(x,1)$ or to $(x+1,1)$, the choice being made so that more weight is captured along the way. After arrival, the zigzag makes an immediate microscopic jump, moving by $\big(-2 n^{-2/3},n^{-1} \big)$.
Zigzag formation continues as if the point of arrival plays the role that $(0,0)$ did at the outset. That is, the zigzag continues by travelling to one of the points whose displacement from its present location is $(x,1)$ or $(x+1,1)$, the selection made to maximize weight; then a further microscopic jump is made; and the process iterates indefinitely. 
If an arbitrarily small constant is subtracted from the right-hand side of~(\ref{e.limitsuplhs}), the left-hand supremum is seen to exceed the right-hand side for all sufficiently high $K$. 
Thus, we obtain~(\ref{e.limitsuplhs}).

However, the left-hand side of~(\ref{e.limitsuplhs}) is at most zero almost surely.
Indeed, by the scaling principle, $\sup_{z \in \R} K^{-1} \weight_n \big[ (0,0) \to (z,K) \big]$ equals $K^{-2/3} \sup_{z \in \R}  \weight_{nK} \big[ (0,0) \to (z,1) \big]$ in law; and the latter supremum converges as $K \to \infty$ in law to the Tracy-Widom $GOE$ distribution~$\nu$, because the process $z \to  \weight_m \big[ (0,0) \to (z,1) \big]$ converges in law as $m \to \infty$ in a compact uniform topology to the parabolic Airy process, whose maximum has the law $\nu$ (see e.g. \cite{BFPS07}). 

Thus,  the mean of $\weight_n \big[ (0,0) \to (z,1) \big]$ is at most $- d_1$. This completes the proof of Proposition~\ref{p.negativemean}. \qed 

\subsection{Two tools}

\subsubsection{Gaussian increments for weight profiles}

For a shortly upcoming use, we record a result bounding the tail of increments for the weight of polymers subject to horizontal endpoint perturbation.
The result is better expressed using parabolically adjusted weight, so that a slope arising from a difference of parabolas is eliminated and much higher choices of horizontal endpoint discrepancy may be treated. 
The parabolic weight notation $\weight^\cup_n$ was specified in Subsection~\ref{s.parabolicweight}. We now specify a variant notation in order to describe {\em differences}
in parabolic weight.  
Let $(x_1,x_2)$ and $(y_1,y_2)$ belong to $\R^2_\leq$.
The {\em parabolically adjusted} weight difference
$$
 \Delta^{\cup} \, \weight_n \big[ ( \{ x_1,x_2\} ,s_1) \to ( \{y_1,y_2\},s_2) \big] 
$$
denotes 
$$
\bigg( \weight_n \big[ (x_2,s_1) \to (y_2,s_2) \big] + 2^{-1/2} \frac{(y_2 - x_2)^2}{s_2-s_1} \bigg) \, - \, \bigg( \weight_n \big[ (x_1,s_1) \to (y_1,s_2)  \big] +  2^{-1/2} \frac{(y_1 - x_1)^2}{s_2-s_1} \bigg) \, ,
 $$
\begin{proposition}\label{p.fluc}
Positive constants $C$ and $c$ exist for which the following holds. Let $a \in (0,2^{-4}]$. 
Let $(n,s_1,s_2) \in \N \times \R_\leq^2$ be a compatible triple for which 
$n \tot \geq 10^{32} c^{-18}$ and let $x,y \in \R$ satisfy  $\big\vert x - y  \big\vert \tot^{-2/3} \leq 2^{-2} 3^{-1} \rsc  (n \tot)^{1/18}$.
Let 
 $K \in \big[10^4 \, , \,   10^3 (n \tot)^{1/18} \big]$.
Then
$$
\PP \, \left( \, \sup_{\begin{subarray}{c} x_1,x_2 \in [x,x+a\tot^{2/3}] \, , \, x_1 < x_2 \\
    y_1,y_2 \in [y,y+a\tot^{2/3}] \, , \, y_1 < y_2 \end{subarray}}  \Big\vert  \Delta^{\cup} \, \weight_n \big[ ( \{ x_1,x_2\} ,s_1) \to ( \{y_1,y_2\},s_2) \big]  \Big\vert \, \geq \, K a^{1/2} \tot^{1/3} \, \right) 
$$
is at most 
$10032 \, C \exp \big\{ - c 2^{-24} K^2 \big\}$.
\end{proposition}
{\bf Proof.}
The special case that $s_1 = 0$ and $s_2=1$ is implied by \cite[Theorem $1.1$]{ModCon}.
(The upper bound in the latter result is $10032 \, C  \exp \big\{ - c_1 2^{-21}   R^{3/2}   \big\}$. But $c_1 = 2^{-5/2} c \wedge 1/8$, where $c > 0$ is a constant that is at most one, so that 
 we obtain the upper bound in Proposition~\ref{p.fluc}.)  
 The scaling principle from Section~\ref{s.scalingprinciple}
then yields the proposition from this special case. \qed

\subsubsection{{A control on weight that is uniform as endpoints vary.}}\label{s.nosubpath}
Here we record a consequence of Proposition~\ref{p.onepoint}. Recall that we used $\low(\zeta,\ell,L,M)$ to denote the event that 
$$
 \tot^{-1/3} \weight^\cup_n \big[  (x,s_1) \to  (y, s_2) \big]  
$$ 
is less than $-\zeta$ for some pair $(x,s_1), (y,s_2) \in \R  \times  n^{-1}\Z \cap [0,1]$ with $|x| \vee |y| \le M$, $\vert x - y \vert \leq 2^{-2\ell/3} L$ and $s_{1,2} \in (2^{-\ell-1},2^{-\ell}]$. Similarly, $\high(\zeta,\ell,L,M)$ denotes the event that the displayed quantity exceeds~$\zeta$  for some such pair. 
For $M=n^{1/20}$ and $L = (n 2^{-\ell})^{1/47}$, specify the {\em uniform boundedness} event   
 $$
 {\uni}_n(\zeta)
 \, = \, \bigcap_{\ell}\Big( \neg\low\big(\zeta,\ell,L,M\big) \cap  \neg \high\big(\zeta,\ell,L,M\big) \Big) \, ,
 $$ where the intersection ranges over $\ell$ such that $\zeta^{20} n^{-1} \le 2^{-\ell}\le 1$.

From Proposition~\ref{p.onepoint} and a union bound, it follows that, for  $\zeta\geq \Theta\big( (\log n)^{2/3} \big)$,
\begin{equation}\label{uniformbound}\P\big({\uni}_n(\zeta)\big) \geq 1-e^{-\Theta(1)\zeta^{3/2}}.
\end{equation}

\subsection{Excursions constrained at given heights are uncompetitive}\label{s.excursions-uncomp}

 Let $r \in \N$, and let $\kappa > 0$
satisfy 
\begin{equation}\label{e.rkappa}
\kappa^{3/2} \in r^{-1} \N \, \,  \textrm{and} \, \, \kappa^{-3/2} \in \N \, .
\end{equation}
The parameter $\kappa$ will be positive but small, and these conditions are then ensured if need be by slight adjustment to its value. They ensure that an $r$-zigzag~$\phi$ of lifetime $[0,1]$ begins and ends at moments that are multiples of~$\kappa^{3/2}$, and that every intervening such multiple, being an element of $r^{-1}\N$, is the vertical coordinate of a horizontal interval in $\phi$. 
Further  let $b \in (0,1)$.  
A {\em segment} is a horizontal planar line segment of length $b\kappa$ whose height is an integer multiple of $\kappa^{3/2}$. The role of $b$, which will be taken to be a small enough absolute constant, is elucidated in the discussion following the statement of Proposition \ref{p.closelow}.

Let $\chi \in (0,1)$. A {\em plentiful segment collection} is a set of segments that numbers at least $(1-\chi) \kappa^{-3/2}$ whose elements have distinct heights that include $0$ and $1$ and that belong to $[-r^{1/20},r^{1/20}] \times [0,1]$.
Let $\mc{C}$ denote the set of plentiful segment collections.

Let $\mathsf{c} \in \mc{C}$. A $\mathsf{c}$-path is an $r$-zigzag 
from an element $(z_1,0)$ in the lowest of $\mathsf{c}$'s segments to an element $(z_2,1)$ in its highest 
that intersects every segment in $\mathsf{c}$. Let 
$\weight_r \big[ \mathsf{c}\textrm{-path} \big]$
denote the supremum of the weights of $\mathsf{c}$-paths. 
\begin{proposition}\label{p.closelow}
There exist positive parameters $\kappa_0$, $b$, $d_1$, $d_2$,  and $\chi_0 \in (0,1/2]$ such that,  if $r \in \N$ and $\kappa \in (0,\kappa_0)$
satisfy (\ref{e.rkappa}); if $\chi \in (0,\chi_0)$ satisfies $\chi \geq 2\kappa^{3/2}$ and $2\chi \kappa^{-3/2} \in \N$; 
 and if $\mathsf{c} \in \mc{C}$; then
$$
\PP \Big( 
 \weight_r \big[ \mathsf{c}\textrm{-path} \big] \geq - d_1 \kappa^{-1}  \Big) \leq \exp \big\{ - d_2 \kappa^{-3/2} \big\} \, .
$$
\end{proposition}
To derive this result, our task is to show that  typically  
 $\weight_r \big[ \mathsf{c}\textrm{-path} \big]$ is a large negative number. To argue this, 
let $\psi$ be a $\mathsf{c}$-path. 
A value of the form $j \kappa^{3/2}$ for $j \in \llbracket 0, \kappa^{-3/2} - 1 \rrbracket$ is said to be $\psi$-{\em useful}---but we will simply say `useful'---if  
$\psi$  intersects a segment in $\mathsf{c}$ of height $j \kappa^{3/2}$ and another of height $(j+1) \kappa^{3/2}$.
We will write  $\weight_r(\psi)$ in the form $W_u(\psi) + W_o(\psi)$ in a sense that we now explain. We divide~$\psi$ into sub-zigzags by splitting at points of departure of~$\psi$ from levels  that are integer multiples of~$\kappa^{3/2}$ lying in $(0,1)$. Each sub-zigzag is called {\em useful}, or {\em otherwise}, according to whether the height of its starting point is useful or not. Then $W_u(\psi)$ and $W_o(\psi)$ are the respective sums of the weights of the useful, or otherwise, sub-zigzags. 
We aim to carry out the needed task by finding bounds on the upper tail of the two right-hand terms in the inequality
\begin{equation}\label{e.supweightuo}
  \weight_r \big[ \mathsf{c}\textrm{-path} \big] \leq \sup \big\{ W_u(\psi) : \psi \, \, \textrm{a $\mathsf{c}$-path} \big\} \, + \,
  \sup \big\{ W_o(\psi) : \psi \, \, \textrm{a $\mathsf{c}$-path} \big\} \, .
\end{equation}
We will first analyse the useful weight sum supremum $\sup \big\{ W_u(\psi) : \psi \, \, \textrm{a $\mathsf{c}$-path} \big\}$; and  then do likewise for the otherwise counterpart.
The resulting bounds will then permit a quick proof of  Proposition~\ref{p.closelow}. 

Analysing the useful sum is the principal component in the proof of this proposition, and a few words in summary of this analysis will, we hope, be helpful. There are two elements: we will show that weights of the sub-zigzags that contribute to $W_u(\psi)$ have negative mean; and we will then appeal to concentration inequalities for sums of independent random variables.

Regarding the first element, it follows from Proposition \ref{p.negativemean} and the scaling principle that the weight of a useful sub-zigzag with {\em fixed} endpoints has mean at most $-d_1\kappa^{1/2}.$ However, the endpoints of useful sub-zigzags are not fixed, but in fact vary over horizontal segments of length $b\kappa$. 
An effect of Brownian oscillation for polymer weight that will be controlled in Lemma~\ref{errorcontr2} causes our upper bound on the mean to rise by an order of $(b\kappa)^{1/2}$.
 It is at this moment that we will select the value of $b > 0$. By choosing this parameter to be small enough,  the mean supremum weight of useful sub-zigzags 
 traversing between vertically consecutive elements of $\mathsf{c}$ will be shown to be at most $-\frac{d_1}{2}\kappa^{1/2}.$  

In consecutive subsections, we analyse the useful sum; and the otherwise sum; and give the proof of  Proposition~\ref{p.closelow}. 

\subsubsection{The useful sum}

\begin{definition}\label{d.yij}
Let $I$ and $J$ be compact intervals in $[- r^{1/20},r^{1/20}]$ of length $b \kappa$. Set
$$
Y_{I,J} =  \kappa^{-1/2} \sup_{\begin{subarray}{c} u \in I,  v \in J  \end{subarray}}   \weight_r \big[ (u,0) \to (v,\kappa^{3/2}) \big]  \, . 
$$
\end{definition}
Let $\mc{U}$ denote the set of useful values. 
A useful sub-zigzag starts in  a segment belonging to $\mathsf{c}$ and intersects that segment only at this starting point. Its ending point is $\kappa^{3/2}$ higher than its starting point. The law of the weight of a useful sub-zigzag with given starting height is at most the supremum of weights of $r$-zigzags
that begin and end in two given segments and whose lifetime has duration~$\kappa^{3/2}$. Since the sub-zigzag immediately departs from its starting height, the weights of distinct sub-zigzags in our partition are independent. Thus, we find that the useful weight sum supremum
  \begin{equation}\label{e.usefulweightsum}
 \sup \big\{ W_u(\psi) : \psi \, \, \textrm{a $\mathsf{c}$-path} \big\} \, \,  \textrm{is stochastically dominated by} \, \,  \sum_{i=1}^{\vert \mc{U} \vert} U_i \, ,
\end{equation}  
   where the  latter quantity is a  sum of independent random variables, $U_i$ having the law of $\kappa^{1/2}Y_{I,J}$, where $Y_{I,J}$ has been specified in Definition~\ref{d.yij},  and where the pair $(I,J)$ satisfies  the hypothesis in that definition. 

The next result is our conclusion regarding the useful sum; indeed, with $U_i = \kappa^{1/2} X_i$, it will permit analysis of the right-hand sum in~(\ref{e.usefulweightsum}).
\begin{proposition}\label{p.weightuseful}
 Let $j \in \N$ satisfy $\kappa^{-3/2}/2 \leq j \leq \kappa^{-3/2}$. Let $\big\{ X_i: i \in \intint{j} \big\}$ be an independent sequence of random variables, where  $X_i$ has the law of $ Y_{I,J}$ for a possibly $i$-dependent pair $(I,J)$ satisfying the hypothesis in Definition~\ref{d.yij}.
Then
$$
 \PP \bigg( \sum_{i=1}^j X_i \geq -
 \Theta(1)
 \kappa^{-3/2} \bigg) \leq \exp \big\{ - 
 \Theta(1)
  \kappa^{-3/2} \big\} \, .
$$ 
\end{proposition}

To derive this result, we will need the second element to which we alluded in summary of the useful sum analysis: a concentration result for independent random variables.
\begin{proposition}\cite[Theorem 2.8.1]{Vershynin}\label{p.Vershynin}
Let $\big\{ X_i: i \in \N \big\}$ be a sequence of independent real-valued random variables of zero mean.
For any $C > 0$, there exist positive $c_1$ and $c_2$ such that, for $k \in \N$ and $t \ge 0$,
\begin{enumerate}
\item if
 $\E \exp \big\{ |X_i|/C \big\} \le 2$ then
$$
\P\bigg( \Big\vert \sum_{i=1}^kX_i \Big\vert \ge t \bigg) \, \le \, \exp \Big\{ - \min \big\{ c_1 t^2 k^{-1}, c_2 t \big\} \Big\} \, ;
$$
\item and if  $\E \exp \big\{  X^2_i /C \big\} \le 2$, then
$$
\P \bigg( \Big\vert \sum_{i=1}^k X_i \Big\vert \geq t \bigg) 
\, \leq \, \exp \big\{- c_1 t^2 k^{-1} \big\} \, .
$$
\end{enumerate}
\end{proposition}
The variables satisfying the hypotheses in (1) and (2) are known to be sub-exponential and sub-Gaussian variables respectively.
Now let $(u',v') \in I \times J$ satisfy 
$\vert v' - u' \vert = \inf \big\{ \vert v - u \vert: u \in I, v \in J\big\}$.  Write $Y_{I,J} = Z_{I,J} + E_{I,J}$, where $Z_{I,J} =  \kappa^{-1/2}  \weight_r \big[ (u',0) \to (v',\kappa^{3/2}) \big]$
 and 
 $$
  E_{I,J} = \kappa^{-1/2} \sup_{\begin{subarray}{c} u \in I,  v \in J  \end{subarray}}  \Big( \weight_r \big[ (u,0) \to (v,\kappa^{3/2}) \big] - \weight_r \big[ (u',0) \to (v',\kappa^{3/2}) \big] \Big) \, .
  $$
 This decomposition reflects the argument promised in the first element in our summary: $Z_{I,J}$ is a point-to-point weight (normalized to be of unit-order by the factor $\kappa^{-1/2}$), and  the {\em error} term $E_{I,J}$ is a weight difference due to horizontal endpoint perturbation. We offer a Gaussian form of control on the latter next. In a usage also found later in this section, $\Cmac$ and $\cmac$ denote positive constants whose value may change from line to line.

\begin{lemma}  There exist positive  constants $\Cmac$ and $\cmac$ such that, for $r \geq \Cmac  \kappa^{-195}$ and $0<b \le 2^{-4}$, the following holds.
For all $I$ and $J$ as above, there exists an event $\mathsf{E}_{I,J}$ such that 
\begin{equation}\label{e.errorprob}
\PP \big( \mathsf{E}_{I,J} \big) \leq \Cmac \exp \big\{ - \cmac r^{1/9} \kappa^{1/6}  \big\}
\end{equation}
 and that
\begin{equation}\label{e.ijh}
 \PP \Big( E_{I,J} {\bf 1}_{\mathsf{E}_{I,J}^c} \geq b^{1/2}h \Big) \leq \Cmac \exp \big\{ - \cmac  h^2 \big\}
\end{equation}
for $h \geq 0$.  
 \end{lemma}
{\bf Proof.} 
Note that, for  $u \in I$ and $v \in J$, 
\begin{eqnarray*}
 & &  \weight_r \big[ (u,0) \to (v,\kappa^{3/2}) \big] -  \weight_r \big[ (u',0) \to (v',\kappa^{3/2}) \big]  \\
 & = & \Delta^{\cup} \, \weight_r \big[ ( \{ u',u\} ,0) \to ( \{v',v\},\kappa^{3/2}) \big] \, + \,  2^{-1/2} \frac{(v' - u')^2}{\kappa^{3/2}} -  2^{-1/2} \frac{(v - u)^2}{\kappa^{3/2}} \\
  & \leq &  \Delta^{\cup} \, \weight_r \big[ ( \{ u',u\} ,0) \to ( \{v',v\},\kappa^{3/2}) \big] \, ,
\end{eqnarray*}
where the inequality is due to~(\ref{e.uvinfimum}). Thus, $E_{I,J} \leq \kappa^{-1/2} \Delta^{\cup} \, \weight_r \big[ ( \{ u',u\} ,0) \to ( \{v',v\},\kappa^{3/2}) \big]$. 
We now apply Proposition~\ref{p.fluc}  with parameter settings ${\bf n} = r$, ${\bf s_{1,2}} = \kappa^{3/2}$, ${\bf a} = b$ and ${\bf K} =h$, and with ${\bf x}$ and ${\bf y}$ equal to the left endpoints of $I$ and 
$J$. Note that the hypothesis ${\bf a} \leq 2^{-4}$ holds due to $b \leq 2^{-4}$.
The hypothesis 
 $\big\vert {\bf x} - {\bf y}  \big\vert s_{1,2}^{-2/3} \leq 2^{-2} 3^{-1} \rsc  (r s_{1,2})^{1/18}$ holds due to $\vert {\bf x} \vert , \vert {\bf y} \vert \leq r^{1/20}$, which follows by hypothesis on $I$ and $J$;  and to the hypothesised lower bound on $r$. 

The hypothesis ${\bf n} \, {\bf \tot} \geq 10^{32} c^{-18}$ is due to $r \kappa^{3/2} \geq 10^{32} c^{-18}$, a consequence of the hypothesised lower bound on $r$ alongside $c < 1$ and $\kappa<1$. The hypothesis  ${\bf K} \in \big[10^4 \, , \,   10^3 (r \kappa^{3/2})^{1/18} \big]$ holds provided that we impose this condition on $h$. This application of Proposition~\ref{p.fluc} yields that
$$
 \PP \big( E_{I,J} \geq  b^{1/2}h \big) \leq 10032 \exp \big\{ - c 2^{-24} h^2 \big\}
$$ 
for  $h \in \big[10^4 \, , \,   10^3 r^{1/18} \kappa^{1/12} \big]$. 
Define the error event $\mathsf{E}_{I,J} = \big\{ E_{I,J} \geq h_0 (b\kappa)^{1/2}  \big\}$ where $h_0$ equals the maximal value  $10^3 r^{1/18} \kappa^{1/12}$  for the range of $h$.
We obtain~(\ref{e.errorprob}), and~(\ref{e.ijh}) 
for $h \geq 0$. \qed

To address the point-to-point normalized weight $Z_{I,J}$, we introduce a parabolically adjusted version:
\begin{equation}\label{centered}
\bar Z_{I,J}=Z_{I,J} + 2^{-1/2}\frac{(v'-u')^2}{\kappa^{2}} \, .
\end{equation}
By Lemma \ref{l.onepointbounds} with parameter setting ${\bf n} = r \kappa^{3/2}$, and the scaling principle, it follows that, for $t \ge 0$,
\begin{equation}\label{e.ijh1}
 \PP \Big( |\bar Z_{I,J}| \geq t \Big) \leq  C\exp \big\{ - ct^{3/2} \big\} \, , \, \, \text{  and hence\, }
\PP \Big( |\bar Z_{I,J}|{\bf 1}_{\mathsf{E}_{I,J}^c}  \geq t \Big) \leq  C \exp \big\{ - ct^{3/2} \big\} \, .
\end{equation}
This  implies that 
\begin{equation}\label{e.ijh2}
 \PP \Big( \left|\bar Z_{I,J} {\bf 1}_{\mathsf{E}_{I,J}^c}-\E[\bar Z_{I,J} {\bf 1}_{\mathsf{E}_{I,J}^c}]\right|\geq t \Big) \leq  C\exp \big\{ - ct^{3/2} \big\} 
\, .
\end{equation}
Concentration of the sums $\sum_{i=1}^{j}\bar Z_{I,J}$ and $\sum_{i=1}^{j} Z_{I,J}$ will be related by means of 
\begin{align}\label{decompose1}{{\sum_{i=1}^j \bar Z_{I,J} {\bf 1}_{\mathsf{E}_{I,J}^c}-\E[\bar Z_{I,J} {\bf 1}_{\mathsf{E}_{I,J}^c}]}}&={{\sum_{i=1}^j \big[ Z_{I,J} {\bf 1}_{\mathsf{E}_{I,J}^c}-\E[ Z_{I,J} {\bf 1}_{\mathsf{E}_{I,J}^c}]\big]}}+
{{\sum_{i=1}^j  2^{-1/2}\frac{(v'-u')^2}{\kappa^{2}}[{\bf 1}_{\mathsf{E}_{I,J}^c}-\E({\bf 1}_{\mathsf{E}_{I,J}^c})]}}
\end{align}
with the next result offering control on the latter parabolic term.
\begin{lemma}With probability at least $1- \Cmac \kappa^{-3/2} \exp \big\{ - \cmac r^{1/9} \kappa^{1/6}  \big\},$ 
\begin{equation}\label{parabolicterm}
\sum_{i=1}^j 2^{-1/2}\frac{(v'-u')^2}{\kappa^{2}}[{\bf 1}_{\mathsf{E}_{I,J}^c}-\E({\bf 1}_{\mathsf{E}_{I,J}^c})]\le \Cmac r^{1/10}\kappa^{-4}\exp \big\{ - \cmac r^{1/9} \kappa^{1/6}  \big\} \, ,
\end{equation}
Note that, by our hypothesis on $\kappa$, the right-hand side is less than one when $r$ is large enough.
\end{lemma}
{\bf Proof.}
 Since by \eqref{e.errorprob}, $\E({\bf 1}_{\mathsf{E}_{I,J}^c})\ge 1- \Cmac \exp \big\{ - \cmac r^{1/9} \kappa^{1/6}  \big\},$ 
on the event ${\bf 1}_{\mathsf{E}_{I,J}^c}=1$ we have 
$$ 
2^{-1/2}\frac{(v'-u')^2}{\kappa^{2}}[{\bf 1}_{\mathsf{E}_{I,J}^c}-\E({\bf 1}_{\mathsf{E}_{I,J}^c})]\le \Cmac r^{1/10}\kappa^{-2}\exp \big\{ - \cmac r^{1/9} \kappa^{1/6}  \big\} \, ,
$$ 
since $|v'-u'|\le r^{1/20}.$
A union bound over $i \in \intint{j}$ now implies the lemma. \qed

We now note that the tail bound \eqref{e.ijh2} allows us to invoke Proposition \ref{p.Vershynin}(1). For brevity's sake, let $W_{I,J}=\bar Z_{I,J} {\bf 1}_{\mathsf{E}_{I,J}^c}-\E[\bar Z_{I,J} {\bf 1}_{\mathsf{E}_{I,J}^c}].$ By this proposition, we may find $d > 0$ so that
\begin{align*}
\PP \bigg( \sum_{i=1}^j W_{I,J} \ge \frac{d}{10}\kappa^{-3/2}\bigg) \leq \exp \big\{ - \cmac \dmac \kappa^{-3/2} \big\} \, .
\end{align*}
Now, by our assumptions on $r$ and $\kappa$, \eqref{decompose1} and \eqref{parabolicterm}, and the above lemma, it follows that
\begin{align}\label{conc10}
\PP \bigg( \sum_{i=1}^j Z_{I,J}{\bf 1}_{\mathsf{E}_{I,J}^c} -\E[Z_{I,J}{\bf 1}_{\mathsf{E}_{I,J}^c}] \ge \frac{\dmac}{8}\kappa^{-3/2}\bigg) \leq \exp \big\{ - \cmac d \kappa^{-3/2} \big\} \, .
\end{align}
Here we used $\Cmac \kappa^{-3/2} \exp \big\{ - \cmac  r^{1/9} \kappa^{1/6}  \big\}\le e^{-\cmac\kappa^{-3/2}}$, which follows from our hypotheses on $r$ and~$\kappa$.

In the last display, we see, in a mildly truncated form, the point-to-point mean weight that was the mainstay of our overview of the first element of useful sum analysis. Indeed, we next argue that this truncated mean is suitably negative.
\begin{lemma} 
There exists $\dmac>0$ such that, for $r \geq \Theta(1)$, 
$$
 \sum_{i=1}^j\E[Z_{I,J}{\bf 1}_{\mathsf{E}_{I,J}^c}] \le -\frac{\dmac}{4}\kappa^{-3/2} \, .
 $$
\end{lemma}
{\bf Proof.}
Since  Proposition \ref{p.negativemean} and the scaling principle imply that $\E[Z_{I,J}]\le-\dmac$,
it suffices to show that $\sum_{i=1}^j\E[Z_{I,J}{\bf 1}_{\mathsf{E}_{I,J}}]\le \frac{d}{10}\kappa^{-3/2}.$
Now note that $\E[Z_{I,J}{\bf 1}_{\mathsf{E}_{I,J}}]\leq \big( \E(Z^2_{I,J})\P(\mathsf{E}_{I,J}) \big)^{1/2}$. Since $r$ is assumed to be large enough, we are now done by  \eqref{e.errorprob}, and $\sqrt{\E(Z^2_{I,J})}\le O(\frac{r^{1/10}}{\kappa^2})$, the latter due to 
 Lemma \ref{l.onepointbounds}, which implies that, in the decomposition \eqref{centered}, the first term  is $O(1),$ and hence the above bound arises from the maximum value of the parabolic term.  \qed

A final result needed to deliver Proposition~\ref{p.weightuseful} concerns the error terms $E_{I,J}$.
\begin{lemma}\label{errorcontr2} There exist positive constants $C_1$ and $c_1$ such that, for small enough $b$ and $\kappa,$
$$
\P \bigg( \sum_{i=1}^jE_{I,J}{\bf 1}_{\mathsf{E}_{I,J}^c} \ge  b^{1/2} C_1 \kappa^{-3/2} \bigg) \, \leq \, \exp \big\{ -c_1{\kappa^{-3/2}} \big\} \, .
$$
\end{lemma}
{\bf Proof.}
Note that, by definition, the random variables $b^{-1/2}E_{I,J}{\bf 1}_{\mathsf{E}_{I,J}^c}$ are independent.
Further, by~\eqref{e.ijh}, they are sub-Gaussian. Thus, Proposition \ref{p.Vershynin}(2)'s hypotheses are satisfied, and the lemma is obtained. 
 \qed

{\bf Proof of Proposition~\ref{p.weightuseful}.}
By (\ref{conc10}) and the two just stated lemmas, 
we find that, for a suitably small choice of $b > 0$, it is 
 with probability at least $1-e^{-\cmac\kappa^{-3/2}}$ that
\begin{eqnarray*}
\sum_{i=1}^jX_i=\sum_{i=1}^jX_i{\bf 1}_{\mathsf{E}_{I,J}^c}&=&\sum_{i=1}^jZ_{I,J}{\bf 1}_{\mathsf{E}_{I,J}^c}+\sum_{i=1}^jE_{I,J}{\bf 1}_{\mathsf{E}_{I,J}^c}\\
&=&\sum_{i=1}^j\E[Z_{I,J}{\bf 1}_{\mathsf{E}_{I,J}^c}]+ \sum_{i=1}^j[Z_{I,J}{\bf 1}_{\mathsf{E}_{I,J}^c} -\E[Z_{I,J}{\bf 1}_{\mathsf{E}_{I,J}^c}]+\sum_{i=1}^jE_{I,J}{\bf 1}_{\mathsf{E}_{I,J}^c}\\
&\le & \frac{-\dmac}{4}\kappa^{-3/2}+\frac{\dmac}{8}\kappa^{-3/2}+ b^{1/2}C_1\kappa^{-3/2} \le \frac{-\dmac}{16}\kappa^{-3/2} \, .
\end{eqnarray*}
 The first equality follows from \eqref{e.errorprob} and a union bound.
\qed

\subsubsection{The otherwise sum}

Our upper bound on the otherwise sum, namely on the latter right-hand term in~(\ref{e.supweightuo}), will depend partly on there being few otherwise summands. To this end, we begin by stating and proving a simple claim giving a lower bound on the number of useful summands. 
Recall that it is hypothesised in Proposition~\ref{p.closelow} that  $2\chi \kappa^{-3/2} \in \N$ and $2\chi < 1$. 

We {\em claim} that, 
\begin{equation}\label{e.usefulclaim}
\textrm{for at least $(1-2\chi)\kappa^{-3/2} - 2$ indices $j \in \llbracket 0, \kappa^{-3/2}  - 1 \rrbracket$, the value $j \kappa^{3/2}$ is useful} \,, 
\end{equation}
where the notion of usefulness was specified after Proposition \ref{p.closelow}. 
In verifying this, we will describe 
two segments in~$\mathsf{c}$ as being {\em vertically consecutive} if their heights differ by~$\kappa^{3/2}$.
Consider the set formed from $\N \kappa^{3/2} \cap [0,1]$ by the removal of those elements that are the heights of members of $\mathsf{c}$. This set has cardinality at most $\chi \kappa^{-3/2} + 1$. Any element $j \kappa^{3/2}$ in the set forbids the values  $(j-1) \kappa^{3/2}$ and $j \kappa^{3/2}$ from being useful. Cumulatively, at most $2 \chi \kappa^{3/2} +2$ elements of $\kappa^{3/2 }\llbracket 0, \kappa^{-3/2}  - 1 \rrbracket$ are thus forbidden. The remainder, numbering at least $\kappa^{-3/2} - 2 \chi \kappa^{-3/2} -2$, are useful. This is as we claimed. 

We now present a bound on the upper tail of  the otherwise weight sum supremum
$\sup \big\{ W_o(\psi) : \psi \, \, \textrm{a $\mathsf{c}$-path} \big\}$.
Let $\psi$ again denote a given $\mathsf{c}$-path.
Let $O \in \N$ denote the number of otherwise sub-zigzags of~$\psi$. Recalling that $\mc{U}$ denotes the set of useful values, with $U = \vert \mc{U} \vert$, we have $(U+O) \kappa^{3/2} \leq 1$, because every sub-zigzag, useful or otherwise, has height at least $\kappa^{3/2}$;
so that~(\ref{e.usefulclaim}) implies that
\begin{equation}\label{e.oub}
O \leq 2\chi \kappa^{-3/2} + 2 \, .
\end{equation}
The expression $W_o(\psi)$ is the sum of weights of the otherwise sub-zigzags of $\psi$. Each otherwise sub-zigzag begins, but immediately leaves, a given segment in~$\mathsf{c}$, and ends in another such segment.
The two segments may be called the {\em starting} and {\em finishing} segments of the sub-zigzag. 
Labelling the otherwise sub-zigzags  $\big\{ Z_i: i \in \intint{O} \big\}$ in order of increasing height, we denote by $S_i$ and $F_i$ the starting and ending segments of $Z_i$. The respective heights of $S_i$ and $F_i$ will be denoted by $s_i$ and~$f_i$.  Note  that $s_i < f_i \leq s_{i+1} < f_{i+1}$ for $i \in \intint{O-1}$.

For $i \in \intint{O}$, let  $(u_i,s_i) \in S_i$ and $(v_i,f_i) \in F_i$ be chosen so that  the gradient of the line segment connecting $(u_i,s_i)$ and $(v_i,f_i)$ is maximal given that these endpoints lie in $S_i$ and $F_i$. This implies that 
 \begin{equation}\label{e.uvinfimum}
\vert v_i - u_i \vert = \inf \Big\{ \vert v - u \vert:  u,v \in \R, (u,s_i) \in S_i , (v,s_{i+1}) \in F_i \Big\} \, . 
 \end{equation}

We find then that the otherwise weight sum supremum
\begin{equation}\label{e.otherwisesup}
  \sup \big\{ W_o(\psi) : \psi \, \, \textrm{a $\mathsf{c}$-path} \big\} \, \, \textrm{is stochastically dominated by} \, \,
 \sum_{i = 1}^{O}  
 \big( \, \mathsf{W}_i + E_i \, \big) 
   \, ,
\end{equation}
where $\big\{ \mathsf{W}_i : i \in \llbracket 1, O \rrbracket \big\}$
is an independent sequence whose term $W_i$ has the distribution of $\weight_r \big[ (u_i,s_i) \to (v_{i},f_{i}) \big]$; and
where  $\big\{ E_i : i \in \llbracket 1, O \rrbracket \big\}$ is an independent sequence of error terms given by
$$
 E_i \, =  \, \sup_{\begin{subarray}{c} u,v \in \R : \\ (u,s_i) \in S_i \\ (v,f_{i}) \in F_{i}  \end{subarray}}  \Big( \, \weight_r \big[ (u,s_i) \to (v,f_{i}) \big] -  \weight_r \big[ (u_i,s_i) \to (v_{i},f_{i}) \big] \, \Big) \, .
$$
Reminiscently of useful sum analysis, the right-hand quantity in~(\ref{e.otherwisesup}) is a sum of a point-to-point $W$-sum and an error $E$-sum. The next two lemmas, which provide tail bounds on these two sums, are the outcomes of otherwise sum analysis that will be needed to prove Proposition~\ref{p.closelow}. 
\begin{lemma}\label{l.weighto}
There exist  $y \in \R$ and $s \in [0,1]$ that satisfy $\vert y \vert \leq r^{1/20} + 2\chi b \kappa^{-1/2}$
 and $s \leq 2 \chi$ such that
the random variable $\sum_{i = 1}^{O} \mathsf{W}_i$
is stochastically dominated by 
 $\weight_r \big[ (0,0) \to (y,s) \big]$.
 \end{lemma}
 {\bf Proof.} Set $s = \sum_{i=1}^O (f_i - s_i)$ and $y = \sum_{i=1}^O (v_i - u_i)$. Set $S$
 equal to the set of planar points $p_j = \big(  \sum_{i=1}^j (v_i - u_i)  , \sum_{i=1}^j (f_i - s_i) \big)$ where $j$ varies over $\intint{O}$.
 The value $\weight_r \big[ (0,0) \to (y,s) \big]$ is the supremum of the weights of $r$-zigzags from $(0,0)$ to $(y,s)$. This value is at least the supremum~$W$ of weights of $r$-zigzags with these endpoints but that also contain the set $S$. Note that $W$ has the law of  $\sum_{i = 1}^{O} \mathsf{W}_i$, because, if $Z$ denotes the maximizer zigzag in the optimization that specifies~$W$, the weight of the sub-zigzag of $Z$ between $p_i$ and $p_{i+1}$ equals $\mathsf{W}_i$ in law, and the various sub-zigzags have independent weights. \qed
 
 \begin{lemma}\label{l.e}
There exist $\Cmac,\cmac > 0$ such that the following holds. Suppose that 
 $r \geq  \Cmac  \kappa^{-195}$,
 $h \in \big[10^4 \, , \,   10^3 (r s_{i,i+1})^{1/18} \big]$ and $b \leq 2^{-4}$.
\begin{enumerate}
\item 
 For $i \in \llbracket 1, O \rrbracket$,
$$
\PP \Big(  E_i  \geq  h  (b \kappa)^{1/2} \Big) \leq  \Cmac \exp \big\{ - \cmac  h^2 \big\} \, .
$$
\item Suppose that $\chi \geq 2\kappa^{3/2}$. There exist $h_0 > 0$ and an error event $\mathsf{E}$ satisfying 
\begin{equation}\label{e.erroreventbound}
\PP \big( \mathsf{E} \big) \leq \Cmac \chi \kappa^{-3/2}   \exp \big\{ - \cmac  r^{1/9}  \kappa^{1/6} \big\}
\end{equation}
such that, for $h \geq h_0$,
\begin{equation}\label{e.gaussianld}
 \PP \Big( 
 \sum_{i = 1}^{O}  
  E_i  {\bf 1}_{\mathsf{E}^c} \geq h  \chi b^{1/2} \kappa^{-1} \Big) \leq \Cmac \exp \big\{ - \cmac \chi \kappa^{-3/2} h^2 \big\} \, .
\end{equation}
\end{enumerate}
\end{lemma}
{\bf Proof: (1).}
Note that, for  $u,v \in \R$ with $(u,s_i) \in S_i$ and $(v,s_{i+1}) \in S_{i+1}$, 
\begin{eqnarray*}
 & &  \weight_r \big[ (u,s_i) \to (v,s_{i+1}) \big] -  \weight_r \big[ (u_i,s_i) \to (v_{i+1},s_{i+1}) \big]  \\
 & = & \Delta^{\cup} \, \weight_n \big[ ( \{ u_i,u\} ,s_1) \to ( \{v_i,v\},s_2) \big] \, + \,  2^{-1/2} \frac{(v_i - u_i)^2}{s_2-s_1} -  2^{-1/2} \frac{(v - u)^2}{s_2-s_1} \\
  & \leq &  \Delta^{\cup} \, \weight_n \big[ ( \{ u_i,u\} ,s_1) \to ( \{v_i,v\},s_2) \big] \, ,
\end{eqnarray*}
where the inequality is due to~(\ref{e.uvinfimum}). Thus, $E_i \leq  \Delta^{\cup} \, \weight_n \big[ ( \{ u_i,u\} ,s_1) \to ( \{v_i,v\},s_2) \big]$. 
We now apply Proposition~\ref{p.fluc} with parameter settings ${\bf n} = r$, ${\bf s_{1,2}} = s_{i,i+1}$, ${\bf a} = s_{i,i+1}^{-2/3}  b \kappa$ and ${\bf K} =h$, and with ${\bf x}$ and ${\bf y}$ equal to the left endpoints of $S_i$ and 
$S_{i+1}$. Note that the hypothesis ${\bf a} \leq 2^{-4}$ holds due to $b \leq 2^{-4}$ and $s_{i,i+1} \geq \kappa^{3/2}$.
The hypothesis 
 $\big\vert {\bf x} - {\bf y}  \big\vert s_{i,i+1}^{-2/3} \leq 2^{-2} 3^{-1} \rsc  (r s_{i,i+1})^{1/18}$ holds due to $\vert {\bf x} \vert , \vert {\bf y} \vert \leq r^{1/20}$, which follows from $S_1, S_2 \in \mathsf{c}$; to $s_{i,i+1} \geq \kappa^{3/2}$; and to the hypothesised lower bound on $r$. 
The hypothesis ${\bf n} \, {\bf \tot} \geq 10^{32} c^{-18}$ is due to $r \kappa^{3/2} \geq 10^{32} c^{-18}$, a consequence of the hypothesised lower bound on $r$ alongside $c < 1$ and $\kappa<1$. The hypothesis  ${\bf K} \in \big[10^4 \, , \,   10^3 (n \tot)^{1/18} \big]$ holds because this condition is imposed on $h$. Lemma~\ref{l.e}(1) follows from this application of Proposition~\ref{p.fluc}. 

{\bf (2).} Set the error event $\mathsf{E}$ by defining  $\mathsf{E} =\cap_{i \in \llbracket 0, O - 1 \rrbracket} \mathsf{E}_i$, where
$$
 \mathsf{E}_i =  
 \Big\{ E_i \geq (b\kappa)^{1/2}10^3 \big( r s_{i,i+1}\big)^{1/18}  \Big\} \, .
$$
Since $\chi \geq 2\kappa^{3/2}$,~(\ref{e.oub}) implies that  $\vert O \vert \leq 3 \chi \kappa^{-3/2}$.  
  From this bound, and since
$s_{i,i+1} \geq \kappa^{3/2}$ for all concerned indices~$i$, we obtain~(\ref{e.erroreventbound}), from the conclusion of Lemma~\ref{l.e}(1) and a union bound.

The random variables $E_i {\bf 1}_{\mathsf{E}_i^c}$ verify the conclusion of Lemma~\ref{l.e}(1) for all $h \geq 10^4$; that is, even after the removal of the upper bound on $h$ hypothesised in that result. Hence, they are,  after scaling by the factor $(b\kappa)^{-1/2}$,  sub-Gaussian variables as in Proposition~\ref{p.Vershynin}(2). Since~(\ref{e.oub}) holds, (\ref{e.gaussianld}) is implied by the bound $\PP \big( \sum_{i=1}^K E'_i \geq h K \big)   \leq e^{-\Theta(1) K h^2}$ for all large $h,$
where $K = \lceil 2\chi \kappa^{-2/3} \rceil + 2$ and $E'_i = (b\kappa)^{-1/2} E_i {\bf 1}_{\mathsf{E}_i^c}$. Since the $E_i'$ are by definition independent, the desired bound follows from Proposition~\ref{p.Vershynin}(2).  \qed

\subsubsection{Proof of Proposition~\ref{p.closelow}}
We are ready to return to the bound (\ref{e.supweightuo}) in order to prove the upper tail bound stated by this proposition. Indeed, by~(\ref{e.supweightuo}),~(\ref{e.usefulweightsum}) and~(\ref{e.otherwisesup}),
 $$
 \PP \Big( \weight_r \big[ \mathsf{c}\textrm{-path} \big] \geq - d_1 \kappa^{-1} \Big) \leq A_1 + A_2 + A_3 \, ,
 $$
 where here we set $A_1 = \PP \big( \sum_{i=1}^{\vert \mc{U} \vert} U_i \geq -  2d_1 \kappa^{-1} \big)$,
 $A_2 = \PP \big( \sum_{i=1}^{O} \mathsf{W}_i \geq  2^{-1} d_1 \kappa^{-1} \big)$ and $A_3 = \PP \big( \sum_{i=1}^{O} E_i \geq  2^{-1} d_1 \kappa^{-1} \big)$.

To find an upper bound on $A_1$, note that 
$\sum_{i=1}^{\vert \mc{U} \vert} U_i \geq - 2d_1 \kappa^{-1}$ entails that $\sum_{i=1}^{\kappa^{-3/2}} U_i \geq - 2d_1 \kappa^{-1}$, where on the right-hand side, further $U_i$-terms have been introduced consistently with the conditions on this sequence. Provided that  $b \leq b_0$ and $d_1 \leq d/8$, we may apply 
Proposition~\ref{p.weightuseful}
 with $j = \kappa^{-3/2}$ to find that 
$$
A_1 \leq \exp \big\{ - d \kappa^{-3/2} \big\} \, .
$$ 
Let the parameters $y \in \R$ and $s \in [0,1]$ satisfy the hypotheses of Lemma~\ref{l.weighto}. This result implies that 
$A_2 \leq  \PP \big( \weight_r \big[ (0,0) \to (y,s) \big]  \geq 2^{-1} d_1 \kappa^{-1} \big)$.
From $s \leq 2\chi$,  and
the one-point upper tail bound
 $\PP \big( \weight_r \big[ (0,0) \to (y,s) \big]  \geq h s^{1/3}  \big) \leq C \exp \big\{ - c h^{3/2} \big\}$ offered by  Lemma~\ref{l.onepointbounds}(1) via the scaling principle, we see that
$$
A_2 \leq C \exp \big\{ -  2^{-2} c \chi^{-1/2} d_1^{3/2} \kappa^{-3/2}   \big\} \, . 
$$
In the notation of Lemma~\ref{l.e}(2), $A_3 \leq  \PP \big( \sum_{i=1}^{O} E_i {\bf 1}_{\mathsf{E}^c} \geq  2^{-1} d_1 \kappa^{-1} \big) + \PP \big(\mathsf{E} \big)$. Choose $h$ in  Lemma~\ref{l.e}(2) so that $h  \chi b^{1/2} = 2^{-1} d_1$; we ensure the needed condition that $h \geq h_0$ by insisting that $\chi > 0$ be small enough (as we do by demanding that $\chi \leq \chi_0$ in Proposition~\ref{p.closelow}). From Lemma~\ref{l.e}(2), we thus learn that 
$$
A_3 \leq  \Cmac \exp \big\{ - \cmac  \kappa^{-3/2} \chi^{-1} b^{-1}  2^{-2} d_1^2 \big\}  +  \Cmac \chi \kappa^{-3/2}   \exp \big\{ - \cmac   r^{1/9}  \kappa^{1/6} \big\} \, .
$$
Applying $r \geq \kappa^{-15/2}$
 in the guise $r^{1/9}  \kappa^{1/6} \geq \kappa^{-2/3}$, we obtain Proposition~\ref{p.closelow} by choosing (or adjusting) the positive parameters $\kappa_0$, $d_1$ and $d_2$ to be suitably small. \qed

\subsection{Deriving Theorem~\ref{t.nocloseness}}

Throughout this section, we suppose that $\tza^{-1/4} > C\log n$ and  
 that $\ell \in \N$ satisfies $2^{\ell}\leq n \tza^{40}$, since these conditions are hypothesised by the result that we seek to show.

To apply Proposition~\ref{p.closelow}, let $\kappa > 0$ satisfy $\kappa^{3/2} \in n^{-1}\Z$, $\kappa^{-3/2} \in \N$ and {$b\kappa/8 \in  2^{-2\ell/3} \tza \cdot [1,2]$}, where $b$ appears in the statement of the proposition. {Let $\kappa$ also satisfy $\kappa^{3/2} \leq 2^{-1-\ell}\chi$.}

\begin{figure}[t]
\centering{\epsfig{file=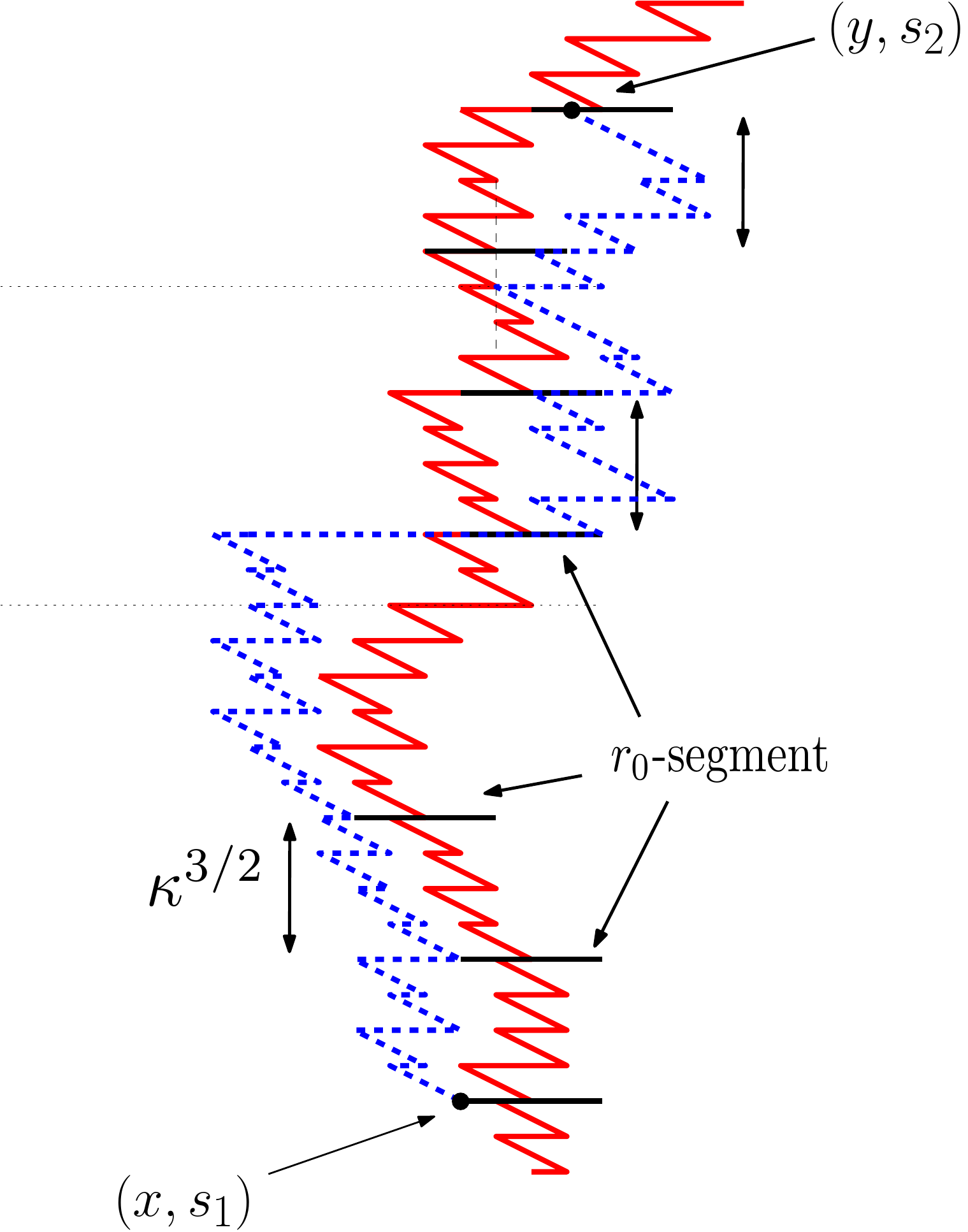, scale=0.4}}
\caption{The solid zigzag is $\phi$ and the dashed zigzag is a $\mathsf{c}$-path which passes through the elements of $\mathsf{c}$, denoted by the black horizontal planar lines of length $b\kappa$, a plentiful $r_0$-segment collection. Such segments occur regularly at vertical separation being integer multiples of $\kappa^{3/2}$. 
 }\label{f.plentifulsegment}
\end{figure} 

As we derive Theorem~\ref{t.nocloseness}, we will define the terms segments, plentiful segment collection and $\mathsf{c}$-path.
In doing so, we abuse the notation employed in  Proposition~\ref{p.closelow}. However, as we will explain shortly, the different usages coincide when suitable parameters are specified and a simple change of coordinates is made.

Suppose given an $n$-zigzag $\phi$ from $(0,0)$ to $(0,1)$. 
For now, we take an arbitrary such $\phi$, though we later impose the condition, seen in Theorem~\ref{t.nocloseness}, {that $\phi$ be $\theta^{-1/40}$ regular.}

Further suppose given $s_1,s_2 \in n^{-1}\Z \cap [0,1]$ for which $2^{-1-\ell} \leq \tot \leq 2^{-\ell}$; Recall that our choice of~$\kappa$ ensures that {$\kappa^{3/2} \leq \chi \tot$}. Let $\remainder \in n^{-1} \llbracket 0, n \kappa^{3/2} - 1 \rrbracket$. An {\em  $\remainder$-segment} is a horizontal planar interval of the form 
$$
I_{j/n \, + \, \remainder} := \big[ \phi\big(j/n + \remainder \big) -  b \kappa/2,   \phi\big(j/n + \remainder \big) + b \kappa/2 \big] \times \{ j n^{-1} + \remainder \}
$$ 
where $j/n \in \kappa^{3/2}\Z$ and  $j/n + r_0 \in n^{-1} \llbracket ns_1,ns_2    \rrbracket$: see Figure \ref{f.plentifulsegment}.

 A plentiful $\remainder$-segment collection is a subset of $\big\{ I_{j/n + \remainder}:  j/n \in \kappa^{3/2}\Z \, , \, j/n + r_0 \in n^{-1} \llbracket ns_1,ns_2    \rrbracket \big\}$
whose cardinality is at least $(1-\chi) \tot \kappa^{-3/2}$, where note that the quantity $\tot \kappa^{-3/2}$ differs from the cardinality of the set
$\big\{ j/n \in \kappa^{3/2}\Z : j/n + r_0 \in n^{-1} \llbracket ns_1,ns_2    \rrbracket \big\}$ by at most one. If $v^-$ and $v^+$ denote the lowest and highest vertical coordinates assumed by elements in a plentiful $\remainder$-segment collection, note that
\begin{equation}\label{e.lowesthighest}
  \big\vert v^-  - s_1 \big\vert   \vee  \big\vert v^+ - s_2 \big\vert  
  \leq 2\chi \tot \, .
\end{equation}
Indeed, the left-hand side is at most $\chi \tot + \kappa^{3/2}$, and $\kappa^{3/2} \leq \chi \tot$.

Let $\mc{C}_{\remainder}$ denote the set of plentiful $\remainder$-segment collections. Set $\mc{C}$ equal to the union of $\mc{C}_{\remainder}$ as $\remainder$ varies over $n^{-1} \llbracket 0, n \kappa^{3/2} - 1 \rrbracket$. For given such $\remainder$, let $\mathsf{c} \in \mc{C}_{\remainder}$.
A $\mathsf{c}$-path is an $n$-zigzag 
whose starting moment is the lowest height of an element of $\mathsf{c}$;
whose ending moment is the greatest such height; and that 
 intersects every element of $\mathsf{c}$.
Let $\weight_n \big[ \mathsf{c}\textrm{-path} \big]$ denote the supremum of the weights of $\mathsf{c}$-paths. 

Define  
$$
{\mathsf{HighSlenderWeight} \big( s_1,s_2,1-\chi;\phi \big) \, :=} \, \bigcup_{\mathsf{c} \in \mc{C}} \,  \Big\{ \, \weight_n \big[ \mathsf{c}\textrm{-path} \big] \geq -  2^{-1} \redconst \kappa^{-1} \tot \, \Big\} \, . 
$$

Recall the event $\uni_n(\cdot)$ from Subsection~\ref{s.nosubpath}.
The argument to prove Theorem \ref{t.nocloseness}  proceeds by showing that, on the event $\uni_n(\cdot)$---with a choice of this event's parameter that will make it typical---the non-occurrence of $\mathsf{LowSlenderWeight}^*(\ell, \tza, 1 - \chi;  \phi)$ implies the occurrence of ${\mathsf{HighSlenderWeight} \big( s_1,s_2,1-\chi;\phi \big) \,}$ for some $s_1$ and $s_2.$  The latter event is then shown to be rare, implying the desired rarity of $\mathsf{LowSlenderWeight}^*(\ell, \tza, 1 - \chi;  \phi)$.

Recall the notion of regularity from \eqref{regular-criterion}.
 and that the definition of $\mathsf{LowSlenderWeight}^*(\ell, \tza , 1 - \chi;  \phi)$ involves a constant $d_0.$
\begin{lemma}\label{l.lowcontain} 
 Suppose that $\chi^{1/3} \leq 2^{-11/3}\redconst$. 
Then, for any zigzag $\phi$ from $(0,0)$ to $(0,1)$ that is $\theta^{-1/4}$-regular, 
\begin{eqnarray*}
 & &  \uni \big( \frac{b}{10}\tza^{-1}\big) \, \cap \, \neg \,  \mathsf{LowSlenderWeight}^*(\ell, \tza , 1 - \chi;  \phi) \\
 & \subseteq & \bigcup_{\begin{subarray}{c} 0 \leq s_1 \leq s_2 \leq 1 \\
\tot \in [2^{-1-\ell},2^{-\ell}] \\   \end{subarray}} \mathsf{HighSlenderWeight}(s_1,s_2,1-\chi; \phi) \, .  
\end{eqnarray*}
\end{lemma}
{\bf Proof.}
When $\neg \,  \mathsf{LowSlenderWeight}^*(\ell, \tza , 1 - \chi;  \phi)$
occurs, there exist $(x,s_1),(y,s_2) \in \R \times n^{-1}\Z$ for which $2^{-1-\ell}\leq \tot \leq 2^{-\ell}$, and a $(\phi,\tza,1-\chi)$-close zigzag $\psi$
from $(x,s_1)$ to $(y,s_2)$ 
that satisfies 
\begin{equation}\label{e.weightpsi}
\weight_n (\psi) > -  \redconst \kappa^{-1}\tot \, . 
\end{equation}
By our choice of $b$ and $\kappa$, the zigzag $\psi$ satisfies the bound 
\begin{equation}\label{e.psiphi}
\big\vert \psi(s)  - \phi(s) \big\vert \leq b\kappa/2
\end{equation}
 for at least $(1 - \chi) \big\vert  [s_1,s_2] \cap n^{-1}\Z \big\vert$ values of $s \in [s_1,s_2] \cap n^{-1}\Z$. When elements of $[s_1,s_2] \cap n^{-1}\Z$  are identified if they differ by a multiple of $\kappa^{3/2}$, they are partitioned into classes which are naturally indexed by $\remainder$. Since $\tot \geq \kappa^{3/2}$, the number of classes equals $n \kappa^{3/2}$. 
At least one of these classes---call it $\mc{D}_{r_0}$---contains at least $n^{-1}\kappa^{-3/2} (1 - \chi)  \big\vert  [s_1,s_2] \cap n^{-1}\Z \big\vert = \kappa^{-3/2} (1 - \chi)  \big( \tot + n^{-1}\big)$ elements~$s$ that satisfy~(\ref{e.psiphi}).
Let $\mathsf{c}$ denote the set of $I_{j/n + r_0}$ indexed by those $j/n + r_0 \in \mc{D}_{r_0}$ such that~(\ref{e.psiphi}) is satisfied with $s = j/n + r_0$ (where note that $j/n \in \kappa^{3/2} \Z$ and $j/n + r_0 \in n^{-1} \llbracket ns_1,ns_2    \rrbracket$).

Note that $\vert \mathsf{c} \vert \geq \kappa^{-3/2} (1 - \chi) \tot$.
Thus, $\mathsf{c}$ is a plentiful $r_0$-segment collection; so that $\mathsf{c} \in \mc{C}$. We claim that 
\begin{equation}\label{e.fortyeight}
\weight_n \big[ \mathsf{c}\textrm{-path} \big] \, \geq \, -   2^{-1}\tot\redconst \kappa^{-1} \, . 
\end{equation}
To verify this, let $\psi^0$ denote the sub-zigzag of $\psi$ from the entry of $\psi$ to the lowest vertical coordinate $v^-$ in $\mc{D}_{r_0}$ to its departure from the highest such coordinate~$v^+$. Note that 
$\weight_n \big[ \mathsf{c}\textrm{-path} \big]$  is at least $\weight_n (\psi^0)$. To bound below the latter weight, 
 we write $\psi$ as a concatenation $\psi^- \circ \psi^0 \circ \psi^+$.

{We now claim that, since $\phi$ is regular, the occurrence of the event $\uni_n\big(b\tza^{-1}\big)$ entails that $\weight_n(\psi^-)$ and $\weight_n(\psi^+)$
are at most $$2\big( 4\chi \tot \big)^{1/3} \tot^{2/3}\kappa^{-1}= 2^{5/3}\chi^{1/3} \tot\kappa^{-1}.$$

Admitting the claim, note that, by weight additivity, $\weight_n(\psi^0)$ equals  $\weight_n(\psi) -  \weight_n(\psi^-) - \weight_n(\psi^+)$; by~(\ref{e.weightpsi}), and the claim, this weight is thus seen to be at least  
$- \tot \redconst \kappa^{-1} + 2^{8/3}\chi^{1/3} \tot \kappa^{-1}$. Since $2^{8/3}\chi^{1/3} \leq 2^{-1}\redconst$, we have verified~(\ref{e.fortyeight}) for all $n$ high enough. In view of the definition of the $\mathsf{HighSlenderWeight}$ event, this reduces the proof of Lemma~\ref{l.lowcontain} to deriving the claim. 

To this end, note first that bounds on $\weight_n(\psi^-)$ and $\weight_n(\psi^+)$ cannot be obtained directly from Proposition~\ref{p.onepoint}, since the duration of these zigzags may be too small for this result to offer a meaningful bound. This said, the claimed bounds follow easily from superadditivity. 
We will provide only the argument for $\psi^-$.  Let  $(x_1,s_1)$ and $(x_2, v_-)$ denote this zigzag's endpoints (recall from \eqref{e.lowesthighest}, that $v_--s_1\le 2\chi \tot$). Now consider the point $(x_2, v_-+\chi \tot).$ 
We then have 
$$
\weight_n(\psi_-)\le \weight_n \big[ (x_1,s_1)\to (x_2,v_-+\chi \tot) \big] - \weight_n \big[ (x_2,v_-)\to (x_2,v_-+\chi \tot) \big] \, .
$$

{Now, recall that, when $\phi$ is $\theta^{-1/4}$-regular, we have that,  for any $(u,h_1),(v,h_2) \in \phi$,  
\begin{equation*}
\big\vert v -u \big\vert \leq h_{1,2}^{2/3}\theta^{-1/4}\le  h_{1,2}^{2/3} \max(\theta^{-50/4}, nh_{1,2})^{1/50}.
\end{equation*}
In particular, this means that $x_1$ and $x_2$ are at most $n^{1/50}$ in absolute value, and $|x_1-x_2|\le (4\chi\tot)^{2/3}(n4\chi \tot)^{1/50}.$ Note that this last deduction needs $4\chi\tot\ge \frac{\tza^{-50/4}}{n}$ for all fixed $\chi$ and all large enough $n$, which condition is implied by our stronger standing assumption that  $2^{-\ell}\ge \frac{\tza^{-40}}{n}.$  } 

We can now conclude that the event $\uni_n\big(\frac{b}{10}\tza^{-1}\big)$ implies that the quantities $\weight_n \big[ (x_1,s_1)\to (x_2,v_-+\chi \tot) \big]$ and  $\weight_n \big[ (x_2,v_-)\to (x_2,v_-+\chi \tot) \big]$ are at most $\big( 4\chi \tot \big)^{1/3} (\frac{b}{8}\theta^{-1}).$
We prove only the {\em claim} that this bound holds for the first term, because a similar argument works for the second. The occurrence of $\uni_n\big(\frac{b}{10}\tza^{-1}\big)$ entails that the parabolically adjusted weight satisfies 
$$
\weight^\cup_n \big[ (x_1,s_1)\to (x_2,v_-+\chi \tot) \big] \le \frac{b}{10}\tza^{-1} \, ,
$$ 
while the parabolic correction term $ 2^{-1/2}(x_1-x_2)^2 (v_-+\chi \tot-s_1)^{-1}$ is at most $O(\big( 4\chi \tot \big)^{1/3} \theta^{-1/2})\le \big( 4\chi \tot \big)^{1/3} \frac{b}{100}\theta^{-1}$ for all large enough $n$, since $\theta $ is assumed to satisfy $\theta^{-1/4}\ge C \log n$ throughout this section.

Since $b \kappa\le 8\tot^{2/3}\tza$, we find that the above bound is at most $\big( 4\chi \tot \big)^{1/3} \tot^{2/3}\kappa^{-1},$ finishing the proof of the claim that we sought to show.
\qed

\begin{lemma}\label{l.highslenderweight} There exists $\chi_0 \in (0,1)$ such that, when $\chi \in (0,\chi_0)$,
 $s_1,s_2$  satisfy $0 \leq s_1 \leq s_2 \leq 1$, and 
$\tot \in [2^{-1-\ell},2^{-\ell}]$ and $\phi$ is a $\theta^{-1/4}$-regular $n$-zigzag from $(0,0)$ to $(0,1)$,
$$
 \PP \Big(  \mathsf{HighSlenderWeight} \big( s_1,s_2,1-\chi;\phi \big)  \Big) \, \leq  \, \exp \big\{ - d_2 \kappa^{-3/2} 2^{-\ell}  \big\}  \, .
$$ 
\end{lemma}
{\bf Proof.} 
Note that
$$
 \vert \mc{C} \vert \leq 
 n \kappa^{3/2} \sum_{k=0}^{\lceil  \chi \tot \kappa^{-3/2} \rceil} {\lceil \tot \kappa^{-3/2} \rceil \choose k}  \leq  n \kappa^{3/2} \cdot \tot \kappa^{-3/2} { \lceil \tot \kappa^{-3/2} \rceil \choose \lceil \chi \tot \kappa^{-3/2} \rceil} \, , 
$$
where the latter bound invokes $\chi \leq 1/4$. 

Now let $\mathsf{c} \in \mc{C}$. We now proceed to express our present circumstance in the notation of Proposition~\ref{p.closelow}. Given $\mathsf{c},$ let $r=n(v^+-v^-),$ where $v^-$ and $v^+$ are the lowest and highest vertical coordinates assumed by elements in the plentiful $\remainder$-segment collection given by $\mathsf{c}.$  Note that, when $\chi_0$ is small enough, $\frac12 {n \tot}\le { r} \le  n \tot.$ We will now apply Proposition~\ref{p.closelow} with  ${\bf r}=r$ and ${\bm \kappa} = \kappa \big(\frac{n}{{r}}\big)^{2/3} \in b^{-1}\tza \cdot [8,16]$. In a detail to ensure formal accuracy of the application, we apply a vertical translation that sends $v^-$ to zero. Furthermore, given the collection of horizontal segments of length $b\kappa$ forming~$\mathsf{c}$, let $\mathsf{c}_*$ be the collection obtained by multiplying each of them by the factor $\big(\frac{n}{{\bf r}}\big)^{2/3},$ and their vertical heights by $\frac{n}{{ r}}.$ Now, by the scaling principle, we conclude that  $\big(\frac{{ r}}{n}\big)^{1/3} \weight_n \big[ \mathsf{c}\textrm{-path} \big]$ is equal in law to $\weight_{ r} \big[ \mathsf{c}_*\textrm{-path} \big].$
 Note that, by the assumed regularity of $\phi$,  the  horizontal segments of $\mathsf{c}$ are confined in an interval of length ${\tot}^{2/3}({n\tot})^{1/30}\le \big( r n^{-1} \big)^{2/3}{{r}}^{1/20},$ and hence the elements of $\mathsf{c}_*$ are contained in a horizontal interval of length ${{r}}^{1/20}$.  Thus, by Proposition~\ref{p.closelow},  we learn that
$$
\PP \Big(  \weight_n \big[ \mathsf{c}\textrm{-path} \big] \geq  - \tot d_1 \kappa^{-1} \Big) \leq   \exp \big\{ - d_2 \kappa^{-3/2}2^{-\ell} \big\} \,.
$$

Hence, if $2^{-1}d_0\ge d_1,$
$$
 \PP \Big(  \mathsf{HighSlenderWeight} \big( s_1,s_2,1-\chi;\phi \big)  \Big) \leq n \kappa^{3/2} \cdot \tot \kappa^{-3/2} { \lceil \tot \kappa^{-3/2} \rceil \choose \lceil \chi \tot \kappa^{-3/2} \rceil}
  \exp \big\{ - d_2\kappa^{-3/2}2^{-\ell}  \big\} \, .
$$
Since
$\kappa \in  2^{-2\ell/3} b^{-1}\tza \cdot [8,16]$, $\tot \in [2^{-1-\ell},2^{-\ell}]$, $\tza^{-1/4} \ge C(\log n)$ and $b > 0$ is given, we may choose $\chi_0 \in (0,1)$ small enough that this right-hand side is at most $\exp \big\{ - 2^{-1}d_2 \kappa^{-3/2}   \big\}$.
Lemma~\ref{l.highslenderweight} follows by relabelling $d_2 > 0$. \qed

{\bf Proof of Theorem~\ref{t.nocloseness}.} By Lemmas~\ref{l.lowcontain} and~\ref{l.highslenderweight},
\begin{align*}
\PP \Big( \uni \big( \frac{b}{10}\tza^{-1}\big) \cap  \neg \,  \mathsf{LowSlenderWeight}^*(\ell, \tza , 1 - \chi;  \phi) \Big) &\leq  (n+1)^2 \exp \big\{ - d_2 (b^{-1}\tza)^{-3/2}  \big\}\\
&\le  \exp \big\{ - \frac{d_2}{2} (b^{-1}\tza)^{-3/2}  \big\}\, .
\end{align*}
The factor $(n+1)^2$ arises from the union bound taken over values of $s_1$ and $ s_2$ in Lemma \ref{l.lowcontain}; it is absorbed during the second inequality in view of $\tza^{-1/4} \ge C(\log n)$.
Further, by \eqref{uniformbound},
$$
\PP \Big(\neg \uni \big( \frac{b}{10}\tza^{-1}\big)\Big)\le \exp \big\{ -C \theta^{-3/2}\big\} \, .
$$

 For a suitably high choice of $n_0 \in \N$, the theorem is obtained from Proposition~\ref{p.onepoint} by assembling the preceding estimates and by adjusting the constant $d_2 > 0$. \qed

We are ready to prove Theorem \ref{t.slenderpolymer}.  Alongside the result just proved, the main ingredients are Theorem \ref{t.toolfluc}, which asserts that the polymer $\rho_n$ is typically regular with high probability; and
  the FKG inequality. Indeed, from the latter, we will learn that conditioning on $\rho_n$ has a negative effect, so that the proof will be completed by invoking Theorem~\ref{t.nocloseness} and the observation that 
  $\mathsf{LowSlenderExcursion}(\ell, \tza , 1 - \chi;  \rho_n)$ is a decreasing event in the remaining environment.
  
  \noindent  
 {\bf Proof of Theorem \ref{t.slenderpolymer}.} We start by recording some notation in order to state a stochastic domination lemma.
 A noise field will be viewed as a random function sending $\R \times n^{-1}\Z$  to~$\R$.
Let $X$ and~$Y$ denote two such. For any subset $A$ of $\R \times n^{-1}\Z$, $Y$ stochastically dominates $X$ on~$A$ if  there exists a coupling of $X$ and $Y$ such that, 
whenever $(j,u,v) \in \Z \times \R^2$ satisfies $u < v$ and $\{ j/n \} \times [u,v] \subset A$,
the bound
$Y(v,j/n)- Y(u,j/n) \geq X(v,j/n)- X(u,j/n)$ holds. An event $E$ that is measurable with respect to the natural $\sigma$-algebra generated by the increments of the Brownian motion on $A$ is called decreasing on $A$ if $\PP (Y \in E) \leq \PP (X \in E)$ whenever $Y$  stochastically dominates $X$ on $A$.  

For a given $n$-zigzag $\phi$, let the {\em exterior} ${\rm Ext}(\phi)$ of $\phi$ denote  $\big( \R \times n^{-1} \Z \big) \setminus  \phi$.  {Recall from Sections \ref{s.brlpp} and \ref{s:scaledcoordinates} that our scaled noise environment is given by an ensemble of independent  two-sided Brownian motions, thought of as a function $\R \times \frac{1}{n}\Z \to \R$.}
\begin{lemma}\label{FKG}Given a zigzag  $\phi$, and two independent noise environments $\Omega$ and $\tilde \Omega$, the restriction of $\widetilde \Omega$ to  ${\rm Ext}(\rho_n)$ stochastically dominates $\Omega$ on this set.
\end{lemma}
{\bf Proof.}
Consider the noise environment that is given by $\Omega$ on $\rho_n$ and by~$\widetilde\Omega$ on ${\rm Ext}(\rho_n)$. When this environment is conditioned on the event that there exists no $n$-zigzag from $(0,0)$ to $(0,1)$ whose weight determined by this environment exceeds that of $\rho_n$, the result is a distributional copy of $\Omega$.
 The event in the conditioning is negative for $\Omega$ on ${\rm Ext}(\rho_n)$. The system~$\Omega$  on ${\rm Ext}(\rho_n)$ is a countable collection of Brownian motions whose domains are either copies of the real line or semi-infinite real intervals; indeed, to each height in $y \in n^{-1}\Z$ are associated one or two intervals, formed by the sometimes vacuous removal from $\R \times \{y \}$ of this set's intersection with~$\rho_n$. The FKG inequality for products of independent Brownian motions is implied by \cite[Theorems~$3$ and~$4$]{Barbato}. Applying it, we obtain the lemma.  \qed

The next lemma says that the polymer is typically regular.
\begin{lemma}Given $C>0,$ there exists $c>0$ such that, for all large $n$, and for $\theta$ with $C\log n< \tza^{-1/4} <C n^{\frac{1}{10}},$ with probability at least $1-\exp(-c \theta^{-1/2})$, the polymer $\rho_n$ is $\theta^{-1/4}$-{\em regular}.
\end{lemma}
\begin{proof}
By Theorem \ref{t.toolfluc} and a union bound over all $h_1, h_2 \in \frac{1}{n}\Z \cap [0,1]$, we get 
\begin{align*}
\P\big( \rho_n \text{ is } \theta^{-1/4}\text{-regular} \big) &\geq 1 - \exp \big\{ -c\tza^{-3/4}/\log n) \big\}\\
&\ge 1 - \exp \big\{ -c\tza^{-1/2}\big\},
\end{align*} 
using the upper bound on $\theta.$
\end{proof}

Note that $\theta$ in the hypothesis of Theorem \ref{t.slenderpolymer} satisfies $C\log n< \tza^{-1/4} < n^{\frac{1}{160}}$ since $2^{\ell}\ge 1.$
Theorem \ref{t.slenderpolymer}  now follows  from Theorem~\ref{t.nocloseness}; the just noted bound; Lemma \ref{FKG}; the event   $\mathsf{LowSlenderExcursion}(\ell, \tza , 1 - \chi;  \rho_n)$ being decreasing on  ${\rm Ext}(\rho_n)$; and that $\mathsf{LowSlenderExcursion}(\ell, \tza , 1 - \chi;  \rho_n)$ implies $\mathsf{LowSlenderExcursion}^*(\ell, \tza , 1 - \chi;  \rho_n).$

We finish with a brief discussion regarding the point that Theorem \ref{t.nocloseness} bounds the probability of $\mathsf{LowSlenderExcursion}^*(\ell, \tza , 1 - \chi;  \phi)$ for a fixed zigzag $\phi$ while Theorem \ref{t.slenderpolymer} only bounds the probability of $\mathsf{LowSlenderExcursion}(\ell, \tza , 1 - \chi;  \rho_n)$. In short, this is because $\phi$ is deterministic and hence independent of the noise environment, while $\rho_n$ is highly correlated with the latter. Namely,  notice that the proof of Theorem  \ref{t.slenderpolymer}  uses Theorem \ref{t.nocloseness}, along with an FKG inequality; the latter implies that the noise environment off $\rho_n$ is stochastically smaller than a typical environment, this rendering $\mathsf{LowSlenderExcursion}(\ell, \tza , 1 - \chi;  \rho_n)$ more likely. However, the same cannot be said for $\mathsf{LowSlenderExcursion}^*(\ell, \tza , 1 - \chi;  \rho_n)$, since the environment on $\rho_n$ is, in fact,  stochastically larger than a typical one---indeed, it is easily seen that the path $\rho_n$ itself obstructs the event $\mathsf{LowSlenderExcursion}^*(\ell, \tza , 1 - \chi;  \rho_n)$ from occurring.

\qed
\section{There are few cliffs along the geodesic}\label{s.fewcliffs}

Here we derive Theorem~\ref{t.notallcliffs}. In a first subsection, we reduce to a principal component, Proposition~\ref{p.gammapsibound}; and, in a second, we prove this proposition. Theorem~\ref{t.notallcliffs} will find application in the investigation of Brownian LPP under dynamical perturbation in \cite{Dynamics}. This study is undertaken in scaled coordinates, and uses a scaled counterpart to Theorem~\ref{t.notallcliffs}. In the third and final subsection, the counterpart, Proposition~\ref{p.notallcliffs}, is presented and proved.  

\subsection{Proving Theorem~\ref{t.notallcliffs}, a main component admitted}

Let $\gamma \subset [0,n]^2$ denote any staircase between $(0,0)$ and $(n,n)$. We may associate to $\gamma$ the index set $\mc{I}(\gamma)$ (that is specified before the theorem), just as we did to the geodesic staircase $\Gamma_n$.

For now, let $\alpha$ be any given value in $(1/2,1)$ for which $\alpha m \in \N$; the lower bound $\alpha_0 > 1/2$ on this parameter's value will be set later in the proof of Theorem~\ref{t.notallcliffs}. Consider the class $\Theta$ of difference functions $\Psi: \llbracket 0,m \rrbracket \to \llbracket 0, n \rrbracket$ that are associated to staircases $\gamma \subset [0,n]^2$ with $(0,0),(n,n) \in \gamma$
and $\big\vert \mc{I}(\gamma) \big\vert \geq \alpha m$. 

Let $\Gamma(\Psi) \subset [0,n]^2$ denote the staircase of maximum energy that contains $(0,0)$ and $(n,n)$ and whose $Z$-difference function as specified by~(\ref{e.zdifferencefunction})
is equal to $\Psi$.

Our approach to proving Theorem~\ref{t.notallcliffs} is governed by the bound 
 $$
 \PP \big( \big\vert \mc{I} \big\vert \geq \alpha m \big) \leq \sum_{\Psi \in \Theta} 
 \PP \Big( E \big( \Gamma(\Psi) \big) \geq E(\Gamma_n) \Big)  \, .
 $$
 This right-hand side is in fact equal to  $\sum_{\Psi \in \Theta} 
 \PP \Big( E \big( \Gamma(\Psi) \big)=E(\Gamma_n) \Big).$
 
 The next two results form the backbone of the proof of Theorem~\ref{t.notallcliffs}.
\begin{proposition}\label{p.gammapsibound} There exist positive constants $H$ and $h$ such that, 
for $A$ high enough and $\Psi \in \Theta$, 
\begin{equation}\label{e.gammapsibound}
 \PP \Big( E \big( \Gamma(\Psi) \big) \geq E(\Gamma_n) \Big) \leq  H \exp \big\{-hn \big\} \, .
\end{equation}
\end{proposition}
The quantity $\Theta$ grows at a rate that is exponential in $n/A$
 when $\alpha$ is close to one.
\begin{lemma}\label{l.thetaentropy}
\begin{equation}\label{e.thetaentropy}
 \big\vert \Theta  \big\vert \, \leq \, 3^{\alpha m}{m+1\choose{\alpha m}} { n+1 \choose{(1-\alpha)m + 1}} \, .
 \end{equation}
\end{lemma}
We close out the proof of Theorem~\ref{t.notallcliffs} before deriving these two inputs.

{\bf Proof of Theorem~\ref{t.notallcliffs}.}
By Proposition~\ref{p.gammapsibound} and Lemma~\ref{l.thetaentropy},  we see that 
 $$
 \PP \big( \big\vert \mc{I} \big\vert \geq \alpha m \big)  
 \, \leq \, 3^{\alpha m}{m+1\choose{\alpha m}} { n+1 \choose{(1-\alpha)m + 1}} \cdot  H \exp \big\{-hn \big\} \, .
 $$

  Since $m \in [n/A,n/A+1)$, the right-hand factor arising from Lemma~\ref{l.thetaentropy} is at most 
 $$
  3^{\alpha (A^{-1} n + 1)} \cdot (n/A + 1) \cdot \exp \big\{ ( n A^{-1} + 1 )K(\alpha) \big\} \cdot {(n+1)} \exp \big\{ (n+1)  K\big((1-\alpha) A^{-1}\big) \big\} \, ,
 $$ 
where $K:(0,1) \to \R$ denotes the entropy rate $K(p) = p \log p + (1-p) \log (1-p)$.
Our choice of $A \in \N$ made so that~(\ref{e.gammapsibound}) holds as well as $3^{\alpha (A^{-1} n + 1)}\le e^{hn/4}$, we specify $\alpha_0 \in (1/2,1)$ to be high enough that the last display with any $\alpha \in [\alpha_0,1]$ is less than~$e^{hn/2}$ when $n$ is supposed to be sufficiently high. 
We obtain Theorem~\ref{t.notallcliffs} by further relabelling the parameter $h > 0$ to be one-half of its present value. \qed

{\bf Proof of Lemma~\ref{l.thetaentropy}.} For any staircase $\gamma$ that offers a function $\Psi$ belonging to $\Theta$, let $J(\gamma)$ denote the set of the~$\alpha m$ lowest elements of $\mc{I}(\gamma)$. Further set  $J^c(\gamma)  = \llbracket 0, m \rrbracket \setminus J(\gamma)$. The element $\Psi$ of $\Theta$
associated to $\gamma$ may be surmised from three pieces of data:
\begin{itemize}
\item the set $J(\gamma)$;
\item the values $\Psi(j)$ indexed by $j \in J(\gamma)$; 
\item and the remaining values, namely $\Psi(i)$ for $i \in J^c(\gamma)$.
\end{itemize}
The number of choices for $J(\gamma)$ is equal to ${m+1\choose{\alpha m}}$. 
For each index $j \in J(\gamma)$, $\Psi(j)$ is valued in $\{ 0,1,2\}$, so that there are at most $3^{\alpha m}$ choices for the second piece of data. 
The remaining values, in the third piece of data, are indexed by $i \in J^c(\gamma)$; to each such index~$i$ is associated the partial sum $p(i)$ of $\Psi(j)$ over $j \in J^c(\gamma)$ with  $j \leq i$. The index set $J^c(\gamma)$ has cardinality $m+1 - \alpha m$ and may be identified with the integer interval  $\llbracket 0,(1-\alpha)m \rrbracket$ via an increasing  map $I:\llbracket 0,(1-\alpha)m \rrbracket \to J^c(\gamma)$. The third item data is specified by the function mapping $\llbracket 0,(1-\alpha)m \rrbracket$  to $\llbracket 0,n \rrbracket$ given by $i \to (p \circ I)(i)$. This function is increasing, so that it is determined by its values; thus, the number of such functions is at most 
${n+1 \choose (1-\alpha) m + 1}$.

The right-hand side of~(\ref{e.thetaentropy}) is a product of three factors. These factors have been verified to offer respective upper bounds on the cardinality of the set of choices for the second, first and third pieces of displayed data. Thus, the proof of Lemma~\ref{l.thetaentropy} is complete.  \qed

 \subsection{Energy near a given cliff-strewn route is unlikely to attain the maximum}
 
To complete the proof of Theorem~\ref{t.notallcliffs}, it remains to give the next derivation.

{\bf Proof of Proposition~\ref{p.gammapsibound}.} Let
 $\Psi \in \Theta$. Let $P$ denote the set of points of the form  $\big(\sum_{j=0}^i \Psi(j),iA\big)$ for $i \in \llbracket 0,m\rrbracket$. 
 Let $\Gamma_0(\Psi)$ denote the almost surely unique staircase between $(0,0)$ and $(n,n)$ that has maximum energy among those that visit every element in $P$ and that, on arrival at any such element, immediately jump upwards by one unit. The staircase $\Gamma_0(\Psi)$ offers a coarse-grained description of any staircase specifying $\Psi$, including the staircase $\Gamma(\Psi)$ among these of maximum energy.  
 
 The plan of attack for proving Proposition~\ref{p.gammapsibound} has three steps:
 \begin{enumerate}
 \item We wish to argue that the energies of $\Gamma(\Psi)$ and its coarse-grained cousin $\Gamma_0(\Psi)$ are typically similar. 
 Indeed, we will find positive constants $H$ and $h$ such that, for $r \geq 0$,
\begin{equation}\label{e.gammazerogamma}
 \PP \Big(  E\big(\Gamma(\Psi)\big) - E\big( \Gamma_0(\Psi) \big)
  \geq (5 + r) n A^{-1/2} 
  \Big) \leq H \exp \big\{ - h n r^2 \big\} \, .
\end{equation}
 \item We will then need to analyse $E \big( \Gamma_0(\Psi) \big)$. This is the maximum energy of a staircase from $(0,0)$ to $(n,n)$ that visits every point $\big(\sum_{j=0}^i \Psi(j),iA\big)$ for $i \in \llbracket 0,m\rrbracket$. We will argue that this maximum energy is unchanged in law if the vector of differences between consecutively visited points is reordered so that these vectors are presented in decreasing order of gradient. Because $\Psi \in \Theta$, the reordered collection of points-to-be-visited contains an element whose distance from the diagonal has order $n$.
 \item Thus, $E \big( \Gamma_0(\Psi) \big)$ has the law of the maximum energy of a staircase from $(0,0)$ and $(n,n)$ that visits a given point at a distance of order $n$ from the diagonal. We will exploit this information to argue that there exist positive constants $K$ and $\kappa$ such that
\begin{equation}\label{e.gammazeropsi}
 \PP \Big(  E\big(\Gamma_0(\Psi)\big) \geq (2 - \kappa) n \Big) \leq K \exp \big\{ - \kappa n \big\} \, .
\end{equation}
The sought bound~(\ref{e.gammapsibound}) will then emerge directly from~(\ref{e.gammazerogamma}) and~(\ref{e.gammazeropsi}).
 \end{enumerate}
In three subsections, we accomplish these respective steps. 
 
\subsubsection{The coarse-grained energetic approximation: deriving~(\ref{e.gammazerogamma})} 
To derive~(\ref{e.gammazerogamma}), we split the staircases  $\Gamma(\Psi)$ and  $\Gamma_0(\Psi)$ into pieces that traverse consecutive strips of height $A$. The staircase $\Gamma(\Psi)$
is divided into pieces by splittling at its points of entry to the levels $\{iA \} \times \R$  indexed by $i \in \llbracket 0,m\rrbracket$. The coarse-grained counterpart $\Gamma_0(\Psi)$
is partitioned by splitting at the points $\big(\sum_{j=0}^i \Psi(j),iA\big)$ indexed by the same set. The elements of the two partitions may be paired according to which strip of height $A$ they cross. If the difference in energy between the fragment of $\Gamma(\Psi)$ crossing the $i\textsuperscript{th}$ strip and its counterpart for $\Gamma_0(\Psi)$ is denoted by $E_i$, then $E\big(\Gamma(\Psi)\big) - E\big( \Gamma_0(\Psi) \big) = \sum_{i=1}^{\lceil n/A \rceil} E_i$. 
The pair of fragments involved in specifying  $E_i$ each begin with a unit vertical movement, from level $(i-1)A$ to level $(i-1)A + 1$. 
With $h_1 = (i-1)A + 1$ and $h_2 = iA$, note thus that, for $i \in \intint{\lfloor n/A \rfloor}$, $E_i$ takes the form  
$M \big[ (u,h_1) \to (v,h_2) \big] - M \big[ (x,h_1) \to (y,h_2) \big]$ for a choice of $(u,v,x,y)$ that satisfy the hypotheses of the next result.
{ \begin{lemma}\label{l.adjust}
 Let $u,v \in \N$ and $x,y \in \R$ satisfy $u \leq v$, $x \leq y$, $x \in [u,u+1]$ and $y \in [v,v+1]$. Let $h_1,h_1 \in \N$ satisfy $h_1 \leq h_2$. Writing $h_{1,2} = h_2 - h_1$, we have that
\begin{equation}\label{coarseapproxerror}
      \PP \Bigg(     \sup_{\begin{subarray}{c}    x \in [u,u+1]  \\  y \in [v,v+1]  \end{subarray}}   \Big\vert M \big[ (u,h_1) \to (v,h_2) \big] - M \big[ (x,h_1) \to (y,h_2) \big] \Big\vert \geq  4 \big(h_{1,2}+1\big)^{1/2} + \big(h_{1,2}+1\big)^{-1/6}r \Bigg) 
 \end{equation}
is at most $C \exp \Big\{ - c\frac{r^2}{(h_{1,2}+1)^{1/3}} \Big\}$. 
 \end{lemma}} 
 The tail of the remainder term indexed by $i = \lceil n/A \rceil$
is also treated by Lemma~\ref{l.adjust} with $h_{1,2}$ assuming a value in $\llbracket 0, A-1 \rrbracket$.  We derive \eqref{e.gammazerogamma} and then prove this lemma.
 
{\bf {Proof of \eqref{e.gammazerogamma}}. }
We rely on the bound 
\begin{align*}E\big(\Gamma(\Psi)\big) - E\big( \Gamma_0(\Psi) \big) &\le \sum_{i=1}^{\lfloor n/A \rfloor} E_i +R\le (R-4 A^{1/2})+4 A^{1/2}+\sum_{i=1}^{\lfloor n/A \rfloor}\big[4 A^{1/2}+(E_i-4 A^{1/2})_+\big]\\
&\le (4 n+A)A^{-1/2}+(R-4 A^{1/2})_+\sum_{i=1}^{\lfloor n/A \rfloor}(E_i-4 A^{1/2})_+,
\end{align*} where $R$ is the remainder term (and $a = \max (a,0)$). 
By Lemma \ref{l.adjust}, the random variables $X_i=(E_i-4 A^{1/2})_+$ are independent and satisfy the uniform tail bound $\P(X_i\ge r)\le Ce^{-cr^2}$. This tail bound is also satisfied by $\hat R=(R-4 A^{1/2})_+$.
By Proposition \ref{p.Vershynin}(2), we thus see that, for $t\ge c \lceil n/A \rceil,$ 
$$
\P\bigg(\sum_{i=1}^{\lceil n/A \rceil}X_{i}+\hat R \ge \sum_{i=1}^{\lceil n/A \rceil}\E(X_{i}) +t\bigg)\le e^{-Ct^2A n^{-1}} \, .
$$

Observing that $\E(X_i)=O(1),$  we see that, for a large enough $A$, and  for all large $n$,
$$\P\Big(E\big(\Gamma(\Psi)\big) - E\big( \Gamma_0(\Psi) \big) \ge (5 + r) n A^{-1/2} 
  \Big) \le e^{-h n r^2},$$ for some constant $h>0$.
  \qed
  
{\bf Proof of Lemma~\ref{l.adjust}.} The argument is simple and relies on using superadditivity to bound the expression on the left-hand side of \eqref{coarseapproxerror} by a linear combination of passage times between deterministic points, which is then easy to bound using the well-known correspondence between passage time in Brownian LPP between fixed points and the largest eigenvalue of a matrix drawn from a suitable Gaussian unitary ensemble.

We now carry out the first part of the above strategy. 
Suppose first that $v \geq u+1$. Note that
\begin{eqnarray*}
& &  M \big[ (x,h_1) \to (u+1,h_1) \big] + M \big[ (u+1,h_1) \to (v,h_2) \big] + M \big[ (v,h_2) \to (y,h_2) \big] \\
& \leq &
M \big[ (x,h_1) \to (y,h_2) \big] \\
& \leq &  M \big[ (x,h_1) \to (u+1,h_2) \big] + M \big[ (u+1,h_1) \to (v,h_2) \big] + M \big[ (v,h_1) \to (y,h_2) \big] \, .
\end{eqnarray*}
These bounds hold also when the replacements $x \to u$ and $y \to v$ are made. Thus,
\begin{eqnarray}
& & \Big\vert M \big[ (x,h_1) \to (y,h_2) \big] - 
M \big[ (u,h_1) \to (v,h_2) \big] \Big\vert \nonumber \\
& \leq & M \big[ (x,h_1) \to (u+1,h_2) \big] - M \big[ (x,h_1) \to (u+1,h_1) \big] \nonumber 
\\
& & \qquad  + \, M \big[ (v,h_1) \to (y,h_2) \big] -  M \big[ (v,h_2) \to (y,h_2) \big] \nonumber \\
& \leq & M \big[ (u,h_1) \to (u+1,h_2) \big] - M \big[ (u,h_1) \to (u+1,h_1) \big] \label{e.rightfour} 
\\
& & \qquad  + \, M \big[ (v,h_1) \to (v+1,h_2) \big] -  M \big[ (v,h_2) \to (v+1,h_2) \big] \, . \nonumber
\end{eqnarray}
Here, the latter inequality depended on 
$$
M \big[ (u,h_1) \to (u+1,h_2) \big] \geq   M \big[ (u,h_1) \to (x,h_1) \big] + M \big[ (x,h_1) \to (u+1,h_2) \big]
$$ 
and  
$$
M \big[ (u,h_1) \to (u+1,h_1) \big] = M \big[ (u,h_1) \to (x,h_1) \big] + M \big[ (x,h_1) \to (u+1,h_1) \big]
$$
as well as  
$$
M \big[ (v,h_1) \to (v+1,h_2) \big] \geq  M \big[ (v,h_1) \to (y,h_2) \big] + M \big[ (y,h_2) \to (v+1,h_2) \big]
$$ 
and 
$$
M \big[ (v,h_2) \to (v+1,h_2) \big] = M \big[ (v,h_2) \to (y,h_2) \big] + M \big[ (y,h_2) \to (v+1,h_2) \big] \, .
$$  
Note that the right-hand expression in~(\ref{e.rightfour}) does not depend on $x$ or $y$, so that the upper bound on the left-hand term is valid when the supremum over $x \in [u,u+1]$ and $y \in [v,v+1]$ is taken. For $n \in \N$ and $s \geq 0$, let $G_n(s)$ denote the uppermost eigenvalue of an $n \times n$ random matrix drawn from the Gaussian unitary ensemble with entry variance $s$. 
Viewing the right-hand expression in~(\ref{e.rightfour})
{as a sum of four terms, we see that the first and third are independent and have the law of $G_{h_{12}+1}(1),$}
and the second and fourth are independent and have the law of $G_1(1)$. The latter two clearly have Gaussian tails.
We also use the following tail estimate of $G_{h_{12}+1}(1)$, which is the content of \cite[(5),(6)]{Aubrun}.
For all $r\ge 0,$
\begin{equation}\label{aubrunbound1}
\P\Big(G_{h_{12}+1}(1)- 2(h_{12}+1)^{1/2}\ge \frac{r}{(h_{12}+1)^{1/6}}\Big)\le C \exp \bigg\{ - c \max \Big(r^{3/2}, \frac{r^2}{(h_{1,2}+1)^{1/3}} \Big) \bigg\}.
\end{equation}

The lemma in the case that $v \geq u+1$ now follows by a simple union bound over the possibilities that one of the four quantities is bigger than $\frac{r}{4(h_{12}+1)^{1/6}}.$

Suppose then that $v < u+1$. Since $u,v \in \N$ and $v \geq u$, we have $v =u$.
Thus, $M \big[ (u,h_1) \to (v,h_2) \big] = 0$. 
Note that 
\begin{equation}\label{e.fourm}
 M \big[ (u,h_1) \to (u+1,h_2) \big] \geq  M \big[ (u,h_1) \to (x,h_1) \big] +  M \big[ (x,h_1) \to (y,h_2) \big] +  M \big[ (y,h_2) \to (u+1,h_2) \big] 
\end{equation}
and that
\begin{equation}\label{e.mtriple}
 M \big[ (u,h_1) \to (u+1,h_2) \big] \leq  M \big[ (u,h_1) \to (x,h_2) \big] +  M \big[ (x,h_1) \to (y,h_2) \big] +  M \big[ (y,h_1) \to (u+1,h_2) \big] \, .
\end{equation}
Hence, by rearranging, we obtain 
$$
M \big[ (x,h_1) \to (y,h_2) \big] \geq  M \big[ (u,h_1) \to (u+1,h_2) \big] - M \big[ (u,h_1) \to (x,h_2) \big] -  M \big[ (y,h_1) \to (u+1,h_2) \big] \, . 
$$
Note further that  
$$
M \big[ (u,h_1) \to (x,h_2) \big] +  M \big[ (x,h_2) \to (u+1,h_2) \big] \leq  M \big[ (u,h_1) \to (u+1,h_2) \big] \, ,
$$
and that 
$$
  M \big[ (u,h_1) \to (y,h_1) \big] 
 + M \big[ (y,h_1) \to (u+1,h_2) \big] \leq  M \big[ (u,h_1) \to (u+1,h_2) \big] \, .  
$$ 
Applying the two preceding bounds, we find that
$$
  M \big[ (x,h_1) \to (y,h_2) \big] \geq 
  - M \big[ (u,h_1) \to (u+1,h_2) \big] +  M \big[ (x,h_2) \to (u+1,h_2) \big]
 +  M \big[ (u,h_1) \to (y,h_1) \big]  \, .
$$
From this bound, and~(\ref{e.fourm}), we see that
\begin{eqnarray*}
 & &  - M \big[ (u,h_1) \to (u+1,h_2) \big] +  M \big[ (x,h_2) \to (u+1,h_2) \big]
 +  M \big[ (u,h_1) \to (y,h_1) \big] \\
 & \leq & 
 M \big[ (x,h_1) \to (y,h_2) \big] \\
& \leq &   
 M \big[ (u,h_1) \to (u+1,h_2) \big] -  M \big[ (u,h_1) \to (x,h_1) \big] -  M \big[ (y,h_2) \to (u+1,h_2) \big] \, .
\end{eqnarray*}
If the first line takes the form $A_1 + A_2 + A_3$,
then $A_1$ has the law of $G_{h_{1,2} + 1}(1)$, while the suprema of $\vert A_2 \vert$ 
and $\vert A_3 \vert$ over  $x \in [u,u+1]$ and $y \in [v,v+1]$ are stochastically dominated by the supremum of standard Brownian motion on the interval $[0,1]$. This statement is equally true of the third line.
Since $M \big[ (u,h_1) \to (v,h_2) \big] = 0$, Lemma~\ref{l.adjust} when $v = u$ follows from 
\eqref{aubrunbound1} 
and an upper tail bound on the supremum of Brownian motion obtained from the reflection principle and a standard bound on the Gaussian tail. This completes the proof of Lemma~\ref{l.adjust}. \qed
\subsubsection{Reordering the trajectory of the coarse-grained cousin}  We now analyse $E \big( \Gamma_0(\Psi) \big)$.
We first argue that the path $\Gamma_0(\Psi)$ may be reordered in order to visit a point far away from the diagonal
without a change to the law of its energy. 
We will present notation and a general form Lemma~\ref{l.rearrange} for such a rearrangement result before discussing the energy $E\big(\Gamma_0(\Psi)\big)$.
 
 Let $P$ denote a collection of points in $\llbracket 0 , n \rrbracket^2$
 with distinct $y$-coordinates, which lie in the range of a staircase that begins at $(0,0)$ and ends at $(n,n)$.
 We write $M_n[P]$ for the supremum of the energies of staircases that begin at $(0,0)$; end at $(n,n)$; that visit every element of $P$; and that, on visiting any element  $(m_1,m_2)$ of $P$ for which $m_2 < n$, immediately jump upwards by one step, to $(m_1,m_2 +1)$.
 When $P$ is a singleton set whose element is $(m_1,m_2)$, we abuse notation, and denote  $M_n[P]$ by $M_n[m_1,m_2]$.

Suppose that $(u,0) \in P$ for some $u \in \llbracket 0,n \rrbracket$.
Thus, the staircases involved in specifying $M_n[P]$ begin at 
$(0,0)$; remain on the $x$-axis until $(u,0)$; 
and end at $(n,n)$.  
 
To the collection $P$, we may add the point
$(n,n)$. The resulting set of points may be ordered so that each successive element lies strictly upwards, and to the right, of its predecessor. To each consecutive pair in this sequence, we may associate the rectangle whose lower-left corner is the former element in the pair and whose upper-right corner is the latter. Any staircase from $(u,0)$ to $(n,n)$ that visits every element of $P$ crosses all of these rectangles, passing out of the upper-right corner  of one into the lower-left corner of the next. 

The rectangles may be placed in increasing order of width, and translated so that the lower-left corner of the lowest rectangle equals $(0,0)$, and the upper-right corner of one rectangle is the lower-left corner of the next one. Note that the upper-right corner of the last rectangle is $(n-u,n)$. 

Let $P^{\rightarrow}$ denote the collection of upper-right corners of the rectangles when so placed.
 \begin{lemma}\label{l.rearrange}
Let $u \in \llbracket 0,n \rrbracket$. Let $P$ denote a collection of points in $\llbracket 0 , n \rrbracket^2$ which satisfies the condition in the second paragraph of this section and which contains $(u,0)$. 
Then $M_n[P]$ and $M_n[P^{\rightarrow}]$ are equal in law.
 \end{lemma}
 {\bf Proof.} The quantity $M_n[P]$ equals $A_1 + A_2$, where $A_1 = B(0,u) - B(0,0)$ 
 is the energy accrued along $(0,0) \to (u,0)$ and where $A_2$ is the sum over the rectangles associated to $P$ of the maximum energy available in a staircase that crosses the rectangle from its lower-left to its upper-right corner. The quantity $M_n[P^{\rightarrow}]$ equals $A_1 + B_2$, where $B_2$ is the sum, counterpart to~$A_2$, over rectangles associated to $P^{\rightarrow}$. The two collections of rectangles are disjoint, except for endpoint intersections, and one is a rearrangement of the other given by translations applied to the elements. So $A_2$ and $B_2$ are equal in law, conditionally on the value of $A_1$. This proves the lemma. \qed

We now employ Lemma~\ref{l.rearrange} to find an upper bound on the energy $E\big(\Gamma_0(\Psi)\big)$ of the coarse-grained staircase. 
In the notation of Lemma~\ref{l.rearrange}, this energy takes the form 
 \begin{equation}\label{e.energyform}
E\big(\Gamma_0(\Psi)\big) = 
M_n[P] \, ,
 \end{equation}
where $P = \big\{ \big(\sum_{j=0}^i \Psi(j),iA\big): i \in \llbracket 0,m\rrbracket\big\} \cup \big\{ (n,n) \big\}$.
By this lemma, $M_n[P]$ and $M_n[P^{\rightarrow}]$ are equal in law. Note that if $\alpha\ge 1/2$ then $P^{\rightarrow}$ contains a point of the form $\big(\kappa m/2, A m/2\big)$ where $\kappa \leq 2$ is determined by the given element $\Psi \in \Theta$ that we are considering. Thus, 
\begin{equation}\label{offdiagfluc}
M_n[P^{\rightarrow}] \leq M_n\big[ \big(\kappa m/2, A m/2\big) \big].
\end{equation}

\subsubsection{The energetic penalty for highly off-diagonal travel: deriving~(\ref{e.gammazerogamma})}
In this third step, we present a tool indicating how $M_n\big[ (\kappa \alpha m, A \alpha m) \big]$ typically falls far below the typical energy maximum~$2n$ for the route $(0,0) \to (n,n)$. We will then promptly be able to obtain~(\ref{e.gammazeropsi}). Finally,
from~(\ref{e.gammazerogamma}) and~(\ref{e.gammazeropsi}), we will obtain~(\ref{e.gammapsibound}), the derivation of which is the final step in proving Theorem~\ref{t.notallcliffs}.
 \begin{lemma}\label{l.offdiagonal}
Let $\mu \in (0,1)$. Let $m_1,m_2 \in \llbracket 0,n \rrbracket$ satisfy  $m_2 \in \big[\mu n , (1 - \mu)n\big]$ and $\big\vert m_2 - m_1 \big\vert \geq \mu n$. Then 
there exist positive $\mu$-dependent $K$ and $\kappa$ such that
$$
 \PP \big( M_n[m_1,m_2] \geq (2 - \kappa)n  \big) \leq K e^{-\kappa n} \, . 
$$
 \end{lemma}
 {\bf Proof.} Note that $M_n[m_1,m_2] = M_n \big[ (0,0) \to (m_1,m_2) \big] + M_n \big[ (m_1,m_2) \to (n,n) \big]$
 is a sum of two independent terms having the respective distributions $m_1^{1/2} G_{m_2 + 1}(1)$ and  $(n - m_1)^{1/2}G_{n - m_2 +1}(1)$.
 In an expression for this sum arising by writing the two $G$ terms as a sum of a leading order term and a random fluctuation, the deterministic part is 
 $$
 2 \big( m_1^{1/2} (m_2 + 1)^{1/2} + (n - m_1)^{1/2}(n - m_2 +1)^{1/2} \big) \, . 
 $$
 This term falls short of $2n$ by a quantity that grows linearly in $n$ under the hypotheses of the lemma. The fluctuation term in the sum equals $m_1^{1/2} (m_2 + 1)^{-1/6}R_1  + (n - m_1)^{1/2}(n - m_2 +1)^{-1/6} R_2$ where the $R$-terms are independent random variables that satisfy $\PP \big( R \geq r \big) \leq C e^{-cr^{3/2}}$ for $r \geq 0$. 
 By our hypothesis on $m_2$, the fluctuation term is at most a constant multiple of $n^{1/3} (R_1+R_2)$. Thus the fluctuation term exceeds one-half of the shortfall of the leading term with probability at most  $K e^{-\kappa n}$ for suitable positive $K$ and $\kappa$. Decreasing the value of $\kappa > 0$ completes the proof of Lemma~\ref{l.offdiagonal}. \qed

From Lemma~\ref{l.offdiagonal}, we learn that for $\kappa, A$ as in \eqref{offdiagfluc},
$$
 \PP \Big(  M_n\big[ (\kappa m/2, A m/2 ) \big] \geq (2 - \kappa) n \Big) \leq K \exp \big\{ - \kappa n \big\}
$$
for suitable positive $K$ and $\kappa$ not depending on $A$. Thus, we confirm~(\ref{e.gammazeropsi}) via~(\ref{e.energyform}).

From~(\ref{e.gammazerogamma}) and~(\ref{e.gammazeropsi}), we find that
$$
 \PP \Big( E\big( \Gamma(\Psi) \big) \geq   \big( (2 - \kappa)  + (5 + r)  A^{-1/2} \big) n  \Big) \leq 
  H \exp \big\{ - h n r^2 \big\}  +  K \exp \big\{ - \kappa n \big\} \, .
$$
Setting $r = 1$ and choosing $A$ high enough that $6 A^{-1/2} \leq \kappa/2$, we learn that
$$
 \PP \Big( E\big( \Gamma(\Psi) \big) \geq    (2 - \kappa/2)  n  \Big) \leq 
  H \exp \big\{ - h n \big\} \, ,
$$
where we have relabelled the positive constants $H$ and $h$.

Using the bound Lemma~\ref{l.onepointbounds}(2) on the lower tail of the uppermost GUE eigenvalue, we have  
$$
 \PP \Big( E(\Gamma_n) \leq 2n - x n^{1/3} \Big) \leq C \exp \big\{ - c x^{3/2} \big\}  
$$
for $x \geq 0$. Setting $x$ equal to $(c \wedge \kappa/4)n^{2/3}$,
we confirm from the two preceding displays that~(\ref{e.gammapsibound}) holds, 
after suitable adjustment to the positive values of $H$ and $h$. This bound derived, the proof of Proposition~\ref{p.gammapsibound} is complete, and, with it, the derivation of the elements needed for Theorem~\ref{t.notallcliffs}. \qed

\subsection{Few cliffs in scaled coordinates}

In \cite{Dynamics}, use is made of a scaled counterpart to Theorem~\ref{t.notallcliffs}. We finish by presenting the notation needed to express this result; by stating it as Proposition~\ref{p.notallcliffs}; and by deriving it from Theorem~\ref{t.notallcliffs}.

We start by recalling some notation: $\rho_n$ denotes the polymer $\rho_n \big[ (0,0) \to (0,1) \big]$. 
For $i \in \llbracket 0, n \rrbracket$, $\rho_n(i/n)$ equals  the supremum of the set $\big\{ x \in \R : (x,i/n) \in \rho_n \big\}$. The sequence $\big\{ \rho_n(i/n): i \in \llbracket 0, n \rrbracket \big\}$
records the horizontal coordinates of departures of the polymer $\rho_n$ from the consecutive horizontal intervals that it traverses. Indeed, the projections to~$\R$ of the horizontal intervals of $\rho_n$ take the form $[0,\rho_n(0)]$ and 
$\big[ \rho_n\big((i-1)/n\big) - 2^{-1}n^{-2/3} , \rho_n(i/n) \big]$ for $i \in \intint{n}$. Thus, by writing $\omega_0 = \rho_n(0)$ and $\omega_i = \rho_n(i/n) - \rho_n \big((i-1)/n\big) + 2^{-1}n^{-2/3}$ for $i \in \intint{n}$, the lengths of the consecutive horizontal intervals of $\rho_n$ are recorded in the sequence $\big\{ \omega_i: i \in \llbracket 0 , n \rrbracket \big\}$. 
The unscaled preimage $R_n^{-1}(\rho_n)$ of $\rho_n$ has endpoints with horizontal coordinates zero and $n$, so 
the form~(\ref{e.scalingmap}) of the scaling map $R_n:\R^2 \to \R^2$  implies that $\sum_{i=0}^n \omega_i = 2^{-1}n^{1/3}$.

Let $\beta_1 > 0$ and $\beta_2 \in (0,1)$. The polymer $\rho_n$ is said to advance horizontally with $(\beta_1,\beta_2)$-steadiness if the cardinality of the set of $i \in \llbracket 0,n \rrbracket$
for which $\omega_i \geq \beta_1 n^{-2/3}$ is at least $\beta_2 n$.

\begin{proposition}\label{p.notallcliffs}
There exist  $\beta_1 > 0$, $\beta_2 \in (0,1)$, $h > 0$ and $n_0 \in \N$ such that, for $n \geq n_0$, the probability that $\rho_n$ fails to  advance horizontally with $(\beta_1,\beta_2)$-steadiness is at most $e^{-hn}$.
\end{proposition}
{\bf Proof.}
For $i \in \llbracket 0, n \rrbracket$, let $W_i$ denote the length of the horizontal interval at which $\Gamma_n$ intersects $\R \times \{ i \}$. The geodesic $\Gamma_n$ maps to the polymer $\rho_n$ under the scaling map $R_n$ from Subsection~\ref{s.scalingmap}. Recall that $\omega_i$ denotes the length of the horizontal interval at which $\rho_n$ intersects $\R \times \{ i/n \}$. The form of the scaling map thus dictates that $\omega_i = 2^{-1}n^{-2/3} W_i$.

Suppose that, for some $i \in \llbracket 0, m-1 \rrbracket$, $\Psi(i) \ge 2$.
The sum $\sum_{j= iA}^{(i+1)A - 1} W_j$ is readily seen to be at least two. Thus, there is at least one $j \in \llbracket iA , (i+1)A - 1 \rrbracket$ for which $W_j \geq 2A^{-1}$.

If $\big\vert \mc{I}(\Gamma_n) \big\vert \leq \alpha m$, then the cardinality of the set of $i \in \llbracket 0, mA \rrbracket$ for which $\omega_i \geq 2A^{-1} n^{-2/3}$
is thus seen to be at least $ (1-\alpha)m$. In the language of Theorem~\ref{t.notallcliffs}, the event that $\big\vert \mc{I}(\Gamma_n) \big\vert \leq \alpha m$ entails that $\rho_n$ advances horizontally with $(\beta_1,\beta_2)$-steadiness, for $\beta_1 = 2A^{-1}$, $\beta_2 = \frac{(1-\alpha)}{A}$, and for~$n$ at least a level $n_0$ determined by $\beta_2$.

Thus Theorem~\ref{t.notallcliffs} implies Proposition~\ref{p.notallcliffs}. \qed

\bibliographystyle{plain}

\bibliography{airy}

\end{document}